\documentclass[12pt]{article}

\usepackage{amsmath,amsfonts,amssymb,amsthm,amscd}
\usepackage{latexsym,graphicx}
\usepackage{dynkin-diagrams}
\usepackage{multirow}

\usepackage{titlesec}
\usepackage{csquotes}
\usepackage{xcolor}
\usepackage{authblk}
\usepackage{url}

\titleformat*{\subsection}{\normalsize\bfseries}
\titleformat*{\subsubsection}{\normalsize\itshape}

\newcommand{\C}{\mathbb{C}}

\renewcommand{\P}{\mathbb{P}}

\newcommand{\Z}{\mathbb{Z}}

\newcommand{\fa}{\mathfrak{a}}
\newcommand{\fb}{\mathfrak{b}}

\newcommand{\fC}{\mathfrak{C}}

\newcommand{\fg}{\mathfrak{g}}
\newcommand{\fG}{\mathfrak{G}}

\newcommand{\fk}{\mathfrak{k}}
\newcommand{\fK}{\mathfrak{K}}
\newcommand{\fm}{\mathfrak{m}}
\newcommand{\fM}{\mathfrak{M}}

\newcommand{\fp}{\mathfrak{p}}

\newcommand{\fR}{\mathfrak{R}}
\newcommand{\fs}{\mathfrak{s}}

\newcommand{\fu}{\mathfrak{u}}
\newcommand{\fU}{\mathfrak{U}}

\newcommand{\fz}{\mathfrak{z}}

\newcommand{\fso}{\mathfrak{so}}
\newcommand{\fosp}{\mathfrak{osp}}
\newcommand{\fsl}{\mathfrak{sl}}
\newcommand{\fgl}{\mathfrak{gl}}

\newcommand{\ad}{\mathrm{ad}}
\newcommand{\Ad}{\mathrm{Ad}}

\newcommand{\CK}{\mathrm{CK}}

\newcommand{\Coind}{\mathrm{Coind}}

\newcommand{\Conf}{\mathrm{Conf}}

\newcommand{\Der}{\mathrm{Der\,}}
\newcommand{\diag}{\mathrm{diag}}

\newcommand{\End}{\mathrm{End}}

\newcommand{\Eval}{\mathrm{Eval}}

\newcommand{\Grass}{{\mathrm {Grass}}}

\newcommand{\Hom}{{\mathrm {Hom}}}

\newcommand{\id}{\mathrm{id}}
\newcommand{\Id}{\mathrm{Id}}

\newcommand{\Ind}{\mathrm{Ind}}
\newcommand{\Inn}{\mathrm{Inn}}

\newcommand{\Jor}{\mathrm{Jor}}
\newcommand{\Ker}{\mathrm{Ker}}

\newcommand{\Max}{\mathrm{Max}}

\newcommand{\Mode}{\mathrm{Mode}}

\newcommand{\Pf}{{\mathrm {Pf}}}

\newcommand{\Rad}{\mathrm{Rad}}

\newcommand{\rk}{\mathrm{rk}}
\newcommand{\Skew}{{\mathrm {Skew}}\,}

\newcommand{\Sym}{{\mathrm {Sym}}\,}
\newcommand{\Tens}{{\mathrm {Tens}}}
\newcommand{\Tr}{{\mathrm {Tr}}}

\newcommand{\Sh}{{\mathrm {Sh}}\,}

\newcommand{\Supp}{{\mathrm {Supp}}\,}
\newcommand{\TKK}{\mathrm{TKK}}
\newcommand{\Vir}{\mathrm{Vir}}

\newcommand{\G}{\mathrm{G}}
\newcommand{\K}{\mathrm{K}}
\newcommand{\Kac}{\mathrm{Kac}}
\renewcommand{\L}{\mathrm{L}}

\newcommand{\R}{\mathrm{R}}
\renewcommand{\S}{\mathrm{S}}

\newcommand{\U}{\mathrm{U}}
\newcommand{\W}{\mathrm{W}}

\renewcommand{\div}{\mathrm{div}}
\renewcommand{\d}{\mathrm{d}}
\renewcommand{\mod}{\,\mathrm{mod}\,}
\renewcommand{\v}{\mathrm{v}}

\newtheorem{thm}{Theorem}

\newtheorem*{MainA}{Theorem A}
\newtheorem*{MainB}{Theorem B}
\newtheorem*{MainC}{Theorem C}
\newtheorem*{MainD}{Theorem D}
\newtheorem*{MainE}{Theorem E}
\newtheorem*{MainF}{Theorem F}
\newtheorem*{MainG}{Theorem G}
\newtheorem*{MainH}{Theorem H}

\newtheorem*{thm5.1}{Theorem 5.1}
\newtheorem*{imprecise}{Proposition}
\newtheorem*{definition}{Definition}

\newtheorem*{Dong}{Dong Lemma}
\newtheorem*{CKvdLthm}{Cheng-Kac-van de Leur Theorem}

\newtheorem{lemma}{Lemma}
\newtheorem{prop}[lemma]{Proposition}

\newtheorem{cor}[lemma]{Corollary}
\newtheorem{ax}{Axiom}

\newcommand{\cA}{\mathcal{A}}
\newcommand{\cB}{\mathcal{B}}

\newcommand{\cC}{\mathcal{C}}
\newcommand{\cD}{\mathcal{D}}

\newcommand{\cF}{\mathcal{F}}
\newcommand{\cG}{\mathcal{G}}
\newcommand{\cH}{\mathcal{H}}

\newcommand{\cL}{\mathcal{L}}
\newcommand{\cM}{\mathcal{M}}
\newcommand{\cN}{\mathcal{N}}

\newcommand{\cP}{\mathcal{P}}
\newcommand{\cQ}{\mathcal{Q}}

\numberwithin{lemma}{section}

\makeindex

\begin{document}

\title{Cuspidal modules over superconformal algebras
of rank $\geq 1$}

\author{Consuelo Martínez, Olivier Mathieu and Efim Zelmanov}

\affil{\small Departamento de Matemáticas, Universidad de Oviedo,\\ C/ Calvo Sotelo s/n, Oviedo, 33007, Spain}

\affil{\small Institut Camille Jordan du CNRS, Université Claude Bernard Lyon,\\ 69622 Villeurbanne Cedex, France}

\affil{\small SICM, Southern University of Science and Technology,\\ Shenzhen, 518055, China}

\affil{\small [CM] chelo@pinon.ceu.uniovi.es\\
[OM] mathieu@math.uni-lyon1.fr\\
[EZ] efim.zelmanov@gmail.com}




\date{}

\maketitle

\begin{abstract} According to V. Kac and J. van de Leur, the  superconformal algebras are  the simple
$\Z$-graded Lie superalgebras  of
growth one which contains the Witt algebra. 
We describe an explicit classification of all cuspidal modules over the known supercuspidal algebras of rank
$\geq 1$, and their central extensions. 

Our  approach reveals some unnoticed phenomena.
Indeed the central charge of cuspidal modules is trivial,
except for one specific central extension 
of the contact algebra $\K(4)$. As shown in the paper, this fact also impacts the representation theory
of $\K(3)$, $\CK(6)$ and $\K^{(2)}(4)$. 

Besides these four cases, the classification relies on general methods based on highest weight theory.
\end{abstract}

\tableofcontents

\section{Introduction}\label{Intro}

Throughout the paper all vector spaces are considered over the field $\mathbb{C}$ of complex numbers.

By the centerless Virasoro algebra, also known as the Witt algebra, we mean the Lie algebra 
$\Vir=\Der \C[t,t^{-1}]$ of all derivations of the algebra of Laurent polynomials. 
A search for  superextensions of the Virasoro algebra started in the works of A. Neveu and J. H. Schwarz \cite{NeveuSchwarz}, P. Ramond \cite{Ramond}, followed by M. Ademollo et al. \cite{Ademollo}, K. Schoutens \cite{Shoutens}, A. Schwimmer and N. Seiberg \cite{SchSeiberg}. 
In \cite{KvdL}, V. Kac and J. van de Leur put the program on a more formal footing with the following definition.

\begin{definition}
    A $\Z$-graded Lie superalgebra 
$\L=\sum _{i\in\mathbb{Z}}\L_i=
\sum_{i\in\Z}\L_{\bar{0},i}\oplus \L_{\bar{1},i}$ is called a \textit{superconformal} algebra if 
\begin{enumerate}
 \item[(a)]  $\L$ is simple, 
 \item[(b)]  $\L$ has growth one, i.e. the function $i\mapsto\dim\,\L_i$ is uniformly bounded,
 \item[(c)]  $\Vir\subseteq L_{\bar{0}}$.
    \end{enumerate}
\end{definition}

\noindent Then V. Kac and J. van de Leur \cite{KvdL} conjectured that the only superconformal algebras are 

\begin{enumerate}
\item[(a)]
the Cartan type superalgebras:
$\W(n),\S(n;\gamma )$ and $\K_*(N)$; the contact superalgebras $\K_*(N)$  exist in two forms, the Ramond form $\K(N)$ and the Neveu-Schwarz form 
$\K_{NS}(N)$, which are isomorphic for $N$ even,
\item[(b)] the exceptional superconformal algebra $\CK(6)$,
\item[(c)] the twisted contact superalgebras $\K^{(2)}(2m)$, 
\end{enumerate}

\noindent
see their description in Chapter~\ref{zoology}.
The results in \cite{FattoriKac} and \cite{KMZ} partially confirm this conjecture. The analogous question for Lie algebras is solved, since the $\Z$-graded simple Lie algebras of  growth one were classified in \cite{M86}.

Except explicitly stated otherwise, we assume that all
$\Z$-graded $\L$-mo\-du\-les $M=\oplus_{i\in\Z}\,M_i$ have finite dimensional homogenous components $M_i$. A simple
 $\Z$-graded module is called {\it cuspidal} if its support
 $$\Supp\,M=\{i\in\Z\mid M_i\neq 0\}$$
is neither lower bounded nor upper bounded.

Cuspidal modules over the Virasoro algebra have been classified by the second author \cite{M92} and, independently,
by C. Martin and Piard \cite{MartinPiard}. Y. Cai, D. Liu, R. Lu \cite{CaiLiuLu} and Y. Cai, R. Lu \cite{CaiLu22} classified cuspidal modules over the Ramond and Neveu-Schwarz superalgebras $\K_*(1)$. Y. Cai, R. Lu, Y. Xue \cite{CaiLuXue} and  Y. Billig, V. Futorny, K. Iohara, I. Kashuba \cite{BFIK} obtained a description of cuspidal modules over  $\W(m,n)=\Der \C[t^{\pm 1}_1,\dots , t^{\pm 1}_m , \xi _1,\dots , \xi _n]$. D. Liu, Y. Pei, L. Xia \cite{LiuPeiXia} classified cuspidal modules over the superalgebra $\K(2)$. The authors are aware of the recent preprint \cite{CaiLu25}  of Y. Cai and R. Lü on classification of cuspidal modules over the superalgebras $\K_*(n), n\neq 4$.

In this paper, we present a classification of cuspidal modules over all known superconformal algebras of (semi-simple) rank $\geq 1$, namely 
$$\W(n),\S(n;\gamma ),n\geq 2; \K_*(N), N\geq 3; \K^{(2)}(2m),m\geq 2; \CK(6)$$
 and their
universal central extensions.
 It is different from the description in the series of papers 
\cite{BFIK}\cite{CaiLiuLu} \cite{CaiLu22}\cite{CaiLu25} \cite{CaiLuXue}\cite{LiuPeiXia} which  are complemental.
In addition, it  reveals some unnoticed phenomena.

Our classification is based on the highest weight theory. 
For conciseness, we will only provide
incomplete definitions in the introduction and we refer to the main text for details.

The central extensions of all superconformal algebras
have been determined in \cite{KvdL}. We select one specific class   
$$\psi\in H^2(\K(4))\simeq \C^3$$ 
to form a central extension
$$0\to \C c\to \widehat{\K(4)}\to \K(4)\to 0.$$
The first result of our paper is the following:

\begin{MainA} Let $\L$ be a superconformal
algebra.
\begin{enumerate}
\item[(a)] If $\L\not\simeq\K(4)$, all projective cuspidal $L$-modules have  zero central charge.
\item[(b)] If $\L=\K(4)$, only the specific 
central extension $\widehat{\K(4)}$ admits cuspidal modules with a non-zero central charge.
\end{enumerate}
\end{MainA}

In Chapter~\ref{zoology}, we review the construction of all known superconformal algebras $\L$. We write $\rk\,\L$ for the rank of a maximal
semi-simple subalgebra in $\L_{\bar 0,0}$. It is
also useful to consider a maximal reductive subalgebra
$\fg$ in $\L_{\bar 0,0}\cap C_\L(\Vir)$.
For known superconformal algebras of rank $\geq 1$,
the rank and the reductive subalgebra $\fg$ are described in the following table, where the second line indicates the condition for $\rk\,\L\geq 1$:

\begin{align*}
&\L: \hskip5mm &\W(n)\hskip5mm &\S(n;\gamma)\hskip5mm 
&\K(N)\hskip5mm &\K^{(2)}(2m)\hskip3mm &\widehat{\K(4)}\hskip5mm &\CK(6) \\
&\text{for:}\hskip5mm &n\geq 2\hskip4mm &n\geq2\,\, &N\geq 3\hskip6mm &m\geq 2&&\\
&\fg: &\fgl(n)\hskip5mm  &\fsl(n) 
&\fso(N)\hskip4mm &\fso(2m-1) &\fso(4)\oplus\C\hskip2mm &\hskip2mm\fso(6)\\
&\rk\,\L\hskip5mm &n-1\hskip5mm &n-1 &[N/2]\hskip6mm &m-1 &2\hskip1cm &\hskip5mm3
\end{align*}

\noindent Assume now that $\L$ is an untwisted 
superconformal algebra.
We use a Cartan subalgebra 
$H$ of $\fg$ and a specific element $F\in H$ to
define a triangular decomposition of $\L$
$$\L=\L^+\oplus C_\L(F)\oplus \L^-.$$

The algebra $C_\L(F)$ contains a subalgebra
$\fR$ called the {\it radical} of $C_\L(F)$ and we have
$$C_\L(F)/\fR\simeq \Vir\ltimes H\otimes\C[t,t^{-1}]\hbox{, or }$$
$$C_\L(F)/\fR\simeq \K_*(1)\ltimes H\otimes\C[t,t^{-1},\xi].$$

\noindent In order to use highest weight theory, we first define  the  simple representations 
$\Tens(\lambda,\delta,u)$ of $C_\L(F)/\fR$, where
$(\lambda,\delta,u)\in H^*\times\C\times \C$. 
In
the whole paper we use the Ramond derivation
$D=t\frac{\d}{\d t}$ to simplify the formulas.
In the first case the module $\Tens(\lambda,\delta,u)$ is one copy $\overline{\C[t,t^{-1}]}$ of the vector space $\C[t,t^{-1}]$ with the action of  $\Vir$
$$(fD)\overline{g}=
\overline{fD(g)+\delta D(f)g+ufg}$$
and the action of $H\otimes \C[t,t^{-1}]$
$$(h\otimes f)\overline{g}=\lambda(h)\overline{fg}$$
where $f=f(t)$, $g=g(t)$ are Laurent polynomials and $h$ is an arbitrary element in $H$.

    In the second case, the 
$\K_*(1)\ltimes H\otimes\C[t,t^{-1},\xi]$-modules
$\Tens(\lambda,\delta,u)$ are similarly defined,
see Section \ref{fG}.

Set $B:=C_\L(F)\oplus \L^+$ and 
consider $\Tens(\lambda,\delta,u)$ as a
$B$-module, with a trivial action of
$\fR$ and $\L^+$. Then we define the generalized
Verma module
$$M(\lambda,\delta,u):=
\Ind_B^\L\,\Tens(\lambda,\delta,u)$$
and let $V(\lambda,\delta,u)$ be its  simple quotient. 

We now briefly outline the main steps of the classification for cuspidal modules over
the untwisted superconformal algebras.
The first step is:

\begin{MainB} Any cuspidal $\L$-module is isomorphic to
$V(\lambda,\delta,u)$ for some triple
$(\lambda,\delta,u)$.
\end{MainB}

However in general $V(\lambda,\delta,u)$ is not 
cuspidal because its homogenous components could have
infinite dimension. Thus the classification reduces to the following question:

\centerline{\it For which triples $(\lambda,\delta,u)\in H^*\times\C\times\C$ the $\L$-module is cuspidal?}

\bigskip
Since $[\fg,\fg]$ has type $A$, $B$ or $D$,
we define the simple coroot $h_i$ corresponding
with the  label $i$ of  its Dynkin diagram:

$$\dynkin [backwards,
labels={n-1,2,1},
scale=1.8] A{o...oo}\hskip1cm\dynkin [backwards,
labels={m,2,1},
scale=1.8] B{o...oo}\hskip1cm
\dynkin [backwards,
labels={m,4,3,1,2},
scale=1.8] D{o...oooo}$$

The second step is the following result

\begin{MainC} If the $\L$-module $V(\lambda,\delta,u)$
is cuspidal, then $\lambda$ is dominant and 
$\lambda(h_1)\geq 1$.

Conversely if $\lambda$ is dominant and
$\lambda(h_1)\geq 2$, then
$V(\lambda,\delta,u)$
is cuspidal.
\end{MainC}

Therefore the most interesting part of the classification concerns the dominant  weights
$\lambda$ with $\lambda(h_1)=1$. For 
$\L\neq S(n;\gamma) \hbox{ or }
\widehat{\K(4)}$ we describe the weights
as $m$-uples $(\lambda_1,\lambda_2\cdots,\lambda_m)$.
 For $\L=\widehat{\K(4)}$, weights are triples $(\lambda_1,\lambda_2,\lambda_c)$.
In this setting the values $\lambda(h_i)$ are:
\begin{enumerate}
\item[(a)]  $\lambda(h_i)=\lambda_i-\lambda_{i+1}$ for any $i\leq n-1$  if $\L=\W(n)$,
\item[(b)] $\lambda(h_1)=2\lambda_1$ and 
$\lambda(h_i)=\lambda_{i}-\lambda_{i-1}$ for $2\leq i\leq m$
if $\L=\K(2m+1)$,
\item[(c)]$\lambda(h_1)=\lambda_1+\lambda_2$ and 
$\lambda(h_i)=\lambda_i-\lambda_{i-1}$ for $2\leq i\leq m$
if $\L=\K(2m)$
\item[(d)] $\lambda(h_1)=\lambda_1+\lambda_2$,
$\lambda(h_2)=\lambda_{2}-\lambda_1$ and $\lambda(c)=\lambda_c$ if $\L=\widehat{\K(4)}$,
\item[(d)] $\lambda(h_1)=\lambda_1+\lambda_2$,
$\lambda(h_2)=\lambda_{2}-\lambda_1$ and 
$\lambda(h_3)=\lambda_{3}-\lambda_2$ if $\L=\CK(6)$.
\end{enumerate}

\begin{MainD} Let $\lambda$ be a dominant 
weight with $\lambda(h_1)=1$. The $\L$-module
$V(\lambda,\delta,u)$ is cuspidal in the following cases
\begin{enumerate}
\item[(a)] $\L=\W(n)$ and $\lambda_1=1-\delta$,
\item[(b)] $\L=\S(n;\gamma)$ and $\delta=1$,
\item[(c)] $\L=\K(3)$ and $\delta=\frac{1}{4}$,
\item[(d)] $\L=\widehat{\K(4)}$, $\delta=\frac{\lambda_1}{2}$ and $\lambda_c=2\lambda_2$,
\item[(e)]$\L=\widehat{\K(4)}$, $\delta=\frac{1+\lambda_2}{2}$ and $\lambda_c=-2\lambda_2$,
\item[(f)]$\L=\CK(6)$, $\delta=\frac{\lambda_1}{2}$
and $\lambda(h_3)=0$.
\end{enumerate}

Otherwise the $\L$-module $V(\lambda,\delta,u)$ is not cuspidal.
\end{MainD}

Among these  series  we single out
\begin{enumerate}
\item[(a)] the $\K(3)$-modules
$S(u):=V((\frac{1}{2}),\frac{1}{4},u)$,
of conformal dimension $4$,
\item[(b)] the $\widehat{\K(4)}$-modules 
$$S^+(u):=V(((\frac{1}{2},\frac{1}{2},1),\frac{1}{4},u),\,\,
S^-(u):=V((\frac{1}{2},\frac{1}{2},-1),\frac{3}{4},u)$$ 
of conformal dimension $4$, and
\item[(c)]the $\CK(6)$-modules $T(u):=
V((\frac{1}{2},\frac{1}{2},\frac{1}{2}),\frac{1}{4},u)$
of conformal dimension $8$.
\end{enumerate}
Indeed $\CK(6)$ and $\widehat{\K(4)}$  share the same Cartan algebra $H$,  with  different bases. The  basis 
$\{h_1,h_2,h_3\}$ is replaced by $\{h_1,h_2,c\}$,
where $c=h_1+h_2+2h_3$. Hence the higest weights
of $T(u)$ and $S^+(u)$ are identical. 

 These modules are attracting because they are very small, they involve spin representations and interesting degrees. 
Moreover for  these three series  the methods of Chapter \ref{LR}  do not apply. Therefore, both by necessity and by interest, we  look at these modules in great details.

We obtain

\begin{MainE} There are embeddings
$$\K(3)\subset \widehat{\K(4)}\subset\CK(6).$$
As a $\widehat{\K(4)}$-module, 
$$T(u)=S^+(u)\oplus S^-(u).$$
As $\K(3)$-modules, 
$$S^\pm(u)\simeq S(u)\hbox{ up to parity change}.$$
\end{MainE}

The final result concerns the classification
of cuspidal modules over the
twisted superconformal algebras $\K^{(2)}(2m)$.
In order to state our result, we need a preliminary result.
 
 The involution of the Dynkin diagram $D_m$ induces an involution of $\K(2m)$. By definition
$$\K^{(2)}(2m)=\K(2m)^\sigma.$$
A cuspidal $\K(2m)$-module $V$ is called of {\it first kind} if
$$\sigma_* V\not\simeq V,$$
otherwise it is called of {\it second kind}.  
We similarly define the $\widehat{\K(4)}$-modules of
first or second kind.

\begin{MainF} 
\begin{enumerate} Let
$(\lambda,\delta,u)\in H^*\times\C\times\C$ be an arbitrary triple.
\item[(a)] If $m\geq 3$, the $\K(2m)$-module $V(\lambda,\delta,u)$ is of second type if and only if 
$$\lambda_1=1.$$
\item[(b)] The $\widehat{\K(4)}$
-module $V(\lambda,\delta,u)$ is 
 of second type if and only if 
$$\lambda_1=1\text{ and }(\lambda_c,\delta)\neq(\pm2\lambda_2, \frac{1}{2}).$$

\end{enumerate}
\end{MainF}

The involution $\sigma$ acts on any cuspidal module $V$ of second type and
$$\sigma(g){\sigma}(v)= {\sigma}(gv)$$
for any $g\in \K(2m)$ and  $v\in V$. In fact this action is defined up to a sign.
The subspace
$V^\sigma$ of fixed points is obviously
a $\K^{(2)}(2m)$-submodule.

\begin{MainG} 
Let $m\geq 3$. A cuspidal $\K^{(2)}(2m)$-module is
\begin{enumerate}
\item[(a)] either isomorphic to a $\K(2m)$-module of first kind, or

\item[(b)]  to the component $V^\sigma$ of a
$\K(2m)$-module $V$ of second kind.
\end{enumerate}
\end{MainG}

It should be noted that $\K^{(2)}(4)$ has a trivial center. However the next theorem shows that its cuspidal modules inherit the central extension of $\K(4)$.

\begin{MainH} 
 A cuspidal $\K^{(2)}(4)$-module is
\begin{enumerate}
\item[(a)] 
either isomorphic to a $\widehat{\K(4)}$-module of first kind, or 
\item[(b)] it is isomorphic to $V^\sigma$ for some
$\widehat{\K(4)}$-module $V$ of second kind.
\end{enumerate}
\end{MainH}

\bigskip\noindent
{\it Plan of the paper.} Since the paper is quite long,
we briefly describe its content.
Chapter \ref{zoology} complements the introduction
with precise definitions of the known
superconformal algebras $\L$. For untwisted $\L$ the
triangular decomposition 
$$\L=\L^+\oplus C_\L(F)\oplus \L^-$$
is defined. The description of $\widehat\K(4)$ is postponed
to Chapter \ref{projcusp}.

Then Part I, namely 
Chapters \ref{uniformbound}-8, contains general results for the highest weight theory of superconformal algebras
$\L$. In Chapter \ref{uniformbound} we show that all cuspidal modules have
growth one. 

In Chapter \ref{LR}, we investigate coinduction functors to define some $\L$-mo\-du\-les of growth one. Unlike \cite{M2000} we do not use the more concrete approach based on abelian and nonabelian cocycles to define these representations. 
In our setting  the cocycle formulas are quite involved, especially for the Neveu-Schwarz algebras,
thus we have adopted a more formal approach. The example of Section \ref{exceptionalcocycle} makes visible the complexity of a 
merely abelian cocycle.
The result of Chapter \ref{LR} implies that the modules
$V(\lambda,\delta,u)$ are cuspidal whenever $\lambda$ is dominant and $\lambda(h_1)\geq 2$.

The subalgebra $C_\L(F)$ of an untwisted superconformal algebra $\L$ contains a semi-direct product
$\Vir\ltimes H\otimes\C[t,t^{-1}]$ or 
$K_*(1)\ltimes H\otimes\C[t,t^{-1},\xi]$. We determine 
their simple representations of growth one in
Chapters \ref{split} and \ref{splitII}.

In Chapter \ref{projcusp} we show that only one specific central extension $\widehat{\K(4)}$ of $\K(4)$ admits
  cuspidal modules with a non-trivial central charge.
This chapter also describes the structure
of $\widehat{\K(4)}$ whose Cartan subalgebra $H$  has dimension three.
 
In the final Chapter \ref{HW} of Part 1, we explain the highest theory
for all untwisted superconformal algebras, including 
$\widehat{\K(4)}$. 

The second part involves a case-by-case classification 
of cuspidal modules over superconformal algebras.
The cases of superconformal algebras $\W(n)$ and
$\S(n;\gamma)$ are quite easy. However, tedious computations are required for $\widehat{\K(4)}$. Some formal multidistributions are introduced in Chapter \ref{distribution}
to alleviate those computations. Fortunately they  are also used to determine the cuspidal 
representations for all other cases,
including $\CK(6)$ and the twisted $\K^{(2)}(2m)$.

The third part investigates 
the twisted superconformal algebras. 
In Chapter \ref{chlocal} we prove that, under minimal hypotheses, some
distributions are mutually local. The aim of Chapter \ref{ChClassK2} is the classification of 
a cuspidal $\K^{(2)}(2m)$-modules $V$. We use the locality property to define another cuspidal $\K^{(2)}(2m)$-module $\tilde{V}$ such that $V\oplus \tilde{V}$ carries
a structure of $(\sigma,\L)$ module, where
$\L=\widehat{\K(4)}$ if $m=2$ and
$\L=\K(2m)$ otherwise. Then the classification is deduced from the results of Chapter \ref{twokinds} about modules of first and second kinds.

\section{Definition of all  known superconformal\\ algebras}
\label{zoology}

In this chapter we describe all known superconformal
algebras $\L$ and their $\Z$-grading. 

We define a commutative subalgebra $H\subset \L_0$ that we call the {\it Cartan subalgebra}. The Cartan  subalgebra acts diagonally on $\L$ and commutes with the subalgebra $\Vir$.
We also choose an element $F\in H$ to define a
triangular decomposition
$$\L=\L^+\oplus C_\L(F)\oplus \L^-.$$

All known superconformal algebras $\L$ are 
 superalgebras of {\it vector fields on
$(\C^*)^{1,n}$}  that is  
$$\L\subset \Der \C[t,t^{-1},\xi_1,\cdots,\xi_n]$$ 
for some $n\geq 1$, where
$\C[t,t^{-1},\xi_1,\cdots,\xi_n]$ is the polynomial algebra on the Laurent variable $t$ and the odd variables $\xi_1,\cdots,\xi_n$.

Recall that $\Vir=\Der \C[t,t^{-1}]$ has basis
$E_n=-t^{n+1}\frac{\d}{\d t}$, $n\in \Z$.
We  denote by
$$D=t\frac{\d}{\d t}=-E_0$$  
the {\it Ramond derivation} and we use it systematically 
to simplify formulas. 

By definition, $\L$ contains a copy of
$\Vir$, though usually $\Vir$ can be embedded in more than one way.  The natural embeddings
of superconformal algebras are usually compatible with their subalgebras $\Vir$. However it is not the case for
the subalgebra  $\S(n;\gamma)$ of $\W(n)$,

\subsection{Convention concerning the $\Z$-gradings}
\label{conventionZgrading}

It will be convenient to use the following definition of
a $\Z$-graded space. 
Let $\tau\in\{0,1/2\}$.
A superspace $M=M_{\overline 0}\oplus M_{\overline 1}$ 
endowed with decompositions
$$M_{\overline 0}=\oplus_{i\in \Z}\, M_{{\overline 0},i}$$
$$M_{\overline 1}=\oplus_{i\in \Z}\, M_{{\overline 0},i+\tau}$$ will be called a {\it $\Z$-graded} vector space
and $\tau$ is called the {\it shift} of the $\Z$-grading.
In what follows, when we consider a $\Z$-graded 
superconformal algebra $\L$ and a $\Z$-graded $\L$-module,
it should be understand that $\tau$ has been fixed.

For  all known superconformal algebra $\L\not\simeq \S(n;\gamma)$  there is an
element $\ell_0\in \Vir_0$ such that the decomposition
$$\L=\oplus_{n\in\Z}\, \L_n$$
is the eigenspace decomposition of the operator
$\ad(\ell_0)$. In such case, $\ell_0$ is called the {\it grading element} of $\L$.

According to their grading and their shift, we define four type of superconformal algebras:
\begin{enumerate}
\item[(a)] The {\it Ramond type superconformal algebras} admits  $\ell_0=-E_0$ as a grading element  and  shift $\tau=0$.
Therefore they are $\Z$-graded in the usual sense, and the subalgebra $\Vir$ is endowed with its standard $\Z$-grading.
The Ramond type superconformal algebras are
$\W(n)$, $\K(N)$ or $\CK(6)$. 

\item[(b)] The {\it Neveu-Schwarz type  superconformal algebras} admits  $\ell_0=-E_0$ as a grading element  and  shift  
$\tau=1/2$. The subalgebra 
$\Vir$ is still endowed with its standard $\Z$-grading
but $\L_{\overline{1}}$ is $(1/2+\Z)$-graded.
The superalgebras $\K_{NS}(N)$ are Neveu-Schwarz type  superconformal algebras.

\item[(c)] The {\it twisted superconformal algebras}
admits  $\ell_0=-2E_0$ as a grading element  and  
shift  $\tau=0$. The subalgebra 
$\Vir$ is now  endowed with an upscaled $\Z$-grading by a factor of two, that is the subalgebra
$\Vir$ is $2\Z$-graded. The twisted contact superalgebras
$\K^{\bf (2)}(2m)$ belong to this series.

\item[(d)] The superconformal algebras $\S(n;\gamma)$ are defined as a $\Z$-graded subalgebra of $\W(n)$.
This series differs from the other ones, since for 
$\gamma\neq 0$ the grading is not inner.
\end{enumerate}

With these definitions
\begin{enumerate}
\item[(a)] $\K(N)$ and $\K_{NS}(N)$ share the same even part,
\item[(b)]  $\K^{\bf (2)}(2m)$ is viewed as a $\Z$-graded subalgebra of $\K(2m)$.
\end{enumerate}

\subsection{Structure of $\C[t,t^{-1}]$-module}\label{RNS}
We will see that, except $\K(4)$, all untwisted known superconformal algebras $\L$ have a structure of $\Z$-graded $\C[t,t^{-1}]$-module, relative to which the map $\ad(x)$ is a differential operator of order $\leq 1$, for any $x\in \L$. The extended 
superalgebra $\K(4,D)$ defined in Section \ref{defK}
and  the specific central extension $\widehat{\K(4)}$ of $\K(4)$, defined in Chapter \ref{projcusp} have the same property.

Thus for a Ramond type superconformal algebra
$\L$, we have 
$$\L=\L_0\otimes\C[t,t^{-1}].$$
For $f\in C[t,t^{-1}]$ and $x\in \L_0$, we write
$$x(f):=x\otimes f.$$
Since any differential operator on $\C[t,t^{-1}]$
extends to $\C[t^{\frac{1}{2}},t^{-\frac{1}{2}}]$
we can define the {\it Neveu-Shwarz form}   
$\L_{NS}$ of the Ramond type superconformal algebra $\L$ as
$$(\L_{NS})_{\bar 0}=\L_{\bar 0}
\text{ and } (\L_{NS})_{\bar 1}=
t^{\frac{1}{2}}\L_{\bar 1}.$$

In is proved in \cite{KvdL} that
$\K_{NS}(2m)\simeq \K(2m)$. Simarly it is easy to see that for $\L=\W(n)$ or $\CK(6)$,
their Neveu-Schwarz forms 
are isomorphic to the initial algebra $\L$.
The proof relies on the existence of an element
$h\in[\L_0,\L_0]$ such that the spectrum of
$\ad(h)\vert_{\L_{\bar 0}}$ consists of integers
and the spectrum 
$\ad(h)\vert_{\L_{\bar 1}}$ consists of half integers.

By contrast, there is no  such an element in $\K(2m+1)$ and
$$\K_{NS}(2m+1)\not\simeq \K(2m+1).$$

\subsection{The superconformal algebra $\W(n)$}\label{defW}

Recall that $D$ denotes the Ramond derivation
$D:=\frac{t\d}{\d t}$.
The Lie superalgebra 
$$\W(n):=\Der \C[t,t^{-1},\xi_1,\cdots,\xi_n]=\{ f(t,\xi)D + \sum _{i=1}^n f_i (t,\xi)\frac{\partial }{\partial \xi _i} \}$$
contains a copy of $\Vir$
$$\Vir=\{f(t) D\mid f(t)\in\C[t,t^{-1}]\}.$$
It is a superconformal algebra with grading element 
$\ell_0=D=-E_0$. Its degree zero component $\W(n)_0$ contains
the Lie algebras
$$\fgl(n)=\{\sum_{1\leq i,j\leq n} a_{i,j}\,\xi_i
\frac{\partial }{\partial \xi _j} \mid a_{i,j}\in \C
\}\text{ and }$$
$$\fsl(n)=\{\sum_{1\leq i,j\leq n} a_{i,j}\,\xi_i
\frac{\partial }{\partial \xi _j} \mid \sum_{i=1}^n a_{i,i}=0\}.$$ 

\noindent whose Cartan subalgebras are respectively
$$H:=\{ \sum _{i=1}^{n} a_i\xi _i \frac{\partial }{\partial \xi _i} \mid a _i\in\C \}\text{ and }
H^0=\{\sum _{i=1}^{n} a_i\xi _i \frac{\partial }{\partial \xi _i} \mid \sum_{i=1}^n a_i=0\}.$$ 
A weight $\lambda\in H^*$ can be written as a $n$-uple $(\lambda_1,\cdots,\lambda_n)$ where
$$\lambda_i=\lambda(\xi _i \frac{\partial }{\partial \xi _i}).$$

The $\fsl(2)$-triples
$$e_i:=\xi _i\frac{\partial }{\partial \xi _{i+1}},
\,h_i:=\xi _{i+1}\frac{\partial }{\partial \xi _{i+1}}-\xi _i\frac{\partial }{\partial \xi _i}
\hbox{ and } f_i:=\xi _i\frac{\partial }{\partial \xi _{i+1}},$$ 

\noindent where $1\leq i\leq i-1$, form the usual
Chevalley generators of $\fsl(n)$. 
The action of $H$ on $\L$ provides an eigenspace
decomposition

$$\L=C_\L(H)+\sum _{\alpha\in \Delta}\,\L^{(\alpha)},$$

\noindent where $\Delta:=\{\alpha\in H^*\mid \alpha\neq 0 \hbox{ and } \L^{(\alpha)}\neq 0\}$ is called the {\it set of roots}.
For a root $\alpha$ of the form
$\alpha=\epsilon_{i_1}+\cdots+\epsilon_{i_r}$ with 
$1\leq i_1<\cdots<i_r<n$ with $r\neq 0$, we have
$$\L^{(\alpha)}=\C[t,t^{-1}]\xi_{i_1}\cdots\xi_{i,r} D
\oplus H\otimes \C[t,t^{-1}]\xi_{i_1}\cdots\xi_{i,r}.$$
For a root $\alpha$ of the form
$\alpha=\epsilon_{i_1}+\cdots+\epsilon_{i_r}-\epsilon_k$ with 
$1\leq i_1<\cdots<i_r<n$ and $k\notin\{i_1,\cdots,i_r\}$,
 we have
$$\L^{(\alpha)}=\C[t,t^{-1}]\xi_{i_1}\cdots\xi_{i,r}
\frac{\partial}{\partial\xi_k}.$$

Let consider the element $F\in H^{0}$ defined by
$$\epsilon_1(F)=2^n-1\hbox { and }\epsilon_k(F)=-2^{k-1}
\hbox{ for } k>1,$$
and set
$$\Delta^+=\{\alpha\in\Delta\mid \alpha(F)>0\},$$
$$\Delta^-=\{\alpha\in\Delta\mid \alpha(F)<0\},\hbox{ and }$$
$$\,\Delta^0=\{\alpha\in\Delta\mid \alpha(F)=0\}
=\{\epsilon_1+\cdots+\epsilon_n\}.$$

There is a decomposition triangular decomposition
$$\L=\L^+\oplus C_\L(F)\oplus \L^{-},$$

\noindent where $\L^+=\oplus_{\alpha\in\Delta^+} \L^{(\alpha)}$ and $\L^-=\oplus_{\alpha\in\Delta^-} \L^{(\alpha)}$.

We observe
\begin{enumerate}
\item[(a)] The subalgebra $\C[t,t^{-1}]\xi_1\cdots\xi_n D$,
which lies in $C_\L(F)$ is called the {\it radical of 
$C_\L(F)$} and its is denoted as $\Rad(C_\L(F))$. We have

$$C_\L(F)=(\Vir\ltimes H\otimes \C[t,t^{-1}])\ltimes
\Rad(C_\L(F)).$$

\item[(b)]  Set
$$\Delta^{+}(\fsl(n))=\Delta(\fsl(n)\cap \Delta^+,$$
where  $\Delta(\fsl(n)$ is the root system of the 
subalgebra $\fsl(n)\subset \W(n)$.
Then $\Delta^{+}(\fsl(n))$ consists of roots
$\epsilon_i-\epsilon_j$ with $i<j$ and the simple roots
$\Delta^{+}(\fsl(n))$ are 
$\alpha_i:=\epsilon_i-\epsilon_{i+1}$.
Thus $\Delta^{+}(\fsl(n))$ is the usual set of positive roots of $\fsl(n)$.
\end{enumerate}

\subsection{The superconformal algebra $\S(n;\gamma)$,   $n\geq 2$}\label{defS}

We now describe the superconformal algebras $\L:=S(n;\gamma )$ which were introduced in  \cite{Ademollo} and \cite{CantariniKac} under the name 
``$\operatorname{SU}_2$-superconformal algebras''.

The {\it divergence} of a derivation
$\partial=f D+\sum_{i=1}^n\frac{\partial}{\partial\xi_i}
\in 
\W(n)_{\overline 0}\cup \W(n)_{\overline 1}$ is
$$\div(\partial)=D(f)+
\sum_{i=1}^n (-1)^{|f_i|}\frac{\partial f_i}
{\partial\xi _i}.$$
With our convention, the Ramond derivation $D$ is divergence-free.

Let $\gamma\in\C$. The space

$$\{ \partial\in W(n)\mid  
\div(t^{-\gamma}\partial)=0 \} $$
is a subalgebra of $W(n)$. When $\gamma\notin\Z$,
this algebra is denoted as $\S(n;\gamma)$. Otherwise, we denote as $\S(n;\gamma)$ the derived algebra
$$\S(n;\gamma):=[\{ \partial\in W(n)\mid  
\div(t^{-\gamma}\partial)=0 \},\{ \partial\in W(n)\mid  
\div(t^{-\gamma}\partial)=0 \}].$$
In the case $\gamma\in\Z$  we have
$$\{ \partial\in W(n)\mid  
\div(t^{-\gamma}\partial)=0 \}=
\S(n;\gamma )\oplus\C t^\gamma\xi_1\cdots\xi_n D.$$
The superalgebra $\S(n;\gamma )$ is simple 
for all $\gamma\in\C$, see \cite{KvdL}. 

For an arbitrary $k\in\mathbb{Z}$, the elements 
$$E_k = -\left(t^k D + (k+\gamma)t^k\xi _1\frac{\partial }{\partial \xi _1} \right)$$ 
lie in $\S(n;\gamma)$ and form a basis of 
a Lie subalgebra isomorphic to $\Vir$. 

Since
$\S(n;\gamma)$ is a $\Z$-graded subalgebra of $\W(n)$,
it is a $\Z$-graded superconformal algebra. 
Except when $\gamma=0$, the grading of $\S(n;\gamma)$ is not inner. The superconformal algebras $\S(n,\gamma ')$ and $\S(n,\gamma '')$ are isomorphic if and only if $\gamma ' -\gamma ''\in \mathbb{Z}$. However their gradings differ by an outer derivation, so the relation between their $\Z$-graded representations is  evolved.

The Lie algebra 
$$\fsl(n):=\{\partial=\sum_{1\leq i,j\leq n}
a_{i,j} \xi _i\frac{\partial }{\partial \xi _j}
\mid \sum_{i=1}^n\,a_{i,i}=0\}$$ 
and its Cartan Lie subalgebra 
$$H^{(0)}:=\{\partial=\sum_{1\leq i\leq n}
a_{i} \xi _i\frac{\partial }{\partial \xi _i}
\mid \sum_{i=1}^n\,a_{i}=0\}$$ 
lie in $\S(n;\gamma)$.

The action of $H^{(0)}$ on $\L$ provides an eigenspace
decomposition

$$\L=C_\L(H^{(0)})+\sum _{\alpha\in \Delta}\,\L^{(\alpha)},$$

\noindent where $\Delta:=\{\alpha\in (H^{(0)})^*\mid \alpha\neq 0 \hbox{ and } \L^{(\alpha)}\neq 0\}$ is called the {\it set of roots}.

Using the same  $F\in H^{(0)}$ as in the previous section,
we define  
$$\Delta^+=\{\alpha\in\Delta\mid \alpha(F)>0\}
\hbox{ and }\Delta^-=
\{\alpha\in\Delta\mid \alpha(F)<0\}.$$

There is a decomposition triangular decomposition

$$\L=\L^+\oplus C_\L(F)\oplus \L^{-},$$

\noindent where $\L^+=\oplus_{\alpha\in\Delta^+} \L^{(\alpha)}$ and $\L^-=\oplus_{\alpha\in\Delta^-} \L^{(\alpha)}$. 

We observe that
$$C_\L(F)=
\Vir\ltimes H^{(0)}\otimes \C[t,t^{-1}].$$

\subsection{The contact bracket algebras}
\label{defK}

\subsubsection{Definition of the Lie superalgebras $\K(N;D)$}
Let $A=A_{\overline 0}\oplus A_{\overline 1}$ be an associative commutative superalgebra. It is called a {\it bracket algebra} if $A$ is additionally
endowed with a $\Z/2\Z$-graded
product 
$$[-,-]:A\times A\to A,$$ 
satisfying
 
\begin{enumerate}
\item[(a)] 
$[a,b]=(-1)^{\vert a\vert\vert b\vert} [b,a]$
\item[(b)] 
$[ab,c]=a[b,c]+(-1)^{\vert a\vert \vert b\vert}
b[a,c]-ab[1,c]$
\end{enumerate}
for abitrary elements 
$a,b, c\in A_{\overline 0}\cup A_{\overline 1}$.

It follows from the definition that the map
 $\partial:c\mapsto [1,c]$ is an even derivation.
 The bracket $[-,-]$ is called a {\it contact bracket}
 if $(A,[-,-])$ is a Lie superalgebra. A contact bracket
 with $\partial=0$ is called a {\it Poisson bracket}.
 
We will use a more intuitive definition, see \cite{MSZ}. Let $\cP$ be a Poisson
superalgebra and let $A$ be a commutative subalgebra.
A {\it contact superalgebra}
is an free $A$-module of rank one
$P\subset\cP$, generated by an even element $D$, such that 
$$[P,P]\subset P\text{ and }
[P,A]\subset A.$$
Accordingly we will define the contact 
superconformal algebras $\K(N;D)$ as 
Lie subalgebras of the Poisson superalgebra
$\Grass(N)\otimes\cH$ defined below.

The Grassmann superalgebra
$$\Grass(N):=\langle 1,\xi_1,\cdots,\xi_N\mid
\xi_i\xi_j+\xi_j\xi_i\rangle$$
is the free commutative superalgebra over the odd variables $\xi_1,\cdots,\xi_n$.
Its Poisson structure is defined 
on the generators by
$$[\xi_i,\xi_j]=\delta_{i,j}.$$

The  Grassmann superalgebra $\Grass(N)$ has a 
natural grading
$$\Grass(N)=\oplus_{k=0}^N\,\Grass_k(N)$$
relative to which all generators $\xi_i$ are homogenous of degree
$1$. This grading is not compatible with the Poisson structure. Indeed for
$a\in\Grass_k(N)$ and $b\in \Grass_l(N)$ we have
\begin{equation}\label{2stepsG}
ab\in \Grass_{k+l}(N)\hbox{ and } [a,b]\in \Grass_{k+l-2}(N).
\end{equation}

We also consider the Poisson algebra ${\mathcal H}$ symbols of pseudo-differential operators \cite{IoharaMathieu}. Its basis  consists of symbols
$$[t^nD^k]\text{ for } n\in\Z \text{ and } 
k\in\frac{1}{2}\Z,$$
where $D:=t\frac{\d}{\d t}$ is the Ramond derivation. The commutative product
$\cdot$ and the Poisson bracket $\{-,-\}$
of  symbols 
are defined by

$$[t^n D^k]\cdot[t^mD^l]=[t^{n+m}D^{k+l}]
$$
$$\left\lbrace[t^nD^k],[t^mD^l]\right\rbrace
=(km-ln)\left[t^{n+m-1} D^{n+m-1}\right].$$
In what follows we simply denote the symbol $[t^nD^k]$ as $t^nD^k$ and we will use the notation $[-,-]$ for the Poisson bracket.

The Poisson algebra $\cH$ admits a decomposition
$$\cH=\oplus_{\delta\in\frac{1}{2}\Z}\,\,\Omega^\delta,$$
where $\Omega^\delta:=\C[t,t^{-1}] D^{-k}$. This grading is not compatible with the Poisson structure. Indeed we have
$$\Omega^\delta\cdot\Omega^{\delta'}\subset \Omega^{\delta+\delta'}\text{ and }
[\Omega^\delta,\Omega^{\delta'}]\subset\Omega^{\delta+\delta'+1}.$$

It follows that
$$\K(N;D):=\oplus_{k=0}^N\, 
\Grass_k\otimes \Omega^{\frac{k-2}{2}}$$
is a Lie subalgebra. 

There is an obvious decomposition
$$\K(N;D)=\oplus_{k=0}^n\,\K_k(N;D)$$
where $\K_k(N;D)=\Grass_k(N)\otimes \C[t,t^{-1}] D^{\frac{2-k}{2}}$.
Equation (\ref{2stepsG}) implies
\begin{equation}\label{2stepsK}
[\K_k(N;D), \K_l(N;D)]
\subset \K_{k+l}(N;D)\oplus \K_{k+l-2}(N;D)
\end{equation}

\subsubsection{Definition of the superconformal algebras $\K(N)$ and $\K_{NS}(N)$}
We observe that the subspace $\Omega^{-1}$ of $\cP$ is a Lie algebra isomorphic to $\Vir$ and the
tensor $\Omega^\delta$ is a $\Vir$-module. As a $\Vir$-module, it is clear that
$\Omega^0$ is the space of functions $\C[t,t^{-1}]$, $\Omega^1$ is the space of $1$-forms, 
$\Omega^{\frac{1}{2}}$ the space of half-densities and so on. Therefore $\K(N;D)$ contains
$\Vir$ as a subalgebra and 
its grading element is
$\ell_0=D$, that is $\K(N;D)$ is $\Z$-graded by the degree in $t$.

For any $N\neq 4$, the superalgebra $\K(N;D)$
is simple. Therefore 
$$\K(N):=\K(N;D)\text{ for } N\neq 4$$ 
is a superconformal algebra.

In $\K(4;D)$, the top component
$\K_4(4;D)$ is isomorphic to 
$\xi_1\xi_2\xi_3\xi_4\otimes \Omega^1$.
Since $\Omega^1/(\Vir\cdot\Omega^1)=\C$, it follows that
$$\K(4;D)=\C\xi_1\xi_2\xi_3\xi_4\oplus 
[\K(4;D),\K(4;D)].$$
The Lie superalgebra $[\K(4;D),\K(4;D)]$ is simple
and therefore
$$\K(4):=[\K(4;D),\K(4;D)]$$
is a superconformal algebra.

The superalgebra $\K(N)$ 
are called the {\it Ramond contact
superconformal algebra}. As in Section \ref{RNS},
the {\it Neveu-Schwarz form} $\K_{NS}(N)$ is defined by
$$(\K_{NS}(N))_{\bar 0}=\K(N)\text{ and }
(\K_{NS}(N))_{\bar 1}=t^{\frac{1}{2}}\K(N)_{\bar 1}.$$
We refer to \cite {KvdL} for the isomorphism
$\K_{NS}(N)\simeq \K(N)$
when $N$ is even.

\subsubsection{Explicit formulas.}
We now provide explicit formulas for the bracket
of $\K(N;D)$ and the embedding $\nabla:\K(N;D)\to
\W(N)$.

For $a\in \Grass_k(N)$ and 
$f(t)\in \C[t,t^{-1}]$ we write
$$a(f):=a\otimes f(t)\otimes D^{\frac{k-2}{2}}
\in \K(N).$$
When $k$ is even, $a(f)$ is also an element of
$\K_{NS}(N)$. For $k$ odd and
$f(t)\in t^{\frac{1}{2}}\C[t,t^{-1}]$ the element
$a(f)\in \K_{NS}(N)$ is similarly defined.
E.g. $\xi_i\notin \K_{NS}(N)$ but the expression
$\xi_i(t^{\frac{1}{2}})$ is a well-defined element in
$\K_{NS}(N)$.


Let $f(t)$, $g(t)$ be arbitrary elements in 
$\C[t,t^{-1}]$, let $a\in\Grass_k(N)$ and $b\in \Grass_l(N)$ be arbitrary
homogenous elements in $\Grass(N)$. Then 
 
\begin{equation}\label{Kbracket}
[a(f),b(g)]= ab\left({\frac{2-k}{2}}fD(g)-{\frac{2-l}{2}}gD(f)
\right)+ [a,b](fg).
\end{equation}
For $\K_{NS}(N)$ the formula is the same, except that
it involves some half power of $t$ for odd elements.

It is clear that $\K(N;D)$ is a free module of rank one over the subalgebra of $\Grass(N)
\otimes{\mathcal H}$
$$\C[t,t^{-1},\xi_1 D^{1/2},\cdots,
\xi_N D^{1/2}]\subset 
\Grass(N)\otimes{\mathcal H}.$$
Since this algebra is isomorphic to
$\C[t,t^{-1},\xi_1,\cdots,\xi_n]$, there is an embedding
$$\nabla: \C[t,t^{-1},\xi_1,\cdots,\xi_n]\simeq K(N;D)\subset W(N).$$
 Let 
$$E:=\sum_{i=1}^N \xi_i\ \frac{\partial}{\partial\xi_i}$$
 be the Euler derivation.
As in \cite{KvdL} we can show that
\begin{equation}\label{contactformula1}
\nabla(F)=\left(F-1/2 E(F)\right) D
+\frac{1}{2} D(F) E -
(-1)^{\vert F\vert}\sum_{i=1}^N
 \frac{\partial}{\partial\xi_i}
F  \frac{\partial}{\partial\xi_i},
\end{equation}
\noindent for any $\Z/2\Z$-homogenous function 
$F\in K(N;D)\simeq 
\C[t,t^{-1},\xi_1,\cdots,\xi_N]$. The small differences with  \cite{KvdL} 
is  due to  different normalizations.

It follows that the subalgebra $\Vir$ has basis
$$\{E_n=-(t^nD +\frac{n}{2} t^n E)\mid n\in\Z\}.$$

Theses formulas  extend to the Neveu-Schwarz forms. 
Set 
$$\C_{NS}[t,t^{-1},\xi_1,\cdots,\xi_N]:
=\C_{NS}[t,t^{-1},t^{\frac{1}{2}}\xi_1,\cdots,t^{\frac{1}{2}}\xi_N].$$
Then the same formula provides an embedding
$$\nabla:\C_{NS}[t,t^{-1},\xi_1,\cdots,\xi_N]\to
W_{NS}(N)=\Der \C_{NS}[t,t^{-1},\xi_1,\cdots,\xi_N].$$

\subsubsection{ The root system of $\K_*(N)$, $N\geq 3$}
In order to describe the root system of $\K_*(N)$, it is convenient to use the split form of $\Grass(N)$.
For $N=2m$, we consider the new odd variables
$\zeta_1,\eta_1,\cdots,\eta_m$ satisfying
$$[\zeta_i,\eta_i]=1,$$
where all other unwritten brackets are zero.
For $N=2m+1$, we add an additional variable
$\xi$ satisfying
$$[\xi,\xi]=1, [\xi,\eta_i]=[\xi,\zeta_i]=0.$$

The superconformal algebra 
$\L:=\K_*(N)$ contains the subalgebra
$$\Grass_2(N)\otimes \Omega^0\simeq 
\fso(N)\otimes\C[t,t^{-1}].$$
Its Cartan subalgebra $H$ has basis
$\{\zeta_i\eta_i\mid 1\leq i\leq m\}$
and let $\{\epsilon_i\mid 1\leq i\leq m\}$
be the dual basis of $H^*$.
The action of $H$ provides an eigenspace decomposition
$$\L=C_\L(H)\oplus \oplus_{\alpha\in\Delta} \,\L^{(\alpha)},$$
where $\Delta=\{\alpha\in H*\mid \alpha\neq 0\text{ and }\L^{(\alpha)}\neq 0\}$ is called the {\it set of roots.} Indeed an  element
$$\alpha=\sum_{i=1}^m  a_i \epsilon_i\in H^*
,\,\,\,\alpha\neq 0$$ is a root 
if and only if all $a_i$ belongs to
$\{-1,0,1\}$.

Set $F=\sum_{i=1}^m 2^i \zeta_i\eta_i.$
The weights $\epsilon_i\in H^*$ satisfy
$\epsilon_i(\zeta_j\eta_j)=\delta_{i,j}$.
We decompose $\Delta=\Delta^+\cup\Delta^-$, where
$$\Delta^{\pm}=\{\alpha\in\Delta\mid \pm\alpha(F)>0\}.$$
There is a triangular decomposition
$$\L=\L^+\oplus C_\L(F)\oplus \L^-,$$
where $\L^{+}=\oplus{\alpha\in\Delta^{+}}
\, \L^{(\alpha)}.$ and $\L^{-}=\oplus{\alpha\in\Delta^{-}}
\, \L^{(\alpha)}.$

We will label the simple roots of $\fso(N)$ as follows:

\begin{enumerate}
\item[(a)] If $N=2m$, $\alpha_1=\epsilon_2+\epsilon_1$ and $\alpha_k=\epsilon_k-\epsilon_{k-1}$
for $k>1$,
\item[(a)] If $N=2m+1$, $\alpha_1=\epsilon_1$ and the other simple roots are unchanged.
\end{enumerate}

Accordingly the Dynkin diagrams are labelled as follows:

$$\dynkin [backwards,
labels={m, m-1,4,3,1,2},
scale=1.8] D{oo...oooo}\hskip2cm
\dynkin [backwards, arrows=false,
labels={m, m-1,3,2,1},
scale=1.8] B{oo...ooo}.$$

As usual, we denote by $h_1,\cdots,h_m$ the corresponding simple coroots. A weight $\lambda\in H^*$ is
written as $(\lambda_1,\cdots,\lambda_m)$, where
$\lambda_i=\lambda(\zeta_i\eta_i)$.

We also observe that
$$C_\L(F)=C_\L(H).$$

\subsubsection{A remark about the contact structures}

The following obvious observation will be used repeatedly:

\begin{lemma}\label{repeat} Let $A$ be a contact superalgebra and
let $\fs$ be a Lie subalgebra. Then the ideal
$\fs A$ generated by $\fs$ is a Lie subalgebra.
\end{lemma}

 \subsection{The superconformal algebra $\CK(6)$}
\label{defCK(6)}
In 1996, S. Cheng \& V. Kac \cite{CK} and, independently, P. Grozman, D. Leites \& I. Shchepochkina \cite{GLS} introduced a new superconformal algebra. Following \cite{CK}, we will denote it as $\CK(6)$. 
We will mostly follow the construction 
of the first and third author \cite{MZ01}. However the
chosen embedding 
$\Vir\subset\CK(6)$ is different from \cite{MZ01}.
The present choice is compatible with the embedding of
$\Vir$ in the subalgebra
$\widehat{\K(4)}\subset\CK(6)$ defined in Section 
\ref{K(4)}.

Let $V$ be a $n$-dimensional vector space
and let ${\bf V}$ be the superspace
$${\bf V}_{\bar 0}=V\text { and } {\bf V}_{\bar 1}=V^*.$$
The elements of $\bf{V}$ are written as pairs
$(v,\lambda)$ where $v\in V$ and $\lambda\in V^*$.
In brief, it means that the $(0,1)$-tensor
$v$  represents the first $n$ coordinates and the
$(1,0)$-tensor $\lambda$ represents the last $n$ coordinates. Accordingly, 
the associative superalgebra $\End({\bf V})$ is  the space of $2n\times 2n$ matrices
$$\begin{pmatrix}
    A & B\\
    C  & D
\end{pmatrix}$$
where $A\in \End(V)$, $B\in \Hom(V^*,V)$,
$C\in \Hom(V,V^*)$ and $D\in \End(V^*)$.

Let $\Skew(V^*,V)$ be the space of 
skew-symmetric maps $s:V^*\to V$. Similarly let
$\Sym(V,V^*)$ be the space of 
symmetric maps $h:V\to V^*$. In terms of tensors
$$\Skew(V^*,V)=\wedge^2 V\hbox{ and }
\Sym(V,V^*)=\Sym^2\,V.$$

The Lie superalgebra $P(n-1)$ is the space
of all linear transformations 
$g\in \End({\bf V})$
which preserves the natural bilinear form on
${\bf V}$. Since the bilinear form is odd,
it is rather unclear if it should be described as a symmetric or a skew-symmetric bilinear form.
Accordingly, the series of superalgebras
$P(n-1)$ is called the {\it strange series}, see \cite{Kac77}.
More explicitly, $P(n-1)$ is the space
of all $2n\times 2n$-matrices 
$$g=\begin{pmatrix}
    A & S\\
    H  & -A^T
\end{pmatrix},$$
where $A\in\fsl(V)$, $S\in \Skew(V^*,V)$,
$H\in\Sym(V,V^*)$ and the symbol $T$ stands for the transposition. 

We now assume that $\dim\, V=4$. 
We write the Pfaffian of an arbitrary $s\in \Skew(V^*,V)$ as $\Pf(s)$. Indeed $\Pf(s)$ is well
defined once  a volume element 
$v\in \wedge^4 V\setminus 0$ is given once and for all. There is a linear map 
$$\phi:\Skew(V^*,V)\to\Skew(V,V^*)$$
such that 
$$\phi(s)\circ s=\Pf(s)\, \Id_{V^*}
\text{ and } s\circ\phi(s) =\Pf(s)\, \Id_V,$$
where $\Id_V$ and $\Id_{V^*}$ denotes the identity operator on $V$ and $V^*$.
The map $\phi$ has been used to describe the universal central extension of $P(3)$ in \cite{Kac77} and it was used in \cite{MZ01} to describe $\CK(6)$.

Let 
$${\bf A}_1=\{\sum_{n\geq 0} f_n(t) D^n\mid f_n(t)\in\C[t,t^{-1}]\text{ and } f_n=0\text{ for } n>>0\}$$ be the
Weyl algebras of differential operators of the algebra $\C[t,t^{-1}]$, where, as usual, $D$ is the Ramond derivation $D:=t\frac{\d}{\d t}$.

We now define 
$$\CK(6)\subset \End({\bf V})\otimes
{\bf A}_1$$ 
as the subspace of all 
$8\times 8$-matrices
$$\begin{pmatrix}
    A & B\\
    C  & D
\end{pmatrix},$$
where the $4\times 4$ matrices $A, B, C$ and $D$ have the following form

$$A=a+\Bigl(fD +\frac{D(f)}{4}\Bigr) \Id_V \text{ and }$$
$$D=-a^T+\Bigl(fD +\frac{3 D(f)}{4}\Bigr) \Id_{V^*},$$
for some $f\in\C[t,t^{-1}]$ and 
$a\in \fsl(V)\otimes \C[t,t^{-1}]$,

$$B\in \Skew(V^*,V)\otimes\C[t,t^{-1}]\text{ and }$$
$$ B=H + \phi(C) D +\frac{D(\phi(C))}{2},$$ 
for some $H\in\Sym(V,V^*)\otimes\C[t,t^{-1}]$. More explicitly if $C$ is written as a finite sum
$$\sum_i\, a_i f_i(t)\mid a_i\in \Skew(V^*,V)
\text{ and } f_i(t)\in \C[t,t^{-1}],$$ 
we write
$$\phi(C) D=\sum_i\,\phi(a_i) f_i(t)D\text{ and }D(\phi(C)):=
\sum_i\,\phi(a_i) D(f_i).$$
The choice of coefficients $\frac{1}{4},\frac{1}{2}$
and $\frac{3}{4}$ in the definition insures that $\CK(6)$ is stable by the Lie bracket.

The Virasoro subalgebra is the subspace
$$\Vir=\{fD\, \Id_{\bf V}+\frac{D(f)}{4} \Id_V 
+\frac{3 D(f)}{4} \Id_{V^*}\mid f\in\C[t,t^{-1}]\}.$$
The superconformal algebra $\L=\CK(6)$ is graded by the degree in $t$ and its grading element is
$\ell_0=D\cdot \Id_{\bf V}$. It is clear that
$$L_{\bar 0}\simeq \Vir\ltimes \fsl(V)\otimes \C[t,t^{-1}],
\text{ and }$$
$$L_0=\widehat{P(3)},$$
that is the universal central extension of $P(3)$.
For other superconformal
algebras, we followed the usual convention that
latin letters denote even variables and greek letters
the odd variables. However 
for $P(3)$ and its superconformal extension $\CK(6)$ it looks difficult to have notations that keep track of the parity.
Let $x_1,\cdots,x_4$ be a basis of $V$ and let
$\frac{\partial}{\partial x_1},\cdots, \frac{\partial}{\partial x_4}$ be the dual basis. We also identify
$$\fsl(V)\simeq\{\sum_{1\leq i,j\leq 4}\, 
a_{i,j}\, x_i \frac{\partial}{\partial x_j}
\mid \sum_{i=1}^4\,a_{i,i}=0\}.$$

The Cartan subalgebra $H\subset\fsl(V)$ has basis
$h_1, h_2, h_3$ where
$$h_1=x_1\frac{\partial}{\partial x_1}-x_2
\frac{\partial}{\partial x_2},\,
h_2=x_3\frac{\partial}{\partial x_3}-x_4
\frac{\partial}{\partial x_4},
h_3=x_2\frac{\partial}{\partial x_2}-x_3
\frac{\partial}{\partial x_3}.$$

As in \cite{MZ01} we write as $w_1,\cdots,w_4\in H^*$ the
weights of $x_1,\cdots, x_4$. They satisfy
$w_1+w_2+w_3+w_4=0.$
It will be  convenient to describe the Lie algebra
$\fsl(V)\simeq \fso(6)$ as a simple Lie algebra of type $D_3$. Thus set
$$\epsilon_1=w_1+w_4,\,\, \epsilon_2=w_1+w_3, \,\,
\epsilon_3=w_1+w_2.$$
As before let $F\in H$ be the element defined by
$$\epsilon_1(F)=2, \epsilon_2(F)=4 \hbox{ and }
\epsilon_3(F)=8.$$
Thus we decompose the set of roots as
$\Delta=\Delta^+\cup\Delta^-$, where
$$\Delta^\pm=\{\alpha\in \Delta\mid \pm\alpha(F)>0\}.$$

Then the simple roots of $\fsl(V)$ are
$$\alpha_1=\epsilon_2+\epsilon_1=w_1-w_2,$$
$$\alpha_2=\epsilon_2-\epsilon_1=w_3-w_4,\text{ and },$$
$$\alpha_3=\epsilon_3-\epsilon_1=w_2-w_3.$$
Accordingly, the Dynkin diagram of 
$\fsl(V)\simeq\fso(6)$ is labelled as
$$\dynkin [backwards,
labels={2,3,1},
scale=1.8] D{ooo}$$

\noindent Since the map $\phi$ mixes $(0,2)$-tensors and 
$(2,0)$-tensors, the diagram involution of the Dynkin diagram does not extends to $\CK(6)$, so the 
the simple roots $\alpha_1,\alpha_2$ 
are not  exchangeable.

We observe that
$$C_{\CK(6)}(F)\simeq\Vir\ltimes H\otimes \C[t,t^{-1}].$$

There is a  natural isomorphism of vector spaces
$$\CK(6)\simeq \L_0\otimes\C[t,t^{-1}].$$
Exactly as form the contact superconformal algebras
we can define the Neveu-Schwarz form 
$\CK_{NS}(6)$.
It is easy to check that
$$CK_{NS}(6)\simeq\CK(6).$$


\subsection{The twisted contact algebras}
\label{defK2(2m)}

Let $m\geq 1$. Let $\sigma$ be the involution of
$\C[t,t^{-1},\zeta_1,\eta_1,\cdots,\zeta_m,\eta_m]$ 
such that 
$$\sigma(t^n)=(-1)^n t^n\text{ for all }n\in\Z
\text{ and }$$
$$\sigma(\zeta_1)=\eta_1, 
\sigma(\eta_1)=\zeta_1, \sigma(\zeta_k)=\zeta_k
\text{ and } \sigma(\eta_k)=\eta_k\text{ for } k>1.$$

\noindent We write 
$\C[t,t^{-1},\zeta_1,\eta_1,\cdots,\zeta_m,\eta_m]^\sigma$
for the subalgebra of fixed points under $\sigma$. 
By definition $\K^{(2)}(2m)$  is the contact algebra
$\C[t,t^{-1},\zeta_1,\eta_1,\cdots,\zeta_m,\eta_m]^\sigma$.
Since we will not use it, we skip the description of the root system.

\vskip15mm
\centerline{\bf PART I: General Results}

\bigskip\noindent
In Chapter \ref{zoology} we have seen the list of all
known superconformal algebras. In \cite{KvdL} and 
\cite{CK} it has been conjectured that there are no
other superconformal algebras.

Part I is devoted to general results concerning
superconformal algebras. Some of them 
are only proved for the superalgebras of the list.
Since  all proofs are based on  general ideas,  we believe the results could be extended to unknown superconformal algebras if any.

\bigskip

\section{Uniform bounds of dimensions}
\label{uniformbound}

Let $\L$ be a superconformal algebra. 
By a {\it cuspidal projective $\L$-module} we mean
a cuspidal module over a central extension
$\hat{\L}$ of $\L$.  In this chapter, an ordinary
cuspidal $\L$-module will be viewed as 
a cuspidal projective $\L$-module with a trivial central charge.

A $\Z$-graded vector space $M=\oplus M_i$ has {\it growth one} if its homogenous components have 
uniformly bounded dimension. 
In this chapter, we prove:

\begin{thm} Any  projective cuspidal
$\L$-module has growth one.
\end{thm}

\subsection{General Conventions}

Let $M=\oplus_{n\in\Z}\,M_n$ be a $\Z$-graded vector space. Except stated otherwise we always assume that the
homogenous components $M_n$ have finite dimension.
The {\it support of $M$} is
$$\Supp\,M=\{ n\; |\; M_n\neq (0) \} .$$
For any integer $n$, we write as $M_{\geq n}$ the
subspace 
$$M_{\geq n}=\oplus _{m\geq n}M_m.$$ 
The subspaces $M_{>n},M_{\leq n},M_{<n}$ are defined similarly. For a $\Z$-graded superalgebra $\L$, it is convenient to write $\L^+$ for $\L_{>0}$ and
$\L^-$ for $\L_{<0}$.

For a graded vector space $M$, the graded vector space $T^kM$ defined by $(T^kM)_n=M_{n+k}$ is called a \textit{shift} of $M$. In what follows some graded vector spaces $M$ appear with a natural grading by a $\mathbb{Z}$-coset $u$ in $\mathbb{C}$. Obviously, it induces a $\mathbb{Z}$-grading of $M$, up to a shift.

\subsection{Simple $\Vir$-modules}

In what follows we will denote as $D$ the Ramond derivation $D=t\frac{\d}{\d t}$. The Lie algebra
$$\Vir:=\Der \C[t,t^-1]$$
has basis $\{E_i=-t^i D\mid \in\Z\}$ and bracket
$$[E_i,E_j]=(i-j)E_{i+j}.$$

We define the $\Vir$-module $\Omega^\delta_u$ as one copy 
$\overline{\C[t,t^{-1}]}$ of the vector space $\C[t,t^{-1}]$ with the action of  $\Vir$
$$(fD)\overline{g}=
\overline{fD(g)+\delta D(f)g+ufg}.$$
Informally, the $\Vir$-module 
$\Omega^\delta_u$ can be viewed as the set of symbols
$$z^u\C[t,t^{-1}]  \, 
\Bigl(\frac{\d t}{t}\Bigr)^\delta.$$ 
These $\Vir$-modules $\Omega^\delta_u$ are called {\it tensor density modules}. The $\Vir$-module
$\Omega^\delta$ of Section \ref{defK} is identical 
to $\Omega^{\delta}_0$.

 The eigenvalues of the grading element
$\ell_0:=-E_0$
on $\Omega^\delta_u$
is precisely the $\Z$-coset $u+\Z$, so 
$\Omega^\delta_u$  is a $\Z$-graded
$\Vir$-module. It is clear that 
$$\Omega^\delta_u\simeq \Omega^\delta_v$$ if
$u-v$ is an integer. Except if  $u$ is an integer
and $\delta\in\{0,1\}$, the $\Vir$-module
$\Omega^\delta_u$ is simple. 
However $\Omega^0_0\simeq \C[t,t^{-1}]$ contains the trivial module and $\Omega^0_0/\C\simeq 
\C[t,t^{-1}]/\C$ is simple. Similarly
$\Omega^1_0$ is the space of Khaler $1$-forms 
$\Omega^1_{\C[t,t^{-1}]}$. It contains
the submodule 
$\d \C[t,t^{-1}]\simeq \C[t,t^{-1}]/\C$ and 
$$\Omega^1_{\C[t,t^{-1}]}/\d \C[t,t^{-1}]\simeq \C.$$

The parameter $u$ is called the {\it monodromy}.
Usually $\delta$ is called the {\it degree}. However
in the mathematical  physics literature $-\delta$ is known as the {\it conformal weight}.

The graded dual
of $\Omega^\delta_u$ is $\Omega^{1-\delta}_{-u}$.
Thus the graded dual of $\Vir\simeq \Omega^{-1}_0$
is the space of quadratic differentials
$\Omega^2_0$. The space $\Omega^{\frac{1}{2}}_0$ of
half densities  appears frequently in Lie theory. For Lie superalgebras, the values $\delta=\frac{1}{4}$ and
$\frac{3}{4}$ are also typical degrees.

The universal central extension $\widehat{\Vir}$ of $\Vir$
is an nonsplit extension \cite{KacRaina}
$$0\to\C c\to \widehat{\Vir}\to \Vir\to 0.$$
The Lie algebras  $\widehat{\Vir}$ is also $\Z$-graded and its grading element is $\ell_0=-E_0$.
For a pair $(h,\bar{c})\in\C^2$, 
there is a unique simple $\Z$-graded
$\widehat{\Vir}$-module $V(h,{\bar c})$ generated by
a degree $0$  vector $v$ satisfying
$$\widehat{\Vir}^+\cdot v=0, E_0\cdot v=h v \text{ and } c\cdot v={\bar c} v.$$
Obviously
$$\Supp\,V(h,c)=\Z_{\leq 0}.$$
It is called the  $\widehat{\Vir}$-module with highest weight
$(h,\bar{c})$. The graded dual
$V(h,c)^{(*)}$ is called the $\widehat{\Vir}$-module with 
lowest weight $(-h,-\bar{c})$. These modules are trivial
if $h=\bar{c}=0$.

The classification of all simple $\Z$-graded 
$\widehat{\Vir}$-modules has been obtained by the second author in \cite{M92}. The important case
of  $\Z$-graded 
$\Vir$-modules of growth one has been independently obtained by C. Martin and A. Piard in \cite{MartinPiard}.

\begin{thm}\label{MMP} Any nontrivial simple  $\Z$-graded 
$\widehat{\Vir}$-module is isomorphic to
\begin{enumerate}
\item[(a)]  a simple tensor density module
$\Omega_u^\delta$,
\item[(b)] or to the module $\C[t,t^{-1}]/\C$,
\item[(c)] or to some nontrivial highest weight
$\widehat{\Vir}$-module $V(h,c)$,
\item[(c)] or to some nontrivial lowest weight
$\widehat{\Vir}$-module $V(h,c)^{(*)}$.
\end{enumerate}

\end{thm}

On a simple $\widehat{\Vir}$-module the central
element $c$ acts as a scalar called the
{\it central charge} of the module.
We deduce
\begin{cor}\label{nocharge}
 \begin{enumerate}
\item[(a)] Any simple cuspidal $\widehat{\Vir}$-module
has a trivial central charge
\item[(b)] Any simple $\widehat{\Vir}$-module of growth one is cuspidal. 
\end{enumerate}    
\end{cor}

\begin{proof} The first statement follows from
Theorem \ref{MMP}.

The second statement  amounts to the fact that the dimensions of  the homogeneous components
of $V(h,c)$ or  $V(h,c)^{(*)}$ are  unbounded. This follows from elementary considerations or from
the character formula, see \cite{KacRaina}\cite{IoharaKoga}.
\end{proof}



We also deduce from Theorem \ref{MMP}:
\begin{cor}\label{gap} Given 
     an arbitrary cuspidal $\operatorname{Vir}$-module $M$
\begin{enumerate}
\item[(a)] either $\Supp M=\mathbb{Z}$,  
\item[(b)] or $\Supp M= \mathbb{Z}\setminus\{ m \} $ for some $m\in\mathbb{Z}$, and $\ell_0\vert_{M_i}=i-m$.
\end{enumerate}
\end{cor}

\subsection{A criterion for growth one}
Throughout the whole section,
 $\cL$ denotes  an arbitrary $\Z$-graded Lie superalgebra, which is not necessarily simple. 
In the section we establish a criterion insuring that the cuspidal 
$\cL$-modules have growth one.

\begin{lemma}\label{fg+}
Let $M$ be a simple $\Z$-graded $\cL$-module.

If the subalgebra  $\cL^+$ is  finitely generated, then the $\cL^+$-module $M_{>0}$ is finitely generated.
\end{lemma}

\begin{proof} By hypothesis, 
$\cL^+$ is generated by $\oplus_{1\leq i\leq k}\,\cL_i$ for some $k>0$.

If $\Supp\,M$ has a positive lower bound, say $m> 0$, then $M=U(\cL^+) M_m$ and the lemma is obvious.

Assume  otherwise, that is $M_{\leq 0}\neq 0$.
 Since
$M_{\leq 0}$ is a $\cL_{\leq 0}$-module, we have
$$M=\U(\cL^+)M_{\leq 0}.$$
Hence
$$M_n=\oplus_{i=1}^k\,\cL_k\,M_{n-k},$$
for any $n>0$. Hence the $\cL^+$-module
$M_{>0}$ is generated by $\oplus_{i=1}^k\,M_i$.
\end{proof}

\begin{lemma} Assume that 
\begin{enumerate}
\item[(a)] the subalgebra $\cL^+$ is finitely generated, and
\item[(b)] $\cL=\Ad(U(\cL))(\cL_{\geq m})$, for any $m>0$.
\end{enumerate}
Let $M$ be a simple $\Z$-graded $\cL$-module. If 
$$\cL_{\geq n}\cdot v=0$$ 
for some $0\neq v\in M$ and some $n\in\Z$, then 
$\Supp M$ is upper bounded.
\end{lemma}

\begin{proof}
    The set 
$$M':=\{ u\in M \mid \cL_{\geq k}u=0\; \text{for some integer k} \}$$ 
is clearly a $\cL$-submodule of $M$, which is nonzero by hypothesis. Hence 
$$M=M'.$$ 
By Lemma \ref{fg+}, the $\cL^+$-module $M_{>0}$ is generated by a finite set $S$ of generators. We have
$$\cL_{\geq m}s=0 \text{ for any } s\in S$$
for some integer $m>0$. As $\cL_{\geq m}$ is an ideal of $\cL^+$, it follows that 
$$\cL_{\geq m}\cdot M_{>0}=0.$$

 Set $U^-=\U(\cL^-)$. By hypothesis (b) we have
  $\cL=\Ad(U^-)(\cL_{\geq m})$. Since 
  $\dim\,\cL^+/\cL_{\geq m}<\infty$, there exists an integer $p\geq 1$ such that 
  $$\sum _{k=0}^p \Ad(U^{-}_{-k})(\cL_{>m})\supset \cL^+.$$
 We deduce that
\begin{align*}
 \cL^+\cdot M_{>p} &\subset 
 \sum _{k=0}^p \Ad(U^{-}_{-k})(\cL_{>m})\cdot M_{>p}\\
&\subset \sum _{k=0}^p\sum_{a+b=k} 
U^{-}_{-a}\cdot\cL_{>m}\cdot U^{-}_{-b}\cdot M_{>p}\\
&\subset \sum _{k=0}^p\sum_{a+b=k} 
U^{-}_{-a}\cdot\cL_{>m}\cdot M_{>(p-b)},
\end{align*}   
which proves that   
    $$\cL^+\cdot M_{>p}=0.$$

If $M_{>p}=0$, then 
$$\Supp\,M \subset \,]-\infty, p].$$
Otherwise there is an integer $q>p$ 
such that $M_q\neq 0$. Since $\cL^+\cdot M_q=0$, it follows that $M=\U(\cL)M_q=\U(\cL^-)\cdot M_q$, which implies that
$$\Supp\,M \subset\, ]-\infty, q],$$
which completes the proof of the lemma.
\end{proof}

For  integers $0<p<q$, let $\cL([p,q])$ be the subalgebra of  $\cL^+$ generated by 
$$\oplus_{k\in[p,q]}\,\cL_k.$$

\begin{lemma}\label{criteriongrowth1} Assume that
\begin{enumerate}
 \item[(a)] the subalgebra $\cL^+$ has growth one,
\item[(b)] we have $\cL=\Ad(\U(\cL))(\cL_{>n})$ for
 any $n>0$,
 \item[(c)] there is an integer $q>0$ such that
$$\dim\,\cL^+/\cL([p,p+q]<\infty\hbox{ for any }
p>0.$$ 
\end{enumerate}
    Then any cuspidal $\cL$-module has growth one.
\end{lemma}

\begin{proof}
    
  Choose an integer $q>0$ such that the
subalgebras  $\cL([p,p+q])\subset\cL^+$ have finite codimension for all $p\geq 1$.

Set $d=\max \{ \dim\,\cL_n,n\geq 1 \}$.
We claim that 
$$\dim\,M_{-k}\leq d\sum _{i=0}^{q-1}\dim\,M_{i},$$
for any $k\geq 1$. 

Consider the natural map 
$$\Phi:M_{-k}\to\oplus _{0\leq i<q}\Hom (\cL_{k+i},M_i),$$ 
that maps an arbitrary element $v\in M_{-k}$ to the collection of maps 
$$x\in \cL_{k+i} \mapsto x\cdot v\in M_i.$$ 
 The hypothesis (c) implies that 
$\cL^+$ is finitely generated. Hence by Lemma 4.10, we have 
$$\cL([k,k+q])\cdot v\neq 0\text{ for any }v\in M_{-k}\setminus 0,$$
that is $\Phi(v)\neq 0$. Thus $\Phi$ is injective,
which proves the claim.

It follows that for any cuspidal module $M$ the subspace $M_{<0}$ has growth one. Let $M^{(*)}=\oplus M^{(*)}_n$ be the graded dual of $M$. Since $M^{(*)}_{<0}$ has growth one and $\dim\,M_{-k}=\dim\,M_{k}^{*}$, it follows that
$M_{>0}$ has growth one as well, which complete the proof.
\end{proof}

\subsection{${\Vir}$-modules of growth one}

Let $\cC(\Vir)$ be the category of all $\Z$-graded 
${\Vir}$-modules  of growth one. 
In this section, we establish the technical lemma
\ref{technicalC} about modules in $\cC(\Vir)$.

Let $\ell_0=-E_0$ be the grading element in ${\Vir}$.
For a $\Z$-graded $\C[\ell_0]$-module $M$,
we denote as
$$M^{(u)}:=
\{m\in M\mid (\ell_0-u )^{d}m=0\mid \text{ for }d>>0\}$$ 
the generalized $\ell_0$-eigenspace of eigenvalue $-u\in\C$. By definition the {\it schift set} 
and the {\it real-shift set} of $M$ are
$$\Sh(M):=\{u\in\C\mid M_n^{(u+n)}\neq 0\}
\hbox{ for some }n\in\Z\}, \text{ and }$$ 
$$\Sh_{re}(M):=\{u\in\C\mid 
M_n^{(u+n)}\neq 0 \text{ and } n+u\neq 0
\hbox{ for some }n\in\Z\}.$$ 

Assume now that $M$ is a ${\Vir}$-module. For $u\in \Sh(M)$, set
$$G_u(M)=\oplus_{n\in\Z}\,M_n^{(n+u)}.$$
Obviously
 $$E_n.M_m^{(v)}\subset M_{v+n}^{(\mu+n)}
\text{ for any } n,m \in\Z \text{ and } v\in\C.$$
Therefore each $G_u(M)$ is a ${\Vir}$-submodule and
$$M=\oplus_{u\in \Sh(M)}\,G_u(M).$$

\begin{lemma}\label{fl} Let $M\in\cC(\Vir)$.
\begin{enumerate} 
\item[(a)] The set $\Sh_{re}(M)$ is finite, and
\item[(b)] Each sumodule $G_u(M)$ has finite lenght.
\end{enumerate}
\end{lemma}

\begin{proof} 
Let $u\in \Sh_{re}(M)$. 
By hypothesis $G_u(M)$ admits a 
cuspidal subquotient $M'$ and by 
Corollary \ref{gap} we have
$$\Supp\,M'=\Z\text{ and }\Supp\,M'=\Z-\{-u\}.$$
Therefore
$$\dim\,G_u(M)_n=1 \text { for } n>>0.$$
Thus the set $\Sh_{re}(M)$ has cardinality 
$\leq \Max_{n\in\Z}\,\dim\,M_n$ which proves Assertion (a).

Let $u\in \Sh(M)$. In order to prove Assertion (b), we can assume that $M=G_u(M)$.
If $u\in \Sh(M)\setminus\Sh_{re}(M)$, then 
$M=M_{-u}$ thus $M$ is a trivial module of finite dimension.

Otherwise choose an integer $n\neq -u$. Then
$\ell_0-(n+u)$ acts nilpotently on 
$M_n$. By Corollary \ref{gap}, $n$ lies in the 
support of any cuspidal subquotient. Moreover
if $M$ has a trivial subquotient, its support is
$-u$. Thus a composition series of 
$M$ contains at most $\dim\,M_n$ cuspidal modules and
at most $\dim\,M_{-u}$ trivial subquotients, which shows that
$M$ has finite lenght.
\end{proof}

It follows that any indecomposable ${\Vir}$-module in
$\cC(\Vir)$ has finite length.
Though Theorem \ref{MMP} describes simple objects
in $\cC(\Vir)$, the category $\cC(\Vir)$ is wild \cite{germoni}.

For $k>0$ we write as $\Vir([k,k+1])$ the subalgebra generated by $E_k$ and $E_{k+1}$. 

\begin{lemma} Let $M\in\cC(\Vir)$ be a ${\Vir}$-module with
a finite shift-set $\Sh(M)$.

Then for any $k, l>0$, the $\Vir([k,k+1])$-module 
generated by $M_l\oplus M_{l+1}$ has finite codimension in
$M_>0$.
\end{lemma}

\begin{proof}
First assume that $M$ is cuspidal. 
It follows from 
the explicit description of $M$ that

\begin{enumerate}\label{technicalC}
\item[(a)] for any integer $n\in\Supp\,M$, $E_k.M_n\neq 0$ or $E_{k+1}.M_n\neq 0$
\item[(b)] for some $n_0$, we have
$E_k.M_n\neq 0$ or $E_{k+1}.M_n\neq 0$
for all $n\geq n_0$.
\end{enumerate}

Let $M'\subset M$ be the $\Vir([k,k+1])$-submodule 
generated by $M_l\oplus M_{l+1}$.
By Corollary \ref{gap}, we have $M_l\neq 0$ or $M_{l+1}\neq 0$. By Assertion (a), $M'$ is infinite dimensional. Thus
$\Supp\,M'$ contains an integer $n_1\geq n_0$.
Any integer $m\geq k^2$ can be written as
$$m=ak+b(k+1)\text{ with } a,b\geq 0.$$
Since $E_k^a E_{k+1}^b M_{n_1}\neq 0$, we deduce that
$$\Supp M'\supset\Z_{\geq n_1+k^2}.$$
Since all homogenous components of $M$ have dimension
$\leq 1$, $M'$ has finite codimension in
$M_{>0}$.

An arbitrary $\Vir$-module $M$ with  finite 
shift set has finite length
by Lemma \ref{fl}. Thus the claim is proved in full generality.
\end{proof}

\subsection{Proof of Theorem \ref{uniformbound}}

\begin{lemma}\label{BB}
    Let a Lie superalgebra $\L$ be a sum $\L=A+B$, where $A$ is a Lie subalgebra and $[A,B]\subseteq B$. Then $B+[B,B]$ is an ideal of $\L$.
\end{lemma}

\begin{proof}
    Denote $I=B+[B,B]$. Clearly, $[A,I]\subseteq I$.
    Now, $[B,I]\subseteq [B,B]+[[B,B],B]$.
    We have $[[B,B],B]\subseteq [A+B,B]\subseteq B+[B,B]=I$.
    This completes the proof of the lemma.
\end{proof}

Let $\L$ be a superconformal algebra and let
$\hat L$ be its universal central extension.
For $d\geq 1$, set 
$\Vir(d):=\Der \C[t^d,t^{-d}]$. The Lie algebra $\Vir(d)$
is isomorphic to  Virasoro algebra $\Vir$, with a grading
rescaled by a factor $d$. Its basis is
$$\{E_{nd}=-t^{nd} D\mid n\in\Z\}.$$
By definition $\L_{\overline 0}$ contains 
the Lie algebra $\Vir(d)$ for some $d\geq 1$.
For the known superconformal algebras,
the integer $d$ is always $1$ or $2$,
see Chapter \ref{zoology}.

As before, we set $\ell_0=-E_0$. 
 
 \begin{lemma}\label{finiteshift} The shift set $\Sh(\hat{\L})$ is finite.
\end{lemma}

\begin{proof} Let 
$$A=(\hat{\L})^{(0)}\text{ and }
B=\oplus_{v\neq 0}\,(\hat{\L})^{(v)}.$$
By Lemma \ref{BB}, $B+[B,B]$ is an ideal. Since $B$ is not
central, the simplicity of $\L$ implies that 
 $$\hat{\L}=B+[B,B].$$
The $\C[\ell_0]$-module $B$ can be viewed as a submodule of $\L$.
Therefore $\Sh(B)=\Sh_{re}(\L)$ and the shift set of $B$ is finite by Lemma \ref{fl}.

 It follows that
$$\Sh(\hat{\L})\subset \Sh(B) \cup
\Bigl(\Sh(B)+\Sh(B)\Bigr)$$ is finite.
\end{proof}

\begin{lemma}\label{hypo(c)} Let $\L\supset\Vir(d)$ be a superconformal algebra. 

Then for any $p>0$, the subalgebra
$$\L([p,p+2d-1])$$
has finite codimension in 
$\dim\, \L^+$.
\end{lemma}

\begin{proof} For any $0\leq i\leq d-1$, set
$M(i)_n=\L_{nd+i}$ and
$$M(i)=\oplus_{n\in\Z}\, M(i)_n.$$ 
We have
$$[E_{dn}, M(i)_m]\subset M(i)_{n+m}.$$
With this rescaling, each $M(i)$ can be viewed as a $\Z$-graded module over $\Vir$. 

Let $k$ be the integer satisfying
$$p\leq dk\leq d(k+1)\leq p+2d-1.$$
Also for  $0\leq i<d$, let $l$ be the integer such that
$$p\leq i+ dl\leq i+d(l+1)\leq p+2d-1.$$

By Lemma \ref{finiteshift}, the shift set of each 
$\Vir$-module $M(i)$ is finite. Thus
by Lemma \ref{technicalC}, the
$\Vir([k,k+1])$-module generated by 
$M(i)_l\oplus M(i)_{l+1}$ has finite codimension in
$M(i)_{>0}$. 

We will now express the same result with the original
$\Z$-grading. Let $\fM\subset\Vir(d)$ be the subalgebra generated by $E_{dk}$ and $E_{d(k+1)}$. We have just proved that
the $\fM$-module generated by
$$\oplus_{n=p}^{p+2d-1} \L_n$$ has finite codimension
in $\L^+$. It follows that
$\L([p,p+2d-1])$
has finite codimension in 
$\dim\, \L^+$.
\end{proof}

\bigskip\noindent
{\it Proof of Theorem \ref{uniformbound}.}
Let $M$ be a projective cuspidal $\L$-module.

If $M$ is indeed a cuspidal $\L$-module, set
$\hat{\L}=\L$. Otherwise
there exists a central extension 
$$0\to\fz\to \hat{L}\to \L\to 0$$
with a center $\fz$ of dimension one,  such that $M$ is a cuspidal module over $M$.

We now show that the superalgebra $\hat{\L}$ satisfies
the hypotheses of Lemma \ref{criteriongrowth1}.
Hypothesis (a) is obvious and hypothesis (b) follows from the fact that $\Ad(U(\hat{L})(\hat{L}_{\geq n})$ is a noncentral ideal.

Let $d$ such that $\L$ contains $\Vir(d)$. By
Lemma \ref{hypo(c)}, the subalgebra
$\L([p,p+2d-1])$ has finite codimension in 
$\L^+$, so the similar assertion holds for
$\hat{L}$. Thus Theorem \ref{uniformbound} is proved.

\section{Growth one submodules in \\ coinduced modules}\label{LR}

In the chapter we consider $\Z$-graded Lie 
superalgebras $\L$ 
endowed with with a grading element $\ell_0$ and 
a subalgebra $\L^{\bf (1)}$ of finite codimension.

We assume that
$\L_{\overline 0}=\L^{\bf (1)}_{\overline 0}
\oplus\C \ell_0$ together with 
certain axioms $(AX1-3)$, including that
$\L^{\bf (1)}_{\overline 0}$ contains a Cartan subalgebra $H$ and a specific element $F\in H$.
The element $F$ is used  to define the notion   of positive and negative roots.
Set
$T=\L/ \L^{\bf (1)}$. The weight decomposition of
$T$ induces  a triangular decomposition
$$T=T^+\oplus T^{(0)}\oplus T^-$$

Given a simple finite dimensional $\L^{\bf (1)}$-module $S$ with highest weight $\mu$,
we define a  family,  indexed by $u\in\C/\Z$, of
 $\Z$-graded $\L$-modules 
$\cF(S,u)$. They have   growth $1$ and they lie in
$\Coind_{\L^{\bf(1)}}^\L\,S$. The main result
can be stated as follows:

\begin{imprecise} The $C_\L(F)$-module structures of $T^-$ and $S^{(\mu)}$ determine  

\begin{enumerate}
\item[(a)]the highest weight $\lambda$ of
$\cF(S,u)$ and

\item[(b)] the $C_\L(F)$-module $\cF(S,u)^{(\lambda)}$.
\end{enumerate}
\end{imprecise}

A precise statement, which requires many hypotheses 
and notations, is stated at the end of the chapter.
At first glance the result looks very technical. In fact it applies very easily to
the superconformal algebras $\W(n)$, $\K_*(N)$,
$\CK(6)$ and to the specific central extension $\widehat{\K(4)}$ of $\K(4)$. The axiomatic approach adopted in this chapter  seems more pleasant than repeating identical computations 
in the case-by-case analysis of Part II.
     
A more natural approach of these questions is based on nonabelian cocycles and coherent families as in 
\cite{M2000} but in our setting formulas are quite complicated. For example the reader can compare the definition of certain cuspidal modules for 
$\widehat{\K(4)}$ using  coinduction functor
and the alternative definition  in term of a cocycle in Chapter \ref{projcusp}. The second definition is much more complicated, though the cocycle is merely abelian.

\subsection{Coinduced modules}

Given a simple Lie algebra $\fb$ of arbitrary dimension and a subalgebra $\fa$ of finite codimension $m$,  E. Cartan observed in \cite{Cartan} that
$$\fb\subset\Der \C[[t_1,\cdots,t^m]].$$
We now present,  in the framework of 
Lie superalgebras and coinduced modules,
Cartan's observation. 

 Let $\L$ be a Lie superalgebra acting by derivation
on a commutative superalgebra $R$. An $\L$-module
$V$ endowed with a structure of $R$-module is called 
a $(\L,R)$-module if it satisfies 

$$(\partial\circ r) v
=\partial(r) v +(-1)^{\vert \partial\vert\vert r\vert} 
\,r\circ \partial v,$$

\noindent for any $v\in V$ and any $\Z/2\Z$-homogenous elements $\partial\in\L$ and $r\in R$.

Given a Lie superalgebra $\fb$, a  subalgebra $\fa$ and an $\fa$-module $S$, the coinduced module is
$$\Coind_\fa^\fb\, S:=\Hom_{\fa}(U(\fb), S).$$

\noindent Since $U(\fb)$ is  a  $U(\fa)\otimes U(\fb)$-module, $\Coind_\fa^\fb\, S$ is a $\fb$-module.
 
There is a natural evaluation map
$$\Eval:f\in \Coind_\fa^\fb\, S\mapsto f(1)\in S$$

\noindent and the coinduction functor satisfies the following universal property:

{\it given a    $\fb$-module $V$, an   $\fa$-module $S$
and an homorphism of $\fa$-modules}

{\it   $\phi:V\to S$ 
there is a unique homomorphism of
$\fb$-modules}
$$\psi:V\to \Coind_\fa^\fb\, S$$
{\it such that}
$$\Eval\circ\psi=\phi.$$

\noindent see e.g \cite[ch. 5]{Dixmier}. Roughly speaking
the coinduction functor satisfies the same property as the induction functor for the reverse arrows. 

Given a superspace $M=M_{\overline 0}\oplus M_{\overline 0}$ set
$$\Sym(M)=\S M_{\overline 0}\otimes \wedge M_{\overline 0}.$$

The superalgebra $\Sym(M)$ is a cocommutative
Hopf algebra, therefore its dual $\Sym(M)^*$ is 
a commutative algebra. Our case of interest is
$$\dim\, M_{\overline 0}=1\hbox{ and }\dim\, M_{\overline 0}=n.$$
Then it is clear that
$$\Sym(M)^*\simeq\C[[t]][\xi_1,\cdots,\xi_n].$$

The following lemma summarizes the  the general properties of coinduction functors.

\begin{lemma}\label{univpty1}
Let $\fa\subset \fb$ be Lie superalgebras.
\begin{enumerate}
\item[(a)] Given  $\fa$-modules
$S_1$, $S_2$ there is a natural $\fb$-homomorphism
$$\Coind_\fa^\fb\, S_1 \otimes \Coind_\fa^\fb\, S_2
\to \Coind_\fa^\fb\, (S_1\otimes S_2).$$
\item[(b)]  The superalgebra $\Coind_\fa^\fb\,\C$ is
commutative and 
     $$\Coind_\fa^\fb\,\C\simeq \Bigl(\Sym(\fb/\fa)\Bigr)^*.$$

\item[(c)] In particular,
$\Coind_\fa^\fb\, S$ is a 
$(\Coind_\fa^\fb\,\C, \fb)$-module, for any $\fa$-module $S$.

Moreover if $S$ is finite dimensional,
then there is an isomorphism of  $\Coind_\fa^\fb\C$-modules
$$\Coind_\fa^\fb\,S\simeq (\Coind_\fa^\fb\,\C)\otimes S.$$
\end{enumerate}
\end{lemma}

\begin{proof} Assertion (a) is a  formal consequence of the universal property of coinduction functors.

The Poincar\'e-Birkhoff-Witt isomorphism
$$\Sym(\fb)\to U(\fb)$$
is a canonical isomorphism of coalgebras.
Let $\fp$ be any direct summand of $\fa$ in $\fb$. As a 
right $U(\fa)$-module we have
$$U(\fb)=\Sym(\fp)\otimes U(\fa).$$
Thus 
$$\Coind_\fa^\fb\,\C=\Bigl(\Sym(\fp)\Bigr)^*\simeq 
\Bigl(\Sym(\fb/\fa)\Bigr)^*,
\hbox{ and}$$
$$\Coind_\fa^\fb\,S=\Hom(\Sym(\fp),S)\simeq\Bigl(\Sym(\fp)\Bigr)^*\otimes S,$$

\noindent which proves Assertion (b) and (c). 

\end{proof}

 We will use the following easy result
 
 \begin{lemma}\label{duality} Let $S$ be a finite dimensional
 $\fa$-module. Then
 $$\Coind_\fa^\fb\,S=
 \left( \Ind_\fa^\fb\,S^*\right)^*.$$
 \end{lemma}
  
\begin{proof} In fact 
$$ \Ind_\fa^\fb\,S^*=H_0(\fa, U(\fb)\otimes S^*)$$
hence its dual is
$$H^0(\fa, \left(U(\fb)\otimes S^*\right)^*).$$
Since $S$ is finite dimensional,
$\left(U(\fb)\otimes S^*\right)^*)=\Hom(U(\fb,S)$
which proves the lemma.
\end{proof}

A {\it character} of a Lie superalgebra $\fa$ is a linear map $\chi:\fa\to\C$ such that
$$\chi([\fa,\fa])=0\hbox{ and } \chi(\fa_{\overline 1})=0.$$

\noindent Let $\C_\chi$ be the $1$-dimensional $\fa$-module of character $\chi$. To simplify we denote
$\Ind_\fa^\fb\,\C_\chi$  and $\Coind_\fa^\fb\,\C_\chi$ as
$$ \Ind_\fa^\fb\,\chi\hbox{ and } 
\Coind_\fa^\fb\,\chi.$$

\subsection{The $\L$-module $\cF(S,u)$}\label{cF(Su)}

Coinduced modules are always modules of dimension
$>\aleph_0$. However
under some natural hypotheses, they contain
interesting submodules of growth one as it is shown in this section.

Let $\L$ be a  $\Z$-graded Lie superalgebra endowed with
a grading element  $\ell_0$. By definition there is some
$\tau\in\{0, 1/2\}$ such that $\L_{\overline 1}$ is graded by the coset $\tau+\Z$. 

Let
$\L^{(\bf 1)}$ be a subalgebra satisfying the following conditions

\begin{enumerate}
\item[(a)] $\L_{\overline 0}=
\L^{(\bf 1)}_{\overline 0}\oplus\C \ell_0$,
\item[(b)] $\dim\, \L_{\overline 1}/\L^{\bf(1)}_{\overline 1}=n<\infty$.
\end{enumerate}
However   we never assume that 
$\L^{(\bf 1)}$ is a $\Z$-graded subalgebra.

Let $M=M_{\overline 0}\oplus M_{\overline 1}$  be an arbitrary $\C[\ell_0]$-module. For 
$u\in \C$,
we write  $M^{(u)}$ for the $\ell_0$-eigenspace of eigenvalue $u$.
Thus $\oplus_{u\in\C} M^{(u)}$ is a 
$\C[\ell_0]$-submodule, with equality only
when $\ell_0$ is diagonalizable.

Assume now that $M$ is a $\L$-module. For $u\in \C$, set 

$$F_u(M):=\oplus_{i\in\Z} \left(M_{\overline{0}}^{(u+i)}
\oplus M_{\overline{1}}^{(u+i+\tau)}\right).$$ 
\noindent
It is clear that $F_u(M)$ is a $\L$-submodule of $M$.
Moreover $F_u(M)$ is endowed with a $\Z$-grading
$F_u(M)=\oplus_i\, F_u(M)$ defined by
 
$$\left(F_u(M)\right)_{{\overline 0},i}:=M_{\overline 0}^{(u+i)}$$
$$(F_u(M))_{{\overline 1},i+\tau}:=M_{\overline 1}^{(u+i+\tau)}.$$

For a finite dimensional $\L^{\bf(1)}$-module $S$, set
$$\cF(S,u):=F_u(\Coind_{\L^{\bf (1)}}^\L\, S).$$
Lemma \ref{univpty1}(b) implies that $\cF(\C,0)$ 
is an algebra and $\cF(S,u)$ is a $(\L,\cF(\C,0))$-module.

\begin{lemma}\label{multiplicity} 
The algebra $\cF(\C,0)$ is isomorphic to
$\C[t,t^{-1},\xi_1,\cdots,\xi_n]$, where all odd variables
$\xi_i$ have degree $\tau$.

As a $\cF(\C,0)$-module, we have
$$\cF(S,u)\simeq \cF(\C,0)\otimes S.$$

In particular, $\cF(S,u)$ is a $\Z$-graded $\L$-module of growth one.
\end{lemma}

\begin{proof} First notice that
for the $\C[\ell_0]$-module $\C[\ell_0]^*$ and an arbitrary $u\in\C$
$$(\C[\ell_0]^*)^{(u)}= (C[\ell_0]/(\ell_0-u)\C[\ell_0])^*$$
is $1$-dimensional. It follows that the algebra
$$\oplus_{u\in\Z}\,(\C[\ell_0]^*)^{(u)}$$ 
is isomorphic to the graded algebra $\C[t,t^{-1}]$.
Let $T\subset \L_{\overline 1}$ be a graded subspace such that $\L_{\overline 1}=T\oplus \L^{\bf(1)}_{\overline 1}$.
By Poincar\'e Birkhoff Witt we have
$U(\L)=\C[\ell_0]\otimes\wedge T\otimes U(\L^{\bf(1)})$
from which it follows that

$$\cF(\C,0)\simeq\C[t,t^{-1}]\otimes (\wedge T)^*
\simeq \C[t,t^{-1},\xi_1,\cdots,\xi_n],$$
 where all odd variables $\xi_i$ are homogenous of  degree $\tau$ modulo $\Z$.
After multiplying each $\xi_i$ by the appropriate power of $t$, all odd variables have the same degree $\tau$.

The second assertion follows easily from Lemma \ref{univpty1}(c), and the last assertion is obvious.
\end{proof}

\subsection{First example of modules $\cF(S,u)$}
\label{fg}
We now consider the Lie algebra
$$\fg(H):=\Vir\ltimes H\otimes \C[t,t^{-1}].$$
where $H$ denotes a finite dimensional even vector space. When the context is clear, we will use the short notation $\fg$ for $\fg(H)$.  

Recall that $\Vir=\{fD\mid f(t)\in\C[t,t^{-1}]\}$,
where $D=t\frac{\d}{\d t}$ is the Ramond derivation.
We define the {\it divergence}
of an arbitrary derivation  $\partial=fD\in \Vir$ 
as 
$$\div(\partial)=Df.$$
With our convention we have $\div(D)=0$. Also 
the {\it contraction of $\partial$ by  the form $t\frac{\d}{\d t}$}  is defined as
$$ \langle\partial\mid t\frac{\d}{\d t}\rangle=f.$$

Let
$(\lambda,\delta,u)\in H^*\times\C\times \C$.
We denote by $\overline{\C[t,t^{-1}]}$ one copy of the vector space $\C[t,t^{-1}]$ with the action of  $\Vir$
$$(fD)\overline{g}=
\overline{fD(g)+\delta D(f)g+ufg}$$
and the action of $H\otimes \C[t,t^{-1}]$
$$(h\otimes f)\overline{g}=\lambda(h)\overline{fg}$$
where $f=f(t)$, $g=g(t)$ are Laurent polynomials and $h$ is an arbitrary element in $H$. More 
formally the derivation $fD$ acts
on $\C[t,t^{-1}]$ as
$$fD+\delta\,\div(fD)+
u \langle \frac{\d t}{t}\mid fD\rangle.$$

This module is denoted  as $\Tens(\lambda,\delta,u)$.
When $\lambda\neq 0$ the $\fg$-module $\overline{\C[t,t^{-1}]}$ is irreducible.
For $\lambda=0$, the $\fg$-modules $\Tens(0,\delta,u)$ is merely the  $\Vir$-modules of 
Chapter \ref{uniformbound}.

Let ${\bf m}\subset\C[t,t^{-1}]$ be the ideal
${\bf m}:=(1-t)\C[t,t^{-1}].$
Then 
$$\Vir^{\bf (1)}:={\bf m}\Vir$$ 
is the Lie subalgebra of derivations vanishing at $t=1$. Set
$$\fg^{\bf(1)}= \Vir^{\bf (1)}\ltimes\C[t,t^{-1}].$$

\noindent The elements $\partial\in \fg^{\bf(1)}$ are 
of the form
$$\partial=fD+h\otimes g$$
with $f(1)=0$.
Define the character 
$\chi(\lambda,\delta):
\fg^{\bf (1)}\to\C$ by 

$$\chi(\lambda,\delta)(\partial)=
f'(1)+\lambda(h)g(1)=\div(fD)(1)+\lambda(h)g(1).$$

The Lie algebra $\fg$ is $\Z$-graded,
$\ell_0:=D$ is the grading element and we have
$\fg=\fg^{\bf(1)}\oplus\C\ell_0$.
Set
$$\cF(\chi(\lambda,\delta),u)
=F_u\Bigl(\Coind_{\fg^{\bf (1)}}^\fg\, \chi(\lambda,\delta)\Bigr).$$

\begin{lemma}\label{universalfg} For an arbitrary triple $(\lambda,\delta,u)\in H^*\times\C\times \C$ we have

$$\cF(\chi(\lambda,\delta),u)\simeq\Tens(\lambda,\delta,u).$$
\end{lemma}

\begin{proof} By definition $\Tens(\lambda,\delta,u)$
is a $\C[t,t^{-1}]$-module of rank one, thus
$\Tens(\lambda,\delta,u)/{\bf m}\Tens(\lambda,\delta,u)$
is a $1$-dimensional $\fg^{\bf(1)}$-module. It is clear that its character is  $\chi(\lambda,\delta)$, therefore there is a nonzero $\fg$-equivariant map
$$\phi:\Tens(\lambda,\delta,u)\to
\Coind_{\fg^{\bf(1)}}^\fg\,\chi(\lambda,\delta).$$
Since $\ell_0$ acts diagonaly on $\Tens(\lambda,\delta,u)$ with
all eigenvalues in the coset $u+\Z$, it follows that
$\phi(\Tens(\lambda,\delta,u))$ lies in 
$\cF(\chi(\lambda,\delta),u)$. We observe that
$\phi$ is a nonzero homorphism of $\Z$-graded
 ($\C[t,t^{-1}],\fg$)-module and that
$\Tens(\lambda,\delta,u)$ and $\cF(\chi(\lambda,\delta),u)$ are $\C[t,t^{-1}]$-modules of rank $1$. We conclude that
$$\phi:\Tens(\lambda,\delta,u)\to \cF(\chi(\lambda,\delta),u)$$
is an isomorphism.
\end{proof}

\subsection{Second example of modules $\cF(S,u)$}
\label{fG}

We now consider the Lie superalgebras
$$\fG(H):=\K(1)\ltimes H\otimes \C[t,t^{-1},\xi],$$
$$\fG_{NS}(H):=\K_{NS}(1)\ltimes H\otimes\C[t,t^{-1},t^{1/2}\xi],$$
where $H$ be a finite dimensional even vector space. 
We denote indifferently as $\K_*(1)$
the superalgebra $\K(1)$ or $\K_{NS}(1)$.
Similarly $\fG_*(H)$ denotes $\fG(H)$ or 
$\fG_{NS}(H)$. We set $\tau =0$ in the Ramond 
case and $\tau=1/2$ in the Neveu-Schwarz case. 

Recall that the  superconformal algebra $K_*(1)$ is the space of all derivations 
$\partial\in\Der \C[t,t^{-1},t^\tau\xi]$ of the form
$$
\partial=\Bigl(f+\frac{g\xi}{2}\Bigr)D +
\Bigl(\frac{D(f)}{2}\xi
+g\Bigr) \frac{\partial}{\partial \xi}$$

\noindent where
$f\in \C[t,t^{-1}]$, $g\in t^\tau\C[t,t^{-1}]$.
The {\it contact divergence} of the derivation
$\partial$ is defined as
$$\div(\partial):=D(f+\xi g).$$
In fact the contact divergence is twice the usual divergence.
Also the contraction of $\partial$ by the form
$t\frac{\d}{\d t}$ is
$$\langle\partial\mid t\frac{\d}{\d t}\rangle:=f+\frac{g\xi}{2},$$

We define the {\it tensor density} $\fG_*(H)$-module
$\Tens(\lambda,\delta,u)$  as one copy \\
$\overline{\C[t,t^{-1},t^{\tau}\xi]}$ of
$\C[t,t^{-1},t^{\tau}\xi]$ where any 
$\partial\in \K_*(1)$ acts as
$$\partial+\delta\div(\partial)+u
\langle\frac{\d t}{t}\mid\partial\rangle$$
and any $h\otimes f\in H\otimes\C[t,t^{-1},t^{\tau}\xi]$
acts by multiplication by $\lambda(h)f$.
When $\lambda\neq 0$, these modules are irreducible.
Otherwise they are merely the $K_*(1)$-modules considered in \cite{CaiLiuLu}.

We observe that the even part of
$\fG_*(H)$ is the Lie algebra $\fg(H)$ considered in Section \ref{fg}. 
Using the explicit formulas, we see that the 
$\fg(H)$-module
$$\Tens(\lambda,\delta,u)_{\bar 0}$$ is the
module denoted as $\Tens(\lambda,\delta,u)$ in
Section \ref{fg}, which justifies the chosen normalization for the divergence.
Moreover, as a $\fg(H)$-module, the odd part
$$\Tens(\lambda,\delta,u)_{\bar 1}$$ 
is the $\fg(H)$-module
$\Tens(\lambda,\delta+\frac{1}{2},u+\tau)$.

We now define the subalgebras 
$\K_*(1)^{\bf(1)}\subset \K_*(1)$ and
$\fG_*^{\bf(1)}\subset\fG_*(H)$.
Let ${\bf m}\subset\C[t,t^{-1},t^\tau\xi]$ be the 
maximal ideal
$${\bf m}:=\langle (1-t), t^\tau\xi\rangle$$
and let $\K_*(1)^{\bf (1)}$ be the Lie subalgebra of derivations vanishing at ${\bf m}$. 
Explicitly a derivation
$$\partial=\Bigl(f+\frac{g\xi}{2}\Bigr)D +
\Bigl(\frac{D(f)}{2}\xi
+g\Bigr) \frac{\partial}{\partial \xi}
\in \K_*(1)$$
belongs to $K_*(1)^{\bf (1)}$ if and only if $f(1)=g(1)=0$.

Set
$$\fG_*^{\bf(1)}= \K_*(1)^{\bf (1)}\ltimes\C[t,t^{-1},t^\tau\xi].$$

\noindent For 
$F(t,\xi)= f(t) + \xi g(t) 
\in\C[t,t^{-1},t^\tau\xi]$, set
$F(1)=f(1)=F \mod {\bf m}$.
Define the character 
$\chi(\lambda,\delta):
\fG_*^{\bf (1)}\to\C$ by 

$$\chi(\lambda,\delta)(\partial+h\otimes f)=
\div(\partial)(1)+\lambda(h)f(1),$$
for any 
$\partial+h\otimes f \in\fG(H)$, where
$\partial\in\K_*(1), h\in H$ and $f\in\C[t,t^{-1},t^\tau\xi]$. 

As before set
$$\cF(\chi(\lambda,\delta),u)
=F_u\Bigl(\Coind_{\fG^{\bf (1)}}^\fG\, \chi(\lambda,\delta)\Bigr).$$

\begin{lemma}\label{universalFG} For an arbitrary triple $(\lambda,\delta,u)\in H^*\times\C\times \C$ there is an isomorphism of
$\fG_*(H)$-modules

$$\cF(\chi(\lambda,\delta),u)\simeq\Tens(\lambda,\delta,u).$$
\end{lemma}

\begin{proof} As before there is a natural $\fG(H)$-equivariant map
$$\phi:\Tens(\lambda,\delta,u)\to
\cF(\chi(\lambda,\delta),u).$$
 We observe that
$\Tens(\lambda,\delta,u)$ and $\cF(\chi(\lambda,\delta),u)$ are $\C[t,t^{-1},t^\tau]$-modules of rank $1$ 
and $\phi$ induces an isomorphism  
$$\Tens(\lambda,\delta,u)/{\bf m}\to
\cF(\chi(\lambda,\delta),u)/{\bf m}.$$

\noindent We conclude that
$$\phi:\Tens(\lambda,\delta,u)\to \cF(\chi(\lambda,\delta),u)$$
is an isomorphism.
\end{proof}

\subsection{The highest weight of $\cF(S,u)$}

We arrive to the main at the key result of the chapter,
which concerns cuspidal representations of
$\W(n)$, $\K(n)$,  $\CK(6)$ and their central extensions. Instead of a case-by-case analysis, we will state a general result concerning Lie superalgebras 
$\L$ satisfying some properties.

As before let $\L$ be a  $\Z$-graded Lie superalgebra endowed with a grading element  $\ell_0$ and let 
$\tau\in\{0,1/2\}$ be its shift.

The first property is

\begin{ax}  The superalgebra $\L$ is endowed with
a pair $(H,F)$ where  $H\subset \L_{{\overline 0},0}$
is an abelian subalegebra and $F\in H$.
The
$\ad(H)$-action is diagonalisable,
which gives rise to a root space decomposition
$$\L=C_\L(H)\oplus \oplus_{\alpha\in\Delta} \L^{(\alpha)},$$ 
where $C_\L(H)$ is the centralizer of $F$. It is assumed that 
\begin{enumerate}
\item[(a)]   the set of roots 
$\Delta:=\{\alpha\in H^*\mid \L^{(\alpha)}\neq 0
\hbox{ and }\alpha\neq 0\}$ 
is finite.
\item[(b)] $\alpha(F)$ is an integer  for all $\alpha\in\Delta$.
\end{enumerate}
\end{ax}

Thus we can define
$$\Delta^+=\{\alpha\mid \alpha(F)>0\}\hbox{ and }$$
$$\Delta^-=\{\alpha\mid \alpha(F)<0\},$$
and there is a triangular decomposition
$$\L=\L^-\oplus C_\L(F)\oplus \L^+$$
where 
$$\L^{\pm}=\oplus_{\alpha\in\Delta^{\pm}}\,\L^{(\alpha)}.$$
In some cases there are roots $\alpha$ with $\alpha(F)=0$.

The second property concerns the structure 
of the centralizer  $C_\L(F)$.
Recall that the algebras $\fg(H)$, $\fG(H)$ or $\fG_{NS}(H)$ are defined in Sections \ref{fg} and \ref{fG}.

\begin{ax}
The algebra $C_\L(F)$ contains a subalgebra $\fK$ and an ideal  denoted as
$\Rad(C_\L(F))$ such that
\begin{enumerate}
\item[(a)]\hskip5mm
$C_\L(F)=\fK\ltimes \Rad(C_\L(F))$
\item[(b)]
$\fK$ is isomorphic to 
$\fg(H)$, to $\fG(H)$ or to $\fG_{NS}(H)$.
\end{enumerate}
\end{ax}

Let $\L^{\bf(1)}$ be a subalgebra. We have already
define $\L^{\bf(1)}$ when $\L=\Vir$ or $\K(1)$ and,
in general the definition of
$\L^{\bf(1)}$ will be similar. Thus $\L^{\bf (1)}$ is
called the {\it isotropy subalgebra}. However for the purpose of the chapter we  continue with an axiomatic approach.

The third property concerns the structure of $\L^{\bf(1)}$.
Recall that $\fK=\fg(H)$ or $\fG_*(H)$. Thus set
$$\fK^{\bf (1)}=\fg^{\bf(1)}(H)\hbox{ or }\fG_*^{\bf(1)}(H)$$
accordingly and set 
$$C_\L(F)^{\bf(1)}=\fK^{\bf (1)}\ltimes \Rad(C_\L(F)).$$

For 
$(\lambda,\delta)\in H^*\times \C$ we extend the character
$\chi(\lambda,\delta):\fK^{\bf(1)}\to \C$ to
$C_\L(F)^{\bf(1)}$ by requiring that it vanishes on
$\Rad(C_\L(F))$. This character is again denoted as
$\chi(\lambda,\delta)$.

We also set $T=\L/\L^{\bf (1)}$.

\begin{ax}  We assume  

(a) $\L^{\bf(1)}\cap C_\L(F)= C_\L(F)^{\bf (1)}$
and $\dim\, T<\infty$.

Since $H$ lies in $\L^{\bf(1)}$, the Cartan subalgebra
$H$ acts on $T$ and let
$$T=\oplus_{\alpha} T^{(\alpha)}$$
be its weight decomposition. Set
$\Delta^-(T)=\{\alpha\in\Delta^-\mid T^{(\alpha)}\neq 0
\}.$

(b) The set $\Delta^-(T)$ consists of $m$ linearly
independent roots $\beta_1,\cdots,\beta_m$.

(c) For any $\beta\in \Delta^-(T)$ the space
 $T^{(\beta)}$ has dimension $1$ and 
$C_\L(F)^{\bf (1)}$  acts by multiplication by 
$\chi(\beta,\delta^-)$, for some $\delta^-\in\C$.
\end{ax}

Let $(\mu,\delta)\in H^*\times \C$.
By Axiom 3(a) we have 
$C_\L(F)^{\bf (1)}=C_{\L^{\bf(1)}}(F)$.
Hence there exist a unique simple
$\L^{\bf(1)}$-module $L(\mu,\delta)$
with highest weight $\mu$ such that
$C_\L(F)^{\bf (1)}$ acts on 
$L(\mu,\delta)^{(\mu)}$
by the character $\chi(\mu,\delta)$.
We also assume that the $1$-dimensional space
$\L(\mu,\delta)^{(\mu)}$ has parity $m\hbox{ modulo } 2$.

We now assume that
Axioms (1-3) are satisfied and we set
$$\omega=-\sum_{i=1}^m\,\beta_i.$$

\begin{lemma} \label{lwInd} Assume that
$L(\mu,\delta)$ is finite dimensional and set
$$M:=\Ind_{\L^{\bf (1)}}^\L\,L(\mu,\delta)^*.$$

The lowest weight of the $\L$-module $M$
 is $-\mu-\omega.$
Moreover there is an isomorphism of $C_\L(F)$-modules

$$M^{(-\mu-\omega)}
\simeq \Ind_{C_\L(F)^{\bf (1)}}^{C_\L(F)}\,\,
\chi(-\omega-\mu,m\delta^--\delta).$$
\end{lemma}

\begin{proof} Set  $S=L(\mu,\delta)^*$.
Choose a nonzero $v\in S^{(-\mu)}$ of weight $-\mu$ and
some elements
$$\eta_1\in \L^{(\beta_1)}\setminus (\L^{\bf (1)})^{(\beta_1)},
\cdots,\eta_m\in \L^{(\beta_m)}
\setminus (\L^{\bf (1)})^{(\beta_m)}.$$

 \noindent We claim  that 
$$a.\eta_1\cdots\eta_m v=
\chi(-\omega-\mu,m\delta^--\delta)(a)
\eta_1\cdots\eta_m v,$$
for any $a\in C_\L(F)^{\bf (1)}$.

Let $G\subset \L$ be the subalgebra generated by
the subspaces $\L^{(\beta_i)}$ for $1\leq i\leq m$
and set $G^{\bf (1)}=G\cap \L^{\bf (1)}$. Since the roots
$\beta_1,\cdots,\cdots \beta_m$ are linearly independent,
we have

$$\ad(\L^{(\beta_{i_1})})\circ\ad(\L^{(\beta_{i_2})})\circ
\cdots \circ\ad(\L^{(\beta_{i_k})})(\L^{(\beta_{i_{k+1}})})\subset \L^{\bf (1)}$$
for any $k\geq 1$. It follows that $G^{\bf (1)}$ is an ideal
of $\G$ and that
$G/G^{\bf (1)}$ is an abelian Lie superalgebra.

Set $\chi_i=\chi(\beta_i,\delta^{-})$,
$\chi_0=\chi(-\mu,-\delta)$ and
$\chi=\chi(-\omega-\mu,m\delta^--\delta)$.
Let $a\in C^{\bf(1)}$. By definition we have
$$a.v=\chi_0(a)v$$
and by 
Axiom {(\bf AX3)}(c), we have
$$[a,\eta_i]=\chi_i(a) \eta_i +a_i,$$
for some $a_i\in G^{\bf (1)}$. 
Since $\chi(a)=\chi_0(a)+\sum_{i=1}^m \chi_i(a)$,
we get

$$a.\eta_1\cdots\eta_m v=\chi(a)\eta_1\cdots\eta_m v+\sum_{i=1}^m 
\eta_1\cdots\eta_{i-1} a_i \eta_{i+1}\cdots\eta_m .v$$
 Since $G^{\bf (1)}$ is an ideal,
 
 $$\eta_1\cdots\eta_{i-1} a_i \eta_{i+1}\cdots\eta_m
\in U(G)G^{\bf (1)}U(G)= U(G)G^{\bf (1)}.$$

\noindent However any element in $G^{\bf (1)}$ is a sum of
negative root vectors. Thus we have $G^{\bf (1)}.v$.
It follows that
$$a.\eta_1\cdots\eta_m v=
\chi(a)
\eta_1\cdots\eta_m v,$$
which proves the claim.

Set 
$$T^{\pm}=\oplus_{\alpha\in\Delta^\pm}
T^{(\alpha)}.$$
The hypotheses imply that
$T^{(0)}\simeq W/W^{\bf(1)}$.

First consider the case $W=\Vir$.
In that case $T^{(0)}=\C \ell_0\hbox{ modulo }\L^{\bf(1)}$.
Thus by Poincar\'e-Birkhoff-Witt there is an isomorphism
of $H$-modules
$$M\simeq \C[\ell_0]\otimes \wedge T^+\otimes\wedge T^-\otimes S.$$

\noindent Therefore the lowest weight of  $M$ is
$-\omega-\mu$ and we have

$$
M^{(-\omega-\mu)}\simeq \C[\ell_0]\otimes 
 \eta_1\cdots\eta_m\, v.$$
We observe that  $\eta_1\cdots\eta_m\, v$ is
an even vector. Since $C_\L(H)/C_\L(F)^{\bf (1)}\simeq \C \ell_0$ it follows that

$$M^{(-\mu-\omega)}
\simeq \Ind_{C_\L(F)^{\bf (1)}}^{C_\L(F)}\,
\chi.$$

Next assume that $W=\K_*(1)$ and choose 
 $\xi\in \K_*(1)_{\overline 1}\setminus \L^{(\bf 1)}_{\overline 1}.$
We have
$$T^{(0)}=\C \ell_0\oplus\C\xi\hbox{ modulo }\L^{\bf(1)}.$$
As before, we show that

$$M^{(-\omega-\mu)}= \C[\ell_0]\otimes 
\eta_1\cdots\eta_m\, v \oplus
\C[\ell_0]\,\xi \otimes \eta_1\cdots\eta_m\, v$$ 

\noindent and that $\C\eta_1\cdots\eta_m v$ is a one dimensional $C_\L(F)^{\bf (1)}$-module of character $\chi$.
It follows that
$$M^{(-\mu-\omega)}
\simeq \Ind_{C_\L(F)^{\bf (1)}}^{C_\L(F)}\,
\chi.$$
\end{proof}

We can now state the main result of the chapter.
Consider an arbitrary triple
$(\lambda,\delta,u)\in H^*\times\C\times\C$
such that the $\L^{\bf (1)}$-module 
$\L(\lambda,\delta)$ is finite dimensional.

Recall that the datum  
$$(\omega, m,\delta^-)$$
is defined  by
$\omega=-\sum_{\beta\in T^-}\,\beta$,
$m=\dim\, T^-$ and $C_\L(F)^{\bf (1)}$ acts on 
 each $T^{(\beta)}$ with the character 
 $\chi(\beta,\delta^-)$.

\begin{prop} \label{hwF} The weight $\mu+\omega$ is the highest weight of the $\L$-module 
 $\cF(L(\mu,\delta),u)$.  Moreover
there is an isomorphism of $C_\L(F)$-modules

$$\cF(L(\mu,\delta),u)^{(\mu+\omega)}\simeq 
\Tens(\mu+\omega, \delta-m \delta^-,u).$$
\end{prop}

\begin{proof} By Lemma \ref{duality}
the module
$$M:= \Coind_{\L^{\bf(1)}}^\L\,L(\mu,\delta)$$
is the dual of the module 
$$\Ind_{\L^{\bf(1)}}^\L\,L(\mu,\delta)^*.$$
Thus it follows from Lemma \ref{lwInd} that
$\mu+\omega$ is the highest weight of $M$ and that
$$M^{(\mu+\omega)}\simeq
 \Coind_{C_\L(F)^{\bf (1)}}^{C_\L(F)}
 \chi(\mu+\omega,\delta-m\delta^-).$$
 Thus it follows from Lemmas \ref{universalfg} and
 \ref{universalFG} that
 $$F_u M^{(\mu+\omega)}=\cF(L(\mu,\delta),u)^{(\mu+\omega)}\simeq 
\Tens(\mu+\omega, \delta-m \delta,u).$$
\end{proof}

\section{Representations of split extensions}
\label{split}

Let $\fg$ be the semi-direct product
$$\fg:=\Vir\ltimes H\otimes \C[t,t^{-1}],$$
where $H$ denotes a finite dimensional even vector 
space.

In Section \ref{fg} we have defined the $\fg$-module
$$\Tens(\lambda,\delta,u)$$
for an arbitrary triple
$(\lambda,\delta,u)\in H^*\times\C\times\C$. When
$\lambda\neq 0$ it is a obviously simple $\fg$-module.
In this chapter we prove the following
result:

\begin{thm}\label{Vir+} Let $V=\oplus_{i\in \Z} V_i$ be a simple $\Z$-graded
$\fg$-module of growth one. Then
\begin{enumerate}
\item[(a)] Either
$$V\simeq\Tens(\lambda,\delta,u)$$
for some  $\delta,u\in\C$ and 
some $\lambda\in H^*\setminus 0$,
\item[(b)] $(H\otimes \C[t,t^{-1}])V=0$.
\end{enumerate}
\end{thm}

\subsection{Generalities.}

We summarize some easy statements that are repeatedly used throughout the paper. Let $A=A_0\oplus A_1$ be a commutative superalgebra. By a derivation of
$A$ we mean a superderivation.

\begin{lemma}\label{dernil} Let $a\in A_0$ be an even element with $a^{n}=0$
and let $\partial_1,\dots,\partial_{2n-1}$ be derivations of $A$.
Then
$$(\partial_1 a)(\partial_2 a)\cdots(\partial_{2n-1} a)=0.$$
\end{lemma}

Consequently, if $A=A_0$ its radical $\Rad(A)$ is stable by any derivation. 

\begin{proof} For $x\in A$ we write $L_x\in\End(A)$ for the left multiplication by $x$.
For  an arbitrary linear transformation $\Delta:A\to A$,
set $\ad(a)(\Delta)=L_a\Delta-\Delta L_a $.
Since 
$$\ad(a)^{2n-1}(\Delta)=\sum_{i+j=2n-1}\,(-1)^i {
\binom{2n-1}{i}}
L_a^i\Delta L_a^j$$
we have $\ad^{2n-1}(a)(\Delta)=0$.

Furthermore if $\partial$ is a derivation of $A$, we have
$\ad(a)(\partial)=-L_{\partial a}$ and $\ad^2(a)(\partial)=0$.
It follows that
$$0=\ad(a)^{2n-1}(\partial_1 \partial_2 \cdots \cdots\partial_{2n-1})=- (2n-1)!\,
L_{\partial_1 a}L_{\partial_2 a}\cdots
L_{\partial_{2n-1} a}$$
which shows the formula.
\end{proof}

Let $V=\oplus_{i\in\Z} V_i$ be a $\Z$-graded 
$\fg$-module 
with all dimensions $\dim\,V_i$ uniformly bounded by some integer $d$. (Here we do not assume that $V$ is irreducible.) Abusing notation we use the same symbol
for an element  $a\in\fg$ and  the associated operator on $V$. 

\begin{lemma}\label{index} Let $h\in H$. If the operator 
$h(t)h(t^{-1})$ is  nilpotent, then
$$h(t)^{N(d)}V=0 \hbox{ and }\hskip1mm (h\otimes\C[t,t^{-1}])^{N'(d)} V=0$$
\noindent where $N(d)=9d-6$  and $N'(d)=18d-13$.
\end{lemma}

\begin{proof} 
Set $a=h(t)h(t^{-1})$ and $\partial=t^3\frac{\d}{\d t}$. 
As $\dim\,V_i$ is uniformly bounded by $d$, the operator $a$ is nilpotent of index $d$. By Lemma \ref{dernil}, the operator
$h(t)[\partial, a]$ is nilpotent of index $2d-1$.

Since
$$h(t)[\partial, a]=h(t)h(t^3)h(t^{-1})-h(t)^3= h(t^3) a-h(t)^3$$

\noindent $h(t)^3$ is the difference of two nilpotent elements
of indices $2d-1$ and $d$.  Hence $h(t)^3$ is nilpotent of index $3d-2$, thus
$$h(t)^{N(d)}=0,$$
for $N(d)=9d-6$.  Since
$$h\otimes\C[t,t^{-1}]=[\Vir, h(t)]$$
\noindent Lemma \ref{dernil} implies 
$$(h\otimes\C[t,t^{-1}])^{2N(d)-1}
V=0.$$\end{proof}

\subsection{Reduction to $\dim\, H=1$}

Assume now that  $V$ is a simple $\Z$-graded $\fg$-moduleof growth one.
In what follows we assume that
$$(H\otimes \C[t,t^{-1}])V\neq 0.$$

Let $R$ be the  commutative subalgebra of 
$\End_\C(V)$ generated by $H\otimes\C[t,t^{-1}]$. The algebra $R$ admits a $\Z$-grading $R=\oplus_{i}\,R_i$, where $R_i V_j\subset V_{i+j}$.

\begin{lemma}\label{zero} Let $h\in H$. If the operator 
$h(t)h(t^{-1})$ is  nilpotent, then

\centerline{$(h\otimes\C[t,t^{-1}])V=0$.}
\end{lemma}

\begin{proof} By Lemma \ref{index} 

\centerline{$(h\otimes\C[t,t^{-1}])^N V=0$,}

\smallskip
\noindent for some $N>0$. Since 
$\Ker (h\otimes\C[t,t^{-1}])$ is a nonzero $\fg$-submodule
we conclude that

\centerline{$(h\otimes\C[t,t^{-1}])V=0$.}
\end{proof}

We say
that an element $a\in R$ is {\it locally nilpotent} if for an abitrary $v\in V$ we have $a^{n(v)}v=0$ for some $n(v)\geq 1$.
For an arbitrary $a\in R$ the subspace

\centerline {$\{v\in V\mid a^{n(v)}v=0$ for some $  n(v)\geq 1\}$}

\smallskip
\noindent is a $\fg$-submodule of $V$. Hence  the operator $a$ is either locally nilpotent or  injective. Moreover, since
$\dim\, V_i\leq d$ for any $i\in\Z$, any operator $a\in R_0$ is either nilpotent of index $d$ or is invertible.

\begin{lemma}\label{K} There exists a codimension one subspace
$H'\subset H$ such that 
$$(H'\otimes \C[t,t^{-1}])V= 0.$$

Moreover for $h\notin H'$, the operator 
$h(t)h(t^{-1})$ is invertible.
\end{lemma}

\begin{proof} 
Set 
$$H'=\{k\in H\mid (k\otimes\C[t,t^{-1}])V=0\}$$
and let $h\notin H'$. By Lemma \ref{zero} $a:=h(t)h(t^{-1})\in R_0$ is not nilpotent. Hence $h(t)h(t^{-1})$ is invertible.

Let $h'\in H$. Since the operator $h(t):V_0\to V_1$ is bijective, there is $\lambda\in \C$ and a nonzero $v\in V_0$ such that
$(h'(t)-\lambda h(t))v=0$. It follows 
that 
$$(h'(t)-\lambda h(t))\,
(h'(t^{-1})-\lambda h(t^{-1}))$$

\noindent is nilpotent. Thus by Lemma \ref{zero} 
$(h'-\lambda h)\otimes\C[t,t^{-1}]$ acts trivially on $V$, i.e.
 $h'-\lambda h$ belongs to $H'$. Hence $H=H'\oplus\C h$ and
 $H'$ has codimension one. 
\end{proof}

\subsection{The algebra $R$}

In view of the previous considerations,
 we can assume without loss of generality that
$H=\C h$ has dimension one.
Since the operator $h(t)$ is invertible we have 
$\dim\, V_i=d$ for some integer $d$ independent of $i$.

\begin{lemma}\label{R} \begin{enumerate}

\item[(a)] The operator $h(t)^{-1}$ belongs to $R$.

\item[(b)] The algebra $R$ is isomorphic to 
$\C[h(t),h(t)^{-1}]$.
\end{enumerate}
\end{lemma}

\begin{proof} Let $a=h(t)h(t^{-1})$ and 
$b=a\vert_{V_0}$. 
Since $b$ is an invertible endomorphism of $V_0$, 
we have
$$b f(b)=\id_{V_0}$$
for some polynomial $f$. Since $V_i=h(t)^iV_0$ for any $i\in \Z$, it follows that $af(a)=\id$, which shows that 
$h(t)^{-1}$ belongs to $R$.

It follows that $$R= R_0[h(t),h(t)^{-1}].$$

We claim that $\Rad(R_0)=0$. 
Since
$\Rad(R_0)^d=0$ and 
$$\Rad(R)=\Rad(R_0)[h(t),h(t)^{-1}],$$
 we also have 
$\Rad(R)^d=0$.
By Lemma \ref{dernil} the radical $\Rad(R)$ is stable by $\Vir$, hence $\Ker\,\Rad(R)$ is a $\fg$-submodule.
Since $\Ker\,\Rad(R)\neq0$ we conclude that 
$V=\Ker\,\Rad(R)$, that is $\Rad(R_0)=0$, which proves the claim.

We also claim that any idempotent $a\in R_0$ is trivial.
Let $\partial\in\Vir$. We have
$$0=\partial (a^2-a)=(2a-1)\partial a.$$
Since $(2a-1)$ is invertible it follows that $\partial a=0$, that is $a$
commutes with $\fg$. Since $V$ is irreducible, we conclude that
$a=0$ or $1$.

The previous two claims implies that $R_0=\C$, which completes the proof.
\end{proof}

By definition the algebra $R$ is generated by
$h\otimes\C[t,t^{-1}]$ and let 
$$\rho:\C[t,t^{-1}]\to R$$ 
be the corresponding map.
Set $T=\rho(h(t)$. By Lemma \ref{R}, we have 
$R=\C[T,T^{-1}]$.

\begin{lemma}\label{hnot0} The operator $\rho(h)$ is nonzero.
\end{lemma}

\begin {proof} Assume otherwise. Since $[e_{-1}, h(t)]=-h$,
we have $[e_{-1},T]=0$. Since $e{-1}$ acts by derivation on $R$, we have $[e_{-1},R]=0$. Since
$$h\otimes\C[t,t^{-1}] = \C h(t^{-1})\oplus [e_{-1},
h\otimes\C[t,t^{-1}]]$$
we conclude that 
$$\rho(h\otimes\C[t,t^{-1}])\subset\C T^{-1}$$
which contradicts that $\rho(h\otimes\C[t,t^{-1}])$ generates the algebra $R$. 
\end{proof}

By the previous lemma we can also assume that $\rho(h)=1$. As before set  $T=\rho(h(t))$.

\begin{lemma}\label{formula} For any $n\in\Z$ we have
\begin{enumerate}
\item[(a)] $\rho(h(t^n))=T^n$
\item[(b)] $[e_n,f(T)]=-T^{n+1} \frac{\d}{\d} T f(T)$
\end{enumerate}
\end{lemma}

\begin{proof} Since $\rho(t^n)$ is homogenous 
of degree $n$, we have $\rho(t^n)=\lambda_n T^n$
for some $\lambda_n\in\C$. 

We prove by induction that $\lambda_n=1$ for all $n\geq 0$. By hypothesis, we have $\lambda_0=\lambda_1=1$.
We have $[e_{-1},T]=-h=-1$ hence 
$e_{-1}$ acts over $R$ as $-\d/\d T$.  Thus for any
$n>1$ we have
$$n\lambda_{n-1}T^{n-1}
=\rho(n h(t^{n-1}))
=\rho([-e_{-1},h(t^n)])
=\lambda_n\frac{\d}{\d T} T^n
=n\lambda_n T^{n-1}.$$
Hence $\lambda_n=\lambda_{n-1}$ which proves the claim by induction.

For an arbitrary $n,m \in \Z$ with $m\geq 0$ and $n+m\geq 0$
 we have

 $\hskip4mm[e_n, T^{m}]
 =\rho([e_n, h(t^m)])$

\hskip19.5mm$ =(n-m)\rho(h(t^{n+m})$
 
\hskip19.5mm$=(n-m)T^{n+m}$,

\noindent hence $e_n$ acts over $R$ as
$-T^{n+1}\d/\d T$ which shows Assertion (b).

Moreover  $h(t^n)=-[e_{n-1},h(t)]$
so
$\rho(h(t^n))=T^n$ for any integer $n\in\Z$
which proves Assertion (a).
\end{proof}

\subsection{The Maurer-Cartan equation}

Let $A$ be a commutative algebra and let $\fs$ 
be a Lie algebra of derivations of $A$.
By a $(\fs,A)$-module we mean a $A$-module $V$
endowed with an action of the Lie algebra 
$\fs$ such that
$$\partial(av)=(\partial f)v+f(\partial v)$$
for any $\partial \in\fs$, $f\in A$ and $v\in V$.

Furthermore assume that $V$ is free as an $A$-module that is $V=A\otimes W$ for some vector space $W$. For any $\partial\in\fs$ define an endomorphism $c(\partial)\in A\otimes\End(W)$ by the formula
$$c(\partial)(f\otimes w)=f\partial w$$

\noindent for any $f\in A$ and $w\in W$. 
Then the hypothesis that $V$ is a $\fs$-module is equivalent to 
$$[c(\partial),c(\partial')]=
c([\partial,\partial']) -\partial c(\partial')
+\partial'c(\partial),$$
\noindent which is called the {\it Maurer-Cartan equation}.

Let $V$ be a $\Z$-graded $(\Vir,\C[t,t^{-1}])$-module.  Hence
$$V=\C[t,t^{-1}]\otimes V_0$$
and we can define the cocycle 
$c:\Vir\to \C[t,t^{-1}]\otimes\End(V_0)$ as before.
Since $c(e_n)$ is homogenous of degree $n$, we have
$$c(E_n)=t^n\otimes a_n$$

\noindent for some $a_n\in\End(V_0)$. 
Then the Maurer-Cartan equation is equivalent to
\begin{equation}\label{bracket}
[a_n, a_m]=(n-m) a_{m+n}-na_n +m a_m.
\end{equation}

\noindent for all $n,m\in\Z$.

Consider now the  algebra  $\cG$  
 with a basis
$(a_n)_{n\in \Z}$ with a bracket defined by
Equation \ref{bracket}. It is easy to check that $\cG$ is a Lie algebra. We have just proved the following result.

\begin{lemma}\label{correspondence}
The functor $V\mapsto V_0$ is a bijective correspondence from the class of 
$\Z$-graded $(\Vir,\C[t,t^{-1}])$-modules and the class of $\cG$-modules.
\end{lemma}

\subsection{The Maurer-Cartan Lie algebra $\cG$}
\label{MCalgebra}
Let $\cG^+$ (respectively $\cG^-$ )
be the span of $(a_n)_{n\geq -1}$  (respectively
the span of $(a_n)_{n\leq 1}$). It is easy to see
that $\cG^+$ and $\cG^-$ are Lie subalgebras
in $\cG$.
 
 For $n>0$, set
 $$\Delta(n)=\sum_{k=-1}^n\,(-1)^k\binom{n+1}{k+1} \,a_{k}$$
$$\Delta(-n)=\sum_{k=-1}^n\,(-1)^k
\binom{n+1}{k+1} \,a_{-k},$$

\noindent and let $B^\pm$ be the sets
$$B^\pm=\{\Delta(\pm(n)\mid n>0\}.$$
 
\begin{lemma}\label{MCrelations} 
 \begin{enumerate}
 
 \item[(a)] Let $n>0$. We have
 \begin{equation}\label{delta+}
 [a_{-1},\Delta(n)]=n\Delta(n)
 \end{equation}
 \begin{equation}\label{delta-}
 [a_{1},\Delta(-n)]=-n\Delta(-n).
 \end{equation}

 \item[(b)] An element $g=\sum_k x_k a_k$ belongs to $[\cG,\cG]$ if and only if
\begin{equation}
\sum_k x_k=0\hbox{ and }\sum_k k x_k=0
\end{equation}
 
 \item[(c)] The set $B^+$ is a basis of 
 $[\cG^+,\cG^+]$ and $B^-$ is a basis of 
 $[\cG^-,\cG^-]$.
 
 \item[(d)] We have 
  $$[\cG,\cG]=[\cG^+,\cG^+]+[\cG^-,\cG^-].$$
  \end{enumerate}
 \end{lemma}

 \begin{proof}
The Assertions (a) is easy to prove.

 Set $X:=\C a_0\oplus \C(a_1-a_{-1})$.  For
 $g=\sum_{k} x_k a_k\in[\cG,\cG]$ it is clear that
\begin{equation}\label{GG}\sum_k x_k=0\hbox{ and }\sum_k k x_k=0 
\end{equation}

\noindent It follows that 

\begin{equation}\label{codim2}
X\cap [\cG,\cG]=0.
\end{equation}

 It is easy to see that the set
$\{a_0, a_1-a_{-1}\}\cup B^+$
is a basis of $\cG^+$. 
By Equation (\ref{delta+}) $B^+$ lies in
$[\cG^+,\cG^+]$ and by Equation (\ref{codim2})
we have 
$$\dim\, \cG^+/[\cG^+,\cG^+]\geq 2.$$ 
Thus it follows that  $\cG^+/[\cG^+,\cG^+]$ has dimension exactly two and that $B^+$ is a basis of 
$[\cG^+,\cG^+]$.
The proof that $B^-$ is a basis of 
$[\cG^-,\cG^-]$ is identical. Thus we have proved Assertion (c).

Equation $\ref{codim2}$ shows that
$\dim\, \cG/[\cG,\cG]\geq 2$.
Note that 
$$\{a_0, a_1-a_{-1}\}\cup B^+\cup B^-$$
is a basis of $\cG^+$. (Observe that 
$\Delta(1)=\Delta(-1)$.) Since each element
$\Delta(n)$ lies in $[\cG,\cG]$ we deduce that
 $[\cG,\cG]$ has codimension exactly two.
It has been shown that any element 
$$g=\sum_{k} x_k a_k\hbox{ in }[\cG,\cG]$$ 

\noindent satisfies Equation (\ref{GG}). Hence the codimension two subspace defined by
this equation is precisely $[\cG,\cG]$ which proves 
Assertion (b).
 
 It follows that $B^+\cup B^-$ is a basis of
$[\cG,\cG]$. Since each basis element $\Delta(n)$
belongs to $[\cG^+,\cG^+]$ or  $[\cG^-,\cG^-]$,
Assertion (d) follows.
\end{proof}

\begin{lemma}\label{solvable}
Let $W$ be a finite dimension Lie vector space and let 
 $$\rho:\cG \rightarrow \fgl(W)$$ be a Lie algebra morphism.
 Then  $\rho(\cG)$ is a solvable Lie algebra.
 \end{lemma}

 \begin{proof} By Equation \ref{delta+} we have
 $$[\rho(a_{-1}),\rho(\Delta(n))]=n\rho(\Delta(n))\hbox{ for any } n\geq 1.$$
   Since $\ad(\rho(a_{-1}))$ has a finite spectrum, 
 the nonunital algebra generated by all $\rho(\Delta(n))$
 for $n\geq 1$ is nilpotent. By 
 Lemma \ref{MCrelations}(c) the set $\{\Delta(n)\mid n>0\}$ is a basis of $[\cG^+,\cG^+]$, thus
 the Lie algebra
 $\rho([\cG^+,\cG^+])$ is nilpotent. Similarly the
 Lie algebra  $\rho([\cG^-,\cG^-])$ is nilpotent.
 
 By Lemma \ref{MCrelations}(d) we have
 $$\rho([\cG,\cG]=
 \rho([\cG^+,\cG^+])+\rho([\cG^-,\cG^-]).$$
 However it is shown in \cite{JacobsonLie} that a 
 finite dimensional Lie
 algebra which is sum, as a vector space, of two nilpotent subalgebras is solvable. Hence
  $\rho([\cG,\cG])$ is solvable, which implies that
  $\rho(\cG)$ is solvable.
  \end{proof}

 \begin{lemma}\label{dim1}
Let $\rho:\cG \rightarrow \fgl(W)$ be a simple
 finite dimensional representation of $\cG$. Then $V$ has
 dimension $1$. 
 Moreover there are $\delta,u\in\C$ such that
 
 $$\rho(a_n)=-u-n\delta\hbox{ for any }n\in\Z.$$
 \end{lemma}
 
\begin{proof} The first assertion follows 
immediately from Lemma \ref{solvable}.
Thus $\rho$ can be identified with a linear form
$\rho:\cG\to\C$ with $\rho([\cG,\cG])=0$. 
For any $\delta,u\in\C$ it is easy to check
that the linear form $\rho_{\delta,u}$ with 
$\rho_{\delta,u}(a_n)=-u-n\delta$ vanishes on 
$[\cG,\cG]$. 

By Lemma \ref{MCrelations}(b) $\cG/[\cG,\cG]$ has dimension two, so any linear form on $\cG/[\cG,\cG]$ is one of the $\rho_{\delta,u}$.
\end{proof}

\subsection{Proof of Theorem \ref{Vir+}}

\begin{proof}
Recall that $\fg=\Vir\ltimes H\otimes\C[t,t^{-1}]$.
Let $V=\oplus_{i\in \Z} V_i$ be an irreducible
$\Z$-graded $\fg$-module of growth one. We assume that
$$(H\otimes\C[t,t^{-1}])V\neq 0,$$
otherwise the Assertion follows from Theorem \ref{Vir+}.

Furtermore by Lemma \ref{K} we can assume that 
$H=\C h$ is one dimensional. By Lemma \ref{hnot0}
the operator $h\vert_V$ is a nonzero scalar
so we can assume that $h\vert_V=1$.
Lemma \ref{formula} means that $V$ is a $(\Vir,\C[t,t^{-1}])$-module.

By Lemma \ref{correspondence}, $V_0$ is an irreducible $\cG$-module. Thus by Lemma
\ref{dim1} $V_0$ has dimension 1. In addition
Lemma \ref{dim1} shows that each $a_n$ acts as 
$-u-n\delta$ for some $\delta,u\in\C$.
It follows easily that 
$$V\simeq \Tens(\lambda,\delta,u)$$
where
$\lambda(h)=1$.
This completes the proof of Theorem \ref{Vir+}.
\end{proof} 

\subsection{Three corollaries}

We now state corollaries, which are  used 
in the next chapters.  
As before, let 
$$\fg=\Vir\ltimes(H\otimes \C[t,t^{-1}])$$
where $H=\C h$ is one-dimensional. Let $V$ be a graded
$\fg$-module of growth one.
Here we drop the assumption that $V$ is irreducible.
Set 
$$d=\Max_{n\in\Z}\,\dim\, V_n.$$

\begin{cor}\label{corh=0} Assume that $h\vert_V=0$.
Then
$$h(t)^{N(d)} V=0.$$
\end{cor}

\begin{proof} We claim that $h(t)h(t^{-1})$ is nilpotent.
Assume otherwise. Then there is a simple subquotient of $V$
for which $h(t)h(t^{-1})$ is not nilpotent, which contradicts
Lemma \ref{hnot0} which ensures that $h\neq 0$.

Therefore $h(t)h(t^{-1})$ is nilpotent. The corollary follows from Lemma \ref{index}.

\end{proof}

\begin{cor}\label{corh=1} Assume that $(h-1)\vert_V$
is nilpotent.
Let $I\subset U(h\otimes \C[t,t^{-1}])$ be the linear span of
the set $\{h(t^{n+m})-h(t^m)h(t^n)\mid n,m\in\Z\}$.
\begin{enumerate}
\item[(a)] The operator $h(t)$ is invertible.
\item[(b)] We have $I^d V=0.$
\end{enumerate}
\end{cor}

\begin{proof} Let $V'$ be an arbitrary simple subquotient of $V$. Lemma \ref{zero} implies that 
the operator $h(t)\vert_{V'}$ is invertible and therefore 
$h(t)\vert_V$ is also invertible which proves
Assertion (a).

It follows that all homogenous components $V_i'$ have the same dimension. Hence
the $\fg$-module $V$ has length $\leq d$.
By Lemma \ref{formula}
the elements
$h(t^{n+m})-h(t^m)h(t^n)$ act trivially on $V'$. Therefore $I^d V=0$.
\end{proof}

\begin{cor}\label{finitespec}
The operator $h\vert_V$ admits only finitely many
generalized eigenvalues.
\end{cor}

\begin{proof} Let $x$ be a nonzero generalized eigenvalue
of the operator $h\vert_V$ and let $V^{(x)}$ be the corresponding generalized eigenspace. By 
Corollary \ref{corh=1}, the operator $h(t)$ is invertible.
Therefore $\dim\, V^{(x)}_n\geq 1$ for all $n\in \Z$.
Thus $h\vert_V$ admits at most $d$ nonzero generalized eigenvalues, which proves the corollary.
\end{proof}

 \section{Representations of split extensions II}\label{splitII}

Let $\fG_*$ be the semi-direct product
$$\fG:=\K_*(1)\ltimes H\otimes \C[t,t^{-1},t^\tau\xi],$$
where $H$ denotes a finite dimensional even vector 
space and where $\tau=0$ if $\K_*(1)=\K(1)$ and
$\tau=\frac{1}{2}$ if $\K_*(1)=\K_{NS}(1)$.
In Section \ref{fG} we have defined the $\fG_*$-module
$$\Tens(\lambda,\delta,u)$$
for an arbitrary triple
$(\lambda,\delta,u)\in H^*\times\C\times\C$. When
$\lambda\neq 0$ it is a obviously simple $\fG_*$-module.
In this chapter we prove the following
result:

\begin{thm}\label{K+} Let $V=\oplus_{i\in \Z} V_i$ be a simple $\Z$-graded
$\fG_*$-module of growth one. Then
\begin{enumerate}
\item[(a)] Either
$$V\simeq\Tens(\lambda,\delta,u)$$
for some  $\delta,u\in\C$ and 
some $\lambda\in H^*\setminus 0$,
\item[(b)] $(H\otimes \C[t,t^{-1},t^\tau\xi])V=0$.
\end{enumerate}
\end{thm}

The proof follows the same scheme as the proof of 
Theorem \ref{Vir+} with the addition of 
the universal property stated in Lemma \ref{universalFG}.

\subsection{Reduction to $\dim\, H=1$}

Let $V$ be a $\fG$-module as before and let
$d$ be the upper bound of
$ \dim\, V_i:=\dim\, V_{i,\overline{0}}
+\dim\, V_{i,\overline{0}}$, when $i$ runs over $(1/2)\Z$.
In what follows we assume that
$$(H\otimes \C[t,t^{-1},t^\tau\xi])V\neq 0.$$

\begin{lemma}\label{superzero} Let $h\in H$. If the operator 
$h(t)h(t^{-1})$ is  nilpotent, then

\centerline{$(h\otimes\C[t,t^{-1},t^\tau\xi])V=0$.}
\end{lemma}

\begin{proof} By Lemma \ref{index} 
$(h(t))^{N(d)} V=0$,
for some $N(d)>0$. We have
$$h\otimes\C[t,t^{-1},t^\tau\xi]=[\fG,h(t)]$$
We deduce from Lemma \ref{index} that 
$$(h\otimes\C[t,t^{-1},t^\tau\xi])^{2N(d)-1}V=0$$
Since 
$\Ker\, (h\otimes\C[t,t^{-1},t^\tau\xi])$ is a nonzero $\fG$-submodule
we conclude that
$$(h\otimes\C[t,t^{-1},t^\tau\xi])V=0.$$
\end{proof}

\begin{lemma}\label{superK} There exists a codimension one subspace
$H'\subset H$ such that $(H'\otimes \C[t,t^{-1},t^\tau\xi])V= 0$.

Moreover for $h\notin H'$, the operator 
$h(t)h(t^{-1})$ is invertible.
\end{lemma}

Except that we use Lemma \ref{superzero} instead of Lemma \ref{zero}, the proof is the same as the proof of Lemma \ref{K}.

\subsection{The algebra $R$}

Without loss of generality, we can assume that
$H=\C h$ has dimension one. Set
$$\fm:=h\otimes\C[t,t^{-1},t^\tau\xi]$$
and let $R$ be the 
subalgebra of $\End_\C(V)$ generated by $\fm$.
The superalgebra $R$ is commutative and admits a $\Z$-grading $R=\oplus_{i}\,R_i$, where $R_i V_j\subset V_{i+j}$.

Since the operator $h(t)$ is invertible we have 
$\dim\, V_i=d$ for all $i\in\Z$, for some $d$.

\begin{lemma}\label{superR1} \begin{enumerate}

\item[(a)] The operator $h(t)^{-1}$ belongs to $R$.

\item[(b)] The  radical $\Rad(R)$ of the algebra $R$ is nilpotent of index $d$.

\item[(c)]  The operator $h\vert_V$ acts as some scalar $c\neq 0$. 
\end{enumerate}
\end{lemma}

\begin{proof} The proof of Assertion (a) is the same as the proof of 
Lemma \ref{R}(a). 

It follows that $R=R_0[h(t),h(t)^{-1}]$
and  $\Rad(R)=\Rad(R_0)[h(t),h(t)^{-1}]$.
Since $\Rad(R_0)^d=0$, the radical
of $R$ is nilpotent of index $d$.

Since $h$ lies in the center of $\fG$, it
acts on $V$ by some scalar $c$. By lemma
\ref{superK} the operator $h(t)h(t^{-1})$ is invertible. Hence by Corollary \ref{corh=0}
we have $c\neq 0$ which completes the proof.
\end{proof}

By definition the action of 
$\fm=h\otimes\C[h(t),h(t)^{-1},t^\tau\xi]$ provides an algebra homorphism 
$\rho:U(\fm)\to R$.
By the previous lemma we can assume that $\rho(h)=1$.

\begin{lemma}\label{superR2}  The algebra $R$ is isomorphic to $\C[t,t^{-1},t^\tau\xi]$.
Relative to this isomorphism we have
$$\rho(h(t^m))=t^m\hbox{ for } m\in\Z\hbox{ and }$$
$$\rho(h(t^m\xi))=t^m\xi{ for } m\in\tau+\Z.$$ 
\end{lemma}

\begin{proof} Consider the linear subspaces
$A, B, C$ in $U(\fm)$
which are defined by
$$A:=\C(h-1)$$
$$B:=\langle h(t^{n+m})-h(t^n)h(t^m)\mid n,m\in\Z\rangle$$
$$C:=\langle h(t^{n+m}\xi)-h(t^n)h(t^m\xi)\mid n\in \Z
\hbox{ and }  m\in\tau+\Z\rangle$$
$$D:=\langle h(t^n\xi)h(t^m\xi)\mid n,m\in\tau+\Z\rangle,$$
where the notation 
$\langle X\rangle$ 
stands for the linear span of an arbitrary set $X$. Set $T=A\oplus B\oplus C\oplus D$.

We show that $T$ is a $\K(1)$-module. Let 
$$\rho':U(\fm)\to 
\C[h(t),h(t)^{-1},t^\tau\xi])$$
be the algebra morphism defined by 
$\rho'(h(t^m))=t^m$ and $\rho'(h(t^n\xi))=t^n\xi$
for  arbitraries $m\in\Z$ and $n\in\tau+\Z$. Also set

\smallskip
$$U_2(\fm)=\C\oplus \fm
\oplus S^2(\fm).$$

\noindent Since obviously 
$$T=\Ker\,\rho'\cap U_2(\fm)$$
it follows that $T$ is a $\K(1)$-module.

We claim that $\rho(T)=0$.  Obviously $\rho(A)=0$.
It follows from Corollary \ref{corh=1}
that $\rho(B)$ lies in $\Rad(A)$. Moreover 
$\rho(C\oplus D)$ lies in 
$R_{\overline 1}R\subset \Rad(A)$ which 
shows that $\rho(T)$ lies in the radical of $R$.
Hence by Lemma \ref{superR1} we have
$\rho(T)^d=0$. Since $\Ker\, \rho(T)$ is a nonzero
$\fG$-submodule, we have $\rho(T)=0$ which proves the claim.
  
  It is clear that
$$\C[h(t),h(t)^{-1},t^\tau\xi]= 
U(\fm)/ T U(\fm)$$ 

\noindent therefore the morphism $\rho$ factors through $\C[h(t),h(t)^{-1},t^\tau\xi]$.  Since $\rho(h)=1$, it follows that the morphism 
$$\C[h(t),h(t)^{-1},t^\tau\xi]\to R$$ 
is injective. We conclude that $R=\C[t,t^{-1},t^\tau\xi]$.

The second assertion follows from the proof.
 \end{proof}

\subsection{The  subalgebra $\K_*(1)^{({\bf 1})}$}

Recall that $\K_*(1)$ is the Ramond or the Neveu-Schwarz
superalgebra and that $\tau=0$ in the first case and $\tau=1/2$ in the second case.

Let ${\bf m}\subset \C[t,t^{-1},\xi]$ be the
codimension one ideal generated by $(1-t)$ and $\xi$.
We observe  that 
$$\K_*(1)^{\bf (1)}:=\{\partial\in \K_*(1)\mid 
\partial(\C[t,t^{-1},\xi])\subset {\bf m}\}$$

\noindent is the subalgebra of derivations 
vanishing at the point $\bf m$. We now determine the 
simple finite dimensional $\K_*(1)^{\bf (1)}$-modules.

\noindent Let
 $\cG$ be the Maurer-Cartan Lie algebra
defined in  Section \ref{MCalgebra} and set  
$$\cG'=\{g=\sum_n x_n a_n\mid \sum x_n=0\}.$$ 

\noindent
Let $\psi: \cG\to \K(1)$ be the  linear map 
defined by $\psi(a_n)= E_n$ for any $n\in\Z$.

\begin{lemma}
\label{nocompute} \begin{enumerate}
\item[(a)] An element 
$$g=\sum_n x_n E_n
+\sum_m y_m G_m$$
belongs to
$\K_*(1)^{{\bf (1)}}$ if and only if
$$\sum_n\, x_n=0\hbox{ and }\sum_m\, y_m=0.$$

\item[(b)] The linear map $\psi$ is not an algebra homomorphism, but its restriction
$\psi:\cG'\to \K_*(1)^{\bf(1)}_{\overline 0}$
is a Lie algebra isomorphism. 

\item[(c)] The element $a_0$ is central and
$\cG=\cG'\oplus\C a_0$ as a Lie algebra.
Consequently
$$[\cG',\cG']=[\cG,\cG].$$

\item[(d)] We have 
$$[\K_*(1)^{\bf (1)}_{\overline 1},\K_*(1)^{\bf (1)}_{\overline 1}]
\subset
[\K_*(1)^{\bf (1)}_{\overline 0},\K_*(1)^{\bf (1)}_{\overline 0}].$$
\end{enumerate}
\end{lemma}

\begin{proof} Assertion (a) follows from the explicit formulas of Section \ref{fG}.

By definition $\cG'$ is spanned by the vectors 
$a_n-a_0$ and by Assertion (a) 
$\K_*(1)^{\bf (1)}_{\overline 0}$ is spanned by the vectors
$E_n-E_0$. For any $n,m\in\Z$ we have

\hskip15mm$[a_n-a_0,a_m-a_0]=[a_n,a_m]$

\hskip48mm$=(n-m)a_{n+m}-na_n+ma_m$

\hskip12mm$
[E_n-E_0,E_m-E_0]=[E_n,E_m]+[E_0,E_n]-E_0,E_m]$

\hskip48mm$=(n-m)E_{n+m}-nE_n+mE_m$

\noindent which proves Assertion (b).

Since Assertion (c) is obvious, we skip its proof. 

By Assertion (a)
$\K_*(1)^{(1)}_{\overline 1}$ is spanned by the vectors
$G_i-G_j$, $i,j\in\tau+\Z$. For any 
$i,j,k,l\in \tau+\Z$ we have
$$[G_i-G_j,G_k-G_l]
=E_{i+k}+E_{j+l}-E_{i+l}-E_{j+k}.$$

\noindent By Lemma \ref{MCrelations}(b) the element
$[G_i-G_j,G_k-G_l]$ belongs to 
$\psi([\cG,\cG])$. By Assertion (c)
$[\cG,\cG]=[\cG',\cG']$ hence
$[G_i-G_j,G_k-G_l]$ belongs to 
$$\psi([\cG',\cG'])=[\psi(\cG'),\psi(\cG')]=
[\K_*(1)^{\bf (1)}_{\overline 0},\K_*(1)^{\bf (1)}_{\overline 0}]$$
which completes the proof of the Lemma.
\end{proof}

For $\delta\in\C$, let $S_\delta$ be the  one dimensional $\K_*(1)^{\bf (1)}$-module such that
the elements $E_n-E_m$ act as $(m-n)\delta$
and $\K_*(1)^{\bf (1)}_{\overline 1}$ acts trivially.

\begin{lemma}\label{superLie}
Let $S$ be a simple
 finite dimensional $\K_*(1)^{\bf (1)}$-module. Then 
$$S\simeq S_\delta$$ 

\noindent for some $\delta\in\C$.
 \end{lemma}

\begin{proof}
Let $\fp$ be the image of $\K_*(1)^{\bf (1)}$
in $\fgl(S)$. 
By Lemma \ref{nocompute}(c), 
$$\K_*(1)^{\bf(1)}_{\overline 0}\oplus \C a_0\simeq\cG,$$ hence $\fp_{\overline 0}$ can be realized as  a quotient of $\cG$.  Thus by Lemma \ref{solvable} the
Lie algebra $\fp_{\overline 0}$ is solvable. By Kac classification of finite dimensional 
simple superalgebras \cite{Kac77}, the
superalgebra $\fp$ is solvable.

By Lemma \ref{nocompute}(d), we have 
$$[\fp_{\overline 1}, \fp_{\overline 1}]
\subset [\fp_{\overline 0}, \fp_{\overline 0}].$$
Thus by \cite{Kac77} any simple $\fp$-module 
$S$ has dimension one. 

In particular
$\fp_{\overline 1}$ acts trivially on $S$ and
the element $E_1-E_0$ acts by some scalar $\delta$.
Thus it follows from Lemma \ref{MCrelations}(b) that the elements $E_n-E_m$ act as $(n-m)\delta$. Therefore
$$S\simeq S_\delta.$$
\end{proof}

\subsection{Proof of Theorem \ref{K+}}

\begin{proof}
Let $V=\oplus_{i\in (1/2)\Z} V_i$ be an irreducible
$\fG$-module of growth one. 

We can assume that $(H\otimes\C[t,t^{-1},\xi])V\neq 0$.
By Lemma \ref{superK} we can assume that 
$H=\C h$ has dimension one and that $h$ acts as one.
By Lemma \ref{superR2}, the algebra 
$R$ generated by $h\otimes\C[t,t^{-1},\xi])$ is isomorphic to $\C[t,t^{-1},\xi]$. 

Since $V$ is a free module $\C[t,t^{-1}]$-module of finite rank, we conclude that
$S:=V/{\bf m}V$ is finite dimensional. Moreover since
$\xi^2=0$, the space $S$  is not zero.

For $\delta\in\C$, we consider the $\K_*(1)^{\bf 1)}$-module
$S_\delta$ as a $\fG_*^{({\bf(1)})}$ on which
each $h(t^m)$ acts as $1$.
Observe that $h(t^m)$ acts by $1$ on $S$. Hence
by Lemma \ref{superLie}, the 
$\fG_*^{({\bf(1)})}$-module $S$ has a quotient isomorphic to
$ S_\delta$ for some $\delta\in\C$. Up to
parity-change, we can assume that this quotient is even.

Thus there is an embedding
$$V\subset \Coind_{\fG_*^{\bf(1)}}^{\fG_*}\,  S_\delta.$$

\noindent Let $u$ be one eigenvalue of $-E_0$ on 
$V_{\overline 0}$. Since $V$ is simple, the  eigenvalues of $-E_0\vert_{V_{\overline 0}}$ lie in
$u+\Z$ and the eigenvalues of
$-E_0\vert_{V_{\overline 1}}$  lie $u+\tau+\Z$.
Therefore 
$$V\subset F_u \Coind_{\fG_*^{\bf(1)}}^{\fG_*} \, S_\delta.$$
Thus by Lemma \ref{universalFG}  there is an embedding
$$V\subset\Tens(\lambda,\delta,u)$$
from which we deduce Theorem \ref{K+}.
\end{proof}

\section{Projective cuspidal modules}\label{projcusp}

Let $\L$ be a superconformal algebra
and let 
$$0\to\C c\to \hat{\L}\to \L\to 0$$
be a nonsplit central extension. By definition
a $\hat{\L}$-module $M$ has {\it central charge} $x$ if 
$c\vert_M=x$. Also by a 
{\it projective} $\L$-module we mean 
an $\hat \L$-module with a non-trivial central charge, for some nonsplit central extension $\hat \L$ of $\L$.

The nonsplit central extensions of $\L$ are
classified by 
$$\P H^2(\L):=(H^2(\L)\setminus 0)/\C^*.$$
For $\L\neq \K(4)$,
V.G. Kac and J.W. van de Leur have shown that 
$$\dim\,H^2(\L)\leq 1,$$
thus $\L$ admits at most one central extension. In the $3$-dimensional space
$H^2(\K(4))$  we select one specific cohomology
class $\psi\neq 0$ and we define  
$\widehat{\K(4)}$ as the corresponding  central extension of $\K(4)$.

The  result of the chapter is the following.

\begin{thm}\label{centralcharge} Let $\L$ be one of the known superconformal algebra. Then
\begin{enumerate}
\item[(a)] A superconformal algebra $L\neq \K(4)$
does not admits any cuspidal projective module.
\item[(b)] If $\hat L$ is a nonsplit central extension of $\K(4)$
but $\hat \L\neq \widehat{\K(4)}$, the central charge of any
cuspidal $\hat \L$-module is zero.
\item[(c)] The superconformal algebra $\widehat{\K(4)}$
admits cuspidal modules with an arbitrary  central charge.
\end{enumerate}
\end{thm}

The Theorem is an immediate consequence of Corollaries \ref{YES} and \ref{NO} proved in this chapter.

Among all nontrivial second cohomology classes of all superconformal algebras,
the class $[\psi]$ is the only one which can be represented by a $2$-cocycle of finite rank.
This property plays a role in the proof.

\subsection{A criterion for the existence of projective cuspidal modules}

Let $\L$ be a $\Z$-graded superconformal algebra  and let
$$0\to\C c\to \hat{\L}\to \L\to 0$$
be a nonsplit central extension.

Given a subalgebra $\L^{\bf(1)}\subset \L$ we write 
$\hat \L^{\bf(1)}$ for the inverse image of $\L^{\bf(1)}$ in 
$\hat \L$.
In the next lemma we do not assume that the subalgebra
$\L^{\bf(1)}$ is  a $\Z$-graded subalgebra. 

\begin{lemma}\label{existence}
 Assume that $\L$ admits a grading element 
 $\ell_0$ and that $\L$ contains a subalgebra
$\L^{\bf(1)}$ of finite codimension such that
\begin{enumerate}
\item[(a)] $\L_{\overline 0}=
\L^{\bf(1)}_{\overline 0}\oplus\C \ell_0$
\item[(b)] The central extension
$$\C c\to \hat{\L}^{\bf(1)}\to \L^{\bf(1)}$$
splits.
\end{enumerate}
Then $\hat \L$ admits 
 cuspidal modules of arbitrary central charge.
\end{lemma}

\begin{proof} Let $x$ be an arbitrary scalar.
By hypothesis the algebra
$\hat{\L}^{\bf(1)}$ is isomorphic to
$\L^{\bf(1)}\oplus\C c$. Let $S$ the one-dimensional 
$\hat{\L}^{\bf(1)}$-module on which $\L^{\bf(1)}$ acts
trivially and $c$ acts by $1$.

Then the $\hat{\L}$-module $\cF(S,0)$, which is defined 
in Section \ref{cF(Su)}, has growth one and central charge
$c$. Any simple subquotient is a cuspidal $\L$-module
of central charge $x$.
\end{proof}

\subsection{The Lie superalgebra  $\widehat{\K(4)}$}
\label{seconddefK(4)}

We now describe a specific central extension
$\widehat{\K(4)}$ of $\K(4)$.

The four odd variables are denoted as
$\zeta_1,\eta_1,\zeta_2,\eta_2$ with the convention that
$[\zeta_1,\eta_1]=[\zeta_2,\eta_2]=1$ and all the other
brackets are zero.
The Grassmann algebra 
$\Grass(4)=\C[\zeta_1,\eta_1,\zeta_2,\eta_2]$ has a natural grading
$$\Grass(4)=\oplus_{k=0}^4 \Grass_k(4)$$
relative
to which each odd variable $\zeta_i,\,\eta_i$ has degree
$1$. Set $\omega:=\zeta_1\eta_1\zeta_2\eta_2$ and let
the trace map $\Tr:\Grass(4)\to\C$  defined by
$$\Tr(\omega)=1 \hbox{ and }\Tr(a)=0$$
for any homogenous element $a$ of degree $\neq 4$.

The $2$-cocycle of interest
$$\psi:\K(4)\times \K(4)\to\C$$
is the case $d=0$ $e=1$
in the formula (4.22) of 
V.G. Kac and J.W. van de Leur paper \cite{KvdL}.
For simplicity we will describe  some  extension of the cocycle $\psi$ to $\K(4;D)$. 
Nevertheless we are only interested by  its values on $\K(4)$.

Recall that 
$$\K(4;D)= \C[t,t^{-1}]\otimes\Grass(4).$$ 

\noindent It will be convenient to denote the
elements $f(t)\otimes a\in \C[t,t^{-1}]\otimes\Grass(4)$ as $a(f)$. For $a=1$ it is also denoted as $f(t)$.

Let $a\in \Grass_k(4)$,  $b\in\Grass_l(4)$
and let $n, m$ be integers. Then 
\begin{equation}\label{psi}
\psi(a(t^n), b(t^m))=
\frac{2-k}{2}\, \delta_{n+m,0} \Tr(ab)
\end{equation}

\noindent More explicitly we have
\begin{enumerate}
\item[(a)] $\psi(a(t^m), b(t^n))=0$
if $k=l$ or if the pair $\{k,l\}$ is not
$\{0,4\}$ or $\{1,3\}$.

\item[(b)]  $\psi(a(t^m), b(t^n))=1/2\,\delta_{n+m,0}
\,\Tr(ab)$ if $k=1$ and $l=3$

\item[(c)] $\psi(t^m,\omega(t^n))=\delta_{n+m,0}$
\end{enumerate}

We write $\widehat{\K(4)}$ for the central extension
of $\K(4)$ defined by the cocycle $\psi$ and let $c\in 
\widehat{\K(4)}$ be the corresponding central element.
We are now looking for  convenient formulas for 
the bracket of $\widehat{\K(4)}$.

For $k=0, 1,\cdots, 4$, set  
$\K_k(4;D)=\C[t,t^{-1}]\otimes\Grass_k(4)$.
Recall that the bracket in $\K(4;D)$
of two arbitrary  elements 
$$a(f)\in \K_k(4;D)
\hbox{ and } b(g)\in \K_l(4;D)$$
is given by the formula
$$[a(f), b(g)]=ab
\left(\frac{2-k}{2} fDg -\frac{2-l}{2} gDf\right)
+[a,b](fg)$$

\noindent which implies that 
$$[\K_k(4;D),\K_l(4;D)]\subset 
\K_{k+l}(4;D)\oplus \K_{k+l-2}(4;D).$$

We now define an algebra denoted as
$\widehat{\K(4)}$ that we will identify with the previous definition. First we consider the
vector space
$$\widehat{\Grass}(4)=\oplus_{k=0}^4\,\widehat{\Grass_k}(4)$$ 
defined by
$$\widehat{\Grass_4(4)}=\C c\hbox{ and }
\widehat{\Grass_k(4)}=\Grass_k(4)\hbox{ for }k\neq 4.$$
As a vector space 
$$\widehat{\K(4)}=\oplus_{k=0}^4\,\widehat{\K_k(4)}$$
where $\widehat{\K_k(4)}=
\C[t,t^{-1}]\otimes \widehat{\Grass_k(4)}$.

As before an arbitrary element $f(t)\otimes a\in \C[t,t^{-1}]\otimes\widehat{\Grass}(4)$ is denoted as $a(f)$. We define  the bracket $[a(f),b(g)]$
of two arbitrary  elements 
$$a(f)\in \widehat{\K_k(4)}
\hbox{ and } b(g)\in \widehat{\K_l(4)}$$

\noindent by the following formulas:
\begin{enumerate}
\item[(a)] If $k, l$ and $k+l$ are $\neq 4$
\begin{equation}
[a(f), b(g)]=ab
\left(\frac{2-k}{2} fDg -\frac{2-l}{2} gDf\right)
+[a,b](fg).
\end{equation}
It is the same as the formula in $\K(4)$ and it is well defined since
$[a(f), b(g)]$ lies in $\oplus_{i=0}^3 \K_i(4)
=\oplus_{i=0}^3 \widehat{\K_i(4)}$.

\item[(b)] The formulas
\begin{equation}\label{repetition}
[f, c(g)]=c(fDg)
\end{equation}
\begin{equation}
[a(f), b(g)]=\Tr(\frac{ab}{2})
\, c(fg) +[a,b](fg)\hbox{ if } k=1\hbox{ and }l=3
\end{equation}
\begin{equation}
[a(f), b(g)]=[a,b](fg)\hbox{ if } k=l=2
\end{equation}
define the bracket when $k+l=4$.
\item[(c)] Formula \ref{repetition} and the formulas
\begin{equation}
[a(f), c(g)]=[a,\omega](fDg)\hbox{ if } k=1
\end{equation}
\begin{equation}
[a(f), c(g)]=0 \hbox{ if } k\geq 2
\end{equation}
define the bracket when $l=4$.
\end{enumerate}

\noindent These formulas together with the skew-symmetry
$$[a(f), b(g)]=-(-1)^{kl} [b(g),a(f)]$$
entirely define the bracket on $\widehat{\K(4)}$.

At this point  we have neither check that the defined bracket on
$\widehat{\K(4)}$ satisfies the Jacobi identity
nor that $\widehat{\K(4)}$ is the desired central
extension of $\K(4)$.

Let $\pi:\widehat{\K(4)}\to \K(4)$
and  $\sigma:\K(4)\to \widehat{\K(4)}$
be the linear maps 
$$\pi(a(f))=a(f)\hbox{ if }
a(f)\in\widehat{\K_k(4)}
\hbox{ for }k\neq 4$$
$$\pi(c(f))=\omega(Df)$$
$$\sigma(a(f))=a(f)\hbox{ if }
a(f)\in K_k(4)
\hbox{ for }k\neq 4$$
$$\sigma(\omega(t^n))=1/n\,c(t^n).$$

\noindent The last formula is well defined since
 $\omega(t^n)$ belongs to
$\K(4)$ only if $n\neq 0$. 

It is clear that $\pi$ is an algebra morphism.
It is easy to check that
$$[\sigma(x),\sigma(y)]=\sigma([x,y])
+\psi(x,y)c$$
which shows that $\widehat{\K(4)}$ is 
the central extension of $\K(4)$ corresponding with the cocycle $\psi$.

Moreover we observe that 
$$[\widehat{\K_k(4)},
\widehat{\K_l(4)}]\subset \widehat{\K_{k+l}(4)} \oplus \widehat{\K_{k+l-2}(4)}.$$

\subsection{The subalgebra $\K(4)^{\bf(1)}\subset \K(4)$}\label{defK(4)1}

The set of quadratic elements
$$\cQ=\{
\zeta_1\zeta_2, \eta_1\eta_2,
\zeta_i\eta_j\mid 1\leq i,j\leq 2\}$$
 is a basis of the Lie subalgebra $\fso(4)$ in $\K(4)$. Let $\K(4;D)^{\bf(1)}$ be the ideal of
the commutative superalgebra 
$$\C[t,t^{-1},\zeta_1,\eta_1,\zeta_2,\eta_2]$$
generated by $(t-1)$ and $\fso(4)$.
By Lemma \ref{repeat}, $\K(4;D)^{\bf(1)}$ is a Lie subalgebra of $\K(4;D)$. Set 
$$\K(4)^{\bf(1)}=\K(4;D)^{\bf(1)}\cap\K(4),$$
and let $\widehat{\K(4)}^{\bf(1)}$ be its inverse image in 
$\widehat{\K(4)}$.

\begin{lemma}\label{G} We have 
$$\K(4)_{\overline 0}=
\K(4)^{\bf(1)}_{\overline 0}\oplus\C E_0\hbox{ and }
\dim\, K(4)_{\overline 1}/
\K(4)^{\bf(1)}_{\overline 1}=4.$$
 Moreover 
the central extension
$$0\to\C c\to \widehat{\K(4)}^{\bf(1)}
\to \K(4)^{\bf(1)}\to 0$$ 
splits.
\end{lemma}

\begin{proof} Set 
$$I:=(t-1)\C[t,t^{-1}]=\{f\in\C[t,t^{-1}]\mid f(1)=0\}.$$ 
For $0\leq k\leq 4$ set $\widehat{\K_k(4)}^{\bf(1)}= \widehat{\K(4)}^{\bf(1)}\cap
\widehat{\K_k(4)}$.
It is easy to see that  
$$\widehat{\K(4)}^{\bf(1)}=\oplus_{k=0}^4\,\widehat{\K_k(4)}^{\bf(1)},$$
where $\widehat{\K_0(4)}^{\bf(1)}=I$,
$\widehat{\K_k(4)}^{\bf(1)}=\widehat{\K_k(4)}$ for $k\geq 2$ and
$$ \widehat{\K_1(4)}^{\bf(1)}= I\otimes\zeta_1\oplus
I\otimes\eta_1\oplus I\otimes\zeta_2\oplus I\otimes\eta_2.$$

It follows that
$$\widehat{\K(4)}_{\overline 0}=
\widehat{\K(4)}^{\bf(1)}_{\overline 0}\oplus\C E_0\hbox{, and }$$
$$\widehat{\K(4)}_{\overline 1}=
\widehat{\K(4)}^{\bf(1)}_{\overline 1}\oplus\Bigl(\C\zeta_1\oplus\C\eta_1\oplus
\C\zeta_2\oplus\C\eta_2\Bigr)$$
which implies the first claim.

For $x, y\in \widehat{\K_k(4)}$ 
the degree $4$ component of 
$[x,y]$ is denoted as $[x,y]_{top}$. 
The previous formulas shows that 
$[\widehat{\K_k(4)},\widehat{\K_l(4)}]_{top} =0$
except if $k\neq l$ and $\{k,l\}=\{0,4\}$ or
$\{k,l\}=\{1,3\}$. We deduce that
the bracket $[-,-]_{top}$ is the sum of two components 
$$\mu_0: \widehat{\K_0(4)}\times\widehat{\K_4(4)}
\to \widehat{\K_4(4)}\,\, $$
$$\mu_1: \widehat{\K_1(4)}\times\widehat{\K_3(4)}
\to \widehat{\K_4(4)}$$
and their symmetric counterparts, where
$$\mu_0(f, c(g))=c(fDg)$$
$$\mu_1(a(f), b(g))=\Tr(\frac{ab}{2}) c(fg)$$

\noindent for any $a\in\Grass_1(4), b\in\Grass_3(4)$ and $f, g\in \C[t,t^{-1}]$. 

We observe that $\mu_0$ and $\mu_1$ are 
$\C[t,t^{-1}]$-linear on the first argument.
Hence we have

$\left[\widehat{\K(4)}^{\bf(1)},
\widehat{\K(4)}^{\bf(1)}\right]_{top}$

\hskip20mm$\subset
\mu_0\left(\widehat{\K_0(4)}^{\bf(1)}, \widehat{\K_4(4)}^{\bf(1)}\right)+
\mu_1\left(\widehat{\K_1(4)}^{\bf(1)}, 
\widehat{\K_3(4)}^{\bf(1)}\right)$

\hskip20mm$\subset\mu_0\left((1-t)\widehat{\K_0(4)}, 
\widehat{\K_4(4)}\right)+
\mu_1\left((1-t)\widehat{\K_1(4)}, 
\widehat{\K_3(4)}\right)$

\hskip20mm$\subset (1-t)\C[t,t^{-1}]\otimes c,$

\smallskip\noindent
which implies that
$c\notin \left[\widehat{\K(4)}^{\bf(1)},
\widehat{\K(4)}^{\bf(1)}\right]$.
Hence the central extension
$$0\to\C c\to\widehat{\K(4)}^{\bf(1)}\to 
\K(4)^{\bf(1)}\to 0$$ 
splits.
\end{proof}

\begin{cor}\label{YES} There are
cuspidal $\widehat{\K(4)}$-modules $M$ with a
nonzero central charge.
\end{cor}

\begin{proof} By Lemma \ref{G} the superalgebra 
$\widehat{\K(4)}$ and its subalgebra
$\K(4)^{\bf(1)}$ satisfy  the hypothesis of 
Lem\-ma
\ref{existence} which insures the existence of
cuspidal modules of nonzero central charge.
\end{proof}

\subsection{A Criterion for the vanishing of the central charge}

Let $H=\C h$ be a $1$-dimensional vector space
and let 
$$\fg=\Vir\ltimes H\otimes \C[t,t^{-1}]$$
be the Lie considered in chapter \ref{split}.
Let  
$\phi:\wedge^2\fg\to\C$ be a $E_0$-invariant 2-cochain. Since  $\phi(\fg_n,\fg_m)=0$ whenever
$n+m\neq 0$ we only need to provide the value
$\phi(x,y)$ where $x$, $y$ are basis elements of
opposite degree. 

With this convention, let 
$\phi_i$, $i=1,2$ or $3$, be the 
$E_0$-invariant 2-cochains defined by
\begin{equation}\label{phi1}
\phi_i(h(t^n),h(t^{-n}))=\delta_{i,1} n
\end{equation}
\begin{equation}\label{phi2}
\phi_i(E_n,h(t^{-n}))=\delta_{i,2}n^2
\end{equation}
\begin{equation}\label{phi3}
\phi_i(E_n,E_{-n})= \delta_{i,3}n^3.
\end{equation}

\noindent An easy computation shows that these cochains are cocycles.  The classes of the cocycles $\phi_i$ are denoted as $[\phi_i]$.

\begin{lemma}\label{h2} The $3$ classes 
$[\phi_1]$, $[\phi_2]$ and $[\phi_3]$ form a basis of $H^2(\fg)$.
\end{lemma}

\begin{proof}  
Let  $E_0^*$ and $h^*$ be  the elements of $\fg^*$ defined by
$$E^*_0(E_0)=1,E^*_0(h)=0,  h^*(E_0)=0 \text{ and }
h^*(h)=1$$
while $E^*_0(\fg_n)=h^*(\fg_n)=0$ if $n\neq 0$.

Since $\{E^*_0,h^*\}$ is a basis of $\fg_0^*$,
the set $\{\d E^*_0,\d h^*\}$ is a basis
of the space of $E_0$-equivariant coboundaries.
Furthermore the $5$ cocycles $\phi_1$, $\phi_2$, $\phi_3$, $\d L_1$ and $\d L_2$ are linearly independent. 
Since any $2$-cocycle of $\fg$ is equivalent to an $E_0$-invariant cocycle, the classes 
$[\phi_1]$, $[\phi_2]$ and $[\phi_3]$ are linearly independent.

Next we claim that any $E_0$-equivariant 
$2$-cocycle $\phi$ is equivalent to a linear combination of $[\phi_1]$, $[\phi_2]$ and 
$[\phi_3]$. It is well known that 
$$H^2(\Vir)=\C[\phi_3].$$ 
Hence for some $a\in\C$, $\phi$ is equivalent to
$$a\phi_3+\phi'$$
where $\phi'(\Vir,\Vir)=0$.

Set $a_n=\phi'(E_n,h(t^{-n}))$. The cocycle equation
$$\phi'([E_n,E_m],h(t^{-n-m}))=
\phi'(E_n,[E_m,h(t^{-n-m})])-
\phi'(E_m,[E_n,h(t^{-n-m})])$$
means that
$$(n-m)a_{n+m}=(n+m)(a_n-a_m).$$
It follows that  $a_n=bn+cn^2$ for some 
$b, c\in \C$.
Set 
$$\phi"=\phi'-b \,\d h^* - c\phi_2.$$
It follows from the previous consideration that
$\phi"(\Vir,\fg)=0$, hence the bilinear map
$$B:\C[t,t^{-1}]\times \C[t,t^{-1}],
\,\, (t^n,t^m))\mapsto \phi"(h(t^n),h(t^m))$$
is $\Vir$-equivariant. Such a map is necessarily proportional to $\phi_1$, see e.g.
\cite{IoharaMathieu}, which completes the proof of the lemma.
\end{proof}

Let $\L$ be a superconformal algebra, let
$\hat \L$ be a $1$-dimensional central extension
and let $c\in \hat \L$ be a central element.
Here we will use the convention that 
$\L$ is $1/d\Z$-graded for some integer $d>0$ and that the subalbalgebra $\Vir$ is endowed with its usual $\Z$-grading.

Given a subalgebra $G\subset \L$ we write 
$\hat G$ for the inverse image of $G$ in $\hat \L$.
Recall that $\fg=\Vir\ltimes H\otimes \C[t,t^{-1}]$.

\begin{lemma}\label{nonexistence} Let $G\subset \L$ be a subalgebra isomorphic to $\Vir$ or to $\fg$. 

Assume that the central extension
$$\C c\to \hat{G}\to G$$
is not split. Then $c$ acts trivially on any
cuspidal $\hat \L$-module.
\end{lemma}

\begin{proof} 
First assume that $G=\Vir$. By Corollary \ref{nocharge}
 $c$ acts trivially on any
simple subquotient. Therefore $c$ is nilpotent of index  $d:=\Max_n(\dim\, M_n)$.
Since $\Ker c$ is a nonzero submodule, it follows
that $M=\Ker\,c$, that is $c$ acts trivially.

Next assume that 
$\G=\Vir\ltimes H\otimes \C[t,t^{-1}]$.  By Lemma \ref{h2}
we can assume that the cocycle of the central extension 
$$\C c\to \hat{G}\to G$$
is $\phi=a_1\phi_1+a_2\phi_2+a_3\phi_3$ for some 
nonzero triple $(a_1,a_2,a_3)\in\C^3$.

For $\delta\in\C$ we consider  the embedding
$$\tau_\delta:\Vir\to 
\Vir\ltimes H\otimes \C[t,t^{-1}]$$ 
defined  by $\tau_\delta(E_n)= E_n-n\delta\, h(t^n)$.
Then we have
$$\phi(\tau_\delta(E_n),\tau_\delta(E_{-n}))=(a+2b\delta-c\delta^2)n^3.$$ 
Thus for $a+2b\delta-c\delta^2\neq 0$ the central extension
$$\C c\to\widehat{\tau_\delta(\Vir)}\to\tau_\delta(\Vir)$$
is not split. Thus by the previous assertion
$c$ acts trivially on any cuspidal 
$\hat \L$-module.
\end{proof}

\subsection{The central charge of cuspidal projective $\L$-modules}

In their paper \cite{KvdL} V.G. Kac and J.W. van de Leur have defined  $2$-cocycles for various
superconformal algebras $\L$.

For $\L=\K(4)$, besides the cocycle $\psi$, they define two other $2$-cocycles $\psi_1,\psi_2$. Indeed 
$\psi_1$ and $\psi_2$ satisfy 
\begin{equation}\label{psi1}
\psi_1(t^n,t^m)=n^3\delta_{n+m,0}
\end{equation}
\begin{equation}\label{psi2}
\psi_2(t^n,t^m)=0 
\end{equation}
\begin{equation}\label{psi2bis}
\psi_2(a(t^n),b(t^m))=n\delta_{n+m,0}\Tr(ab)
\hbox{ for } a,b\in\Grass_2(4)
\end{equation}

\noindent where $\Tr:\Grass(4)\to\C$ is the linear map
satisfying $\Tr(\Grass_k(4))=0$ if $k\neq 4$ and
$\Tr(\zeta_1 \eta_1\zeta_2\eta_2)=1$.

A complete definition of the  cocycle $\psi_1$ is is given by the formulas (4.21) of \cite{KvdL}.
The cocycle $\psi_1$ in \cite{KvdL} satisfies
$\psi_1(t^n,t^m)=(n^3-n)\delta_{n+m,0}$ but it provides an equivalent cohomology class.
Similarly 
the cocycle $\psi_2$ is fully  defined by 
the formulas (4.22) of \cite{KvdL} for
$d=1$ and $e=0$. However for our purpose we do not need a complete description of these two cocycles.

For $\L=\W(2):=\Der\,\C[t,t^{-1},\xi_1,\xi_2]$  set 
$$h=\xi_1\frac{\partial}{\partial \xi_2}-\xi_2\frac{\partial}{\partial \xi_1}.$$
There is a $2$-cocycle $\psi_3:\L\times \L\to\C$ which satisfies
\begin{equation}\label{psi3}
\psi_3(h(t^n), H(t^m))=2n\delta_{n+m,0}
\end{equation}

For $\L=\W(1):=\Der\,\C[t,t^{-1},\xi]$
set $h=\xi\frac{\partial}{\partial \xi}$. 
There is a $2$-cocycle $\psi_4:\L\times \L\to\C$ which satisfies
\begin{equation}\label{psi4}
\psi_4(t^n \frac{\partial}{\partial t},h(t^m))=n^2\delta_{n+m,0}.
\end{equation}

For complete formulas for $\psi_3$ and $\psi_4$ we refer to formulas (4.14) and (2.10) of \cite{KvdL}.
The following 
result is  due to V. G. Kac and J.W. van der Leur, see Theorem
4.1 \cite{KvdL}, except Assertion (e)  for the superalgebra $\CK(6)$ which is
Theorem 4.1 of  S.J. Cheng and V.G. Kac paper \cite{CK}.

\begin{CKvdLthm}\label{CKvdL} 
Let $\L$ be a 
superconformal algebra.
\begin{enumerate}
\item[(a)] If $\L=\K(4)$ or $\K_{NS}(4)$,
the cohomology classes
$$[\psi_1], [\psi_1]\hbox{ and }[\psi]$$
 form a basis of $H^2(\L)$.
\item[(b)] Assume $\L=\K(N)$ or $\K_{NS}(N)$ 
with $1\leq N\leq 3$. Then
$H^2(\L)$ is the $1$-dimensional space generated
by the class of $\psi_1\vert_L$. 
  \item[(c)] Assume $\L=\K^{(2)}(2)$ or 
$\K^{(2)}(4)$. Then
$H^2(\L)$ is the $1$-dimensional space generated
by the class of $\psi_1\vert_\L$. 
 
 \item[(d)] We have
$$H^2(\W(1))=\C[\psi_4]$$
$$H^2(\W(2))=\C[\psi_3].$$
\item[(e)] If $\L=\S(2;\gamma)$ then
$H^2(\L)$ is the $1$-dimensional space generated
by the class of $\psi_3\vert_\L$. 
\item[(f)] If $\L$ is one of the known
superalgebra but not from the previous list,
then $H^2(\L)=0$
\end{enumerate} 
\end{CKvdLthm}

\begin{cor}\label{NO} Let $\L$ be one of the known superconformal algebra. Then
\begin{enumerate}
\item[(a)] If $\hat \L$ is a nonsplit central extension of $\K(4)$
but $\hat L\neq \widehat{\K(4)}$, the central charge of any
cuspidal $\hat \L$-module is zero.
\item[(b)] If $\L\neq \K(4)$
then $\L$ does not admit any cuspidal projective module.
\end{enumerate}
\end{cor}

\begin{proof} Let $\hat \L$ be a nonsplit central extension of $\L$. For a subalgebra $G\subset \L$
let $\hat G$ be its inverse image in $\hat \L$.

The general idea is to  find  a subalgebra $G$ isomorphic
to $\Vir$ or to
$\Vir\ltimes \C[t,t^{-1}]$ such that
the central extension
$$0\to\hat{G}\to G\to 0$$
is not split. Then we use 
Lemma \ref{nonexistence} to conclude.

Let start with the case-by-case analysis.
First assume that $\hat \L$ is a nonsplit central extension of $\K(4)$ with 
$\hat \L\neq \widehat{\K(4)}$. By Theorem 
\ref{CKvdL}(a) this extension is represented by
the 
$2$-cocycle  $\phi=a\psi_1+b\psi_2+c\psi$ for some
$a,b$ and $c\in \C$ with
$(a,b)\neq (0,0)$. 

If $a\neq 0$ set $G=\Vir$. By 
Formulas \ref{psi}, \ref{psi1} and \ref{psi2} we have
$$\phi(t^n,t^m)=a n^3 \delta_{n+m,0}$$
and therefore the extension 
$$0\to\hat{G}\to G\to 0$$
is not split.

Otherwise we have $b\neq 0$. 
Choose $h\in\Grass_2(4)$ with
$\Tr(h^2)=1$ and let
$$G=\Vir\ltimes h\otimes\C[t,t^{-1}].$$
By Formula \ref{psi} and \ref{psi2bis}
we have 
$$\phi(h(t^n),h(t^m))=b n \delta_{n+m,0}.$$
Thus by Lemma \ref{h2} 
the extension 
$$0\to\hat{G}\to G\to 0$$
is not split.

Thus we have completed the proof for $\L=\K(4)$. For
$\L\neq \K(4)$ we will find a subalgebra $G$ isomorphic
to $\Vir$ or $\Vir\ltimes \C[t,t^{-1}]$ such that
the restriction map $H^2(\L)\to H^2(G)$ is injective.

First assume that $\L=\K(N)$ or $\K_{NS}(N)$ for
$N\leq 3$. Since $H^2(\L)$ is generated by the class of $\psi_1\vert_L$. Since Formula \ref{psi1} shows that $\psi_1\vert_\Vir$ is the usual
Virasoro cocycle we deduce that
$H^2(\L)\simeq H^2(\Vir)$.

The proof for $\L=\K^{(2)}(2)$ or 
$\L=\K^{(2)}(4)$ is identical. Indeed 
Lemma \ref{CKvdL}(c) shows that 
$H^2(\L)\simeq H^2(\Vir)$.

Assume now that $\L=\W(2)$ or $\S(2;\gamma)$ and let 
$$h=\xi_1\frac{\partial}{\partial \xi_2}-\xi_2\frac{\partial}{\partial \xi_1}$$
and set 
$$G=\Vir\ltimes h\otimes\C[t,t^{-1}].$$
By Theorem\ref{CKvdL} the class of $\psi_3$
generates  $H^2(\L)$ and by formula
\ref{psi3} we have  
$\psi_3(h(t^n\otimes),h(t^m))=2n\delta_{n+m,0}$
It follows that 
$$[\psi_3\vert_G]=2[\phi_1]\hbox{ modulo }
\C[\phi_2]+\C[\phi_3].$$ 
Thus the restriction map $H^2(\L)\to H^2(G)$ is injective. 

For $\L=\W(1)$ set $G=\L_{\overline 0}$. By
Theorem \ref{CKvdL} the class of $\psi_4$
generates  $H^2(\L)$ and by formula
\ref{psi4} we have
$$[\psi_4\vert_G]=[\phi_2]\hbox{ modulo }
\C[\phi_1]+\C[\phi_3],$$ 
which proves that the restriction map $H^2(L)\to H^2(G)$ is injective. 
\end{proof}

\subsection{The exceptional cocycle of $\widehat{\K(4)}$}
\label{exceptionalcocycle}

The proof of Lemma \ref{existence} shows that
$\widehat{\K(4)}$ acts on 
$\C[t,t^{-1},\xi_1,\xi_2,\xi_3,\xi_4]$ by differential
operators of degree one. We will not use it, by
we will write the corresponding abelian cocycle.
All other known cocycles with values in current algebras
are differential operators of order $\leq 1$, e.g the divergence operator (some nonabelian cocycles are described in \cite{M2000}).

By contrast, the exceptional cocycle of $\widehat{\K(4)}$ is a differential operator of order $3$, as shown by the following formulas.

Set $\omega=\xi_1\xi_2\xi_3\xi_4$ and let consider the following linear map
$$\cD: \widehat{\K(4)}\to\C[t,t^{-1},\xi_1,\cdots,\xi_4]$$
defined by:
$$\cD(f)=2\omega(D^3(f)-D(f))$$
$$\cD(f\xi_i)=(-1)^i\frac{\partial}{\partial\xi_i}
\omega(D^2(f)-f/4)$$
$$\cD(f\xi_i\xi_j)=(-1)^{i+j}2
\frac{\partial}{\partial\xi_i}\frac{\partial}{\partial\xi_j}
\omega(Df)$$
$$\cD(f\xi_ix_jx_k)=(-1)^{i+j}2
\frac{\partial}{\partial\xi_i}\frac{\partial}{\partial\xi_j}
\frac{\partial}{\partial\xi_j}\omega(f)$$
$$\cD(c(f))=f, \text{ for all } f\in\C[t,t^{-1}].$$

The subalgebra $\fk$ of $\widehat{\K(4)}$
with basis
$$\{1, t^{\pm 1}\}\cup
\{ \xi_i(t^{\pm 1/2})\mid 1\leq i\leq 4\}
\cup\{\xi_i\xi_j\mid 1\leq i<j\leq 4\}$$
is isomorphic to $\fosp(4,2)$.

\begin{lemma} The linear map 
$$\cD:\widehat{\K(4)}\to\C[t,t^{-1},\xi_1,\cdots,\xi_4]$$
is a $\fk$-invariant cocycle.
In particular $\cD$ is conformally invariant.
\end{lemma}

\begin{proof} This can be checked by direct computation.
\end{proof}

For $\partial\in \widehat{\K(4)}$ write as 
$\overline{\partial}$ its image in 
$\partial\in\K(4)$.

\begin{cor} The map
$$\partial\mapsto
(\overline{\partial}, \cD(\partial))
$$
realizes an embedding
$$\widehat{\K(4)}\subset \K(4)\ltimes \C[t,t^{-1},\xi_1,\cdots,\xi_4].$$
Therefore $\widehat{\K(4)}$ acts faithfully on 
$\C[t,t^{-1},\xi_1,\cdots,\xi_4]$.
\end{cor}

\section{Highest weight theory}\label{HW}

In the chapter we describe the general approach to classify all cuspidal modules for untwisted  superconformal
algebras $\L$, that is for algebras
of the following list
$$\W(n),n\geq 2;\hskip1mm \S(n;\gamma), n\geq 2;
\hskip1mm \CK(6);$$
$$\K(N)\hbox{ and } \K_{NS}(N), N\geq 3; \hskip1mm\widehat{\K(4)}.$$

We recall the general definitions and notations
of  Chapter \ref{zoology}. Let $H\subset \L_0$ be the Cartan subalgebra and let $F\in H$ be the chosen element. The eigenvalues of $\ad(F)$ are integers. The corresponding eigenspace decomposition
$$L=\oplus_{n\in\Z} \L^{(n)}$$
provides a triangular decomposition 
$$\L=\L^-\oplus C_\L(F)\oplus \L^+.$$
where $\L^{\pm}=
\oplus_{n>0}\,\L^{(\pm n)}$.
Except for $\L=\W(n)$, we have $C_\L(F)=C_L(H)$.
For $\L=\W(n)$,
$$C_\L(F)=C_L(H)\oplus L^{(\epsilon_1+\cdots+\epsilon_n)}$$
 but this brings only little differences.

In Chapter \ref{zoology},  we have defined a nilpotent ideal $\Rad\,C_\L(F)$ in $C_\L(F)$. (The ideal
is not a radical in some theoretical sense but it is a case-by-case definition.)  We have
$$C_\L(F)=\G\ltimes \Rad\,C_\L(F)$$
where $\G$ is isomorphic to
the Lie algebra $\fg$ of Chapter \ref{split} or to the Lie superalgebra
$\fG_*$ of Chapter  \ref{splitII}.

Set $B:=C_\L(F)\ltimes \L^+$.
 For $\delta,u\in\C$ and  $\lambda\in H$,
 we consider the $\G$-module
$\Tens(\lambda,\delta,u)$ as a $B$-module, with a trivial
action of $\Rad\,C_\L(H)$ and $\L^+$. The generalized Verma module 
$$M(\lambda,\delta,u):=\Ind_B^L\, \Tens(\lambda,\delta,u)$$
has a unique simple quotient denoted as $V(\lambda,\delta,u)$. 

In chapter \ref{zoology} we have defined the simple coroots $h_1,\cdots h_m\in H$. 
A weight $\lambda\in H^*$ is called {\it  dominant}  if
all $\lambda(h_i)$ are nonnegative integers. 

 The main result of the chapter is the following theorem.

\begin{thm}\label{condition1} 
Let $V$ be a cuspidal $\L$-module. 
Then 
$$V\simeq V(\lambda,\delta,u)$$ 
for some triple
$(\lambda,\delta,u)\in H^*\times\C\times\C$.

Moreover $\lambda$ is dominant and $\lambda(h_1)\geq 1$.
\end{thm}

However the $\L$-module $V(\lambda,\delta,u)$ is not always cuspidal because its homogenous components could be infinite dimensional. Therefore Theorem \ref{condition1} reduces the classification of all cuspidal
modules to the following question:

\bigskip
\hskip15mm{\it Besides the requirement that $\lambda$ is dominant and $\lambda(h_1)\geq 1$,
}

\hskip15mm{\it which condition insures that $V(\lambda,\delta,u)$
is cuspidal?}

\bigskip
\noindent Part II answers it
by a case-by-case analysis.

\subsection{Existence of a highest weight}

Let $\L$ be a superconformal algebra and
let $V$ be a cuspidal $\L$-module. Then $V$
decomposes into weight spaces
$$V=\oplus_{\lambda\in H^*}\,V^{(\lambda)}.$$
As usual, an element $\lambda\in H^*$ is called 
 a {\it weight of $V$} if $V^{(\lambda)}\neq 0$.

\begin{lemma}\label{finitely many} Any cuspidal
$\L$-module $V$ admits only finitely many weights.
\end{lemma}

\begin{proof} 
Let $\lambda\in H^*$ be an arbitrary  nonzero weight of $V$. Choose $h\in H$ with $\lambda(h)=1$ and
let $d$ be the upper bound of the dimensions $\dim\, V_i$. 

Let $\fg=\Vir\ltimes H\otimes\C[t,t^{-1}]$. We observe that $V^{(\lambda)}$ is a $\fg$-module.
Hence by Corollary \ref{corh=1} 
$h(t)\vert_{V^{(\lambda)}}$
is invertible which implies that all dimensions 
$\dim\, V_i^{(\lambda)}$ are nonzero.
It follows that $V$ admits at most $d$ nonzero weights,
which proves the  assertion.
\end{proof}

Given weights $\lambda\neq \mu$ we write 
$$\lambda>\mu$$
if their difference can be written as
$$\lambda-\mu=\sum_{\alpha\in\Delta^+}\,m_\alpha\alpha
\hbox{ and } \mu-\lambda=
\sum_{\beta\in\Delta^-}\,n_\beta\beta$$
where all $m_\alpha$ and $n_\beta$ are nonnegative integers. When $\Delta^-\neq-\Delta^+$ these two decompositions are not equivalent.

\begin{lemma}\label{dominant} Any cuspidal
$\L$-module $V$ admits  a  highest weight $\lambda$.
Moreover 
\begin{enumerate}
\item[(a)] $V^{(\lambda)}$ is a simple $C_\L(F)$-module,
and
\item[(b)] the weight $\lambda$ is dominant and nonzero.
\end{enumerate}
\end{lemma}

\begin{proof} By definition of $F$,
$\alpha(F)$ is an integer for any root $\alpha\in\Delta$.
Hence for two weights $\lambda,\mu$ of $V$,
$\lambda(F)-\mu(F)$ is an integer.
Since $V$ admits only finitely many weights
by Lemma \ref{finitely many}, we can 
choose a weight $\lambda$ of $V$ such that 
$\lambda(F)$ is maximal.

When $C_\L(F)=C_\L(H)$, 
$V^{(\lambda)}$ is a  $C_\L(F)$-module
which is obviously simple. However for $\L=\W(n)$
a proof is required. 
Set $\omega=\epsilon_1+\cdots+\epsilon_n$ and
$$V^+=\oplus_{m\in\Z} V^{(\lambda+m\omega)}.$$
Then $V^+$  is a $C_\L(F)$-module
which is obviously simple. Since $\L^{(\omega)}$ acts
nilpotently on $V$, we deduce that $\L^{(\omega)}$ acts
trivially on $V^+$. It implies that
$$V^+=V^{(\lambda)},$$ 
which shows the claim for $\W(n)$. 

Let $\mu$ be another weight of $V$.
By 
Poincar\'e-Birkhoff-Witt Theorem, we have
$$V=U(\L^-)  V^{(\lambda)}.$$
Thus 
$$\mu-\lambda=\sum_{\beta\in\Delta^-}\,n_\beta
\beta$$
for some nonnegative integers $n_\beta$.
If we consider the graded dual $V^{(*)}$ of $V$, we have
$$V^{(*)}=U(\L^+) (V^{(*)})^{(\lambda)},$$
which proves that 
$$\lambda-\mu=\sum_{\alpha\in\Delta^+}\,m_\alpha\alpha$$
for some nonnegative integers $m_\alpha$.

Hence
$\mu<\lambda$, that is $\lambda$ is the highest weight
of $V$, which proves Assertion (a).

By the definition of Chapter \ref{zoology}, each  $h_i$ is part
of a $\fsl(2)$-triple $(e_i, h_i,f_i)$
lying in $\L_0$ with $e_i\in \L^+$. Since each 
homogenous component $V_n$ is a
finite dimensional module over
$\fsl(2)=\C e_i\oplus\C h_i\oplus \C f_i$
and $e_i V^{(\lambda)}=0$ we deduce  $\lambda$ is dominant.

It is easy to show that $\L$ is generated by
$\L^+$ and the elements $f_i$. Since
$\L.V^{(\lambda)}\neq 0$, it follows that
$f_i.V^{(\lambda)}\neq 0$ for some $f_i$. We similarly deduce that $\lambda(h_i)\neq 0$ for some $i$, which proves that $\lambda\neq 0$.
\end{proof}

\subsection{The subalgebra $C_{\L}(F)$ for the superalgebras $\K_*(N)$}

The proof  Theorem \ref{dominant} requires 
an explicit description of the Lie subalgebra
$\fC:=C_\L(F)$ and its radical $\Rad\,C_\L(F)$ for
contact superalgebras $\L=\K_*(N)$. In this case we have $C_\L(F)=C_\L(H)$.

We 
indistinctly denote as
$\K_*(N)$  a contact superalgebra of
Ramond type or Neveu-Schwarz type.
Indeed if  $\L:=\L_{\overline 0}\oplus \L_{\overline 1}$ denotes the superalgebra $\K(N)$,
we  identify   the 
Neveu-Schwarz superalgebra $\K_{NS}(N)$   as
$\L_{\overline 0}\oplus t^{1/2} \L_{\overline 1}$,
see Section \ref{RNS}. 

We now fix an integer $N\geq 3$ with $N=2m$ or $N=2m+1$ and let $\L:=\K_*(N)$. As in
Section \ref{defK}, 
for $N=2m$ we denote the odd variables as
$\zeta_1,\eta_1,\cdots\zeta_m,\eta_m$ and for $N=2m+1$ the additional variable is denoted as $\xi$.

First we describe the even part $\fC_{\overline 0}$ of
$\fC:=C_\L(H)$. 
For a subset $I=\{i_1,\cdots,i_r\}$ of $\{1,\cdots,m\}$
it will be convenient to set
$$\epsilon_I=\zeta_{i_1}\eta_{i_1}\cdots\zeta_{i_r}\eta_{i_r}.$$
Assume that $N\neq 4$. As a vector space
$$\fC_{\overline 0}=\oplus_I\,\C[t,t^{-1}]\epsilon_I,$$
where $I$ runs over the power set $2^{\{1,\cdots,m\}}$.

The formula of Lie bracket of $\fC_{\overline 0}$ is
$$[f(t)\epsilon_I,g(t)\epsilon_J]=
\Bigl(1-\vert I\vert) f(t) Dg(t) -(1-\vert J\vert) g(t) Df(t)\Bigr)
(\epsilon_I. \epsilon_J)$$
for any $f(t), g(t)\in \C[t,t^{-1}]$, 
$I, J\in 2^{\{1,\cdots,m\}}$.  Here
$D=\frac{t\d}{\d t}$, the symbol $\vert I\vert$ is the cardinality of the set $I$ and  
$\epsilon_I. \epsilon_J$ is  the Grassman product,
namely
$$\epsilon_I. \epsilon_J=\epsilon_{I\cup J} \hbox{ if } I\cap J=\emptyset \hbox{ and } \epsilon_I.\epsilon_J=0 \hbox{ otherwise }.$$

The previous formula shows that the Lie algebra
$\fC_{\overline 0}$ admits a grading 
$$\fC_{\overline 0}=\fC_{{\overline 0},0}
\oplus \fC_{{\overline 0},1}\oplus\cdots 
\fC_{{\overline 0},n}$$
relative to which the elements $f(t)\epsilon_I$ has degree $\vert I\vert$. 

For the special case $N=4$, there is a similar description except
that  $\zeta_{1}\eta_{1}\zeta_{2}\eta_{2}$ is not in
$\K(4)$. Thus $\fC_2$ consists of elements
$f(t)\zeta_{1}\eta_{1}\zeta_{2}\eta_{2}$ where $f(t)$
is a Laurent polynomial without constant term.

When $N=2m$ the algebra $\fC$ is purely even, namely
$\fC_{\overline 1}=0$. We now describe $\fC_{\overline 1}$ for $N=2m+1$. Let $\tau=0$ if 
$\L=\K(N)$ and $\tau=1/2$ otherwise.

As a vector space
$$\fC_{\overline 1}=\oplus_I\,t^{\tau}\C[t,t^{-1}]\epsilon_I\xi,$$
where $I$ runs over  $2^{\{1,\cdots,n\}}$.
The formula of the additional  brackets of $\fC$ are
$$[f(t)\epsilon_I,a(t)\epsilon_I\xi]=
\bigl((1-\vert I\vert) f(t) Da(t) -(\frac{1}{2}-\vert J\vert) a(t) Df(t)\Bigr)
(e_I. e_J\xi)\hbox{ and }$$
$$[a(t)\epsilon_I\xi,b(t)\epsilon_I\xi]=
a(t)b(t) (e_I. e_J)$$

\noindent
for any $f(t)\in \C[t,t^{-1}]$, 
$a(t), b(t)\in t^{\tau} \C[t,t^{-1}]$ and
$I, J\in 2^{\{1,\cdots,n\}}$. 

These formulas shows that the superalgebra
$\fC$ admits a grading 
$$\fC=\fC_{0}
\oplus \fC_{1}\oplus\cdots 
\fC_{n}$$
relative to which the elements $f(t)\epsilon_I$
and $a(t)\epsilon_I\xi$  have degree $\vert I\vert$. 

We now return to the general case $N=2m$ or $2m+1$.
By definition, $\fR:=\Rad\,C_\L(H)$ is precisely the subalgebra 
$$\fR=\fC_{2}\oplus\cdots \oplus\fC_{n}
=\bigoplus_{k=2}^m \fC_k.$$
Obviously $\fR$ is a nilpotent ideal of $\fC$ and 
$\fR_{\overline 0}$ is a nilpotent ideal of 
$\fC_{\overline 0}$.

\begin{lemma}\label{generators} Assume $N\neq 4$.
As a $\fC_{\overline 0}$-module the ideal 
$\fR_{\overline 0}$ is generated by the set
$$\Sigma:=\{\epsilon_I\mid 
I\in 2^{\{1,\cdots,n\}}\hbox{ and }\vert I\vert =2\}.$$
 \end{lemma}

\begin{proof} The lemma is a consequence of the following claims
\begin{enumerate}
\item[(a)] As a $\Vir$-module, 
$\fC_{{\overline 0},2}$ is generated by $\Sigma$.
\item[(b)] For any $k\geq 2$
$$[\fC_{{\overline 0},1},\fC_{{\overline 0},k}]
=\fC_{{\overline 0},k+1}.$$
\end{enumerate}

For any $\epsilon_I\in \Sigma$ we have
$[E_n,\epsilon_I]=-nt^n \epsilon_I$
hence
$$\C[t,t^{-1}]\epsilon_I=\C\epsilon_I
\oplus [\Vir,\epsilon_I]$$
which proves Claim (a).

Let $t^n \epsilon_I$ be an arbitrary
basis element of $\fC_{\overline{0},k}$ for some
subset $I$ of cardinality $k\geq 3$.
Choose $i\in I$ and set $J=I\setminus \{i\}$.
Since 
$$[\epsilon_J, t^n\epsilon_{\{i\}}]
= n(k-2) t^n \epsilon_I$$ 
the element $t^n \epsilon_I$ belongs to
$[\fC_{{\overline 0},1},\fC_{{\overline 0},k-1}]$ if $n\neq 0$.
Moreover the formula
$$[t^{-1}\epsilon_J, t\epsilon_{\{i\}}]
= (k-2) \epsilon_I$$
shows that $\epsilon_I$ also belongs to 
$[\fC_{{\overline 0},1},\fC_{{\overline 0},k-1}]$
which completes the proof of Assertion (b).
\end{proof}

\subsection{Some  nilpotency criteria}

\begin{lemma}\label{nilK1} Let $R$ be an associative algebra
and let $a, b\in R$. Assume
$$a^n=0\hbox{ and } [a,[a,b]]=0.$$
Then
$$[a,b]^{2n-1}=0$$ 
\end{lemma}

\begin{proof} For any $r\in R$, we have
$$\ad^{2n-1}(a)(r)=
\sum_{k=0}^{2n1} (-1)^k \binom{2n-1}{k} a^k r a^{2n-1-k}=0.$$
Therefore
$$0=\ad^{2n-1}(a)(b^{n-1})=
(2n-1)! [a,b]^n$$
which proves the lemma.
\end{proof}

In the following two lemmas we do not assume that 
the Lie algebras or their modules are finite dimensional.

\begin{lemma}\label{nilK2} Let $\G$ be a Lie algebra,
let $\R$ be a nilpotent ideal and let
$M$ be a $\G$-module. Set
$${\cN}:=\{x\in\R\mid x\vert_M \hbox{ is nilpotent }\}.$$
\begin{enumerate}
\item[(a)] The set ${\cN}$ is a vector space.
\item[(b)] The vector space ${\cN}$ is an ideal
of $\G$.
\end{enumerate}
\end{lemma}

\smallskip
\noindent
{\it Proof of Assertion (a).} For integers $d,n\geq 1$
 let $\fu$ be the free nilpotent Lie algebra of index $d$
over two generators $a$ and $b$ and set 
$$S(d,n):=U(\fu)/I$$
where $I$ is the two sided-ideal generated by $a^n$ and $b^n$.

Let $J\subset S\fu$ be the graded ideal associated with $I$ and let $\sqrt{J}$ be its radical.  Since $\fu$ is finite dimensional, 
Gabber's involutivity Theorem 
\cite{gabber} states that $\sqrt{J}$ is stable by the Poisson bracket of $S\fu$. By definition,  
$a$ and $b$ belong to $\sqrt{J}$ and generate $\fu$. We conclude that
$$\sqrt{J}=\fu S\fu.$$ 
Therefore $J$ and
$I$ have finite codimension, that is the algebra
$S(d,n)$ is finite dimensional.
 
The algebra $S(d,n)$ has a grading
$$S(d,n)=\oplus_{k\geq 0} S(d,n)_k$$
relative to which $a$ and $b$ are homogenous of degree
one. It follows that
$$(a+b)^{N(d,n)}\equiv 0 \hskip1mm\mod I$$
where   $N(d,n)$ is the smallest integer $N$
such that $S(d,n)_N=0$.

Now let $d$ be the nilpotency index of the Lie algebra
$\R$ and $x,y\in\cN$ be nilpotent elements of index $n$.
The associative subalgebra $S\subset\End(M)$ generated
by $x$ and $y$ is a quotient of $S(n,d)$. Therefore
$x+y$ is nilpotent of index $N(d,n)$ which completes the proof of Assertion (a). \qed

\smallskip
\noindent
{\it Proof of Assertion (b).} Let $t$ be a formal variable. For an arbitrary vector space $E$
let $E[[t]]$ be the space of formal series
$$\sum_{k\geq 0}t^k e_k$$
where $e_k\in E$.

Let $\partial$ be an arbitrary element in $\G$.
Then
$$\exp t\,\partial:=\sum_{k\geq 0} \frac{t^k}{ k!} \partial^k
\hbox{ and } \exp t\,\ad(\partial):=
\sum_{k\geq 0}\frac{t^k}{k!}\, \ad(\partial)^k$$
are well defined operators on $M[[t]]$ and
$\R[[t]]$. Observe that for any $x\in\cN$
$$\exp(t\,\ad \partial)(x)$$
is  nilpotent of same index  as $x$ and
$\R[[t]]$ is a nilpotent Lie algebra of same index as $\R$.

Thus by the previous proof 
$$\left(\exp(t\,\ad \partial)(x)-x\right)^{N}=0$$
for some $N>0$. Since
$$\left(\exp(t\,\ad \partial\right)(x)-x)^{N}= t^N\, [\partial,x]^N
+o(T^N)$$
we have proved that $[\partial,x]^N=0$ that is
$[\partial,x]$ belongs to $\cN$. Therefore $\cN$ is an ideal. \qed

\begin{lemma}\label{nilK3} Let $M=M_{\overline 0}\oplus
M_{\overline 1}$ be a free $\C[t,t^{-1}]$-module of finite
rank and let $\fu\subset\End_{\C[t,t^{-1}]}(M)$ be a nilpotent Lie superalgebra. Assume that any $x\in \fu_{\overline 0}$
acts nipotently on $M$.

Then  $\fu$ acts nilpotently on $M$.
\end{lemma}

\begin{proof} Let 
$\fU\subset\End_{\C(t)}(M\otimes \C(t))$ be the
$\C(t)$-vector space generated by $\fu$. So
$\fU$ is a nilpotent Lie superalgebra acting
on the finite dimensional $\C(t)$-vector space
$M\otimes \C(t)$. By Engel's Theorem
$\fU_{\overline 0}$ acts nilpotently on 
$M\otimes \C(t)$. It follows that $\fU$ acts nilpotently
which proves the lemma.
\end{proof}

\subsection{Proof that $\Rad\,C_\L(F)$ acts trivially on the highest weight component}

Let $\L$ be a superconformal algebra and 
let $V$ be a cuspidal
$\L$-module with highest weight $\lambda$. 
In this section we prove that
$\Rad\,C_\L(F)$ acts trivially on $V^{(\lambda)}$.
Indeed the statement is  tautological when
$\Rad\,C_\L(F)=0$ and it has been proved for 
$\L=\W(n)$.

Thus we consider the case  $\L:=\K_*(N)$, with $N\geq 4$.

\begin{lemma}\label{NilKN} Assume $N\geq 5$. Then
$$(\Rad\,C_\L(H))V^{(\lambda)}=0.$$
\end{lemma}

\begin{proof} We first prove that each
$\epsilon_I\in\Sigma$ acts nilpotently on $V^{(\lambda)}$. To simplify we will prove
it for $I=\{1,2\}$, the general proof being 
the same up to a permutation of the indices.

If $N=5$, set 
$$a=\zeta_1\eta_1\zeta_2\xi\hbox{ and }b=\eta_2\xi$$ 
and for $N\geq 6$ set 
$$a=\zeta_1\eta_1\zeta_2\eta_3\hbox{ and }
b=\eta_2\zeta_3.$$

\noindent In both cases, $a$ is a root vector. 
By Proposition \ref{finitely many} $V$ admits
only finitely many weights thus the operator $a$ acts nilpotently on $V$. Moreover
$$[a,b]=\zeta_1\eta_1\zeta_2\eta_2=\epsilon_I.$$
Futhermore $a$ and $\epsilon_I$ are divisible by
$\zeta_1\eta_1$ hence
$$0=[a,\epsilon_I]=[a[a,b]].$$
So by Lemma \ref{nilK1} the operator 
$\epsilon_I$ acts  nilpotently on $V^{(\lambda)}$,
which proves the claim.

Set $\fC:= C_\L(H)$ and $\fR:=\Rad\,C_\L(H)$.
Let us prove that any element $x\in\fR_{\overline 0}$ acts nilpotently
on $V^{(\lambda)}$. 
Observe that $\fR_{\overline 0}$
is a nilpotent ideal in $\fC_{\overline 0}$. Thus
by Lemma \ref{nilK2} the set
$${\cN}:=\{x\in\Rad\,C_\L(H)_{\overline 0}\mid x\vert_{V^{(\lambda)}} \hbox{ is nilpotent}\}$$
is also an ideal of $\fC_{\overline 0}$. 
By Lemma \ref{generators}, the set $\Sigma$ generates  the ideal $\fR_{\overline 0}$ and we have proved that
$\Sigma$ lies in ${\cN}$. We conclude
$$\fR_{\overline 0}={\cN}.$$

We finally prove that $\fR$ acts trivially.
Assume otherwise. Recall that $\fR$ admits a grading
$$\fR=\fR_2\oplus\cdots\fR_m.$$
Let  $k\geq 2$ be the largest integer such that
$\fR_k V^{(\lambda)}\neq 0$ and let
$\fu\subset \End(V^{(\lambda)})$ be the homorphic image of the superalgebra
$$\oplus_{l\geq k}\fR_l.$$

\noindent  By Lemma \ref{dominant}(b), the weight 
$\lambda$ is nonzero. 
Thus choose $h\in H$ with $\lambda(h)=1$.
By Corollary \ref{corh=1} the operator $h(t)$ is invertible. Since 
$$[h(t),\fR_k]\subset \fR_{k+1}$$
all operators in $\fu$ commutes with $h(t)$, that is
$$\fu\subset \End_{\C[h(t),h(t)^{-1}]}(V^{(\lambda)}).$$

\noindent We have proved that $\fu_{\overline 0}$ consists of nilpotent elements. So by Lemma \ref{nilK3} the superalgebra $\fu$ acts nilpotently.
Since $\Ker\,\fu$ is a nontrivial submodule
we have $\Ker\, \fu=V^{(\lambda)}$ which contradicts the hypothesis. This completes the proof.
\end{proof}

\begin{lemma}\label{NilK4} Let $\L=\K(4)$. Then
$$(\Rad\,C_\L(H))V^{(\lambda)}=0.$$
\end{lemma}

For $N=4$ the proof is different because
the element $\zeta_1\eta_1\zeta_2\eta_2$ is not in the derived algebra.

\begin{proof} 
For $\L=K(4)$ the  algebra 
 $\Rad\,C_\L(H)$ is the commutative Lie algebra 
$\fR=\fC_2$ with basis
$$t^n\zeta_1\eta_1\zeta_2\eta_2= t^n \epsilon_{{\{ 1,2}\}}\hskip1mm\hbox{ for } n\in\Z\setminus \{0\}.$$
Set 
$\fk:=\Vir\ltimes \fC _2$ and let
$$\fg:\Vir\ltimes \C[t,t^{-1}]$$
 be the Lie algebra considered in chapter 5.

Since $\fC_2$ is isomorphic to
$\Omega^1_0$ as a $\Vir$-module,
the linear map 
$$\psi: (\partial, f(t))\in\fg
\mapsto (\partial, Df(t)\epsilon_{{\{ 1,2}\}})\in \fk$$
is a Lie algebra homorphism with kernel $\C$.
Hence $V^{(\lambda)}$ is a $\fg$-module on which the element
$1$ acts trivially. By Corollary \ref{corh=1}
$\fC _2$ acts nilpotently on $V^{(\lambda)}$.
Since $\Ker\, \fC _2$ is a nontrivial $\fC$-submodule
we have $\Ker\, \fC _2=V^{(\lambda)}$, that is
we have proved 
$$(\Rad\,C_\L(H))V^{(\lambda)}=0.$$ 
\end{proof}

The following statement explain why we view
the ideal 
$$\Rad(C_\L(F))\subset C_\L(F)$$ 
as a radical.

\begin{cor}\label{radical} Let $\L$ be a superconformal algebra 
of the following list:
$$\W(n),n\geq 2;\hskip1mm \S(n;\gamma), n\geq 2;
\hskip1mm \CK(6);
\K_*(N), N\geq 3; \hskip1mm\widehat{\K(4)}.$$
Let $V$ be a cuspidal $\L$-module with highest weight
$\lambda$.
Then 
$$\Rad(C_\L(F)).V^{(\lambda)}=0.$$
\end{cor}

\begin{proof} For $\L=\W(n)$ we have 
$\Rad(C_\L(F))=\L^{(\epsilon_1+\cdots\epsilon_n)}$.
By Lemma \ref{dominant}, 
$\L^{(\epsilon_1+\cdots\epsilon_n)}$ acts
trivially on $V^{(\lambda)}$, which prove the claim for that case.

For $\L=\K_*(N)$, $N\geq 4$, the claim results from
Lemmas \ref{NilKN} and \ref{NilK4}.

For other superconformal algebras $\L$, the
radical $\Rad(C_\L(F))$ is trivial and the claim
is tautological.
\end{proof}

\subsection{Proof of Theorem \ref{dominant}}

\begin{lemma}\label{f1} The subalgebra $\fm$ generated by
$\L^+$ and $f_1$ contains the Virasoro element
$E_0$.
\end{lemma}

\begin{proof} The proof is based on a case-by-case analysis.

First assume $\L=\W(n)$. Then
$$f_1=\xi_1\frac{\partial}{\partial\xi_2},$$
and $\xi_1 D$ and $\frac{\partial}{\partial\xi_2}$
belong to $\L^+$. Since
$$[f_1, \frac{\partial}{\partial\xi_2}]=
\frac{\partial}{\partial\xi_1}\hbox{ and }
[\frac{\partial}{\partial\xi_1},\xi_1 D]=D$$
the algebra $\fm$ contain $D=-E_0$.

For $\L=\S(n;\gamma)$ we have
$$f_1=\xi_1\frac{\partial}{\partial\xi_2},$$
and $\xi_1 (D+\gamma \xi_2\frac{\partial}{\partial\xi_2})$ and $\frac{\partial}{\partial\xi_2}$ belong to $\L^+$.
Since
$$[f_1, \frac{\partial}{\partial\xi_2}]=
\frac{\partial}{\partial\xi_1}\hbox{ and }
[\frac{\partial}{\partial\xi_1},\xi_1 (D
+\gamma  \xi_2\frac{\partial\xi_2})]=D+\gamma \xi_2\frac{\partial}{\partial\xi_2}$$
the algebra $\fm$ contains 
$D+\gamma\xi_2\frac{\partial}{\partial\xi_2}=-E_0$.

For $\L=\K_*(2m+1)$, we have
$$f_1=\xi\eta_1$$
and $\zeta_1$ belongs to $\L^+$. Since
$$[[f_1[f_1,\zeta_1]=2\eta_1\hbox{ and }
[\eta_1,\xi_1]=D,$$
the algebra $\fm$ contains $D=-E_0$.

For $\L=\K(2m)$, we have
$$f_1=\eta_1\eta_2$$
and $\zeta_1$ and $\zeta_2$ belongs to
$\L^+$. Since
$$[f_1,\zeta_2]=\eta_1\hbox{ and }
[\eta_1,\zeta_1]=D$$
the algebra $\fm$ contains $D=-E_0$.

For $\L=CK(6)$, set $s_1=x_1\wedge x_4$ and
$s_2=x_1\wedge x_3$. For $i=1$ or $2$, consider
$$X_i:=\begin{pmatrix}
    0 & s_1\\
    \phi(s_i)D  & 0
\end{pmatrix}$$
The matrix $S_1$ and $S_2$ belong to $\L^+$ and
a computation in $\hat{P(3)}$ shows that
$$[[f_1,X_2],X_1]=D.$$
Thus $\fm$ contains $E_0=-D$, which completes the proof.
\end{proof}

We now turn to the proof of the main result.

\bigskip\noindent
{\it Proof of Theorem \ref{dominant}.}
Let $\L$ be an untwisted superconformal algebra and let $V$ be a cuspidal module. Recall that
$$C_\L(F)=\G\ltimes\Rad_\L(F)$$
where $G$ is the Lie algebra or superalgebra
$$\fg=\Vir\ltimes H\otimes\C[t,t^{-1}]\hbox{ or }
\fG_*=\K_*(1)\ltimes H\otimes\C[t,t^{-1},\xi].$$

By Lemma \ref{dominant}, $V$ admits
a highest weight $\lambda\neq 0$. 
By Corollary \ref{radical}, the ideal
$\Rad(C_\L(F))$ acts trivially on $V^{(\lambda)}$.
Thus by Theorems \ref{Vir+} and \ref{K+},
the simple $C_\L(F)$-module $V^{(\lambda)}$ is
isomorphic to $\Tens(\lambda,\delta,u)$ for some
$\delta,u\in \C$. Therefore
$$V\simeq V(\lambda,\delta,u).$$

Let $\fm\subset\L$ be the subalgebra generated by
$\L^+$ and $f_1$. Since $E_0.V^{(\lambda)}\neq 0$,
it follows from Lemma \ref{f1} that 
$f_1.V^{(\lambda)}\neq 0$.
However 
$$e_1.V^{(\lambda)}= 0$$ 
and $V$ admits only finitely many weights. Hence by the elementary $\fsl(2)$-theory we conclude that $\lambda(h_1)\geq 1$.
\qed

\vskip2cm
\centerline{\bf PART II: Case-by-case classification of cuspidal modules}

\bigskip\noindent
Let $L$ be a superconformal algebra 
of the following list
$$\W(n),n\geq 2;\hskip1mm \S(n;\gamma), n\geq 2;
\hskip1mm \CK(6);\hskip1mm \K_*(N), N\geq 3,$$
where $\K_*(N)$ denotes the contact superalgebra of
 Ramond $\K(N)$ or of Neveu-Schwarz type $\K_{NS}(N)$. 
In Part I we have shown that any cuspidal
$\L$-module is isomorphic to some module $V(\lambda,\delta,u)$, as it is defined in Chapter \ref{HW}.
However the modules $V(\lambda,\delta,u)$ are not
always cuspidal because their homogenous components
could be infinite dimensional.
Therefore the classification of all cuspidal
$\L$-modules is reduced to the following question:

\bigskip
\hskip15mm{\it Besides the requirement that $\lambda$ is dominant and $\lambda(h_1)\geq 1$,
}

\hskip15mm{\it which condition insures that $V(\lambda,\delta,u)$
is cuspidal?}

\bigskip
Part II provides an answer based on  a case-by-case analysis. 

The proofs will always follow the following pattern.
We use Proposition \ref{hwF} to show that $\lambda(h_1)\geq 2$ is a sufficient condition for cuspidality
of $V(\lambda,\delta,u)$. The analysis of the $\lambda(h_1)=1$  case
is based on determining if $\lambda-2\epsilon_1$ is a weight of $V(\lambda,\delta,u)$.  

\bigskip

\section{Distributions and multidistributions}\label{distribution}
 
Part II uses  tedious computations to check
necessary conditions for cuspidality. In order to ease  these computation, we use 
the setting of formal distributions and formal multidistributions.
For  Part III, we also need the notion of 
locality and semi-locality.

Indeed we do not follow the usual normalization, as in \cite{Kacbook},
for the formal distributions. This allows  a uniform treatment of Ramond and Neveu-Schwarz forms and it is better adapted to
twisted superconformal algebras.

\subsection{Algebra of distributions}
Let $A$ be a vector space and let $z$ for a formal variable.
A formal expression 
$$a(z)=\sum_{n\in\Z}\,a_n z^n$$
where all coefficients $a_n$ lie in $A$ is called a 
{\it $A$-valued distribution}. 
We write
$A[[z,z^{-1}]]$ for the space of 
all $A$-valued distributions and the coefficients $a_n$ are called the {\it modes} of $a(z)$. The Ramond derivation $D_z=z\frac{\d}{\d z}$ obviously acts on the space
$A[[z,z^{-1}]]$. Indeed for $a(z)=\sum_{n\in\Z}\,a_n z^n$
$$D_z(a(z))=\sum_{n\in\Z}\,na_n z^n.$$

Assume now that $A$ is a Lie superalgebra and let
$a(z)=\sum_{n\in\Z}\,a_n z^n$ and 
$b(z)=\sum_{n\in\Z}\,b_n z^n$ be distributions. The product
$[a(z),b(z)]$ cannot be defined. However for each integer
$n$ we can
define the product $[-,-]_n$ by
$$[a(z),b(z)]_n:=[a_n z^n, b(z)].$$

A {\it conformal algebra} of distributions over $A$
is a $\Z/2\Z$-graded  subspace $\cC\subset A[[z,z^{-1}]]$ which is stable
by the Ramond derivation $D_z$ and by all the products
$[-,-]_n$. 
 
Two distributions $a(z), b(z)\in A[[z,z^{-1}]]$ are called
{\it mutually local} if
$$(z_1-z_2)^N [a(z_1),b(z_2)]=0\text{ for some } N>0.$$
In the definition, it is tacitly assumed that $a(z)$ and $b(z)$ belong to $A_{\bar 0}\cup A_{\bar 1}$.

In our setting, Dong Lemma \cite{Dong} is sated as:

\begin{Dong}\label{Dong} Let $\cC\subset A[[z,z^{-1}]]$
be a conformal algebra which is generated by a 
family of  mutually local distributions.

Then  $\cC$ consists of mutually local distributions.
\end{Dong}

\subsection{Algebra of multidistributions}

Given  $ z_1,z_2,\cdots,z_d$ be
$d$  formal variables and  $m:=(m_1,\cdots,m_d)\in\Z^d$ we set
$$z^m=z_1^{m_1}\cdots z_d^{m_d}.$$
A {\it $A$-valued  mutidistribution} in the variables $z_1,\cdots,z_d$ 
is a formal expression
$$a(z_1,\cdots,z_d)=\sum_{m\in\Z^d}\,a_m z^m$$
where all coefficients $a_m$ lie in $A$. 
The Ramond derivations 
$D_{z_i}=z_i\frac{\partial}{\partial z_i}$ act on the space $A[[z_1,z_1^{-1},\cdots,z_d,z_d^{-1}]]$ of multidistributions. Moreover the symmetric group $S_d$ acts on $A[[z_1,z_1^{-1},\cdots,z_d,z_d^{-1}]]$ by
permuting the formal variables $z_1,\cdots,z_d$.

When $A$ is an algebra, we can also defined the product
of two multidistributions $a$ and $b$.  First
we proceed with the formal operadic definition.
Let 
$$a(z_1,\cdots,z_d):=\sum_{m\in\Z^d} a_mz^m
\text{ and }b(z_1,\cdots,z_e):=\sum_{n\in\Z^d} b_nz^n
$$ 
be two
multidistributions, where $d,e$ are positive integers.
We  defined
$$a\circ b(z_1,\cdots,z_{d+e}):=
\sum_{m\in \Z^{e+d}}
a_{(m_1,\cdots,m_d)}b_{(m_{d+1},\cdots,m_{d+e})}
z_1^{m_1}\cdots z_{e+d}^{m_{d+e}}.$$
A space $\cM=\oplus_{d\geq 1} \cM_d$ where
$\cM_d\subset A[[z_1,z_1^{-1},\cdots,z_d,z_d^{-1}]]$ is called an {\it algebra of $A$-valued multidistributions}
if 
\begin{enumerate}
\item[(a)] $\cM_d$ is stable by the action of the
symmetric group $S_d$ and the Ramond derivations
$D_{z_1},\cdots D_{z_d}$.
\item[(b)] $\cM$ is stable by the products, that is
$\cM_d\circ \cM_e\subset \cM_{d+e}$.
\end{enumerate}

Informally, we consider that the product 
$a\circ b$ is legal only if $a$ and $b$ are 
multidistributions in disjoint sets of variables, say
$$a(z_1,\cdots,z_d):=\sum_{m\in\Z^d} a_mz^m
\text{ and }b(z_{d+1},\cdots,z_{d+e}):=\sum_{n\in\Z^d} b_n z_{d+1}^{n_1}\cdots z_{d+e}^{n_e}.$$
Then we write
$$a(z_1,\cdots,z_d)b(z_{d+1},\cdots,z_{d+e})
:=\sum_{(m,n)\in\Z^d\times\Z^e}
a_m b_n z_1^{m_1}\cdots z_d^{m_d}z_{d+1}^{m_{d+1}}
\cdots z_{d+e}^{m_{d+e}}.$$

\subsection{Local multidistributions}\label{seclocalmulti}

For $d\geq d$, an $A$-valued multidistribution 
$a(z_1,\cdots,z_d)=\sum_{m\in\Z^d}\,a_m z^m$ in $d$ variables  is called 
{\it local} if 
$$(z_i-z_j)^M\,a(z_1,\cdots,z_d)=0\, \forall i,j$$
for some integer $M$.

For an integer $N$, set 
$$\Lambda_N=\{m=(m_1,\cdots,m_d)\in \Z^d \mid m_1+\cdots+m_d=N\}\text{ and }$$
$$H_N=\{(x_1,\cdots,x_d)\in \C^d \mid x_1+\cdots+x_d=N\}.$$

A function $P_N: \Lambda_N\to \C$ is called a 
{\it polynomial} if it extends to polynomial function
on the hyperplane $H_N\simeq \C^{N-1}$. Since
$\Lambda_N$ is Zariski-dense in $H_N$ the polynomial extension, if any, is unique.

It is well-known that the multidistribution 
$a(z_1,\cdots,z_d)$ is local if and only if
\begin{enumerate}
\item[(a)] For any $N\in \Z$, the function 
$$P_N:\Lambda_N\to A,\,\, (m_1,\cdots,m_d)
\in \Lambda_N \mapsto a_{m_1}\cdots a_{m_d}$$
is polynomial, and
\item[(b)] the degree of $P_N$ is uniformly bounded,
\end{enumerate}
see \cite{Kacbook}.
Let  $\Mode_N(a)\subset A$ be the subspace
generated by the set $\{a_m\mid m\in \Lambda_N\}$.
Since $P_N$ is a polynomial, it follows that:
\begin{cor}\label{mode} Let $a(z_1,\cdots,z_d)=\sum_{m\in\Z^d}\,a_m z^m$ be a local
distribution. Then
\begin{enumerate}
\item[(a)] The subspaces $\Mode_N(a)$ have finite dimensions
and $\dim\,\Mode_N(a)$ is uniformly bounded.
\item[(b)] Let $N\in\Z$ and let
$\Lambda'\subset\Lambda_N$ be a 
 Zariski dense subset in $H_N$. Then the set
$\{a_m\mid m\in \Lambda'\}$ spans $\Mode_N(a)$.
\end{enumerate}
\end{cor}

\subsection{The multidistribution algebra for  algebras of Ramond type}\label{multiRamond}

A Ramond type superconformal algebra $\L$ has a given
structure of $\C[t,t^{-1}]$-module, and
$$\L=\C[t,t^{-1}]\otimes \L_0.$$
For $a\in \L_0$ set
$$a(z_1,\cdots,z_d):=\sum_{n\in \Z^d} a(t^{n_1+\cdots+n_d})\, z_1^n\cdots z_d^{n_d}.$$
Let $\cM_d(\L)$ be the $\C[D_{z_1},\cdots, D_{z_d}]$-module
generated by the multidistributions $a(z_1,\cdots,z_d)$ when $a$ runs over $\L_0$. The formulas of
chapter \ref{zoology} show that
$$[a(z_1), b(z_2)]\in \cM_2(\L),\text{ for any }
a,\,b\in L_0.$$
For example the formula
$$[t^n D, t^m D]=(m-n) t^{n+m} D$$
can be written as
$$[D(z_1), D(z_2)]=(D_{z_2}-D_{z_1}) D(z_1,z_2).$$
It follows that 
$$\cM(\L):=\oplus_{d\geq 2}\cM_d(L)$$
is an algebra of multidistributions. It is called the
{\it algebra of multidistributions} of $\L$.
It is clear that all distributions in $\cM(\L)$
are pairwise local.

\subsection{Multidistributions  for  algebras of Neveu-Shwarz type}\label{multiNS}

Let $\L$ be   superconformal algebra of  Ramond type.
Recall that its Neveu-Schwarz form $\L_{NS}$ is defined as
$$(\L_{NS})_{\bar 0}=\C[t,t^{-1}]\otimes \L_{\bar 0,0}
\text{ and }
(L_{NS})_{\bar 1}=\C[t,t^{-1}]\otimes t^{\frac{1}{2}}\L_{\bar 1,0}.$$
For $a\in \L_{\bar 0,0}$ and $b\in\L_{\bar 1,0}$
we set
$$a(x_1)=\sum_{n\in\Z}\, a(t^n)\, x_1^n
\text{ and } b(y_1)= \sum_{n\in\frac{1}{2}+\Z}\,  b(t^n) \,y_1^n$$

More generally, we consider two infinite sets $\{x_i\mid i\in \Z\}$
and  $\{y_i\mid i\in \Z\}$ of formal variables. Given $d$ formal variables $z_1,\cdots,z_d$ with
$z_1\in \{x_1, y_1\}\cdots z_d\in \{x_d, y_d\}$,
a monomial
$z_1^{m_1}\cdots z_m^{m_d}$ is called {\it legal}
if $m_i$ is an integer if $z_i=x_i$ and $m_i$ is an half-integer otherwise. Let 
$a\in \L_{\bar 0,0}\cup \L_{\bar 1,0}$. Assume that 
the cardinality of the set $\{i\vert z_i=y_i\}$ has the same parity as $a$. Then we define the multidistribution
$$a(z_1,\cdots,z_d)=
\sum_{m}a(t^{m_1+\cdots+ m_d})\, z_1^{m_1}\cdots z_m^{m_d}$$
where the summation only runs over legal
monomials $z_1^{m_1}\cdots z_m^{m_d}$.
Then the multidistribution algebra of $\cM(\L_{NS})$ is defined
by the space generated by the multidistributions
$$a(z_1,\cdots,z_d)$$ and their iterated Ramond derivatives.
Except the rule about the legal monomials, the algebra
$\cM(\L_{NS})$ is similar to the algebra 
$\cM(\L)$.

\subsection{Multidistributions  for  algebras $\K^{(2)}(2m)$ with $m\geq 2$}

Recall that $\K^{(2)}(2m)$ is the algebra $\L^\sigma$
of fixed points under the involution $\sigma$, where
$\L=\widehat{\K(4)}$ for $m=2$, $\L=\K(2m)$ otherwise.
Recall that
$$\L^\sigma=\C[t^2,t^{-2}]\otimes \L_0^\sigma
\oplus \C[t^2,t^{-2}]\otimes t\L_0^{-\sigma},$$
where $\L_0^{-\sigma}=\{x\in\L-0\mid \sigma(x)=-x\}$.

For $a\in \L_0^\sigma$ and $b\in\L_0^{-\sigma}$
we set
$$a(x_1)=\sum_{n\in\Z}\, a(t^{2n})\,x_1^{2n}
\text{ and } b(y_1)= \sum_{n\in\Z}\,  b(t^{2n+1})\, 
y_1^{2n+1}.$$

More generally, we consider two infinite sets $\{x_i\mid i\in \Z\}$
and  $\{y_i\mid i\in \Z\}$ of formal variables. Given $d$ formal variables $z_1,\cdots,z_d$ with
$z_1\in \{x_1, y_1\}\cdots z_d\in \{x_d, y_d\}$,
a monomial
$z_1^{m_1}\cdots z_m^{m_d}$ is called {\it legal}
if $m_i$ is an even integer  if $z_i=x_i$ and $m_i$ is an odd integer otherwise. Let 
$a\in \L_{\bar 0,0}\cup \L_{\bar 1,0}$. Assume that 
the cardinality of the set $\{i\vert z_i=y_i\}$ has the same parity as $a$. Then we define the multidistribution
$$a(z_1,\cdots,z_d)=
\sum_{m}\,a(t^{m_1+\cdots+ m_d})\, z_1^{m_1}\cdots z_m^{m_d}$$
where the summation only run over legal
monomials $z_1^{m_1}\cdots z_m^{m_d}$.
Then the multidistribution algebra $\cM(\L^\sigma)$ is defined
by the space generated by the multidistributions
$$a(z_1,\cdots,z_d)$$ and their Ramond derivatives.

The algebra of multidistributions $\cM(\K^{(2)}(2m)))$
are {\it mutually semi-local} that is
given two multidistributions
$a(z_1,\cdots,z_d)$ and $b(z_{d+1},\cdots,z_{d+e})$
and any integers $i,j\in\{1,\cdots,d+e\}$,
we have:
$$(z_i^2-z_j^2)^M [a(z_1,\cdots,z_d),b(z_{d+1},\cdots,z_{d+e})]=0\text{ for }  M>>0.$$

\section{Classification of cuspidal $\W(n)$-modules}
\label{W(n)}

Recall from Section \ref{defW}  that 
$$\W(n)=\Der\,\C[t,t^{-1},\xi_1,
\cdots,\xi^n].$$

\noindent Its Cartan subalgebra $H$  has basis
$(\xi_i\frac{\partial}{\partial \xi_i})_{1\leq i\leq n}$ and let $(\epsilon_i)_{1\leq i\leq n}$ be the dual basis of $H^*$. For $1\leq i <n$, set  
$$h_i=\xi_{i}\frac{\partial}{\partial \xi_{i}}
-\xi_{i+1}\frac{\partial}{\partial \xi_{i+1}}.$$

For 
$$(\lambda,\delta,\,u)\in H^*\times\C\times \C,$$
we have defined 
the $\W(n)$-module $V(\lambda,\delta,u)$
in Chapter \ref{HW}. 
Recall that
the definition of $V(\lambda,\delta,u)$
depends on the chosen triangular decomposition
$\L=\L^+\oplus \L^{(0)}\oplus \L^-$
of the superconformal algebra $\L:=\W(n)$.
It has been shown in Chapter \ref{HW} that
$V(\lambda,\delta,u)$ is cuspidal only if
$\lambda$ is  dominant and $\lambda(h_1)\geq 1$.

Recall that $\lambda_i=\lambda(\xi_i\frac{\partial}{\partial\xi_i})$.

\begin{thm}\label{ClasW} Let  $n\geq 2$. For an arbitrary triple
$$(\lambda,\delta,\,u)\in H^*\times\C\times \C,$$
with $\lambda$  dominant,
the $\W(n)$-module $V(\lambda,\delta,u)$ is cuspidal
if and only if  
\begin{enumerate}
\item[(a)] either $\lambda(h_1)\geq 2$
\item[(b)] or $\lambda(h_1)=1$ and $\lambda_1=1-\delta$.
\end{enumerate}
\end{thm}

Since we will repeatedly use the same scheme of proof
in our case-by-case analysis, we will provide some details. The sufficiency of Conditions (a) and (b),
namely Corollary\ref{suffitW},  uses  Proposition \ref{hwF} about coinduction functors. 

Their necessity is based on explicit tedious
computations. In order to alleviate the computation, we use 
the setting of  multidistributions, see Section \ref{distribution}.
Then the necessity of Conditions (a) and (b) is shown
by Corollary \ref{necessaryW}.

\subsection{Highest weights of finite dimensional $\fgl(1,n)$-modules}
Consider the Lie superalgebra 
$\fgl(1,n)$. Then 
$\fgl(1,n)_{\overline 0}$ has basis
$$\{e_{0,0}\}\cup \{e_{i,j}\mid 1\leq i,j\leq n\}$$
and $\fgl(1,n)_{\overline 1}$ has basis
$$\{e_{0,j}\mid j\neq 0\}\cup \{e_{i,0}\mid i\neq 0\}.$$
The Lie bracket is

$$[e_{i,j},e_{k,l}]=\delta_{j,k} e_{i,l}
- (-1)^{\vert e_{i,j}\vert\vert e_{k,l}\vert}
\delta_{i,l} e_{k,j}.$$

Set 

$$H=\oplus_{i=1}^n\,\C e_{i,i}$$
and let $\epsilon_i$ be the dual basis of $H^*$.
The action of $H$ on $\fgl(1,n)$ provides an eigenspace
decomposition

$$\L=H^e+\sum _{\alpha\in \Delta(\fk)}\,\fgl(1,n)^{(\alpha)},$$

\noindent where $\Delta(\fgl(1,n)):=\{\alpha\in H^*\mid \alpha\neq 0 \hbox{ and } \fgl(1,n)^{(\alpha)}\neq 0\}$.

Let again   $F\in H$ be the element defined by

$$\epsilon_1(F)=2^n-1\hbox { and }\epsilon_k(F)=-2^{k-1}
\hbox{ for } k>1,$$
and set

$$\Delta^+(\fgl(1,n))=\{\alpha\in\Delta\mid \alpha(F)>0\},\hbox{ and }$$
$$\Delta^-(\fgl(1,n))=\{\alpha\in\Delta\mid \alpha(F)<0\}.$$

More explicitly

$$\Delta^+(\fgl(1,n))=\{\epsilon_1\}\cup
\{-\epsilon_i\mid 1<i\leq n\}\cup \{\epsilon_i-\epsilon_j\mid
1\leq i<j\leq n\}$$
and the corresponding roots vectors are
$e_{1,0}, e_{0,i}$ and $e_{i,j}$.

For a  pair $(\lambda,\delta)\in H^*\times\C$, let $L(\lambda,\delta)$ be the
simple  $\fgl(1,n)$-module with highest weight
$(\lambda,\delta)$. By this we mean that
 $L(\lambda,\delta)$ is generated by a vector $v$ satisfying

$$h.v=\lambda(h) v\hbox{ for } h\in H$$

$$e_{0,0}.v=\delta v$$

$$\fgl(1,n)^{(\alpha)}.v=0
\hbox{ if } 
\alpha\in\Delta^+(\fgl(1,n)).$$

If $L(\lambda,\delta)$ is finite dimensional,
$\lambda$ is dominant, that is
$$\lambda(h_i)\in\Z_{\geq 0}\hbox{ for }
 1\leq i\leq n-1,$$ 
where $h_i=e_{i,i}-e_{i+1,i+1}$. 
However this condition does not insure that  
$\dim\, L(\lambda,\delta)<\infty$. Set
$\lambda_1=\lambda(e_{1,1})$.

\begin{lemma}\label{gl(1,n)} Let  $(\lambda,\delta)\in H^*\times\C$
be an arbitrary pair with $\lambda$ dominant.
If
\begin{enumerate}
\item[(a)] either $\lambda(h_1)\geq 1$

\item[(b)] or $\lambda(h_1)=0$ and $\lambda_1+\delta=0$
\end{enumerate}

\noindent the $\fgl(1,n)$-module $L(\lambda,\delta)$ is
finite dimensional.
\end{lemma}

\begin{proof} We first  define the Kac modules.
To simplify the notation, set $\fk=\fgl(1,n)$.

We observe that 
$\fk_{\overline 0}\simeq \fgl(n)\times \fgl(1)$.
For a  pair $(\mu,\delta)\in H^*\times\C$, let 
$S(\mu,\delta)$ be the
simple $\fk_{\overline 0}$-module with highest weight
$(\mu,\delta)$. It is generated by a vector $s$
 satisfying
$$h.s=\mu(h) s\hbox{ for } h\in H$$
$$e_{0,0}.s=\delta s$$
$$\fk_{\overline 0}^\alpha.s=0
\hbox{ if } 
\alpha\in\Delta_{\overline 0}^+(\fk).$$

Set $T:=\oplus \oplus_{i=1}^n \fk^{-\epsilon_i}$.
Since $T$ is an ideal in the subalgebra
$\fp:=\fk_{\overline 0}\oplus T$ we can consider
$S(\mu,\delta)$ as a $\fp$-module on which 
$T$ acts trivially. The $\fk$-module
$$\Kac(\mu,\delta):=\Ind_\fp^\fk\,S(\mu,\delta)$$
is called the {\it Kac module}.

First assume that $\lambda(h_1)\geq 1$. Thus the weight
$\mu:=\lambda-\epsilon_1$ is dominant.
By Poincar\'e-Birkhoff-Witt Theorem, we have
$$\Kac(\mu,\delta+1)\simeq \C[e_{1,0},\cdots,e_{n,0}]
S(\mu,\delta+1),$$
thus  $\Kac(\mu,\delta+1)$ is finite dimensional.
For $i\geq 2$  the root $\epsilon_i$ is negative, hence $\lambda=\mu+\epsilon_1$ is the highest weight of 
$\Kac(\mu,\delta+1)$. Also we observe that
 $\Kac(\mu,\delta+1)^{(\lambda)}=\C e_{1,0}.s$.

Set $v=e_{1,0}.s$. We have

$$e_{0,0}.v=e_{0,0} e_{1,0} s=[e_{0,0}, e_{1,0}]s
+e_{1,0}e_{0,0} s=-e_{1,0}.s+ (\delta+1)e_{1,0}.s
=\delta v.$$ 

Hence $K(\mu,\delta+1)$  admits a subquotient isomorphic to $L(\lambda,\delta)$ which proves that 
$L(\lambda,\delta)$ is finite dimensional.

Next assume that $\lambda_1+\delta=0$.
Set 
$$\fk^{\pm}:=\oplus_{\alpha\in\Delta^+(\fk)}\,\fk^\alpha.$$

It has been shown that $\lambda+\epsilon_1$ is the highest
weight of $\Kac(\lambda,\delta)$. Let $X$ be the 
$\fk$-submodule generated by $v$. 
 By Poincar\'e-Birkhoff-Witt Theorem we have
 $X=U(\fk^-).v$, and more precisely

$$X=U(\fk_{\overline 0}^-)
\C[e_{2,0},\cdots,e_{n,0}]\C[e_{0,1}].v.$$

However we have
$$e_{0,1}.v=e_{0,1}e_{1,0}.s
=[e_{0,1},e_{1,0.}].s=(e_{0,0}+e_{1,1})s=(\lambda_1+\delta)s=0.$$

Therefore $X=U(\fk_{\overline 0}^-)
\C[e_{0,2},\cdots,e_{0,n}]v,$
which shows that $X^{(\lambda)}=0$. 
Since it has been observed that
$\Kac(\lambda,\delta)^{(\lambda+\epsilon_1)}=\C v$,
it follows that
$\lambda$ is a maximal weight of
$\Kac(\lambda,\delta)/X$. Thus $L(\lambda,\delta)$ is a subquotient of $\Kac(\lambda,\delta)/X$, which shows that
$L(\lambda,\delta)$ is finite dimensional.
\end{proof}

\subsection{Sufficiency of Conditions (a) and (b)}

Let
$(\lambda,\delta,u)\in $ be an arbitrary triple
with $\lambda$ dominant. In this section we prove:

\begin{cor}\label{suffitW} Assume that
\begin{enumerate}
\item[(a)] either $\lambda(h_1)\geq 2$
\item[(b)] or $\lambda(h_1)=1$ and $\lambda_1=1-\delta$.
\end{enumerate}
Then the $\W(n)$-module $V(\lambda,\delta,u)$ is cuspidal.
\end{cor}

\begin{proof}

Let ${\bf m}\subset \C[t,t^{-1},\xi_1,
\cdots,\xi^n]$ be the the maximal ideal
$${\bf m}:=(t,\xi_1,\cdots,\xi_n).$$

Set $\W(n)^{\bf (1)}={\bf m}\W(n)$,
$\W(n)^{\bf (2)}={\bf m}^2\W(n)$, and
$T=\W(n)/\W(n)^{\bf (1)}$.
(Intuitively ${\bf m}$ is a point of the supermanifold
$(\C^*)^{1,n}$, the subalgebra
$\W(n)^{\bf (1)}$
is the isotropy at the point ${\bf m}$,
$\W(n)^{\bf (2)}$ consists of derivations vanishing twice
at ${\bf m}$ and $T$ is the tangent space at this point.)

Clearly $\W(n)^{\bf (2)}$ is an ideal of $\W(n)^{\bf (1)}$
and 
$$\W(n)^{\bf (1)}\simeq \fgl(1,n)\ltimes\W(n)^{\bf (2)},$$
where $\fgl(1,n)$ is the subalgebra with basis
$e_{0,0}=(t-1)D$, $e_{i,0}=\xi_i D$, 
$e_{0,i}=(t-1) \frac{\partial}{\partial \xi_i}$, and
$e_{i,j}=\xi_i\frac{\partial}{\partial\xi_j}$
for $1\leq i,j\leq n$.

Set $\mu=\lambda-\epsilon_1$. We consider
the $\fgl(1,n)$-module $L(\mu,\delta)$ as a
$\W(n)^{\bf (1)}$-module with a trivial action of
$\W(n)^{\bf (2)}$.
By Lemma \ref{gl(1,n)}

$$\dim\,L(\mu,\delta)<\infty.$$   

We check the hypotheses of Proposition \ref{hwF}.
Here the grading element $\ell_0$ is the Ramond derivation $D$ and 
$$\W(n)_{\overline 0}=\C D\oplus \W(n)^{\bf (1)}_{\overline 0}.$$ 

The tangent space $T$ has basis 
$$\{ D\}\cup\{\frac{\partial}{\partial\xi_i}\mid
1\leq i\leq n\}\hbox{ modulo }\W^{\bf(1)}.$$
Thus 
$$T^-\simeq\C \frac{\partial}{\partial\xi_1}.$$
Since $[t-1)D,\frac{\partial}{\partial\xi_1}]=0$ we have
$\delta^-=0$. 

By 
Proposition \ref{hwF} the highest weight of
$\cF(L(\mu,\delta),u)$ is $\lambda=\mu+\epsilon_1$ and
$$\cF(L(\mu,\delta),u)^{(\lambda)}\simeq \Tens(\lambda,\delta,u).$$
Since $V(\lambda,\delta,u)$ is a subquotient of
$\cF(L(\mu,\delta),u)$, we conclude that 
$V(\lambda,\delta,u)$ is cuspidal.
\end{proof}

\subsection{Explicit computations}\label{computW}

Let
$(\lambda,\delta,\,u)\in H^*\times\C\times \C,$
and set $V=V(\lambda,\delta,\,u)$. 
We are looking 
for a condition which insures that
$V^{\lambda-2\epsilon_1}\neq 0$. 

We freely use the definition of section \ref{multiRamond}.
Here our  formal variables are $v,w,x,y$ and $z$.
Recall that for any element $a\in \L_0$, we define the formal series
$$a(v)=\sum_{n\in\Z}\, (t^na) v^n\in \L[[v,v^{-1}]].$$
We will also consider $V$-valued multidistribution,
so there all considered multidistribution are valued in the  Lie superalgebra $\L\ltimes V$. 

We have

$$[\xi_1D(v),\frac{\partial}{\partial\xi_1}(w)]
=D_w \xi_1\frac{\partial}{\partial\xi_1}(v,w) + D(v,w),$$

$$[\xi_1\xi_k\frac{\partial}{\partial\xi_k}(v),
\frac{\partial}{\partial \xi_1}(w)]=
\xi_k\frac{\partial}{\partial\xi_k}(v,w)\hbox{ for }k\neq 1.$$

We deduce

$$\left[\xi_1D(v),
[\xi_1D(w),\frac{\partial}{\partial\xi_1}(x)]\right]
=(D_w-D_v)\,\xi_1D(v,w,x),$$

$$
\left[\xi_1\xi_k\frac{\partial}{\partial\xi_k}(v),
[\xi_1D(w),\frac{\partial}{\partial\xi_1}(x)]\right]
=-(D_v+D_x)\xi_1\xi_k\frac{\partial}{ \partial\xi_k}(u,w,x),$$

$$\left[\xi_1\xi_k\frac{\partial}{\partial\xi_k}(v),
[\xi_1\xi_l\frac{\partial}{\partial\xi_l}(w),
\frac{\partial}{\partial \xi_1}(x)]\right]=
0.$$

In what follows we will always implicitly exclude the case $\lambda=0$.
Thus $V^{(\lambda)}=\Tens(\lambda,\delta,u)$ is
a $\C[t,t^{-1}]$-module. Thus for
$\v\in V^{(\lambda)}_0$ the element $\v(t^n):=t^n\otimes \v$ is well 
defined. As before, we can define the formal distribution
$$\v(z):=\sum_{n\in\Z}\,\v(t^n) z^n.$$
We have
$$D(v)\v(z)= (\delta D_v+\Delta_z)\v(u,z),$$
where $\Delta_z=D_z+u$.

\begin{lemma}\label{formulaW} For $\v\in  V^{(\lambda)}_0$, 
set $\tilde{\v}=\v(v,w,x,y,z)$. Let
$2\leq k,l\leq n$. 
We have

\bigskip\noindent
 (a) $ \xi_1D(v)\,
\xi_1D(w)\,\frac{\partial}{\partial\xi_1}(x)\,
\frac{\partial}{\partial\xi_1}(y)\,\v(z)=
\delta(1-\delta-\lambda_1) (D_v-D_w)(D_x-D_y)\tilde{\v}$

\bigskip\noindent
(b) $\xi_1\xi_k\frac{\partial}{\partial\xi_k}(v)\,
\xi_1D(w)\,\frac{\partial}{\partial\xi_1}(x)\,
\frac{\partial}{\partial\xi_1}(y)\,\v(z)=
\lambda_k(\lambda_1+\delta-1)(D_x-D_y)\tilde{\v}$

\bigskip\noindent
(c) $\xi_1\xi_k\frac{\partial}{\partial\xi_k}(v)\,
\xi_1\xi_l\frac{\partial}{\partial\xi_l}(w)\,\frac{\partial}{\partial\xi_1}(x)\,
\frac{\partial}{\partial\xi_1}(y)\,\v(z)=0$

\end{lemma}

\begin{proof}

Set $\L=\W(n)$. Let $A_1,A_2$ be formal distributions in  $\L^{(\epsilon_1)}$ and
$B_1,B_2$ be formal distributions in 
$\L^{(-\epsilon_1)}$. Since the distribution
$[A_1,[A_2,B_2]]\v$ has weight 
$\lambda +\epsilon_1$, we have $[A_1,[A_2,B_2]]\v=0$.

We deduce the formula
$$A_1A_2B_1B_2\v= A+B-C$$
where
$$A:=[A_1,[A_2,B_1]B_2\,\v,$$
$$B:=[A_2,B_1][A_1,B_2]\,\v, \hbox{ and }$$
$$C:=[A_1,B_1][A_2,B_2])\,\v$$
which will be used in the next computations. 
Also we have 
$$C=[A_1,[A_2,B_1]B_2.\v=[[A_1,[A_2,B_1]B_2]\,\v.$$

{\it Proof of (a)}  In the Formula (a), we have
$A_1= \xi_1D(v)$, $A_2=
\xi_1D(w)$, $B_1=\frac{\partial}{\partial\xi_1}(x)$ and 
$B_2=\frac{\partial}{\partial\xi_1}(y)$. Thus we  decompose
$ \xi_1D(v)\,
\xi_1D(w)\,\frac{\partial}{\partial\xi_1}(x)\,
\frac{\partial}{\partial\xi_1}(y)\v(z)$ as $A+B-C$.
We compute now each term $A$, $B$ and $C$.

$$A=\left[\xi_1 D(v),
[\xi_1 D(w),\frac{\partial}{\partial\xi_1}(x)]\right]
\,\frac{\partial}{\partial\xi_1}(y)\,\v(z)$$

$$=\left((D_w-D_v)\xi_1D(v,w,x)\right)\, \frac {\partial}{\partial\xi_1}(y)\v(z)$$

$$=\left[(D_w-D_v)\xi_1D(v,w,x), \frac{\partial}{\partial\xi_1}(y)\right]\v(z)$$

$$=(D_w-D_v)\left(D_y 
\xi_1\frac{\partial}{\partial\xi_1}(v,w,x,y)
+D(v,w,x,y)\right)\,\v(z)
$$

$$= (D_w-D_v) \left(\lambda_1 D_y
+(\delta(D_v+D_w+D_x+D_y)+\Delta_z\right)
\,\tilde{\v}$$

$$=(D_w-D_v) \left((\delta(D_v+D_w)+\delta D_x
+(\delta+\lambda_1)D_y +\Delta_z\right)
\,\tilde{\v}$$

\bigskip
$$B=[ \xi_1D(w),\frac{\partial}{\partial\xi_1}(x)]
\,[\xi_1D(v),\frac{\partial}{\partial\xi_1}(y)]\,\v(z)$$

$$=\left(D_x \xi_1\frac{\partial}{\partial\xi_1}(w,x) + D(w,x)\right)
\,\left(D_y \xi_1\frac{\partial}{\partial\xi_1}(v,y) + D(v,y)\right)
\,\v(z)$$

$$=\left(\lambda_1 D_x+ \delta(D_x+D_w)
+D_v+D_y +\Delta_z\right)
\left(\lambda_1 D_y + \delta(D_v+D_y)+\Delta_z
\right)\tilde{\v}$$

$$=\left(D_v+\delta D_w+(\delta+\lambda_1)D_x+D_y+\Delta_z\right)\,
\left(\delta D_v+(\lambda_1+\delta) D_y  +\Delta_z\right)
\,\tilde{\v}$$

Since 
$$C=[ \xi_1D(v),\frac{\partial}{\partial\xi_1}(x)]
\,[\xi_1D(w),\frac{\partial}{\partial\xi_1}(y)]\,\v(z)$$

\noindent the formula for $C$ is the same  as for $B$
 by exchanging the roles $v$ and $w$. Therefore
$$C=(\delta D_v+ D_w+(\delta+\lambda_1)D_x+D_y
+\Delta_z)
(\delta D_w+(\lambda_1+\delta) D_y +\Delta_z)
\tilde{\v}$$

Since  $[\xi_1D(v),
\xi_1D(w)]=0$ and $[\,\frac{\partial}{ \partial\xi_1}(x),
\frac{\partial}{\partial\xi_1}(y)]=0$ Formula (a)
is skew symmetric in the variables $(v,w)$ and in the variables $(x,y)$. Hence we have

$$A+B-C=f(\delta,\lambda_1,\Delta_z) (D_v-D_w)(D_y-D_x)\tilde{\v}$$
for some function $f$. Thus it is enough to evaluate
the coefficient of the monomial $D_v D_y$. This coefficient is
$-(\delta+\lambda_1)$ in expression $A$,
$\lambda_1+2\delta$ in expression $B$ and
$\delta(\lambda_1+\delta)$ in expression $C$.
Therefore $A+B-C$ contains the monomial
$\delta(1-\delta-\lambda_1)$, which proves that

$$A+B-C=
\delta(1-\delta-\lambda_1)(D_v-D_w) (D_x-D_y)\,\tilde{\v}, 
$$
which proves Formula (a).

{\it Proof of (b)}
Assume $k\neq 1$. We similarily decompose Formula (b) as a sum $A+B-C$. We have

$$A=-(D_v+D_x)\xi_1\xi_k\frac{\partial}{ \partial\xi_k}(v,w,x)\,\frac{\partial}{ \partial\xi_1}(y)\,\v(z)$$

$$=-(D_v+D_x)
\left[\xi_1\xi_k\frac{\partial}{ \partial\xi_k}(v,w,x),\frac{\partial}{ \partial\xi_1}(y)\right]\,\v(z)$$

$$=-(D_v+D_x)\xi_k\frac{\partial}{ \partial\xi_k}(v,w,x,y)
\,\v(z)$$

$$=-\lambda_k(D_v+D_x)\tilde{\v}$$

$$B= [\xi_1D(w),\frac{\partial}{ \partial\xi_1}(x)]
\,[\xi_1\xi_k\frac{\partial}{ \partial\xi_k}(v),
\,\frac{\partial}{\partial \xi_1}(y)]\,\v(z)$$

$$=\left(D_x \xi_1\frac{\partial}{ \partial\xi_1}(w,x) + D(w,x)\right)\,
\xi_k\frac{\partial}{ \partial\xi_k}(v,y)\,\v(z)$$

$$=\lambda_k(\lambda_1 D_x+D_v+D_y+\Delta_z+
\delta D_w + \delta D_x)\tilde{\v}$$

$$=\lambda_k(D_v+\delta D_w + (\lambda_1+ \delta) D_x
+D_y+\Delta_z)\,\tilde{\v}$$

$$C=[\xi_1\xi_k\frac{\partial}{ \partial\xi_k}(v),
\frac{\partial}{\partial \xi_1}(x)]\,
[\xi_1D(w),\frac{\partial}{ \partial\xi_1}(y)]\v(z)
$$

$$=\xi_k\frac{\partial}{ \partial\xi_k}(v,x)
\left(D_y \xi_1\frac{\partial}{ \partial\xi_1}(w,y) + D(w,y)\right)\,\v(z)$$

$$=\lambda_k(\lambda_1 D_y+
\delta D_w+\delta D_y +\Delta_z)\,\tilde{\v}$$
Thus
$$A+B-C=\lambda_k(\lambda_1+\delta-1) 
(D_x-D_y)\tilde{\v},$$
which proves Formula (b).

Formula (c) is obvious.
\end{proof}

\subsection{Necessity of Conditions (a) and (b)}

Let
$(\lambda,\delta,\,u)\in H^*\times\C\times \C,$
and set $V:=V(\lambda,\delta,\,u)$. 

\begin{lemma}\label{necW} Assume that $\lambda(h_1)=1$. If
$V^{(\lambda-2\epsilon_1)}=0$, then
$$\lambda_1=1-\delta.$$
\end{lemma}

\begin{proof} Set $\L=\W(n)$. By hypothesis
we have
$$\L^{(\epsilon_1)}\L^{(\epsilon_1)}\L^{(-\epsilon_1)}\L^{(-\epsilon_1)} V^{(\lambda)}=0.$$
Therefore by Lemma \ref{formulaW}, we have
$\delta(1-\delta-\lambda_1)=0$ and 
$\lambda_2(\lambda_1+\delta-1)=0$.
Thus
$$(\lambda_2+\delta)(\lambda_1+\delta-1)=0.$$
Since $\lambda(h_1)=\lambda_1-\lambda_2$, we have
$\lambda_2+\delta=\lambda_1+\delta-1$. We deduce
$$(\lambda_1+\delta-1)^2=0,$$
which proves the Lemma.
\end{proof}

\begin{cor}\label{necessaryW}
Assume that the $\W(n)$-module 
$V(\lambda,\delta,u)$ is cuspidal.
Then
\begin{enumerate}
\item[(a)] either $\lambda(h_1)\geq 2$,
\item[(b)] or $\lambda(h_1)=1$ and $\lambda_1=1-\delta$.
\end{enumerate}
\end{cor}

\begin{proof} By Proposition \ref{condition1}, 
$\lambda(h_1)$ is an integer $\geq 1$.
Set $V=V(\lambda,\delta,u)$.

If $V^{(\lambda-2\epsilon_1)}=0$, then Condition (a) or
(b) is satisfied by Lemma \ref{necW}. 

Assume otherwise. Any
negative root $\alpha\in\Delta^-$ satisfies

\centerline{$\alpha=-\epsilon_1 \hbox{ modulo }\Lambda$, or}

\centerline{$\alpha=-\epsilon_1+\epsilon_2 \hbox{ modulo }\Lambda$, or}

\centerline{$\alpha=\epsilon_2 \hbox{ modulo }\Lambda$, or}

\centerline{$\alpha$ belongs to $\Lambda$,}

\noindent where $\Lambda:=\oplus_{k\geq 3}\Z\epsilon_k$.
Hence
 $-\epsilon_1-\epsilon_2$ is not
a positive combination of negative roots. Thus
$V^{(\lambda-\epsilon_1-\epsilon_2)}=0$.

It follows that $e_1 V^{(\lambda-2\epsilon_1)}=0$. Since
$V$ contains only finitely many weights, we have
$(\lambda-2\epsilon_1)(h_1)\geq 0$, which proves that
$\lambda(h_1)\geq 2$.
\end{proof}

\section{Classification of cuspidal $\S(n;\gamma)$-modules}

 To insure that $\S(n;\gamma)$ is simple, we assume $n\geq 2$. By definition, $\S(n;\gamma)$ is a $\Z$-graded subalgebra of $\W(n)$. Its Cartan subalgebra is
$$H^0:=\oplus_{i=1}^{n-1}\,\C h_i.$$
In Section \ref{defS} an embedding
$\Vir\subset S(n,\gamma)$ has been chosen.
A weight $\lambda\in (H^0)^*$ is dominant if and only if
all $\lambda(h_i)$ are nonnegative integers.

For any triple 
$(\lambda,\delta,\,u)\in (H^0)^*\times\C\times \C,$ we have  defined the $\S(n;\gamma)$-module 
$V(\lambda,\delta,u)$ in Chapter \ref{HW}. 
The  result of the chapter is:

\begin{thm}\label{ClasS} Let
$$(\lambda,\delta,\,u)\in (H^0)^*\times\C\times \C,$$
be an arbitrary triple with $\lambda$ dominant. Then
the $\S(n;\gamma)$-module $V(\lambda,\delta,u)$ is cuspidal if and only if 

\begin{enumerate}
\item[(a)] either $\lambda(h_1)\geq 2$

\item[(b)] or $\lambda(h_1)=1$ and  
$\delta=1$.
\end{enumerate} 
\end{thm}

The proof is an immediate consequence of
Lemma \ref{suffitS} and Corollary \ref{necessaryS}
proved below.

\subsection{Sufficiency of Conditions (a) and (b)}

Let $(\lambda,\delta,u)\in (H^0)^*\times\C\times\C$
be an arbitrary triple with $\lambda$ dominant.

\begin{lemma}\label{suffitS}
Assume that
\begin{enumerate}
\item[(a)] either $\lambda(h_1)\geq 2$
\item[(b)] or $\lambda(h_1)=1$ and $\lambda_1=1-\delta$.
\end{enumerate}
Then the $\S(n;\gamma)$-module $V(\lambda,\delta,u)$ is cuspidal.
\end{lemma}

\begin{proof} By definition $\S(n;\gamma)$ is a subalgebra of $\W(n)$ and $H^0$ is a $1$-co\-dimen\-sional subspace
of the Cartan subalgebra $H$ of $\W(n)$. The weight $\lambda$ admits a unique
extension $\tilde{\lambda}\in H^*$ with 
$\tilde{\lambda}_1=0$. 

The subalgebras $\Vir$ of $\W(n)$ and $\S(n;\gamma)$
are different. However
on the
$C_{\W(n)}(F)$-module $\Tens(\tilde{\lambda},\delta,u)$ the Virasoro element
$$E_n=-\left(t^nD +(n+\gamma)t^n\xi_1\frac{\partial}
{\partial \xi_i}\right)\in \S(n;\gamma)$$
acts as the Virasoro
element $-t^nD\in \W(n)$. Therefore as a
$C_{\S(n;\gamma)}(F)$-modules 
$\Tens(\tilde{\lambda},\delta,u)$ and
$\Tens(\lambda,\delta,u)$ are isomorphic.

It follows that the $\S(n;\gamma)$-module
$V(\lambda,\delta,u)$  is a subquotient of the
$\W(n)$-module $V(\tilde{\lambda},\delta,u)$,
which is cuspidal by Corollary \ref{suffitW}. So
$V(\lambda,\delta,u)$ is also cuspidal.
\end{proof}

\subsection{Necessity of Conditions (a) and (b)}

Let $(\lambda,\delta,u)\in (H^0)^*\times\C\times\C$
be an arbitrary triple. Set
$V=V(\lambda,\delta,u)$.

For $n\in\Z$, define the elements
$A_n\in S(n:\gamma)$ as
$$A_n=t^n \xi_1 D+(n+\gamma) t^n \xi_1\xi_2
\frac{\partial}{\partial \xi_2}.$$
Let consider the formal distribution 
$A(v)=\sum_{n\in\Z} A_n v^n$. With the same notations as
in Subsection \ref{computW}:

\begin{lemma}\label{formulaS}
 Let $\v\in \tilde{V}^{(\lambda)}_0$.
We have
$$ A(v)\,
A(w)\,\frac{\partial}{ \partial\xi_1}(x)\,
\frac{\partial}{ \partial\xi_1}(y)\,\v(z)=
(\delta-\lambda(h_1))(1-\delta) (D_v-D_w)(D_x-D_y)\tilde{\v},$$
where $\tilde{\v}=\v(v,w,x,y,z)$.
\end{lemma}

\begin{proof} As before we extend
$\lambda$ to a linear form
$\tilde{\lambda}\in H^*$ satisfying $\tilde{\lambda}_1=0$.
Let $\tilde{V}$ be the $\W(n)$-module
$V(\tilde{\lambda},\delta,u)$. Since $V(\lambda,\delta,u)$
is a quotient of the $\S(n;\gamma)$-module generated by
$\tilde{V}^{(\lambda)}$ it is enough to show the same formula for $\v\in V(\tilde{\lambda}_0$.

Set $\Delta_v=D_v+\gamma$ and 
$\Delta_w=D_w+\gamma$. We have
$$A(v)=\xi_1D(v)+\Delta_v \xi_1\xi_2
\frac{\partial}{\partial \xi_2}(v),$$
thus $A(v)A(w)= A+B+C+D$ where
$$A:= \xi_1D(v)\xi_1D(w),$$
$$B:=\Delta_v \xi_1\xi_2
\frac{\partial}{\partial \xi_2}(v)
\xi_1D(w),$$
$$C=:\xi_1D(v) \Delta_w \xi_1\xi_2
\frac{\partial}{\partial \xi_2}(w)\hbox{ and }$$
$$D:=\Delta_v \xi_1\xi_2
\frac{\partial}{\partial \xi_2}(v)
\Delta_w \xi_1\xi_2
\frac{\partial}{\partial \xi_2}(w).$$

Hence $ A(v)\,
A(w)\,\frac{\partial}{ \partial\xi_1}(x)\,
\frac{\partial}{ \partial\xi_1}(y)\,\v(z)$ is the sum of four terms that we compute succesively. By
Lemma \ref{formulaW}:
$$ A\,\frac{\partial}{ \partial\xi_1}(x)\,
\frac{\partial}{ \partial\xi_1}(y)\,\v(z)
=\delta(1-\delta) (D_v-D_w)(D_x-D_y)
\tilde{\v},$$

$$B\,\frac{\partial}{ \partial\xi_1}(x)\,
\frac{\partial}{ \partial\xi_1}(y)\,\v(z)=\Delta_v
\tilde{\lambda_2}(\delta-1)(D_x-D_y)\tilde{\v}.$$

\noindent We observe that
$C=-\Delta_w \xi_1\xi_2
\frac{\partial}{\partial \xi_2}(w)\,\xi_1D(v)$, thus

$$C\,\frac{\partial}{ \partial\xi_1}(x)\,
\frac{\partial}{ \partial\xi_1}(y)\,\v(z)=-\Delta_w
\tilde{\lambda_2}(\delta-1)(D_x-D_y)\tilde{\v}.$$

By Lemma \ref{formulaW} we have

$$D\,\frac{\partial}{ \partial\xi_1}(x)\,
\frac{\partial}{ \partial\xi_1}(y)\,\v(z)=0.$$

We observe that $\tilde{\lambda_2}=-\lambda(h_1)$ and
$\Delta_v-\Delta_w)=\Delta_v-\Delta_w$, thus
$$ A(v)\,
A(w)\,\frac{\partial}{ \partial\xi_1}(x)\,
\frac{\partial}{ \partial\xi_1}(y)\,\v(z)=
(\delta-\lambda(h_1))(1-\delta) (D_v-D_w)(D_x-D_y)\tilde{\v}.$$
\end{proof}

\begin{lemma}\label{necS} Assume that $\lambda(h_1)=1$
and $\delta\neq 1$. 

Then $V(\lambda,\delta,u)$ is not cuspidal.
\end{lemma}

\begin{proof} Set $V:=V(\lambda,\delta,u)$.
 For $\W(n)$ we have proved that
$e_1.V^{(\lambda-2\epsilon_1)}=0$. For 
$\S(n;\gamma)$ we  only prove that
$V^{(\lambda-2\epsilon_1)}$ contains a nonzero vector $w$ such that $e_1.w=0$.

By Lemma \ref{formulaS}
 we have
 $$w:=(t^n\frac{\partial}{\partial\xi_1})
(t^m\frac{\partial}{\partial\xi_1}).\v
\neq 0$$
 for some $\v\in V^{(\lambda)}$ and $n,m\in\Z$. Recall
 that $e_1=\zeta_1\frac{\partial}{\partial\xi_2}$
 and 
 
 $$e_1.v=0\hbox{ and } t^{n+m}\frac{\partial}{\partial\xi_2}.v=0.$$
 
 We have
 
 $$\hskip-57mm e_1.w=(\zeta_1\frac{\partial}{\partial\xi_2})(t^n\frac{\partial}{\partial\xi_1})
(t^m\frac{\partial}{\partial\xi_1}).\v$$

$$=[\zeta_1\frac{\partial}{\partial\xi_2},t^n\frac{\partial}{\partial\xi_1})]
(t^m\frac{\partial}{\partial\xi_1}).\v+
(t^n\frac{\partial}{\partial\xi_1})
[\zeta_1\frac{\partial}{\partial\xi_2},t^m\frac{\partial}{\partial\xi_1})].\v$$
 
$$=(t^n\frac{\partial}{\partial\xi_2})
(t^m\frac{\partial}{\partial\xi_1}).\v
+ (t^n\frac{\partial}{\partial\xi_1})
(t^m\frac{\partial}{\partial\xi_2}).\v$$

$$=(t^m\frac{\partial}{\partial\xi_1})
(t^n\frac{\partial}{\partial\xi_2}).\v
+ (t^n\frac{\partial}{\partial\xi_1})
(t^m\frac{\partial}{\partial\xi_2}).\v$$

$$\hskip-64mm=0$$

Since $w$ has weight $\lambda-2\epsilon_1$, we have
$h_1.w=-w$. Thus by elementary $\fsl(2)$-theory, we have
$(f_1)^l.w\neq 0$ for any $l\geq 0$, which implies that
$V$ has infinitely many weights. Thus by Lemma
\ref{finitely many}, the module $V$ is not cuspidal.

 \end{proof}

\begin{cor}\label{necessaryS}
Assume that the $\S(n;\gamma)$-module 
$V(\lambda,\delta,u)$ is cuspidal.
Then
\begin{enumerate}
\item[(a)] either $\lambda(h_1)\geq 2$,
\item[(b)] or $\lambda(h_1)=1$ and $\delta=1$.
\end{enumerate}
\end{cor}

\begin{proof} By Proposition \ref{condition1}, 
$\lambda(h_1)$ is an integer $\geq 1$.
If $\lambda(h_1)=1$, then Condition (b) is 
satisfied by Lemma \ref{necS}. Otherwise
Condition (a) holds.
\end{proof}

\section{The embedding $K(3)\subset\widehat{\K(4)}\subset\CK(6)$}\label{2K(4)}

We freely use the notation of Section \ref{defCK(6)} about $\CK(6)$. We recall that the
Dynkin diagram of its subalgebra $\fso(6)\simeq\fsl(4)$ is labelled as follows
$$\dynkin [backwards,
labels={2,3,1},
scale=1.8] D{ooo}$$
Consider the element
$c=h_1+h_2+2 h_3\in \CK(6)$. We observe that
the centralizer subalgebra $C_{\CK(6)}(c)$ contains the copy of $\Vir$ in $\CK(6)$. 

 In Section \ref{seconddefK(4)}
we have also defined a specific central extension
$\widehat{\K(4)}$ of $\K(4)$. Since the restriction of this central extension to $\Vir$ splits,
$\widehat{\K(4)}$ contains a well-defined 
copy of $\Vir$.

In this chapter, we prove

\begin{thm}\label{K(4)=K(4)} There is an isomorphism
$$\widehat{\K(4)}\simeq C_{\CK(6)}(c).$$

Consequently there are embeddings
$$\K(3)\subset\widehat{\K(4)}\subset \CK(6)
\text{ and}$$
$$\K(3)_{NS}\subset\widehat{\K(4)}_{NS}\subset  \CK(6)_{NS}.$$

All these algebra homomorphisms preserve the subalgebras $\Vir$.
\end{thm}

The superconformal algebras $\widehat{\K(4)}$ and
$ C_{\CK(6)}(c)$ have conformal dimensions $16$,
so computing their multiplication table would have been very long.
However we will used Jordan theory to  reduce the number of computations.
Indeed, it is enough to compare two Jordan superalgebras of conformal dimension $4$.

The proof in the Neveu-Schwarz case is strictly identical to the Ramond case. Thus we will only explain the latter case.

\subsection{$TKK$-construction}

This section uses the theory of 
Jordan algebras, see \cite{JacobsonJordan}\cite{Mac} and \cite{Shestakov}
for references.
A {\it $TKK$-superalgebra} is  a Lie superalgebra $\L$
endowed with a $\fsl(2)$-triple
 $({\bf e},{\bf h},{\bf f})$ in $\L_{\overline 0}$ such  that $\ad({\bf h})$ is 
diagonalizable and its eigenvalues are 
$\pm2$ and $0$. Let

$$\L=\L^{(-2)}\oplus \L^{(0)}\oplus \L^{(2)}$$

\noindent be the corresponding eigenspace decomposisition.
The space 

$$\Jor(\L):=\L^{(2)}$$ 
is a unital Jordan superalgebra for the product

$$a\circ b:= \frac{1}{ 2}[a,[{\bf f},b]].$$

\noindent Its unit is the element ${\bf e}\in \L^{(2)}$.

Conversely let $J$ be a a unital Jordan superalgebra.
For homogenous elements $a, b$ in $J$ the operator 
$$\partial_{a,b}:c\in J\mapsto a(cb)-(-1)^{\vert a\vert\vert b\vert} (ac)b$$
is a derivation. The linear span 
$\Inn(J)$ of these derivations is called the Lie algebra of inner derivations. Then
the space
$$\TKK(J):=\Inn(J)\oplus \fsl(2)\otimes J$$

\noindent has a structure of Lie superalgebra.
It is called the {\it TKK-construction}, in reference of the works of J. Tits \cite{T62}, I.M. Kantor 
\cite{Kan} and M. Koecher \cite{Koe}.

However the $TKK$-construction 
$J\mapsto \TKK(J)$ is not functorial. For a unital Jordan superalgebra, we write as 
$\widehat{\fsl_2}(J)$ the universal central extension of the  perfect Lie superalgebra $\TKK(J)$. It is shown in  \cite{AG} that 
$$J\mapsto\widehat{\fsl_2}(J)$$ 
is functorial.

Recall that two perfect Lie suparalgebra
$\fg$ and $\fg'$ are called {\it isogenous} if their
universal central extension are isomorphic.  
The following result is well-known:

\begin{lemma}\label{basicTKK}
\begin{enumerate}
\item[(a)] Let $\L$ be a $TKK$-superalgebra. If $L$
is generated by $\L^{(\pm2)}$, the $\L$ is a central extension of $\TKK(\L^{(2)})$

\item[(b)] Given a Jordan subalgebra $J'\subset J$, the subalgebra  $\TKK(J')$ is isogenous to a subalgebra of
$\TKK(J)$.
\end{enumerate}
\end{lemma} 

Indeed $\K(4)$ and
$\CK(6)$ are $TKK$-algebras.
We will show that $\Jor(K(4))$ is
a subalgebra of $\Jor(\CK(6))$, though
$\K(4)$ is not a subalgebra of $\CK(6)$. This illustrates Assertion (b) of the previous lemma.

\subsection{The Lie superalgebra $C_{\CK(6)}(c)\subset \CK(6)$}\label{K(4)}\label{defK(4)}

Set $\L=C_{\CK(6)}(c)$. We consider the following $\fsl(2)$-triple
in $\CK(6)$
$${\bf e}=e_1+e_2,\,{\bf h}=h_1+h_2,\,{\bf f}=f_1+f_2.$$
We easily observe that the eigenvalues of the operator
 $\ad({\bf h})_{\CK(6)}$ are $0$ and $\pm2$. Moreover
this $\fsl(2)$-triple belongs to
$\L$.

Recall that the space $V$ has basis
$x_1,x_2,x_3,x_4$ and that 
$\frac{\partial}{\partial x_1},\frac{\partial}{\partial x_2},\frac{\partial}{\partial x_3},
\frac{\partial}{\partial x_4}$ is the dual basis.
Then $c$ and ${\bf h}$ are the diagonal matrices
$c=\diag(1,1,-1,-1)$ and 
${\bf h}=\diag(1,-1,1,-1)$.
Set $V_{+}=\C x_1\oplus \C x_2$ and
$V_{-}=\C x_3\oplus \C x_4$

As a vector space $\CK(6)_0$ is identified  with $P(3)\oplus \C D$ and more precisely
$$\CK(6)_0\simeq\C D\oplus \wedge^2 V^* \oplus \fsl(V) \oplus S^2 V.$$

It follows that 
$$\L_0\simeq 
\C D\oplus V_{+}^*\otimes  V_{-}^*\oplus
(\fsl(2)\oplus\fsl(2))\oplus V_{+}\otimes V_{-}.$$
Indeed the Lie superalgebra $\L_0$ is a $1$-dimensional central extension of $\fsl(2,2)$.
Set 
$${\bf e'}=e_2-e_1, S=\frac{\partial}{\partial x_2}\frac{\partial}{\partial x_4}, \text{ and }
H=x_1x_3.$$
Clearly $\Jor(L_0)=\L_0^{(2)}$ has basis
${\bf e}, {\bf e'}, S, H$. As a $\Vir$-module, we have

$$\Jor(\L)_{\bar 0}=\Bigl(\C{\bf e}\oplus
\bf{e'}\Bigr)\otimes \Omega^0\text {and  }$$
$$\Jor(\L)_{\bar 1}=S\otimes \Omega^{-\frac{1}{2}}
\oplus H\otimes
\Omega^{\frac{1}{2}}.$$

It is easy to establish the multiplication table for the superalgebra $\Jor(\L)$:

\noindent 
\begin{tabular}{ |p{17mm}||p{19mm}|p{24mm}|p{35mm}|p{19mm}|}
 \hline
Product&$g$ &$\bf{e'}(g)$ &$S(g)$ &$H(g)$\\
 \hline
 \hline
$f$& $fg$ &$\bf{e'}(fg)$ &$S(fg)$&$H(fg)$\\

$\bf{e'}(f)$  & $\bf{e'}(fg)$  &$fg$ &
$H(fD(g))$ &$0$\\
$S(f)$&$\zeta_2(fg)$  &$\frac{1}{2}H(
fD(g))$ &$\frac{1}{4}(fD(g)-D(f)g))$&$-\frac{1}{2}{\bf e'}(fg)$\\
$H(f)$ &$H(fg)$ &$0$ &$\frac{1}{2}{\bf e'}(fg)$&$0$\\
 \hline
\end{tabular}

\subsection{The Jordan algebra of $\K(4)$}

It is well known how to  realize $\L=\K(4)$ as a $TKK$-superalgebra, see \cite{MZ02}.

The $\fsl_2$-triple is
$${\bf e}=\zeta_2\xi, {\bf h}=2\,\zeta_2\eta_2, 
{\bf f}=\,\xi\eta_2,$$
where $\xi=\zeta_1+\eta_1$. Set $\xi'=\zeta_1-\eta_1$ and $\zeta_2^*=\zeta_2\zeta_1\eta_1$.
It is clear that

$$\L^{(2)}_{\bar 0}=\Bigl(\C\zeta_2\xi\oplus
\C\zeta_2\xi'\Bigr) \otimes \Omega^0$$
$$\L^{(2)}_{\bar 1}=\zeta_2\otimes \Omega^{-\frac{1}{2}}
\oplus \zeta_2^*\otimes
\Omega^{\frac{1}{2}}.$$

The multiplication table for the superalgebra $\Jor(\K(4))$ is:

\noindent
\begin{tabular}{ |p{17mm}||p{17mm}|p{24mm}|p{35mm}|p{22mm}|}
 \hline
Product  &$g$ &$\zeta_2\xi'(g)$ &$\zeta_2(g)$ &$\zeta_2^*(g)$\\
 \hline
 \hline
$f$& $fg$ &$\zeta_2\xi'(fg)$&$\zeta_2(fg)$&$\zeta_2^*(fg)$\\

$\zeta_2\xi'(f)$  & $\zeta_2\xi'(fg)$  &$fg$ &
$\frac{1}{2}\zeta_2^*(fD(g))$ &$0$\\
$\zeta_2(f)$&$\zeta_2(fg)$  &$\frac{1}{2}\zeta^*_2(
fD(g))$ &$\frac{1}{4}(fD(g)-D(f)g)$&$-\frac{1}{2}\zeta_2\xi'(fg)$\\
$\zeta^*_2(f)$ &$\zeta^*(fg)$ &$0$ &$\frac{1}{2}\zeta_2\xi'(fg)$&$0$\\
 \hline
\end{tabular}

\bigskip
By comparaison of the two multiplication tables, we deduce:

\begin{cor}\label{Jor} The Jordan superalgebras
$\Jor(\L)$ and $\Jor(\K(4))$ are isomorphic.
\end{cor}

\subsection{Proof of Theorem \ref{K(4)=K(4)}.}

\begin{lemma}\label{embed2} The Lie superalgebras 
$\widehat{\K(4)}$ and $C_{\CK(6)}(c)$
are isomorphic.
\end{lemma}

\begin{proof}  Set $\L=C_{\CK(6)}(c)$.

It is clear that $c$ generates
the center of $C_{\CK(6)}(c)$ and that 
$c$ lies in $[\L_0,\L_0]$. Hence
$\L$ is a nontrivial central extension of
$C_{\CK(6)}(c)/\C c$ with a $1$-dimensional
center.

It is also easy to see that
$\L$ is generated by
$\L^{(\pm 2)}$. Thus by Lemma \ref{basicTKK} and  Corollary \ref{Jor}, 
$C_{\CK(6)}(c)/\C c$ is isomorphic
$\K(4)$.

The eigenvalues of $\ad(c)\vert_{\CK(6)}$ are $0$ and $\pm 2$. It implies that $\CK(6)$, viewed 
as a $C_{\CK(6)}(c)$-module, admits a simple subquotient
with nontrivial central charge. Hence by 
Theorem \ref{projcusp}, $C_{\CK(6)}(c)$ is isomorphic to the specific central extension
$\widehat{\K(4)}$ of $\K(4)$.
\end{proof}

\begin{lemma}\label{embed1} There is an embedding
$$\K(3)\subset\widehat{\K(4)}$$
which lifts the natural embedding
$\K(3)\subset\K(4)$.
\end{lemma}

\begin{proof}
By Theorem \ref{projcusp}, all projective cuspidal
$\K(3)$-modules have zero central charge. Hence the restriction of the central extension 
$$0\to\C c\to \widehat{\K(4)}\to\K(4)\to 0$$
to $\K(3)$ splits. Therefore 
$\K(3)$ embedds in $\widehat{\K(4)}$.
\end{proof}

We now conclude with the proof of 
Theorem \ref{K(4)=K(4)}.
\begin{proof} The Lemmas \ref{embed1} and  \ref{embed2} show the existence of the required embeddings $\K(3)\subset\widehat{\K(4)}\subset\CK(6)$.

It remains to prove that these superconformal
algebras share the same subalgebra $\Vir$.
This statement  is clear for
$K(3)$ and $\widehat{\K(4)}$. We also
observe that the copy of $\Vir$ in
$\widehat{\K(4)}$ and $\CK(6)$ are respectively the linear span of the sets
$$\{ [\zeta_2(t^n),\zeta_2(t^m)]\mid n,m\in\Z\}$$
$$\{ [S(t^n),S(t^m)]\mid n,m\in\Z\}.$$
Thus $\widehat{\K(4)}$ and $\CK(6)$ share the same
copy of $\Vir$.

\end{proof}

\section{Exceptional $1$-parameter families of cus\-pidal modules over $\K(3)$, $\widehat{\K(4)}$ and $\CK(6)$.}\label{chexceptional}

We  label the Dynkin diagram of the
simple Lie algebra $\fso(6)\simeq\fsl(4)\subset\CK(6)$ as
$$\dynkin [backwards,
labels={2,3,1},
scale=1.8] D{ooo}.$$
The Dynkin diagram of 
$\fso(4)\subset \widehat{\K(4)}$ consists of the vertices $1$ and $2$ and the Dynkin diagram
of $\fso(3)\subset \K(3)$ is the first vertex.

In this chapter, we investigate some one-parameter families of modules over the superconformal  algebras $\K(3)$, $\widehat{\K(4)}$ and $\CK(6)$, 
namely

\begin{enumerate}
\item[(a)] the $\K(3)$-modules
$S(u):=V(\omega_1,\frac{1}{4},u)$,
\item[(b)] the $\widehat{\K(4)}$-modules 
$$S^+(u):=V((\omega_1,1),\frac{1}{4},u), \,
S^-(u):=V((\omega_1,-1),\frac{3}{4},u)\text{ and}$$
\item[(c)]the $\CK(6)$-modules $T(u):=
V(\omega_1,\frac{1}{4},u)$.
\end{enumerate}
where in each case $\omega_1$ denotes the first fundamental weight. For $\widehat{\K(4)}$, we write
$(\omega_1,\pm 1)$ for the weights 
$(\lambda_1,\lambda_2,\lambda_c)$ with
$\lambda_1=\lambda_2=\frac{1}{2}$ and $\lambda_c=\pm1$.

Recall that $\K(3)\subset \widehat{\K(4)}\subset
\CK(6)$, see Theorem \ref{K(4)=K(4)}.

\begin{thm}\label{exceptional} Let $u\in\C$.

\begin{enumerate}
\item[(a)] The $\CK(6)$-module $T(u)$ has
conformal dimension $8$.
\item[(b)] The $\K(4)$-modules $S^{\pm}(u)$ and the
$\K(3)$-module $S(u)$ have conformal dimension $4$.
\item[(c)] As a $\widehat{\K(4)}$-module we have
$$T(u)=S^+(u)\oplus S^-(u), \text{ and}$$
\item[(d)]
as  $\K(3)$-modules, 
$$S^\pm(u)\simeq S(u).$$
\end{enumerate}
\end{thm}

If we take into account the parity,
Assertion (b) is $T(u)=S^+(u)\oplus \Pi\,S^-(u)$
and Assertion (c) is $S^-(u)\simeq \Pi S(u)$,
where $\Pi$ is the parity change functor. 
The result holds for the Neveu-Schwarz forms.
Since the statement and its proof is the same,
we will skip it.

The modules $T(u)$ and $S(u)$ where considered in 
\cite{MZ02}. For $\CK(6)$, the parameter 
$\delta=\frac{1}{4}$ is different, because of
a different  embedding of
$\Vir$.

\subsection{Proof of Theorem \ref{exceptional}}

Recall that $V$ denotes a vector space with basis
$x_1,x_2,x_3,x_4$.
Let ${\bf V}$ be the superspace
$${\bf V}_{\bar 0}=V\text{ and }
{\bf V}_{\bar 0}=V^*.$$

In Section \ref{defCK(6)}, the superconformal 
algebra $CK(6)$ is defined as a subalgebra 
of $\End_\C({\bf W})$, where 
$\bf{W}:={\bf V}\otimes\C[t,t^{-1}]$. The superalgebra $\CK(6)$ acts on ${\bf W}$ by
differential operators of order $\leq 1$. Thus
$t^u{\bf W}$ is a well-defined $\CK(6)$ -module
for any $u\in\C$.

\begin{lemma}\label{exc1} As a $\CK(6)$-module, we have
$$t^u{\bf W}\simeq V(\omega_1,\frac{1}{4},u).$$ 
\end{lemma}

\begin{proof} The highest weight of
$V$ and $V^*$ are respectively  

$$\omega_1=\frac{\epsilon_1+\epsilon_2+\epsilon_3}{2}\text{ and }\omega_3=\frac{-\epsilon_1+\epsilon_2+\epsilon_3}{2}.$$

Hence the highest weight of $\bf{W}$ is $\omega_1$.
By definition, we have
$$(t^u{\bf W})_{\bar 0}\simeq V\otimes \Omega_u^{\frac{1}{4}}.$$
It is easy to show that the $\CK(6)$-module
$t^u{\bf W}$ is simple, thus
$$t^u{\bf W}\simeq V(\omega_1,\frac{1}{4},u).$$
\end{proof}

\begin{lemma}\label{exc2} As a $\widehat{\K(4)}$-module, we have
$$t^u{\bf W}\simeq V((\omega_1,1),\frac{1}{4},u)
\oplus V((\omega_1,-1),\frac{1}{4},u).$$ 
\end{lemma}

\begin{proof} Set $V_{+}=\C x_1\oplus \C x_2$,
$V_{-}=\C x_3\oplus \C x_4$,
$$\bf {V_+}=V_+\oplus V_-^* \text{ and }
\bf {V_-}=V_-\oplus V_+^* .$$
As a $\widehat{\K(4)}$-module, ${\bf W}$ decomposes
as
$${\bf W}={\bf W_+}\oplus {\bf W_-}$$
where ${\bf W_+}={\bf V_+}\otimes \C[t,t^{-1}]$
and ${\bf W_-}={\bf V_-}\otimes \C[t,t^{-1}]$.
As $\fso(4)$-modules, 
$$\bf {V_+}\simeq \bf {V_+}^*$$
and its highest weight is $\omega_1=\epsilon_1+\epsilon_2$. Similarly $\bf {V_-}\simeq \bf {V_-}^*$
and its highest weight is the second fundamental weight $\omega_2=-\epsilon_1+\epsilon_2$. Thus the highest weight
of ${\bf W}_{\pm}$ is $\omega_1$

Since $$h_1+h_2+2 h_3=\diag(1,1,-1,-1),$$
the central element $c$ acts as $1$ on 
$\bf {V_+}$ and $-1$ on $\bf {V_+}^*$.

As a $\Vir$-module we have
$$t^u{\bf W}_{\bar 0}\simeq V\otimes \Omega^{\frac{1}{4}}_u\text{ and } 
t^u{\bf W}_{\bar 1}\simeq V^*\otimes \Omega^{\frac{3}{4}}_u.$$
By the previous considerations, it is easy to identify
$$t^u{\bf W}_{+}\simeq V((\omega_1,1),\frac{1}{4},u)
\text{ and } t^u{\bf W}_{-}\simeq V((\omega_1,-1),\frac{3}{4},u),$$
which proves the lemma.
\end{proof}

\begin{lemma}\label{exc3} As a $\K(3)$-module,
$S^{\pm}(u)$ are isomorphic to $V(\omega_1,\frac{1}{4},u)$.
\end{lemma} 

\begin{proof}The highest weight of the $\fso(3)$-modules, $V_{\pm}$ and $V_{\pm}^*$ is $\omega_1$.
As a $\Vir$-module, there is an isomorphisms

$$S^\pm(u)^{\omega_1}\simeq \Omega_u^{\frac{1}{4}}
\oplus \Omega_u^{\frac{3}{4}}.$$ 

Thus it follows from Section \ref{fG} that
$$S^\pm(u)^{\omega_1}\simeq \Tens(\omega_1,\frac{1}{4},u)$$ as $\K(1)\ltimes \C[t,t^{-1},\xi]$-modules, which proves the lemma.
\end{proof}

It is now easy to deduce Theorem \ref{exceptional}

\begin{proof} Assertion (c) and (d) follow from
Lemmas \ref{exc1}, \ref{exc2} and \ref{exc3}. 
Moreover the formula for the conformal dimensions of the explicit description of the modules. 
\end{proof}

\section{Cuspidal $\widehat{\K(4)}$-modules}

Recall that we denote as $\widehat{\K(4)}$ the
specific central extension
$$0\to\C c\to \widehat{\K(4)} \to \K(4)\to 0$$
defined in Chapter \ref{projcusp}.

The set of quadratic elements 
$$\cQ=\{\zeta_1\zeta_2, \eta_1\eta_2\}
\cup\{\zeta_i\eta_j\mid 1\leq i,j\leq 2\}$$

\noindent form a basis of the Lie subalgebra
$\fso(4)\subset\K(4)$. Its inverse image
in $\widehat{\K(4)}$ is the split central extension
$$\widehat{\fso(4)}=\fso(4)\oplus\C c.$$

The Cartan subalgebra of $\widehat{\fso(4)}$ is
$$H=\C\zeta_1\eta_1\oplus 
\C\zeta_2\eta_2\oplus\C c,$$
Thus a weight $\lambda\in H^*$ is written as a triple
$(\lambda_1,\lambda_2,\lambda_c)$ where
$$\lambda_1=\lambda(\zeta_1\eta_1),
\lambda_2=\lambda(\zeta_2\eta_2)
\hbox{ and } \lambda_c=\lambda(c).$$
Recall that a weight $\lambda$ is {\it dominant} 
if $\lambda_1$ and $\lambda_2$ are simultaneously 
integers or half-integers and 
$$\lambda_2+\lambda_1\geq 0\hbox{ and }
\lambda_2-\lambda_1\geq 0.$$

In the chapter we prove the following result:

\begin{thm}\label{ClassK(4)}
Let  

$$(\lambda, \delta,u)\in H^*\times\C\times \C,$$

\noindent be an arbitrary triple with $\lambda$ dominant.
The $\widehat{\K(4)}$-module 
$V(\lambda,\delta,u)$
is cuspidal if and only if

\begin{enumerate}
\item[(a)] either $\lambda_1+\lambda_2\geq 2$
that is $\lambda(h_1)\geq 2$,

\item[(b)] or
$\lambda_1+\lambda_2=1$, $\delta=\frac{1-\lambda_2}{2}$ and 
$\lambda_c=2\lambda_2$,

\item[(c)] or
$\lambda_1+\lambda_2=1$, $\delta=\frac{1+\lambda_2}{2}$ and 
$\lambda_c=-2\lambda_2$.
\end{enumerate}
\end{thm}

The family of $\widehat{\K(4)}$-modules satisfying Condition (b) or (c) are in duality.
More precisely the graded dual of

$$V((\lambda_2,1-\lambda_2, 2\lambda_2),\frac{1-\lambda_2}{2}, u)$$ 
is the $\widehat{\K(4)}$-module
$$V((\lambda_2,1-\lambda_2, -2\lambda_2),\frac{1+\lambda_2}{2}, -u).$$

Set $V:=V(\lambda,\delta,u)$. 
The pattern of the proof of Theorem \ref{ClassK(4)} is more complicated than 
the proof of  analogous results for the other superconformal algebras.

As usual, we use Proposition \ref{hwF} to show that
Condition (a) is sufficient. Then we use explicit computations to
determine for which triples $(\lambda,\delta,u)$
 $$V^{(\lambda-2\epsilon_1)}=0.$$ 
We deduce the necessity of (a), (b) or (c) for the cuspidality of $V$. Using the involution $\tau$
of the Dynkin diagramm of $\fso(4)$ we
also deduce that
Condition (b) or Condition (c) is sufficient when $\lambda_2\geq 1$. The remaining two families, namely
$$\lambda_1=\lambda_2=1/2,\delta=1/4,\lambda_c=2\hbox{ and }$$
$$\lambda_1=\lambda_2=1/2,\delta=3/4,\lambda_c=-2$$
are cuspidal by results of Chapter \ref{exceptional}.

\subsection{Sufficiency of Condition (a)}\label{SecSuffitK(4)}

Set $\L=\K(4)$ and $\hat\L=\widehat{\K(4)}$.
We write $\hat \L^{\bf (1)}$ for the subalgebra
$\widehat{\K(4)^{\bf(1)}}$ defined in Section
\ref{defK(4)1}.
The space
$$T:=\L/\L^{\bf (1)}.$$
is called the {\it tangent space}. As a 
$\L^{\bf (1)}$-module $T$ is a nonsplit extension
$$0\to T_{\overline 1}\to T\to 
T_{\overline 0}\to 0.$$
Thus we define the ideal 
$\L^{\bf (2)}\subset \L^{\bf (1)}$ by
$$\L^{\bf (2)}=\{\partial\in \L^{\bf (1)}
\mid \partial 
T_{\overline 0}\subset T_{\overline 1}\hbox{ and }
\partial 
T_{\overline 1}=0\}.$$
It is easy to verify that
$$\L^{\bf (1)}=\Bigl(\fso(4)\oplus(t-1)D\Bigr)\ltimes 
\L^{\bf (2)}.$$

Let $\hat{\L}^{\bf (1)}$ be the inverse image of
$\L^{\bf (1)}$ in $\hat \L$. The central
extension
$$0\to\C c\to \hat{\L}^{\bf (1)}\to \L^{\bf (1)}\to 0$$
splits, thus $\L^{(2)}$ can be viewed as a subalgebra
of $\hat{\L}^{\bf (1)}$. We have

$$\hat{\L}^{\bf (1)}=\Bigl
(\fso(4)\oplus\C c\oplus \C(t-1)D\Bigr)
\ltimes \L^{\bf (2)}.$$

For an arbitrary pair $\lambda,\delta$,
let $S(\lambda,\delta)$ be the 
$\hat{\L}^{\bf (1)}$-module defined as follows
\begin{enumerate}
\item[(a)]  The action of $\L^{\bf(2)}$ is trivial,
so $S(\lambda,\delta)$ is a 
$\Bigl(\fso(4)\oplus\C c\oplus \C(t-1)D\Bigr)$-module.
\item[(b)] As a $\fso(4)$-module $S(\lambda,\delta)$
is the simple module with highest weight 
$(\lambda_1,\lambda_2)$,
\item[(c)] the central elements $c$ and $(t-1)D$ acts
respectively as $\lambda_c$ and $\delta$.
\end{enumerate}

As a consequence of Proposition \ref{hwF} we obtain

\begin{cor}\label{suffitK(4)a}
Let $$(\lambda, \delta,u)\in H^*\times\C\times \C,$$

\noindent be an arbitrary triple with $\lambda$ dominant. If

$$\lambda_1+\lambda_2\geq 2\hbox{ that is }
\lambda(h_1)\geq 2,$$
the $\widehat{\K(4)}$-module 
$V(\lambda,\delta,u)$
is cuspidal.
\end{cor}

\begin{proof}
We check the hypothesis of Proposition \ref{hwF}
and we compute the datum $(\omega,m,\delta^-)$.
Here the grading element $\ell_0$ is $D$ and the
tangent space $T$ has basis
$$\{D,\zeta_1,\eta_1,\zeta_2,\eta_2\}\hbox{ modulo }
\L^{\bf (1)}.$$
Thus
$$T^-=\C \eta_1\oplus \C\eta_2,
\omega=\epsilon_1+\epsilon_2\hbox{ and }
m=2.$$
Since $[(t-1)D,\eta_i]=-\frac{1}{2}\eta_i$
we have $\delta^-=-\frac{1}{2}$.

We have $\omega(h_1)=2$ and $\omega(h_2)=0$. 
Since the weight $\mu:=\lambda-\omega$ is dominant,
the simple $\hat{L}^{\bf(1)}$-module 
$S(\mu, \delta+1)$ is finite dimensional. By proposition
\ref{hwF}, the highest weight of the $\hat L$-module
$M:=\cF(S(\mu, \delta+1),u)$ is $\lambda=\mu+\omega$ and
$$M^{(\lambda)}\simeq \Tens(\lambda,\delta,u).$$
Hence $V(\lambda,\delta,u)$ is a subquotient of 
$M$, which proves that $V(\lambda,\delta,u)$ is cuspidal.
\end{proof}

\subsection{Explicit computations}

As before, we will use multidistributions  
to write some series of identities as a single
identity.

Let $v$ be a formal variable. For an element
$a\in\widehat{\Grass}(4)$, set 
$$a(v):=\sum_n a(t^n)\, v^n\in 
\widehat{\Grass}(4)[[v,v^{-1}]].$$
To avoid ambiguities, the unit element will be
denoted as $D$ and we set
$$D(v)=\sum_n t^nD\, v^n.$$

We now describe some identities of 
formal distributions of $\widehat{\K(4)}$.
Set $\zeta_1^*=\zeta_1\zeta_2\eta_2$
and $\eta_1^*=\eta_1\zeta_2\eta_2$.
In the following formulas, $v,w, x,y$ and $z$ are formal variables and $D_v, D_w\cdots$ denotes the Ramond derivations $\frac{u\d}{\d u}, \frac{v\d}{\d v}\cdots$.


$$[D(v),\eta_1(w)]=(D_w-\frac{1}{2} D_v)\eta_1(v,w)$$

$$[\zeta_1(v),\eta_1(w)]
=D(v,w)+1/2(D_w-D_v) \zeta_1\eta_1(v,w)$$

$$[\zeta_2\eta_2(v),\eta_1(w)]
=-1/2 D_v\, \eta_1^*(v,w)$$

$$[\zeta_1\eta_1(v),\eta_1(w)]
=-\eta_1(v,w)$$

$$[\zeta_1^*(v),\eta_1(w)]
=-(1/2) c(v,w)+\zeta_2\eta_2(v,w)$$

$$[c(v),\eta_1(w)]
=-D_v \eta_1^*(v,w)$$

We easily deduce the following additional formulas:

$$[\zeta_1(v),[\zeta_1(w), \eta_1(x)]]
=(D_w-D_v)\zeta_1(v,w,x)$$

$$[\zeta_1(v),[\zeta_1^*(w), \eta_1(x)]]
=(D_w+D_x)\zeta_1^*(v,w,x)$$

$$[\zeta_1^*(v),
[\zeta_1^*(w), \eta_1(x)]]=0$$

Let $(\lambda, \delta,u)\in H^*\times\C\times \C,$
be an arbitrary triple. Set $V:=V(\lambda,
\delta,u)$. We always assume that $\lambda\neq 0$. Therefore
$V^{(\lambda)}$ is a $\C[t,t^{-1}]$-module.
Thus for $\v\in V^{(\lambda)}_0$ we set
$$\v(z)=\sum_n\,\v(t^n)z^n.$$

Set $\Delta_z=D_z+u$. Since $V^{(\lambda)}=\Tens(\lambda,\delta,u)$, we have

$$D(v)\,\v(z)=(\delta D_v+\Delta_z)\, \v(v,z)$$
$$\zeta_1\eta_1(v)\,\v(z)=\lambda_1\, \v(v,z)$$
$$\zeta_2\eta_2(v)\,\v(z)=\lambda_2\, \v(v,z)$$
$$c(v)\,\v(z)=\lambda_c \,\v(v,z)$$

Set ${\tilde \v}=\v(v,w,x,y,z)$.

\begin{lemma}\label{formulasK(4)} We have

\bigskip\noindent
$(a)\hskip1mm \zeta_1(v)\,\zeta_1(w)\,\eta_1(x)\,\eta_1(y)\,\v(z)=
(\delta-1+\frac{\lambda_1}{2})(\delta-\frac{\lambda_1}{2})(D_v-D_w)(D_x-D_y)\,{\tilde \v}$

\bigskip\noindent
$(b)\hskip1mm\zeta_1(v)\,\zeta_1^*(w)\,\eta_1(x)\,\eta_1(y)\,\v(z)=
(\lambda_2-\frac{\lambda_c}{2})
(\delta-1+\frac{\lambda_1}{2})(D_y-D_x)\,{\tilde \v}$

\bigskip\noindent
$(c)\hskip1mm \zeta_1^*(v)\,\zeta_1^*(w)\, \eta_1(x)\,\eta_1(y)\,\v(z)=0$

\bigskip\noindent
 $(d)\hskip1mm\zeta_1(v)\,\zeta_1(w)\, \eta_1(x)
\,\eta_1^*(y)\,\v(z)=(\delta-\frac{\lambda_1}{2})(\lambda_2+\frac{\lambda_c}{2})
(D_w-D_v)\,{\tilde \v}$

\bigskip\noindent
 $(e)\hskip1mm\zeta_1(v)\,\zeta_1^*(w)\, \eta_1(x)
\,\eta_1^*(y)\,\v(z)=
(\lambda_2+\frac{\lambda_c}{2})(\lambda_2-\frac{\lambda_c}{2})
\,{\tilde \v}.$

\bigskip\noindent
$(f)\hskip1mm\zeta_1^*(v)\,\zeta_1^*(w)\, \eta_1(x)\,\eta_1^*(y)\,\v(z)=0$

\bigskip\noindent
$(g)\hskip1mm\zeta_1(v)\,\zeta_1(w)\, \eta_1^*(x)\,\eta_1^*(y)\,\v(z)=0$

\bigskip\noindent
$(h)\hskip1mm\zeta_1(v)\,\zeta_1^*(w)\, \eta_1^*(x)\,\eta_1^*(y)\,\v(z)=0$

\bigskip\noindent
$(k)\hskip1mm\zeta_1^*(v)\,\zeta_1^*(w)\, \eta_1^*(x)\,\eta_1^*(y)\,\v(z)=0$
\end{lemma}

\noindent {\it Proof.} Set $\hat{\L}=\widehat{\K(4)}$. 
Let $A_1,A_2$ be  distributions in  $\L^{(\epsilon_1)}$ and
$B_1,B_2$ be distributions in $\L^{(-\epsilon_1)}$. Since the distribution
$[A_1,[A_2,B_2]]\v$ has weight 
$\lambda +\epsilon_1$, we have $[A_1,[A_2,B_2]]\v=0$.
We deduce the formula
$$A_1A_2B_1B_2\v= A+B-C$$
where
$$A:=[A_1,[A_2,B_1]B_2\,\v,$$
$$B:=[A_2,B_1][A_1,B_2]\,\v, \hbox{ and }$$
$$C:=[A_1,B_1][A_2,B_2])\,\v$$
which will be used in the next computations. 

When $A_1(v), A_2(w), B_1(x)$ and $B_2(y)$ are the formal distributions of Formulas (1-9), it is clear
that 
$$A_1(v)A_2(w)B_1(x)B_2(y).\v(x)=P(D_v,D_w,D_x,D_y,
\Delta_z; \lambda_1,\lambda_2,\lambda_c,\delta)
.\tilde{\v}$$ 

\noindent
where $P$ is a degree $4$ polynomial which of degree $\leq 2$ over the first variables $(D_u,D_v,D_w,D_x,\Delta_z)$
and over the last variables  $(\lambda_1,\lambda_2,\lambda_c)$.

\noindent{\it Proof of Formula (a).}

$$\zeta_1(v),\zeta_1(w)\,\eta_1(x)\,\eta_1(y)\,\v(z)= A+B-C$$
where 

$A=[\zeta_1(v),
[\zeta_1(w), \eta_1(x)]]\,\eta_1(y)\,\v(z),$

$B=
[\zeta_1(w), \eta_1(x)]\,
[\zeta_1(v),\eta_1(y)]\,\v(z),$
and

$C=[\zeta_1(v), \eta_1(x)]\,[\zeta_1(w),\eta_1(y)]
\,\v(z).$

\smallskip
We successively compute the terms $A$, $B$ and $C$. 
We have 
$$[\zeta_1(v),[\zeta_1(w), \eta_1(x)]]
=(D_w-D_v)\zeta_1(v,w,x).$$ 
Since
$$[\zeta_1(v,w,z),\eta_1(y)]
=D(v,w,x,y)+1/2(D_y-D_v-D_w-D_x)\zeta_1\eta_1(v,w,x,y)$$ 
we deduce
$$A=\left( D_w-D_v\right) \left((\delta+\frac{\lambda_1}{2})D_y +
(\delta-\frac{\lambda_1}{2})(D_v+D_w+D_x)+\Delta_z \right)\,
{\tilde \v}$$ 

We have 
$$B=
(D(w,x)+\frac{1}{2}(D_x-D_w)\zeta_1\eta_1(w,x))
(D(v,y)+\frac{1}{2}(D_y-D_v)\zeta_1\eta_1(v,y))\,\tilde{\v}$$
$$=\Bigl((\delta-\frac{\lambda_1}{2})D_w+(\delta+\frac{\lambda_1}{2})D_x
+D_v+D_y+\Delta_z\Bigr)
\Bigl((\delta-\frac{\lambda_1}{2})D_v+(\delta+\frac{\lambda_1}{2})D_y
+\Delta_z\Bigr)\,\tilde{\v}.$$

We obtain $C$ from the previous formula by exchanging
 the roles of $v$ and $w$, thus $C$ equals
 $$\Bigl((\delta-\frac{\lambda_1}{2})D_v+(\delta+\frac{\lambda_1}{2})D_x
+D_v+D_y+\Delta_z\Bigr)
\Bigl((\delta-\frac{\lambda_1}{2})D_w+(\delta+\frac{\lambda_1}{2})D_y
+\Delta_z\Bigr)\,\tilde{\v}.$$ 
 
 Since 
$$[\zeta_1(v),\zeta_1(w)=
[\eta_1(x),\eta_1(y)]=0$$
the left side is skew symmetric in $v,w$ and in
$x,y$. Therefore the right side is divisible by
$(D_v-D_w)(D_x-D_y)$ so it will be of the form
$$f(\lambda_1,\lambda_2,\lambda_c,\delta)
(D_v-D_w)(D_x-D_y)\,\tilde{\v}.$$
So  it is enough to consider the contribution 
of the monomial $D_vD_x$ in $A$ $B$ and $C$.
These contributions are respectively
$$-(\delta-\lambda_1/2) D_vD_x\,\tilde{\v},\,\,
(\delta+\frac{\lambda_1}{2})
(\delta-\frac{\lambda_1}{2})D_v D_x\,\tilde{\v}
\hbox{ and } 0.$$

 Thus  the contribution 
of the monomial $D_vD_x$ in $A+B-C$ is 
$$(\delta+\lambda_1/2-1)(\delta-\lambda_1/2)D_v D_x\,\tilde{\v},$$
which proves  Formula (a).\qed

\bigskip
\noindent{\it Proof of Formula (b).}
Similarly we have

$$\zeta_1(v)|,\zeta_1^*(w),\eta_1(x)|,\eta_1(y)|,\v(z)= A+B-C$$
where 

$A=[\zeta_1(v),
[\zeta_1^*(w), \eta_1(x)]]\,\eta_1(y)\,\v(z),$

$B=
[\zeta_1^*(w), \eta_1(x)]\,
[\zeta_1(v),\eta_1(y)]\,\v(z),$
and

$C=[\zeta_1(v), \eta_1(x)]\,[\zeta_1^*(w),\eta_1(y)]
\,\v(z)$.

\smallskip
We now compute $A$, $B$ and $C$. We have

$$A= (D_w+D_x)
\zeta_1^*(v,w,x)\,\eta_1(y)\,\v(z).$$
Since 
$$[\zeta_1^*(v,w,x),\eta_1(y)]
=-1/2 c(v,w,x,y) +\zeta_2\eta_2(v,w,x,y)$$
 we deduce

$$A=\left(\lambda_2-\frac{\lambda_c}{2}\right)
(D_w+D_x) {\tilde \v}.$$

We have 

$$B=\Bigl(-\frac{1}{2} c(w,x)+\zeta_2\eta_2(w,x)\Bigr)
\Bigl(D(v,y)+\frac{1}{2}(D_y-D_v)\zeta_1\eta_1(v,y)
\Bigr)\,\v(z)$$

$$\hskip-21mm=\Bigl(\lambda_2-\frac{\lambda_c}{2}\Bigr)
\Bigl(\delta(D_v+D_y)+\Delta_z+
\frac{\lambda_1}{2}(D_y-D_v)\Big)\,{\tilde \v}.$$

Thus

$$B=\Bigl(\lambda_2-\frac{\lambda_c}{2}\Bigr)
\Bigl((\delta-\frac{\lambda_1}{2}) D_v+(\delta+\frac{\lambda_1}{2}) D_y
+\Delta_z\Bigr)\, {\tilde \v}$$

\bigskip
Since 
$$[\zeta_1^*(w),\eta_1(y)]\,\v(z)=(\lambda_2-\frac{\lambda_c}{2})\,\v(w,y,z)$$
we deduce

$$C=(\lambda_2-\frac{\lambda_c}{2})
[\zeta_1(v), \eta_1(x)]\,\v(w,y,z)$$

$$=(\lambda_2-\frac{\lambda_c}{2})
\Bigl(D(v,x)+\frac{1}{2}(D_v-D_x)\,\zeta_1\eta_1(v,x)\Bigr)
\,\v(w,y,z).$$

$$=(\lambda_2-(1/2)\lambda_c)
\Bigl(\delta(D_v+D_x)+(D_w+D_y+\Delta_z)+
\frac{\lambda_1}{2}
(D_u-D_w)\Bigr)\,{\tilde \v}.$$

Thus 

$$C=\left(\lambda_2-\frac{\lambda_c}{2}\right)
\Bigl( (\delta-\lambda_1/2) D_v+
(\delta+\lambda_1/2)D_x+D_w+D_y+\Delta_z\Bigr)
\,{\tilde \v}$$

We deduce

$$A+B-C=(\lambda_2-\frac{\lambda_c}{2})
(\delta-1+\lambda_1/2)(D_y-D_v)\,{\tilde \v},$$
which prove Formula (b).
\qed

\bigskip
\noindent{\it Proof of Formula (c).}
We decompose again the formula into three pieces $A$, $B$ and $C$.

Since 
$$[\zeta_1^*(v),[\zeta_1^*(w),\eta_1(x)]]=0$$
we have $A=0$.

Moreover
\begin{align*}
B&=[\zeta_1^*(w), \eta_1(x)] 
[\zeta_1^*(v),\eta_1(y)]\,\v(z)\\
&=(1/2 c(w,x)+\zeta_2\eta_2(w,x))
(1/2 c(v,y)+\zeta_2\eta_2(v,y))\,\v(z)\\
&=(\frac{1}{2}\lambda_c+\lambda_2)^2\,\tilde{\v}
\end{align*}

\noindent Since $C$ can be deduced from $B$ by symmetry,
we deduce that $B=C$, thus $A+B-C=0$
which proves Formula (c).

\bigskip
\noindent{\it Proof of Formula (d).} As before,
we decompose the left side of Formula (d) as $A+B-C$.
We have

\begin{align*}
A&=[\zeta_1(v),[\zeta_1(w),\eta_1(x)]]\eta_1^*(y)\,\v(z)
\\
&=(D_w-D_v)\zeta_1(v,w,x) \eta_1^*(y)\,\v(z)
\\
&=(D_w-D_v)(\zeta_2(v,w,x,y) + 1/2 c(v,w,x,y))\,\v(z)\end{align*}

\noindent Thus

$$A=(\lambda_2+\frac{\lambda_c}{2})(D_w-D_v)
\, \tilde{\v}.$$

\begin{align*}
B&=[\zeta_1(w),\eta_1(x)][\zeta_1(v),\eta_1^*(y)]
\,\v(z)\\
&=(D(w,x)+1/2(D_x-D_v)\zeta_1\eta_1(w,x))
(\zeta_2\eta_2(v,y)+1/2 c(v,y))\v(z)\\
&=
(\lambda_2+\frac{\lambda_c}{2})\Bigl(D_v+(\delta-\frac{\lambda_1}{2}) D_w+(\delta+\frac{\lambda_1}{2}) D_x +D_y +\Delta_z\Bigr).\tilde {\v}.
\end{align*}

Exchanging $v$ and $w$, we obtain
$$C=
(\lambda_2+\frac{\lambda_c}{2})\Bigl((\delta-\frac{\lambda_1}{2}) D_v+D_w+(\delta+\frac{\lambda_1}{2}) D_x +D_y +\Delta_z\Bigr).\tilde {\v}.$$

Thus

$$A+B-C= (\lambda_2+\lambda_c/2)(\delta-\lambda_1/2)
(D_w-D_v)\,\tilde{\v},$$

which proves Formula (d). \qed

{\it Proof of Formula (e)}
We have $[\zeta_1^*(w) \eta_1^*(x)]=0$. It follows that

$$\zeta_1(v)\zeta_1^*(w) \eta_1^*(x)
\eta_1(y)\,\v(z)=-[\zeta_1(v)\eta_1^*(w)][ \zeta_1^*(x)
\eta_1(y)]\,\v(z).$$

Thus

$$\zeta_1(v)\zeta_1^*(w) \eta_1^*(x)
\eta_1(y)\,v(z)=-(\lambda_2+\lambda_c/2)(\lambda_2-\lambda_c/2)\,{\tilde v}.$$\qed

Moreover, Formulas $(f), (g), (h)$ and $(k)$ are easy. Thus the proof is complete. \qed

\bigskip
\noindent{\it Remark} Since  $\L_{NS}:=\widehat{\K_{NS}(4)}$ is isomorphic
to the Ramond type contact superalgebra $\L:=\widehat{\K(4)}$, all results
concerning $\widehat{\K(4)}$ can be reformulated for
$\widehat{\K_{NS}(4)}$. Instead of describing an explicit isomorphism and the way to reformulate results, it is simpler to consider that $\L_{NS}$ is defined as in Section \ref{RNS} as
$$(\L_{NS})_{\bar 0}=\L_{\bar 0}
\hbox{ and } (\L_{NS})_{\bar 1}=t^{\frac{1}{2}}\L_{\bar 1}.$$
Then identities concerning $\L_{NS}$ are just
polynomial extensions of identities in $\L$,
where the   distributions $a(v)$ are defined
as in Section \ref{multiNS}.

\bigskip

We now state as a corollary the main consequence of these tedious computations.
As before let $(\lambda, \delta,u)\in H^*\times\C\times \C,$
be an arbitrary triple and set $V:=V(\lambda,
\delta,u)$. For simplicity we assume that $\lambda\neq 0$. We obtain:

\begin{cor}\label{small} We have
$$V^{(\lambda-2\epsilon_1)}=0$$
if and only if
\begin{enumerate}
\item[(a)]
either $\lambda_c=2\lambda_2$ and
$\delta=1-\lambda_1/2$, 
\item[(b)]
or $\lambda_c=-2\lambda_2$ and
$\delta=\lambda_1/2$.
\end{enumerate}
\end{cor}

\begin{proof} Set $\G=\widehat{\K(4)}$.
Recall that
$$\Delta=\{\pm\epsilon_1,\pm\epsilon_2,\pm\epsilon_1+\pm\epsilon_2\}$$ and consider a
new triangular decomposition
$$G=G^+\oplus G^{0}\oplus G^-,$$
where 
$$G^+=\oplus_{k=-1}^1\,\G^{(\epsilon_2+k\epsilon_1)},$$
$$G^0=\oplus_{k=-1}^1\,\G^{(k\epsilon_1)}\hbox{ and }$$
$$G^-=\oplus_{k=-1}^1\,\G^{(-\epsilon_2+k\epsilon_1)}.$$
By highest weight theory, the 
$G^0$-module 
$$M=\oplus_{k\geq 0}\,V^{(\lambda-k\epsilon_1)}$$
is irreducible. It follows that 
$$V^{(\lambda-2\epsilon_1)}=0$$
if and only if
$$\Bigl(G^{(\epsilon_1)}\Bigr)^2
 \Bigl(G^{(-\epsilon_1)}\Bigr)^2
 \,V^{(\lambda)}=0.$$

We observe that 
\begin{enumerate}
\item[(a)]
$G^{(-\epsilon_1)}$ is the span
the modes of $\eta_1(x)$ and $\eta^*_1(x)$ 
\item[(b)]
$G^{(\epsilon_1)}$ is the span of
the modes of $\zeta_1(v)$ and $\zeta^*_1(v)$
\item[(c)] for any nonzero
$\v\in V^{(\lambda)}_0$, the modes of $\v(z)$ generates
$V^{(\lambda)}$. 
\end{enumerate}

It follows that 
$$V^{(\lambda-2\epsilon_1)}=0$$
if and only if all Formulas (a-k) of Lemmas \ref{formulasK(4)} vanish. It is equivalent to 
the equations:
$$(\delta-1+\frac{\lambda_1}{2})(\delta-\frac{\lambda_1}{2})=0$$
$$(\lambda_2-\frac{\lambda_c}{2})
(\delta-1+\frac{\lambda_1}{2})=0$$
$$(\delta-\frac{\lambda_1}{2})(\lambda_2+\frac{\lambda_c}{2})=0$$
 $$(\lambda_2+\frac{\lambda_c}{2})(\lambda_2-\frac{\lambda_c}{2})=0,$$
 which are equivalent to Condition (a) or (b).
 \end{proof}

\subsection{Sufficiency of Condition (b) and (c)}

By definition, we have
$$\Delta^+=
\{\epsilon_1,\epsilon_2,\pm\epsilon_1+\epsilon_2\}$$
and $\Delta^+=-\Delta^+$. 

The Dynkin diagram
of $\fso(4)$

$$\dynkin [backwards,
labels={1},
scale=1.8] A{o}\hskip7mm \dynkin [backwards,
labels={2}, scale=1.8] A{o}$$
admits an involution $\tau$ exchanging the two vertices.
This involution can be lifted to
$\tilde{\Grass(4)}$ as
$$\tau(\zeta_1)=\eta_1, \tau(\eta_1)=\zeta_1,
\tau(\zeta_2)=\zeta_2, \tau(\eta_2)=\zeta_2$$
thus it extends to $\widehat{\K(4)}$.

The involution $\tau$ stabilizes $H$, and its action
on $H^*$ is
$$\tau(\lambda_1,\lambda_2,\lambda_c)=
(-\lambda_1,\lambda_2,\lambda_c).$$

We have 
$$\tau(\Delta^+)=
\{-\epsilon_1,\epsilon_2,\pm\epsilon_1+\epsilon_2\}.$$

\begin{lemma}\label{suffitK(4)(bc)}
Let $(\lambda, \delta,u)\in H^*\times\C\times \C,$
be an arbitrary triple. Assume that $\lambda$ is dominant, $\lambda(h_1)=1$ and $\lambda_2\geq 1$.
If

\begin{enumerate}
\item[(b)]  $\delta=\frac{1-\lambda_2}{2}$ and 
$\lambda_c=2\lambda_2$,

\item[(c)] or
$\delta=\frac{1+\lambda_2}{2}$ and 
$\lambda_c=-2\lambda_2$,
\end{enumerate}
the $\widehat{\K(4)}$-module $V(\lambda,\delta,u)$ is cuspidal.
\end{lemma}

\begin{proof} Assume Condition (b).

Set $\tilde{\lambda}=(\lambda_2,\lambda_2,\lambda_c)$
and set $V:=V(\tilde{\lambda},1-\frac{\lambda_2}{2},u)$.
By hypothesis
$$\tilde{\lambda}(h_1)\geq 2,$$
thus $V$ is cuspidal by Lemma \ref{suffitK(4)a}.
By Lemma \ref{small}, we have
$$V^{\tilde{\lambda}-2\epsilon_1}=0.$$
Set 
$$\mu=(\lambda_2-1,\lambda_2,\lambda_c)=
\tilde{\lambda}-\epsilon_1).$$

 It follows easily from Formula (a) of Lemma
\ref{formulasK(4)} that $V^{\mu}\neq 0$. Moreover it is clear that
$$V^{(\mu+\alpha)}=0\hbox{ for all }\alpha\in \tau(\Delta^+).$$

Hence the highest weight of $V':=\tau(V)$ is 
$$\tau(\mu)=(1-\lambda_2,\lambda_2,\lambda_c)=\lambda$$
We have
$$(V')^{(\lambda)}=\Tens(\lambda,\delta',u'),$$
for some $\delta',u'\in\C$. Since the eigenvalues of 
$-E_0$ are in $u\mod \Z$, we have $u=u'$. Since
$(V')^{(\lambda-2\epsilon_1)}=0$ and $\lambda_c=2\lambda_2$, we deduce from Corollary \ref{small} that $\delta'=\delta$. Hence under Condition (a), the module $V(\lambda,\delta,u)$ is cuspidal.

The proof of cuspidality under Condition (c) is identical.
\end{proof}

\begin{cor} Let $(\lambda, \delta,u)\in H^*\times\C\times \C,$
be an arbitrary triple. Assume that $\lambda$ is dominant and $\lambda(h_1)=1$. If

\begin{enumerate}
\item[(b)] either  $\delta=\frac{1-\lambda_2}{2}$ and 
$\lambda_c=2\lambda_2$,

\item[(c)] or
$\delta=\frac{1+\lambda_2}{2}$ and 
$\lambda_c=-2\lambda_2$,
\end{enumerate}
the $\widehat{\K(4)}$-module $V(\lambda,\delta,u)$ is cuspidal.
\end{cor}

\begin{proof} If $\lambda_2\geq 1$, the corollary follows Lemma \ref{suffitK(4)(bc)}. 

Otherwise 
we have $\lambda_1=\lambda_2=\frac{1}{2}$ and
$(\delta,\lambda_c)=(\frac{1}{4},1)$ or 
$(\frac{3}{4},-1)$. In both case the module
$V(\lambda,\delta,u)$ is cuspidal by Theorem
\ref{exceptional}.
\end{proof}

\subsection{Necessity of Conditions (a),(b) or (c) for cuspidality}

\begin{cor}\label{necessaryK(4)}
Let $(\lambda, \delta,u)\in H^*\times\C\times \C,$
be an arbitrary triple with $\lambda$ dominant
If $V(\lambda,
\delta,u)$ is cuspidal, then
\begin{enumerate}
\item[(a)] either $\lambda_1+\lambda_2\geq 2$
that is $\lambda(h_1)\geq 2$,

\item[(b)] or
$\lambda_1+\lambda_2=1$, $\delta=\frac{1-\lambda_2}{2}$ and 
$\lambda_c=2\lambda_2$,

\item[(c)] or
$\lambda_1+\lambda_2=1$, $\delta=\frac{1+\lambda_2}{2}$ and 
$\lambda_c=-2\lambda_2$.
\end{enumerate}
\end{cor}

\begin{proof}  If $\lambda(h1)\geq 2$,
Condition (a) is satified. By Theorem \ref{condition1}, we have
$$\lambda(h_1)\geq 1,$$
thus we can assume that 
$\lambda(h_1)=1$.
Set $\mu=\lambda-2\epsilon_1$.  It is clear that
$\lambda+\epsilon_2-\epsilon_1$ is not
$<\lambda$. Since 
$\alpha_1=\epsilon_1+\epsilon_2$ we deduce
$$e_1.V^{(\mu)}\subset V^{(\lambda+\epsilon_2-\epsilon_1)}=0.$$

We observe that $\mu(h_1)=-1$. Since $V$ contains only finitely many weights, the elementary $\fsl(2)$-theory implies that $V^{(\lambda-2\epsilon_1)}=V^{(\mu)}$ is zero.
Therefore by Corrolary \ref{small} Condition (b) or
Condition (c) is satisfied.
\end{proof}

\section{Classification of cuspidal $\K(2m)$-modules for $2m\geq 4$}

Let $\K(2m)$ be a contact superalgebra  for some integer $m\geq 2$. 
The superalgebra $\K(2m)$  is the contact algebra on 

$$\C[t,t^{-1},\zeta_1,\eta_1,\cdots, \zeta_m,\eta_m].$$

The set of quadratic elements 
$$\{\xi_i\xi_j, \xi_i\eta_j,
\eta_i\eta_j \mid 1\leq i, j\leq m\}$$ 

\noindent
form a basis of the Lie subalgebra 
$\fso(2m)\subset \K(2m)$.  
 The Cartan subalgebra $H$ of the Lie algebra  
 $\fso(2m)$ has basis
$$\{\xi_i\eta_i\mid 1\leq i\leq m\}.$$
Consequently a weight $\lambda\in H^*$ is written as the
$m$-uple $(\lambda_1,\cdots,\lambda_m)$ where
$\lambda_i=\lambda(\xi_i\eta_i)$. 

With the conventions of chapter 2, the simple roots
of $\fso(2m)$ are 
$$\alpha_1=\epsilon_1+\epsilon_2,\,
\alpha_2=\epsilon_2-\epsilon_1,\cdots,\,
\alpha_m=\epsilon_m-\epsilon_{m-1},$$

\noindent and corresponding Dynkin diagrams
is

$$\dynkin [backwards,
labels={m, m-1,,,4,3,2,1},
scale=1.8] D{oooo...oooo}$$

It is an elementary fact
that a weight $\lambda$ is dominant if and only if

\begin{enumerate}
\item[(a)] all $\lambda_i$ are integers or all $\lambda_i$ are half-integers, 
\item[(b)] $\lambda_1\leq\lambda_2\leq\cdots\leq \lambda_m$
\item[(c)] $\lambda_1+\lambda_2\geq 0$.

\end{enumerate}

In this section we will show the following result.

\begin{thm}\label{K(2m)} Assume that $2m\geq 4$.
Let  

$$(\lambda, \delta,u)\in H^*\times\C\times \C,$$

\noindent be an arbitrary triple with $\lambda$ dominant.
Then the $\K(2m)$-module $V(\lambda,\delta,u)$
is cuspidal  if and only if 
$$\lambda_1+\lambda_2\geq 2, \hbox{ that is }
\lambda(h_1)\geq 2.$$
\end{thm}

The proof is an obvious consequence of
Corollaries   \ref{suffitK(2m)} and
\ref{necessaryK(2m)} proved in this chapter.

\subsection{Sufficiency of the condition $\lambda(h_1)\geq 2$}\label{SecSuffitK(2m)}

Set $\L=\K(2m)$.
As before, we define a subalgebra 
$\L^{\bf (1)}$, called the {\it isotropy subalgebra} and we identify its  highest weight modules.

Let $\L^{\bf (1)}$ be the ideal of the algebra
$$\C[t,t^{-1},\zeta_1,\eta_1\cdots,\eta_m]$$
generated by $(t-1)$ and $\fso(2m)$. By Lemma \ref{repeat}, $\L^{\bf (1)}$ is a 
Lie subalgebra and we set 
$$T:=\L/\L^{\bf (1)}.$$
As a 
$\L^{\bf (1)}$-module $T$ is a nonsplit extension

$$0\to T_{\overline 1}\to T\to 
T_{\overline 0}\to 0.$$
Thus we define the ideal 
$\L^{\bf (2)}\subset \L^{\bf (1)}$ by
$$\L^{\bf (2)}=\{\partial\in \L^{\bf (1)}
\mid \partial 
T_{\overline 0}\subset T_{\overline 1}\hbox{ and }
\partial 
T_{\overline 0}=0\}.$$
It is easy to verify that
$$\L^{\bf (1)}=(\fso(2m)\oplus(t-1)D)\ltimes 
\L^{\bf (2)}.$$

For an arbitrary pair $\lambda,\delta$,
let $S(\lambda,\delta)$ be the 
$\hat{\L}^{\bf (1)}$-module defined as follows
\begin{enumerate}
\item[(a)]  The action of $\L^{\bf(2)}$ is trivial,
so $S(\lambda,\delta)$ is a 
$\Bigl(\fso(2m)\oplus \C(t-1)D\Bigr)$-module.
\item[(b)] $S(\lambda,\delta)$
is the simple $\fso(2m)$-module with highest weight 
$\lambda$,
\item[(c)] the  element $(t-1)D$ acts
as $\delta$.
\end{enumerate}

As a consequence of Proposition \ref{hwF} we obtain

\begin{cor}\label{suffitK(2m)}
Let $$(\lambda, \delta,u)\in H^*\times\C\times \C,$$

\noindent be an arbitrary triple with $\lambda$ dominant. If

$$\lambda_1+\lambda_2\geq 2\hbox{ that is }
\lambda(h_1)\geq 2,$$
the $\K(2m)$-module 
$V(\lambda,\delta,u)$ is cuspidal.
\end{cor}

\begin{proof}
We check the hypothesis of Proposition \ref{hwF}
and we compute the datum $(\omega,m,\delta^-)$.
Here the grading element $\ell_0$ is $D$ and the
tangent space $T$ has basis
$$\{D\}\cup\{\zeta_i,\eta_i\mid 1\leq i\leq m\}\hbox{ modulo }\L^{\bf (1)}.$$
Thus
$$T^-=\oplus_{i=1}^m\,\C \eta_i,\,\,
\omega=\sum_{i=1}^m \epsilon_i.$$
Since $[(t-1)D,\eta_i]=-\frac{1}{2}\eta_i$
we have $\delta^-=-\frac{1}{2}$.

We have 
$$\omega(h_1)=2\hbox{ and } \omega(h_i)=0
\hbox{ for } i\geq 2.$$
Since the weight $\mu:=\lambda-\omega$ is dominant,
the simple $\L^{\bf(1)}$-module 
$$S(\mu, \delta+m/2)$$ 
is finite dimensional. By proposition
\ref{hwF}, the highest weight of the $\K(2m)$-module
$M:=\cF(S(\mu, \delta+1),u)$ is $\lambda=\mu+\omega$ and
$$M^{(\lambda)}\simeq \Tens(\lambda,\delta,u).$$
Hence $V(\lambda,\delta,u)$ is a subquotient of 
$M$, which proves that $V(\lambda,\delta,u)$ is cuspidal.
\end{proof}

\subsection{Necessity of Condition $\lambda(h_1)\geq 2$}

Recall that $2m\geq 4$.

\begin{cor}\label{necessaryK(2m)} Assume that the
$\K(2m)$-module $V(\lambda,\delta,u)$ is cuspidal.
Then 
$$\lambda_1+\lambda_2\geq 2.$$
\end{cor}

\begin{proof} Assume that $V:=V(\lambda,\delta,u)$
is cuspidal. We claim that
$$V^{(\lambda-2\epsilon_1)}\neq 0.$$

We identify $\K(4)$ with the bracket subalgebra
$\C[t,t^{-1},\zeta_1,\eta_1,\zeta_2,\eta_2]$.
Let $V'$ be the $K(4)$-submodule generated
by $V^{(\lambda)}$ and set

Since by Theorem \ref{condition1},  $\lambda_1+\lambda_2\geq 1$, we deduce that
$$\lambda_2\geq 1.$$
Next, let $V\subset V(\lambda,\delta,u)$ be the 
$\K(4)$-submodule generated by $V^{(\lambda)}$.
By Corollary \ref{small}
$$\left(\L^{(-\epsilon_1)}\right)^2 V^{(\lambda)}\neq 0$$.

\noindent 
Therefore $\mu:\lambda-2\epsilon_1$ is a weight of
$V$. We observe that
$$\mu+\alpha_1=\mu+\epsilon_1+\epsilon_2\not\leq\lambda.$$
Since $e_1 V^{(\mu)}=0$ we have $\mu(h_1\geq 0$ which
implies
$$\lambda_1+\lambda_2\geq 2.$$
\end{proof}

\section{Classification of cuspidal modules\\ over $\K_*(2m+1)$ for $2m+1\geq 5$}\label{K(2m+1)}

In Section \ref{defK} we have defined the 
superalgebras 
$\K_*(2m+1)$ as the contact bracket algebra over 

$$\C[t,t^{-1},t^\tau\zeta_1,t^\tau\eta_1,\cdots, t^\tau\zeta_m,t^\tau\eta_m,t^\tau \xi]$$

\noindent where  $\zeta_1,\cdots,\eta_m,\xi$ are odd variables. The superconformal 
algebra is a Ramond contact superalgebra if $\tau=0$
and it is a Neveu-Schwarz 
contact superalgebra if $\tau=1/2$.

The set of quadratic elements 

$$\cQ:=\{\xi_i\xi_j,\,\eta_i\eta_j\mid 1\leq i<j\leq m\}
\cup\{ \,\xi_i\xi,\, \eta_j\xi \mid 
1\leq i\leq m\}$$ 

\noindent
form a basis of the Lie subalgebra $\fso(2m+1)\subset 
\K(2m+1)$. The
 Cartan subalgebra $H$ of the Lie algebra  $\fso(2m+1)$
 has basis
$$\{\xi_i\eta_i\mid 1\leq i\leq m\}.$$
Consequently a weight $\lambda\in H^*$ is written as a
$m$-uple $(\lambda_1,\cdots,\lambda_m)$ where
$\lambda_i=\lambda(\xi_i\eta_i)$. 

With respect with the triangular decomposition
defined in  Chapter 2,  the simple roots
of $\fso(2m+1)$ are
$$\alpha_1=\epsilon_1,\,
\alpha_2=\epsilon_2-\epsilon_1,\cdots,\,
\alpha_m=\epsilon_m-\epsilon_{m-1},$$
and the corresponding simple coroots are
$$h_1=2\zeta_1\eta_1, 
h_2=\zeta_2\eta_2-\zeta_1\eta_1,\cdots\,
h_m=\zeta_m\eta_m-\zeta_{m-1}\eta_{m-1}.$$

\noindent The corresponding Dynkin diagram is

$$\dynkin [backwards, arrows=false,
labels={m, m-1,3,2,1},
scale=1.8] B{oo......ooo}.$$

It is an elementary fact
that a weight $\lambda$ is dominant if and only if

\begin{enumerate}
\item[(a)] all $\lambda_i$ are integers or all $\lambda_i$ are half-integers, 
\item[(b)] $0\leq \lambda_1\leq\lambda_2\leq\cdots\leq \lambda_m$.
\end{enumerate}

For 
$$(\lambda,\delta,\,u)\in H^*\times\C\times \C,$$
we have defined 
the $\W(n)$-module $V(\lambda,\delta,u)$
in Chapter \ref{HW}. 
Observe that $\lambda(h_1)=2\lambda_1$.
In this section we show the following result.

\begin{thm}\label{ClasK(2m+1)} Assume that $2m+1\geq 5$.
Let  

$$(\lambda, \delta,u)\in H^*\times\C\times \C,$$

\noindent be an arbitrary triple with $\lambda$ dominant.
Then the $\K_*(2m+1)$-module $V(\lambda,\delta,u)$
is cuspidal  if and only if 
$$\lambda(h_1)\geq 2 \hbox{ that is } \lambda_1\geq 1.$$
\end{thm}

The proof results from Corollaries
\ref{suffitK(2m+1)} and \ref{necessaryK(2m+1)}.

\subsection{Sufficiency of the condition $\lambda(h_1)\geq 1$}

In this Section, we assume $2m+1\geq 3$.
Set $\L=\K_*(2m+1)$. As before, we  define a subalgebra $\L^{\bf(1)}\subset \L$ and
we identify its  highest weight modules.
 
For
simplicity, we first treat the case
$\L:=\K(2m+1)$. Let $\L^{\bf(1)}$ be the ideal of the commutative superalgebra
$\C[t,t^{-1},\xi_1,\cdots,\eta_n,\xi]$
generated by $t-1$ and the set $\cQ$ of quadratic elements. 
By Lemma \ref{repeat}, $\L^{\bf (1)}$ is a 
Lie subalgebra and we set 
$$T:=\L/\L^{\bf (1)}.$$
As a 
$\L^{\bf (1)}$-module $T$ is a nonsplit extension

$$0\to T_{\overline 1}\to T\to 
T_{\overline 0}\to 0.$$
Thus we define the ideal 
$\L^{\bf (2)}\subset \L^{\bf (1)}$ by
$$\L^{\bf (2)}=\{\partial\in \L^{\bf (1)}
\mid \partial 
T_{\overline 0}\subset T_{\overline 1}\hbox{ and }
\partial 
T_{\overline 0}=0\}.$$
It is easy to verify that
$$\L^{\bf (1)}=(\fso(2m+1)\oplus(t-1)D)\ltimes 
\L^{\bf (2)}.$$

For $\lambda\in H^*$ we denote by
$S(\lambda,\delta)$ the simple
$\L^{\bf(1)}/\L^{\bf(2)}$-module defined as follows
For an arbitrary pair $\lambda,\delta$,
let $S(\lambda,\delta)$ be the 
$\hat{\L}^{\bf (1)}$-module defined as follows
\begin{enumerate}
\item[(a)]  The action of $\L^{\bf(2)}$ is trivial,
so $S(\lambda,\delta)$ is a 
$\Bigl(\fso(2m+1)\oplus \C(t-1)D\Bigr)$-module.
\item[(b)] $S(\lambda,\delta)$
is the simple $\fso(2m+1)$-module with highest weight 
$\lambda$,
\item[(c)] the  element $(t-1)D$ acts
as $\delta$.
\end{enumerate}

For the superconformal algebra 
$\L_{NS}=\K_{NS}(2m+1)$, the definition is similar, namely

 $$\L_{NS}^{\bf(1)}=
 \L^{\bf(1)}_{\overline 0}
 \oplus t^{1/2} \L^{\bf(1)}_{\overline 1},
 \hbox{ and }$$
 
 $$\L_{NS}^{\bf(2)}=
 \L^{\bf(2)}_{\overline 0}
 \oplus t^{1/2} \L^{\bf(1)}_{\overline 2}.$$
 
Also we consider $S(\lambda,\delta)$ as an even 
vector space. Therefore it can be viewed as a 
$\L_{NS}^{\bf(1)}$-module as well.

 \begin{cor}\label{suffitK(2m+1)} Assume that $2m+1\geq 3$.
Let  

$$(\lambda, \delta,u)\in H^*\times\C\times \C,$$

\noindent be an arbitrary triple with $\lambda$ dominant. If
$$\lambda(h_1)\geq 2$$ 
the $\K_*(2m+1)$-module $V(\lambda,\delta,u)$
is cuspidal.
 \end{cor}
 
\begin{proof}
We check the hypothesis of Proposition \ref{hwF}
and we compute the datum $(\omega,m,\delta^-)$.
Here the grading element $\ell_0$ is $D$ and the
tangent space $T$ has basis
$$\{D,\xi\}\cup\{\zeta_i,\eta_i\mid 1\leq i\leq m\}\hbox{ modulo }\L^{\bf (1)}\hbox{ if } \L=\L_R.$$
$$\{D,t^{\frac{1}{2}}\xi\}\cup\{t^{\frac{1}{2}}\zeta_i,t^{\frac{1}{2}}\eta_i\mid 1\leq i\leq m\}\hbox{ modulo }\L^{\bf (1)}\hbox{ if } \L=\L_{NS}.$$

Thus
$$T^-=\oplus_{i=1}^m\,\C t^\tau\eta_i,\,\,
\omega=\sum_{i=1}^m \epsilon_i.$$
Since 
$$[(t-1)D,t^\tau\eta_i]=-\frac{1}{2}\eta_i
+\tau (t-1)t^\tau\eta_i\equiv -\frac{1}{2}\eta_i
\hbox{ modulo } \L^{\bf (1)}$$
we have $\delta^-=-\frac{1}{2}$.

We observe that
$$\omega(h_1)=2\hbox{ and } \omega(h_i)=0
\hbox{ for } i\geq 2.$$
Since the weight $\mu:=\lambda-\omega$ is dominant,
the simple $\hat{\L}^{\bf(1)}$-module 
$$S(\mu, \delta+m/2)$$ 
is finite dimensional. By proposition
\ref{hwF}, the highest weight of the $\K(2m+1)$-module
$M:=\cF(S(\mu, \delta+1),u)$ is $\lambda=\mu+\omega$ and
$$M^{(\lambda)}\simeq \Tens(\lambda,\delta,u).$$
Hence $V(\lambda,\delta,u)$ is a subquotient of 
$M$, which proves that $V(\lambda,\delta,u)$ is cuspidal.
\end{proof}

\subsection{Necessity of the condition $\lambda(h_1)\geq 2$}

In this Section, we assume $2m+1\geq 5$.
Using the  natural embedding
$$\C[\zeta_1,\zeta_2,\eta_1,\eta_2]\subset
\C[\zeta_1,\cdots,\zeta_m,\eta_1,\cdots,\eta_m,\xi]$$
in the Ramond case or the embedding
$$\C[t^{\frac{1}{2}}\zeta_1,t^{\frac{1}{2}}\zeta_2,t^{\frac{1}{2}}\eta_1,t^{\frac{1}{2}}\eta_2]\subset
\C[t^{\frac{1}{2}}\zeta_1,\cdots,t^{\frac{1}{2}}\zeta_m,t^{\frac{1}{2}}\eta_1,\cdots,t^{\frac{1}{2}}\eta_m,t^{\frac{1}{2}}\xi]$$
in the Neveu-Schwarz case,
we consider $\K(4)$ as a subalgebra of $\K(2m+1)$.

\begin{cor}\label{necessaryK(2m+1)}
Assume that $2m+1\geq 5$.
Let  
$$(\lambda, \delta,u)\in H^*\times\C\times \C,$$
\noindent be an arbitrary triple with $\lambda$ dominant. 
If the $\K(2m+1)$-module $V(\lambda,\delta,u)$
is cuspidal, then
$$\lambda(h_1)\geq 2.$$ 
\end{cor}

\begin{proof} Set $V=V(\lambda,\delta,u)$. We claim
that 
$$V^{(\lambda-2\epsilon_1)}\neq 0.$$
By Theorem \ref{condition1} we have $\lambda(h_1)\geq 1$, that is $\lambda_1\geq \frac{1}{2}$. Since $\lambda$
is dominant, we deduce that
$$\lambda_2\neq 0.$$
The $\K(4)$ submodule generated by $V^{(\lambda)}$
has a quotient $V'$ isomorphic to
$V((\lambda_1,\lambda_2),\delta,u)$. We consider
$V'$ as a $\widehat{\K(4)}$-module with trivial central charge. Since $\lambda_2>0$, we have
$0=\lambda_c\neq \pm 2 \lambda_2$. Thus by Corollary \ref{small}
we have $$V'^{(\lambda-2\epsilon_1)}\neq 0$$
which proves that
$$V^{(\lambda-2\epsilon_1)}\neq 0.$$

Since $V$ has only finitely many weights, there are invariant by the Coxeter involution $s_1$. Thus
$$\lambda+(2-\lambda(h_1))\epsilon_1=
s_1(\lambda-2\epsilon_1)\leq \lambda,$$
which proves that $\lambda(h_1)\geq 2$.
\end{proof}

\section{Cuspidal $\K_*(3)$-modules}

In Section \ref{defK} we have defined
$\K_*(3)$  as the contact bracket algebras over 

$$\C[t,t^{-1},t^\tau\zeta_1,t^\tau\eta_1, t^\tau \xi]$$

\noindent where  $\zeta_1,\eta_1$ and $\xi$ are odd variables. The superconformal 
algebra $\K_*(3)$ is the Ramond contact superalgebra 
$\K(3)$ if $\tau=0$
and it is the Neveu-Schwarz 
contact superalgebra 
$\K_{NS}(3)$ otherwise.

The set of quadratic elements 
$$\cQ:=\{\xi_1\eta_1,\xi_1\xi, \eta_1\xi\}$$ 

\noindent
form a basis of the Lie subalgebra $\fso(3)\simeq
\fsl(2)\subset \K_*(3)_{\bar 0}$. 

Its Cartan subalgebra is the
$1$-dimensional Lie algebra $H:=\C \zeta_1\eta_1$.
Consequently a weight $\lambda\in H^*$ can be described as the scalar  $\lambda_1=\lambda(\xi_1\eta_1)$. 
The only roots of $\K_*(3)$ are $\pm\epsilon_1$.

The  simple coroot of $\fso(3)$ is
$$h_1=2\zeta_1\eta_1.$$  
Hence a weight $\lambda$ is dominant if and only if
$\lambda_1$ is a nonnegative
integer or  half-integer. For 
$$(\lambda,\delta,\,u)\in H^*\times\C\times \C,$$
we have defined 
the $\K(3)$-module $V(\lambda,\delta,u)$
in Chapter \ref{HW}. 
Observe that $\lambda(h_1)=2\lambda_1$.
In this section we show the following result.

\begin{thm}\label{ClasK(3)} 
Let  
$$(\lambda, \delta,u)\in H^*\times\C\times \C,$$

\noindent be an arbitrary triple with $\lambda$ dominant.
Then the $\K_*(3)$-module $V(\lambda,\delta,u)$
is cuspidal  if and only if
\begin{enumerate}
\item[(a)]  $\lambda_1+\lambda_2\geq 2$, that is 
$\lambda(h_1)\geq 2$, or
\item[(b)] $\lambda(h_1)=1=\frac{1}{2}$ and $\delta=
\frac{1}{4}$.
\end{enumerate}
\end{thm}

The proof for $\K(3)$ and $\K_{NS}(3)$  are identical, so we mostly consider the Ramond form 
$\K(3)$. When needed, we explain how to adapt the proof in the Neveu-Schwarz case.
The proof of the previous theorem results from Corollaries \ref{suffitK(3)} and
\ref{necessaryK(3)}
proved below.

\subsection{Sufficiency of Conditions (a) and (b)}

\begin{cor}\label{suffitK(3)}
Let  
$$(\lambda, \delta,u)\in H^*\times\C\times \C,$$

\noindent be an arbitrary triple with $\lambda$ dominant. If
\begin{enumerate}
\item[(a)]  $\lambda_1\geq 1$, that is 
$\lambda(h_1)\geq 2$, or
\item[(b)] $\lambda_1=\frac{1}{2}$ and $\delta=
\frac{1}{4}$,
\end{enumerate}
then the $\K(3)$-module $V(\lambda,\delta,u)$
is cuspidal.
\end{cor}

\begin{proof} The sufficiency of Condition (a) follows from Corollary \ref{suffitK(2m+1)} and the sufficiency
of Condition (b) follows from Theorem \ref{exceptional}. 
\end{proof}

\subsection{Necessity of Conditions (a) and (b)}

The proof of the necessity of Conditions (a) and (b)
requires some computations which are easy consequences
of Lemma \ref{formulasK(4)}.

First assume that $\L:=\K(3)$ is the Ramond contact superalgebra. Let $v$ be a formal variable.
As before, for any $a\in\Grass(3)=\L_0$ we
we consider the  distribution
$$a(v)=\sum_{n\in\Z}\,a(t^n) v^n\in \L[v,v^{-1}].$$

To avoid ambiguities, the unit element is denoted $D$ and we set
$$D(v)=\sum_{n\in\Z}\,(t^nD) v^n.$$

Let  
$$(\lambda, \delta,u)\in H^*\times\C\times \C,$$

\noindent be an arbitrary triple with $\lambda\neq 0$
and set
$$V:=V(\lambda,\delta,u).$$
Let $\v$ be a non zero element in $V^{(\lambda)}_{\bar 0,0}$ and let $z$ be another formal variables.
Since $V^{(\lambda)}$ is a $\C[t,t^{-1},\xi]$-module
we set
$$\v(z)=\sum_{n\in\Z}\,\v(t^n) v^n\hbox{ and }
(\xi\v)(z)=\sum_{n\in\Z}\,(\xi\v)(t^n) v^n.$$

For $\L_{NS}=\K_{NS}(3)$ we follow the conventions of
Section \ref{multiNS}, thus for an odd
element $a\in\L_0$ we write
$$a(v):=\sum_{n\in\frac{1}{2}+\Z}\,a(t^n) v^n\text{ and }$$
$$(\xi\v)(v):=
\sum_{n\in\frac{1}{2}+\Z}\,(\xi\v)(t^n) v^n.$$

As before, set
$${\tilde \v}=\v(v,w,x,y,z)\hbox{ and }
{\widetilde {\xi\v}}={\xi\v}(v,w,x,y,z).$$  

\begin{lemma}\label{formulasK(3)} We have

\bigskip\noindent
$(a)\hskip1mm \zeta_1(v)\,\zeta_1(w)\,\eta_1(x)\,\eta_1(y)\,\v(z)=
(\delta-1+\frac{\lambda_1}{2})(\delta-\frac{\lambda_1}{2})(D_v-D_w)(D_x-D_y)\,{\tilde \v}$

\bigskip\noindent
$(b)\hskip1mm \zeta_1(v)\,\zeta_1(w)\,\eta_1(x)\,\eta_1(y)\,\xi\v(z)=
(\delta+\frac{\lambda_1-1}{2})(\delta-\frac{\lambda_1-1}{2})(D_v-D_w)(D_x-D_y)\,{\widetilde {\xi\v}}$
\end{lemma}

\begin{proof} Set $D_v=\frac{v\d}{\d v}$ and
$\Delta_z=u+\frac{z\d}{\d z}$. We have
$$D(v)\,\v(z)=(\delta D_v+\Delta_z)\, \v(v,z)$$
$$\zeta_1\eta_1(v)\,\v(z)=\lambda_1\, \v(v,z).$$
Thus the first formula follows from Lemma \ref{formulasK(4)}(a).

We have 
$$D(v)\,(\xi\v)(z)=((\delta+\frac{1}{2}) D_v+\Delta_z)\, (\xi\v)(v,z)\hbox{ and }$$
$$\zeta_1\eta_1(v)\,(\xi\v)(z)=
\lambda_1\,(\xi\v)(v,z),$$
which similarly proves Formula (b).
\end{proof}

\begin{cor}\label{necessaryK(3)} 
Let  
$$(\lambda, \delta,u)\in H^*\times\C\times \C,$$

\noindent be an arbitrary triple with $\lambda$ dominant.
If the $\K(3)$-module $V(\lambda,\delta,u)$
is cuspidal  then
\begin{enumerate}
\item[(a)]  $\lambda_1\geq 1$, or
\item[(b)] $\lambda_1=\frac{1}{2}$ and $\delta=
\frac{1}{4}$.
\end{enumerate}
\end{cor}

\begin{proof} 
 By Theorem \ref{condition1} we have
$\lambda_1\geq 1$, that is $\lambda_1\geq \frac{1}{2}$.

Set $V:=V(\lambda,\delta,u)$. First assume that
$$V^{(\lambda-2\epsilon_1)}\neq 0.$$
The weights of $V$ are stable by the Coxeter involution
$s_1$, we deduce that
$$\lambda\geq s_1(\lambda-2\epsilon_1)=
\lambda+(2-2\lambda_1\epsilon_1),$$
which implies that $\lambda_1\geq 1$.

Consequently when $\lambda_1=\frac{1}{2}$ we have
$$V^{(\lambda-2\epsilon_1)}= 0.$$
Lemma \ref{formulasK(3)} implies that
$$(\delta-1+\frac{\lambda_1}{2})(\delta-\frac{\lambda_1}{2})=0\hbox{ and }$$
$$(\delta+\frac{\lambda_1-1}{2})(\delta-\frac{\lambda_1-1}{2}),$$
that is 
$$(\delta-\frac{3}{4})(\delta-\frac{1}{4})=0\hbox{ and }(\delta-\frac{1}{4})(\delta+\frac{1}{4}),$$
which implies that $\delta=\frac{1}{4}$.
\end{proof}

\section{Cuspidal modules over the Chen-Kac superalgebra $\CK(6)$}

The Chen-Kac superconformal algebra contains 
the Lie algebra $\fso(6)\supset H$.
Relative to the chosen triangular decomposition,
the Dynkin diagram is labelled as:

$$\dynkin [backwards,
labels={2,3,1},
scale=1.8] D{ooo}$$

Weights $\lambda$ for $\CK(6)$ can be described
by their values $\lambda(h_i)$ or as 
triples $(\lambda_1,\lambda_2,\lambda_3)$.
Indeed
$$\lambda(h_1)=\lambda_2-\lambda_1,\,
\lambda(h_2)=\lambda_2+\lambda_1,\, 
\text{ and } \lambda(h_3)=\lambda_3-\lambda_2.$$

In this Section we prove:

\begin{thm}\label{ClasCK(6)} 
Let  
$$(\lambda, \delta,u)\in H^*\times\C\times \C,$$

\noindent be an arbitrary triple with $\lambda$ dominant.
The $\CK(6)$-module $V(\lambda,\delta,u)$
is cuspidal  if and only if
\begin{enumerate}
\item[(a)]  $\lambda(h_1)\geq 2$, or
\item[(b)] $\lambda(h_1)=1$, $\lambda(h_3)=0$ and $\delta=
\frac{\lambda_1}{2}$.
\end{enumerate}
\end{thm}

\subsection{Proof of Theorem \ref{ClasCK(6)}}

\begin{lemma}\label{suffitCK(6)} 
Let  
$$(\lambda, \delta,u)\in H^*\times\C\times \C,$$

\noindent be an arbitrary triple with $\lambda$ dominant.
If 
\begin{enumerate}
\item[(a)]  $\lambda(h_1)\geq 2$, or
\item[(b)] $\lambda(h_1)=1$, $\lambda(h_3)=0$ and $\delta=
\frac{\lambda_1}{2}$.
\end{enumerate}
the $\CK(6)$-module $V(\lambda,\delta,u)$
is cuspidal.
\end{lemma}

\begin{proof} The paper \cite{MZ02} investigates
the cuspidal $\Conf(\CK(6))$-module. The embedding of $\Vir$ in \cite{MZ02} and here are not the same.
However it follows from \cite{MZ02} that the
$\CK(6)$-module
$V(\lambda,\delta,u)$ is cuspidal.
\end{proof}

\begin{thm}\label{necessaryCK(6)} 
Let  
$$(\lambda, \delta,u)\in H^*\times\C\times \C,$$

\noindent be an arbitrary triple with $\lambda$ dominant. If the $\CK(6)$-module $V(\lambda,\delta,u)$
is cuspidal then
\begin{enumerate}
\item[(a)]  $\lambda(h_1)\geq 2$, or
\item[(b)] $\lambda(h_1)=1$, $\lambda(h_3)=0$ and $\delta=
\frac{\lambda_1}{2}$.
\end{enumerate}
\end{thm}

\begin{proof} Asume that $V(\lambda,\delta,u)$ is cuspidal
and write $\lambda=\lambda_1,\lambda_2,\lambda_3$. 
By Theorem \ref{K(4)=K(4)},
the superconformal algebra $\widehat{\K(4)}$ embedds in $\CK(6)$ and its center $c$ of 
$\widehat{\K(4)}$ is $c=h_1+h_2+2h_3$.

Set $\lambda_c=\lambda(h_1)+\lambda(h_2)+2\lambda(h_3)$.
Obviously the $\widehat{\K(4)}$-module
$$V((\lambda_1,\lambda_2,\lambda_c),\delta,u)$$ 
is a subquotient of $\CK(6)$-module $V(\lambda,\delta,u)$, thus
this $\widehat{\K(4)}$-module
is cuspidal.
Hence by Theorem \ref{ClassK(4)}, one of the following three assertions hold:

\begin{enumerate}
\item[(1)] either $\lambda_1+\lambda_2\geq 2$
that is $\lambda(h_1)\geq 2$,

\item[(2)] or
$\lambda_1+\lambda_2=1$, $\delta=\frac{1-\lambda_2}{2}$ and 
$\lambda_c=2\lambda_2$,

\item[(3)] or
$\lambda_1+\lambda_2=1$, $\delta=\frac{1+\lambda_2}{2}$ and 
$\lambda_c=-2\lambda_2$.
\end{enumerate}

In the first case, Assertion (a) holds. 

We now consider the second case.
We observe that $\lambda(h_1)=\lambda_1+\lambda_2=1$
and that $\delta=\frac{1-\lambda_2}{2}=
\frac{\lambda_1}{2}$. Moreover
$$\lambda(c)=\lambda(h_1)+\lambda(h_2)+2\lambda(h_3)
=2\lambda_3.$$
Therefore $\lambda_2=\lambda_3$, which is equivalent to $\lambda(h_3)=0$. Thus Assertion (b) holds.

The third case is impossible. Indeed
$\lambda(c)=\lambda(h_1)+\lambda(h_2)+2\lambda(h_3)$
is necessarily positive, though $-2\lambda_2$ is negative.
\end{proof}

\vskip2cm
\centerline{\bf PART III: Classification of cuspidal modules}

\centerline{\bf for twisted superconformal algebras}

\bigskip\noindent
The  superconformal algebra $\K(2n)$ of Ramond type
admit an involution $\sigma$ which extends the Dynkin 
diagram involution of their subalgebra $\fso(2n)$.
Set $\K^{\bf (2)}(2n):=\K(2n)^\sigma$. 
We will now deduce from the classification
of  cuspidal $K(2n)$-modules the classification 
of cuspidal $\K^{\bf (2)}(2n)$-modules.

\section{Locality for modules of growth one}\label{chlocal}

As in Chapter \ref{split} we consider the 
$\Z$-graded Lie algebra 
$$\fg:=\Vir\ltimes h\otimes \C[t,t^{-1}].$$
Let $\alpha, \beta\in\C$ and let $A, B$ be $\fg$-modules, 
such that
$$A=A^{(\alpha)}\text{ and } B=B^{(\beta)},$$
that is the operators $h\vert_A$ and $h\vert_B$ act by multiplication by, respectively, $\alpha$ and $\beta$.
 Assume moreover that $A$ and $B$ are $\Z$-graded $\fg$-modules of growth one and the scalars $\alpha,\beta\in\C$ are nonzero. By Corollary \ref{corh=1}, the operator 
 $h(t)\vert_A$ and  $h(t)\vert_B$ 
are invertible. 
Thus for $a\in A_0$, $b\in  B_0$  we define the  distributions:
 $$a(z)=\sum_{n\in\Z}\, \Bigl(\frac{h(t)}{\alpha}\Bigr)^n\cdot a \,z^n\text{ and }b(z)=
 \sum_{n\in\Z}\, \Bigl(\frac{h(t)}{\beta}\Bigr)^n\cdot  b\, z^n.$$

In this chapter, we prove the following result:

\begin{thm}\label{local} Let $C$ be another
$\Z$-graded $\fg$-module of growth one, and let
$$\pi:A\otimes B\to C$$
be an homomorphism of $\Z$-graded $\fg$ modules.

The distributions $a(z)$ and $b(z)$ are mutually local, that is 
$$(z_1-z_2)^N\pi(a(z_1),b(z_2))=0\text{ for some } N>>0.$$
 \end{thm}

\subsection{Nilpotency in the commutative algebras}

Let $R$ be a commutative unital algebra endowed with a derivation $\partial$.  By Lemma \ref{dernil}, its radical
$\Rad(R)$  is stable by $\partial$.

\begin{lemma}\label{key} Let $r,s\in R$ such that 
$\partial^d(r)=\alpha$, $\partial^d(s)=\beta$ for 
some nonnegative integer $d$ and some 
$\alpha,\,\beta\in\C\setminus 0$. Suppose further that
$f(r,s)=0$ for some nonzero homogenous polynomial $f(x,y)$.

Then $r/\alpha-s/\beta$ is nilpotent.
\end{lemma}

\begin{proof} Without loss of generality, we can assume that
$\Rad(R)=0$ and that the polynomial $f$ is square free.
Thus 
$$f=L_1\cdots L_n,$$
 where $L_1,\dots,L_n$ are linear forms which are pairwise nonproportional.

By hypothesis $\partial^d(L_i(r,s))=L_i(\alpha,\beta)$ and
therefore $\partial^{d+1}(L(r,s))=0$. It follows that
$$0=\partial^{nd}(f(r,s))=
c\prod_{i=1}^{i=n}\,\partial^d(L_i(r,s))=
c \prod_{i=1}^{i=n}\,L_i(\alpha,\beta)$$

\noindent for some $c\neq 0$. Permuting some factors, we can assume that $L_1(\alpha,\beta)=0$. 

We claim that $L_1(r,s)=0$, which proves the lemma. Suppose otherwise. It has been proved that $\partial^d(L_1(r,s))=0$, thus there is an integer $q\geq 0$  such that $\partial^q(L_1(r,s))\neq 0$
and $\partial^{q+1}(L_1(r,s))= 0$. We have

\medskip
\hskip3cm$0=\partial^{q+(n-1)d}(f(r,s))$

\hskip33mm$=c'\partial^q(L_1(r,s))\prod_{i=2}^{i=n}\,\partial^d(L_i(r,s))$

\hskip33mm$=
c' \partial^q(L_1(r,s))\prod_{i=2}^{i=n}\,L_i(\alpha,\beta)$

\medskip
\noindent for some $c'\neq 0$. Since for $i\geq 2$ the linear forms $L_i$ are not
proportional to $L_1$, $\prod_{i=2}^{i=n}\,L_i(\alpha,\beta)$
is not zero. Hence $\partial^q(L_1(r,s))=0$ which contradicts the choice of $q$. This completes the proof of the lemma.
\end{proof}

\subsection{Structure of $\fg$-modules}
Consider the $\fg$ module  $S=A\otimes B$. 
The image of $U(h\otimes\C[t,t^{-1}])$ in
$\End(S)$, denoted by $\cA$, is the algebra generated by the operators $h(t^i)\otimes 1+1\otimes h(t^i)$ for $i\in\Z$.

Set $x_{\pm}=h(t^{\pm1})\otimes 1$, $y_{\pm}=1\otimes h(t^{\pm1})$
and let $\mathcal{B}\subset \End(S)$ be the algebra 
generated by $\mathcal{A}$ and the  operators
$x_{\pm}$ and $y_{\pm}$.

\begin{lemma}\label{fingen} Assume that $\alpha+\beta\neq 0$. Then $S$ is a finitely generated 
$\mathcal{A}$-module.
\end{lemma}

\begin{proof}  
The space $S$ admits a $\Z^2$-grading
$S=\oplus_{i,j\in\Z} S_{i,j}$, where
$S_{i,j}=A_i\otimes B_j$. By Corollary \ref{corh=1} the
operators $x_\pm$ and $y_{\pm}$ are  invertible. Since these four operators are bi-homogenous of degree 
$(\pm1,0)$ and $(0,\pm1)$, we have 

\noindent \centerline{$S=\mathcal{B}.S_{0,0}$,}

\noindent which shows that $S$ is finitely generated as a 
$\mathcal{B}$-module.
Therefore, it is enough to prove that that $\mathcal{B}$
is finitely generated as an $\mathcal{A}$-module. Equivalently, we will prove that $x_\pm$ and $y_\pm$ are integral over $\cA$.

By definition, the operators 
$$u:=x_+ +y_+=h(t)\otimes 1+1\otimes h(t)\text{ and }
v:=h(t^2)\otimes 1 +1\otimes h(t^2)$$
belong to
$\mathcal{A}$. By Corollary \ref{corh=1}, the operators

 $$\frac{x_+^2}{\alpha}-(h(t^2)\otimes 1)\text{ and }
\frac{y_+^2}{\beta}-(1\otimes h(t^2))$$ 
 are  nilpotent.
Thus $\frac{x_+^2}{\alpha}+ \frac{y_+^2}{\beta}-v$ is nilpotent, say 
$$(\frac{x_+^2}{\alpha}+ \frac{y_+^2}{\beta}-v)^m=0.$$
The substitution $y_+^2=x_+^2-2 ux_++u^2$ in the previous equation
provides the equation
$$\Bigl(\frac{(\alpha+\beta)}{\alpha\beta}x_+^2
-\frac{2u}{\beta} x_+ +( u^2-v)\Bigr)^m=0.$$
which shows that $x_+$ is integral over $\cA$.
Since $y_+=u-x_+$, the element $y_+$  is also integral over $\mathcal{A}$.

Similarly, $x_-$ and $y_-$ are integral over 
$\mathcal{A}$,
which completes the proof of the lemma.
\end{proof}

\noindent Let $\cA_A$ and $\cA_B$ be the image of 
of $U(h\otimes\C[t,t^{-1}])$ in, respectively, $\End(A)$
and $\End(B)$. By definition we have
$$\cA\subset \cB\subset \cA_A\otimes \cA_B.$$
Let $I:=\{a\in\cA_A\otimes \cA_B\mid aS\subset \Ker\,\pi\}$.

\begin{lemma}\label{existf(x,y)} Assume $\alpha+\beta\neq 0$. There exists a nonzero homogenous polynomial
$f(x,y)\in\C[x,y]$ such that $f(x_+,y_+)$ belongs to $I$.
\end{lemma}

\begin{proof} For integers $k\leq l$ set 
$S_{[k,l]}=\oplus_{k\leq i,j\leq l} S_{i,j}$ and
$C_{[k,l]}=\oplus_{k\leq i\leq l} C_{i}$.
By Lemma \ref{fingen}, there are integers
$k\leq l$ such that $S_{[k,l]}$  generates the 
$\mathcal{A}$-module $S$.

For any homogenous polynomial $f(x,y)$ of degree $N$, we observe that $\pi(f(x_+,y_+)(S_{[k,l]}))$ lies in 
$C_{[N+2k,N+2l]}$. Since $C$ has growth one, there is an integer $D$ such that
$$\dim\,\Hom(S_{[k,l]},C_{[N+2k,N+2l]})\leq D\,\,\, \forall N\in\Z.$$

 Thus there is a nonzero degree $D$
homogenous polynomial $f$ in two variables  such that 

\centerline{$\pi(f(x_+,y_+)(S_{[k,l]}))=0$.}

Since $S=U(h\otimes\C[t,t^{-1}]).S_{[k,l]}$, it follows
that 
\begin{align*}
\pi(f(x_+,y_+).S)
&= \pi\Bigl(f(x_+,y_+)U(h\otimes\C[t,t^{-1}]).S_{[k,l]}\Bigr)\\
&=U(h\otimes\C[t,t^{-1}]).\pi\Bigl(f(x_+,y_+).S\Bigr)\\
&=0,
\end{align*}
that is $f(x_+,y_+)$ belongs to the ideal $I$.
\end{proof}

\begin{lemma}\label{idn}  Assume $\alpha+\beta\neq 0$. There exists $m\geq 2$ such that

\centerline{$\pi\Bigl((\frac{h(t)}{\alpha}\otimes 1-1\otimes \frac{h(t)}{\beta})^m A\otimes B\Bigr)=0$.}
\end{lemma}

\begin{proof} Consider te commutative algebra 
$$R=\cA_A\otimes \cA_B/I.$$ 
The ideal $I$ is invariant with respect to $\Vir$, henceforth the Lie algebra $\Vir$ acts on  algebra $R$.

Set $\partial=\frac{\d}{\d t}$. We have  $\partial x_+=\alpha$, $\partial x_-=\beta$
and $f(x_+,y_+)$ vanishes in $R$, where $f$ is the homogenous polynomial defined in Lemma \ref{existf(x,y)}. Hence Lemma \ref{key}
implies that  $\frac{x_+}{\alpha} -\frac{x_-}{\beta}$ is nilpotent modulo
$I$, which is equivalent to the assertion of the lemma.
\end{proof}

The lemma also holds for $\alpha+\beta=0$, but the proof is different.

\begin{lemma}\label{id0} For $\alpha+\beta=0$, there exists $m\geq 2$ such that

\centerline{$\pi\Bigl((\frac{h(t)}{\alpha}\otimes 1-1\otimes \frac{h(t)}{\beta})^m A\otimes B\Bigr)=0$.}
\end{lemma}

\begin{proof} Without loss of generality, we can assume that $C=C^{(0)}$. 
 Thus by Corollary \ref{corh=0} $h(t)^m C=0$  for some $m\geq 1$. We deduce that
 $$\pi\Bigl((h(t)\otimes 1+1\otimes h(t))^m S\Bigr)=0,$$
which is equivalent to the lemma.
\end{proof}

\subsection{Proof of Theorem \ref{local}}

{\it Proof.}
By lemmas \ref{idn} and \ref{id0}, we have
$$\pi\Bigl((\frac{h(t)}{\alpha}\otimes 1-1\otimes 
\frac{h(t)}{\beta})^m A\otimes B\Bigr)=0.$$
However we have
\begin{align*}
&(z_1^{-1}-z_2^{-1})^m a(z_1)\otimes b(z_2)\\
&=\sum_{k+l=m}\,(-1)^k \binom{m}{k}
\sum_{n_1,n_2\in\Z}\, \bigl(\frac{h(t)}{\alpha}\bigr)^{n_1} a \,z^{n_1-k}\otimes \bigl(\frac{h(t)}{\beta}\bigr)^{n_2} b
z_2^{n_2-l} \,\\
&=\Bigl(\sum_{k+l=m}\,(-1)^k \binom{m}{k}
\frac{h(t)}{\alpha}^k\otimes \frac{h(t)}{\beta}^l\Bigr)
a(z_1)\otimes b(z_2)\\
&=\Bigl(\frac{h(t)}{\alpha}\otimes 1-
1\otimes \frac{h(t)}{\beta}\Bigr)^m a(z_1)\otimes b(z_2)
\end{align*}

Thus $\pi((z_1^{-1}-z_2^{-1})^m a(z_1)\otimes b(z_2))=0$
which is equivalent to
$$(z_1-z_2)^m\pi(a(z_1)\otimes b(z_2))=0,$$
which proves the theorem. \qed

\section{Two kinds of cuspidal $\K(2m)$-modules}
\label{twokinds}

In order to state the classification of 
cuspidal $\K^{(2}(2m)$-modules in the next chapter, we need a preliminary result about the types of
$\K(2m)$-modules. 

Recall that
$$\K^{(2)}(2m)=\K(2m)^\sigma,$$
for some involution $\sigma$.
A cuspidal $\K(2m)$-module $M$ is called of {\it first kind} if
$$\sigma_* V\not\simeq V,$$
otherwise it is called of {\it second kind}.  
We similarly define the $\widehat{\K(4)}$-modules of
first or second kind.

\begin{thm}\label{first-second}
Let $(\lambda,\delta,u)\in H^*\times\C\times\C$ be an arbitrary triple.
\begin{enumerate} 
\item[(a)] If $m\geq 3$, the $\K(2m)$-module $V(\lambda,\delta,u)$ is of second type if and only if 
$$\lambda_1=1.$$
\item[(b)] The $\widehat{\K(4)}$
-module $V(\lambda,\delta,u)$ is 
 of second type if and only if 
$$\lambda_1=1\text{ and }(\lambda_c,\delta)\neq(\pm2\lambda_2, \frac{1}{2}).$$
\end{enumerate}
\end{thm}

\subsection{Proof of Theorem \ref{first-second}}
Let $m\geq 2$.
Set $\L:=\K(2m)\text{ if } m\neq 2$ or
$\L:=\widehat{\K(4)}$ if $m=2$. The involution $\sigma$ stabilizes the subalgebras $\fso(2m)\subset\L_{\bar 0, 0}$ and its Cartan subalgebra $H$. Thus $\sigma$ acts on $H^*$ and
$$\sigma(\epsilon_1)=-\epsilon_1\text{ and }
\sigma(\epsilon_k)=\epsilon_k\text{ for } k\neq 1.$$
Indeed $\sigma\vert_H$ is the Dynkin diagram involution.

Let $V$ be a cuspidal $\L$-module with highest weight
$\lambda$. All weights $\mu$ of $V$ are
$\leq \lambda$, that is  
$$\lambda-\mu=\sum_{\alpha\in\Delta^+}\, m_\alpha\alpha,$$
where all $m_\alpha$ are nonnegative integers.
(Obvserve that $\Delta=-\Delta$.)

The module $V$ also admits a unique weight $\lambda'$
such that all weights $\mu$ of $V$ satisfy
$\sigma(\mu)\leq \sigma(\lambda')$ or equivalently
$$\lambda'-\mu=\sum_{\beta\in\sigma(\Delta^+)} m_\beta\,\beta,$$
where all $m_\beta$ are nonnegative integers.
The weight $\lambda'$ is called the
{\it $\sigma$-highest weight} of $V$.

Let $(\lambda,\delta,u)\in H^*\times\C\times\C$ be an arbitrary triple such that the
$\L$-module  $V:=V(\lambda,\delta,u))$ is cuspidal.
Set $\fg(H):=\Vir\ltimes H\otimes \C[t,t^{-1}]$.
We always assume that the $\fg(H)$-module
$V^{(\lambda)}$ is even.

\begin{lemma}\label{sigmahw} Let $\lambda'$ be the 
$\sigma$-highest weight of $V$.
\begin{enumerate}
\item[(a)] Either $\lambda'=\lambda-\epsilon_1$
and the space $V^{(\lambda')}$ is odd,

\item[(b)] or $\lambda'=\lambda-2\epsilon_1$,
$V^{(\lambda')}$ is even and
$$V^{(\lambda')}\simeq \Tens(\lambda',\delta,u)$$
as a $\fg(H)$-module.
\end{enumerate}
\end{lemma}

\begin{proof} By hypothesis
$V^{(\lambda+n\epsilon_1)}=0$ for $n>0$. 
It is clear that $\lambda'=\lambda-k\epsilon_1$,
where  $k$ is the biggest integer such that
$V^{(\lambda-k\epsilon_1)}\neq 0$.

Obviously,  $V^{(\lambda-k\epsilon_1)}$
has parity $k\mod 2$. Thus Assertion (a) holds
whenever  $\lambda'=\lambda-\epsilon_1$.

We now assume that $k\neq 1$.
We have $[\zeta_1,\eta_1]=D$, hence
$$\zeta_1\eta_1\vert_{V^{(\lambda)}}=
D\vert _{V^{(\lambda)}}\neq 0,$$
which implies that 
$V^{(\lambda-\epsilon_1)}\neq 0$.
Thus we have $k\geq 2$. Since
$$\lambda'+\alpha_1\not\leq \lambda,$$
we deduce that $e_1\cdot V^{(\lambda')}=0$,
which implies that $\lambda'(h_1)\geq 0$, that is
$$\lambda(h_1)\geq k\geq 2.$$ 
It follows that the weight 
$\mu:=\lambda-\omega$, where 
$\omega=\epsilon_1+\cdots+\epsilon_m$, is dominant.
 
First we assume $m\geq 3$. With the notations of
 Section \ref{SecSuffitK(2m)}, we consider the simple finite dimensional $\L^{\bf(1)}$-module 
 $S:=S(\mu,\delta+\frac{m}{2})$ and we set
 $$M:=\cF(S,u).$$

When $m=2$ we  use the notations of Section
\ref{SecSuffitK(4)} and we consider the 
simple $\hat{\L}^{\bf(1)}$-module 
$S:=S(\mu,\delta+1)$ and we set
 $$M:=\cF(S,u).$$

By Proposition \ref{hwF}, 
$\lambda$ is the highest weight of the $\L$-module
 $M$ and
there is an isomorphism of $\fg(H)$-modules
$$M^{(\lambda)}\simeq 
\Tens(\lambda, \delta,u).$$
Since $S$ is a simple $\fso(2m)$-module with highest weight $\mu$, its 
$\sigma$-highest weight is also $\mu$.
Hence by proposition \ref{hwF} the $\sigma$-highest 
weight of $M$ is 
$$\mu+\sigma(\omega)=\lambda-2\epsilon_1$$
and there is an isomorphism of $\fg(H)$-modules
$$M^{(\lambda-2\epsilon_1)}\simeq 
\Tens(\lambda-2\epsilon_1, \delta,u).$$

\noindent Since $V$ is a subquotient of $M$, it follows that $\lambda'=\lambda-2\epsilon_1$ and
$$V^{(\lambda')}\simeq \Tens(\lambda',\delta,u)$$
as a $\fg(H)$-module.
\end{proof}

\begin{lemma}\label{critsecondtype}
Let  $V$ be a cuspidal $\L$-module with highest weight $\lambda$.
Then $V$ is of second type if and only if
$$\lambda_1=1\text{ and } 
V^{(\lambda-2\epsilon_1)}\neq 0.$$
\end{lemma}

\begin{proof} The highest weight of $\sigma_* V$ is
$\sigma\lambda'$. 

If $\lambda'=\lambda-\epsilon_1$ the highest component of $\sigma_* V$ 
is odd by Lemma \ref{sigmahw}, thus 
$$V\not\simeq \sigma_* V.$$ 

Otherwise $\lambda'=\lambda-2\epsilon_1$, hence
$$\sigma(\lambda')=\lambda+2(1-\lambda_1)\epsilon_1.$$
Thus $\lambda_1=1$ if $V$ is of second type. Conversely assume that $\lambda_1=1$.
By Theorem \ref{condition1}, $V$ is isomorphic to
$V(\lambda,\delta,u)$ for some $\delta, u\in\C$. 
We observe that $\lambda=\sigma(\lambda')$ is the highest weight of $\sigma_*V$. Moreover by Lemma \ref{sigmahw},  
$$(\sigma_*V)^{(\lambda)}\simeq\Tens(\lambda,\delta,u)$$ 
as a $\fg(H)$-module, thus  $\sigma_*V$ is isomorphic to $V$.
\end{proof}

\section{Classification of cuspidal modules over the\\ twisted superconformal algebras $\K^{(2)}(2m)$}\label{ChClassK2}
Our final result concerns the classification
of cuspidal modules over the
twisted superconformal algebras $\K^{(2)}(2m)$
for $m\geq 2$

Recall that 
$\K^{(2)}(2m)$ is the subalgebra $\K(2m)^\sigma$ of fixed points under some involution $\sigma$ of $\K(2m)$. 
By definition, a $\K(2m)$-module $M$ endowed with an involution $\sigma$ satisfying
$${\sigma}(x.m)=\sigma(x){\sigma}(m),
\hskip2mm\forall x\in \K(2m)\hbox{ and } m\in M$$ 
is called an {\it $(\sigma,\K(2m))$-module.}
The notion of a $(\sigma,\widehat{\K(4)})$-module is similarly defined.

Obviously a cuspidal $\K(2m)$-module of second type carries a structure of $(\sigma,\K(2m))$-module.
More precisely it carries two such structures,
since the involution
$\sigma:M\to M$ is defined up to a sign.

Recall that the type of any cuspidal
$\K(2m)$-module or $\widehat{\K(4)}$-module have been determined by Theorem \ref{first-second}, thus the next two theorems provide an explicit classification 
of all cuspidal modules over the twisted superconformal algebras.

\begin{thm}\label{ClassK2}
Let $m\geq 3$. A cuspidal $\K^{(2)}(2m)$-module is
\begin{enumerate}
\item[(a)] either isomorphic to a $\K(2m)$-module of first kind, or

\item[(b)]  to the component $V^{\sigma}$ of a
$\K(2m)$-module $V$ of second kind.
\end{enumerate}
\end{thm}

The case $m=2$ is special.
The involution 
$\sigma$ can be lifted to $\widehat{\K(4)}$
and we also have
$$\K^{(2)}(4)=\widehat{\K(4)}^\sigma,$$ 
where we use the same notation $\sigma$ for
the involution of $\K(4)$ and for its lift. 
Though $\K^{(2)}(4)$ has a trivial center, its cuspidal modules inherit the central extension of $\K(4)$, as it is shown in the next theorem. 

\begin{thm} \label{ClassK2(4)}
 A cuspidal $\K^{(2)}(4)$-module is
\begin{enumerate}
\item[(a)] 
either isomorphic to a $\widehat{\K(4)}$-module of first kind, or 
\item[(b)] it is isomorphic to $V^{\sigma}$ for some
$\widehat{\K(4)}$-module $V$ of second kind.
\end{enumerate}
\end{thm}

The proof of both Theorems is  based on a doubling construction. That is there is another cuspidal 
$\K^{(2)}(2m)$-module $\tilde{V}$ such that
 $V\oplus\tilde{V}$ carries a structure of
$(\sigma, \L)$-module, where $\L=\K(2m)$ if $m\geq 3$
or $\L=\widehat{\K(4)}$ otherwise.
The proof is based on Lemma \ref{keygenerate}, which shows 
why the case $\K^{(2)}(4)$ is special: roughly speaking 
the naive definition of the conformal
algebra $\Conf(\K^{(2)}(4))$ is defective. However it is corrected by using $\widehat{\K(4)}$  instead of $\K(4)$.

\subsection{A technical lemma}

Let $m\geq 2$. Set $\L=\widehat{\K(4)}$ if $m=2$ or
$\L=\K(2m)$ otherwise. 

Set $\xi=\zeta_1+\eta_1$. Then 
$$e=\zeta_2\xi, h=\zeta_2\eta_2, f=\xi\eta_2$$
is a $\fsl(2)$-triple in $\L_0$.
The eigenvalues of $\ad(h)$ are $0$ and $\pm 2$, and 
we write
$$\L=\L^{(2)}\oplus \L^{(0)}\oplus \L^{(-2)}$$
the corresponding eigenspace decomposition.

\begin{lemma}\label{diffoperator} For all integers $n,m$, we have
$$[h(t^{n_1}),[h(t^{n_2}),\L^{(0)}]]=0.$$
\end{lemma}

\begin{proof} First assume that $\L=\K(2m)$ with $m\geq 3$. We have
$$\L=\Grass(2m)\otimes\C[t,t^{-1}]\text{ and }
\L^{(0)}=\Grass(2m)^{(0)}\otimes\C[t,t^{-1}].$$
Using the explicit formulas for the bracket from Section
\ref{defK}, we observe that

$$[h(t^{n_2}),a(t{n_3})]=[h,a](t^{n_2+n_3})
\mod h\Grass(2m)\otimes \C[t,t^{-1}],$$
for any $a\in \Grass(2m)$ and $n_3\in\Z$. We deduce that
$$[h(t^{n_2}),\L^{(0)}]\subset h\Grass(2m)\otimes \C[t,t^{-1}]\text{ and }$$
$$[h(t^{n_1}),[h(t^{n_2}),\L^{(0)}]]\subset
h^2\Grass(2m)\otimes \C[t,t^{-1}]=0,$$
which proves the lemma for $m\neq 3$.

For $\L=\widehat{\K(4)}$ we have
$$\L=\widetilde{\Grass(4)}\otimes\C[t,t^{-1}].$$
Consider the decomposition
$\L=\oplus_{0\leq k\leq 4}\,\L_k$, where 
$$\L_k=\widetilde{\Grass(4)}\otimes\C[t,t^{-1}].$$
It is easy to observe that
$$[h(t^{n_2}),\L_k^{(0)}]\subset \L_{k+2}^{(0)},$$
therefore
$$[h(t^{n_1}),[h(t^{n_2}),\L^{(0)}]]=
[h(t^{n_1}),[h(t^{n_2}),\Vir]]=0.$$
\end{proof}

\subsection{The conformal algebra $\Conf(\K^{(2)}(2m))$}

Let $m\geq 2$. Set $\L=\widehat{\K(4)}$ if $m=2$ or
$\L=\K(2m)$ if $m\geq 3$. By definition
$$\K^{(2)}(2m)=\L^\sigma$$
where $\sigma$ is the involution defined by 
$$\sigma(\zeta_1(t^n))=(-1)^n \eta_1(t^n),\,
\sigma(\eta_1(t^n))=(-1)^n \zeta_1(t^n)$$
$$\sigma(\zeta_l(t^n))=(-1)^n \zeta_l(t^n),\,
\sigma(\eta_l(t^n))=(-1)^n \eta_l(t^n)\,\,\text{ for }
2\leq l\leq m.$$

For $a\in (L_0)^\sigma$ set $\tau_a=0$ and for 
$a\in (L_0)^{-\sigma}$  set $\tau_a=1$. 
For $a\in (L_0)^\sigma\cup (L_0)^{-\sigma}$, set
$$a[z]=\sum_{n\in\Z}  a(t^{2n+\tau_a})\, z^{2n+\tau_a}.$$
This $\L^\sigma$-valued  distribution $a[z]$ should not be confused with the $\L$-valued distribution  $a(z)$.
By definition, $\Conf(L^\sigma)$ is the conformal
algebra generated by the distributions 
$a[z]$, for $a\in (L_0)^\sigma\cup (L_0)^{-\sigma}$.

Indeed we have $a[z]=\frac{1}{2}(a(z)+\sigma(a)(-z)$,
thus $a[z]$ is not in $\Conf(\L)$. It is clear that
any two distributions $\phi_1[z]$ and $\phi_2[z]$ 
in $\Conf(L^\sigma)$ are {\it mutually
semi-local}, that is
$$(z_1^2-z_2^2)^N \Bigl[\phi_1[z],\phi_2[z]\Bigr]=0\text{ for some } N>>0.$$ 

We observe that the previously defined
$\fsl(2)$-triple $\{e,h,f\}$ lies in $\L^\sigma$.
For $\phi[z]=\sum_{n\in\Z}\phi_n z^n$, we have
$$[h[z],\phi[z]]_0=\sum_{n\in\Z}[h,\phi_n] z^n.$$
Hence $h$ acts over $\Conf(\L^\sigma)$ and it admits an eigenspace decomposition
$$\Conf(\L^\sigma)=\Conf(\L^\sigma)^{(2)}\oplus\Conf(\L^\sigma)^{(0)}\oplus \Conf(\L^\sigma)^{(-2)}.$$

\begin{lemma}\label{keygenerate} The algebra $\Conf(\L^\sigma)$ is generated by
$$\Conf(\L^\sigma)^{(2)}\oplus \Conf(\L^\sigma)^{(-2)}.$$
\end{lemma}

\begin{proof} Let $\cC\subset \Conf(\L^\sigma)$ be the  
subalgebra generated by 
$$\Conf(\L^\sigma)^{(2)}\oplus \Conf(\L^\sigma)^{(-2)}.$$

There is a basis of the set of odd variables
of the form $\{\zeta_2, \eta_2, \xi_1,\cdots \xi_{2m-2}\}$
such that
\begin{enumerate}
\item[(a)] $\sigma(\xi_1)=-\xi_1$ and $\sigma$ fixes all the other variables,
\item[(b)] $[\zeta_2,\eta_2]=1$, $\xi_i^2=1$ and all other brackets are zero.
\end{enumerate}

First assume $m\geq 3$. The elements 
$\omega$ and $\zeta_2\eta_2\omega$, where $\omega$
runs over $\C[\xi_1,\cdots,\xi_{2m-2}]$, span
$\L_0^{(0)}\simeq \Grass(2m)^{(0)}$. Similarly the elements
$\zeta_2\omega$ and $\eta_2\omega$ generate
respectively $\L_0^{(2)}$ and $\L_0^{(-2)}$.
We now prove that the distributions 
$\omega[z]$ and $\zeta_2\eta_2\omega[z]$ belong to
$\cC$.

We have 
$$[\zeta_2,\eta_2\omega(t^n)]+
[\eta_2,\zeta_2\omega(t^n)]=2\omega(t^n),$$
thus 
$2\omega[z]=[\zeta_2,\eta_2\omega[z]]
+ [\eta_2,\zeta_2\omega[z]]$ which proves that
$\omega[z]$ belongs to $\cC$.

Assume now that $\omega$ has degree $k\neq 2$
and $\sigma(\omega)=\pm(\omega)$. Then
$$[\zeta_2,\eta_2\omega(t^{n+2})]-
[\zeta_2(t^2),\eta_2\omega(t^n)]=
(k-2)\zeta_2\eta_2\omega(t^n),$$
which implies that
$$[\zeta_2,\eta_2\omega[z]]-
[\zeta_2(t^2),\eta_2\omega[z]]=
(k-2)\zeta_2\eta_2\omega[z].$$
Thus $\zeta_2\eta_2\omega[z]$ belongs to 
$\cC$.

When $\omega$ has degree $2$, we can assume that
$\omega=\xi_i\xi_j$ for some $i\neq j$. Choose
an index $k\neq 1, i\text{ or }j$. Obviously, we have
$$[\xi_k, \xi_k\zeta_2\eta_2\omega[z]]=
\zeta_2\eta_2\omega[z].$$
It has been already proved  that
$\xi_k[z]$ and $\xi_k\zeta_2\eta_2\omega[z]$ belong to
$\cC$, thus $\zeta_2\eta_2\omega[z]$ belong to
$\cC$, which proves the claim for
$m\geq 3$.

For $m=2$, we have $\L_0=\widetilde{\Grass}(4)$. 
Let  $\theta\in \widetilde{\Grass}(4)^{(0)}$ be homogenous of degree $k$. The same proof shows that
$\theta[z]$ belongs to $\cC$ if $k\neq 4$.
Recall that $c$ generates $\widetilde{\Grass_4}(4)$. The formula
$$2 [\xi_2,\xi_1\zeta_2\eta_2[z]]=
\Tr(\xi_2\xi_1\zeta_2\eta_2) c[z]$$
shows that $c[z]$ also belongs to $\cC$, which completes the proof.
\end{proof}

\subsection{The double of $V^{(\ell)}$}

Let $V$ be a cuspidal $\K^{(2)}(2m)$-module. Since each $V_n$ is finite dimensional,
the operator $h=\zeta_2\eta_2$ acts diagonally with integer eigenvalues.
 Let $$V=\oplus_{n\in\Z} V^{(n)}$$ 
be the corresponding eigenspace 
decomposition  of $V$ under  $h$.

\begin{lemma}\label{prepa} The operator $h\vert_V$ admits 
only finitely many eigenvalues.

Moreover 
\begin{enumerate}
\item[(a)] its highest eigenvalue $\ell$ is a positive integer and
\item[(b)] the operator $h(t^2)\vert_{V^{(\ell)}}$ is invertible.
\end{enumerate}
\end{lemma}

\begin{proof} Set $\fg=\Vir\ltimes h\otimes\C[t^2,t^{-2}]$.
Since $V$ is a $\fg$-module of growth one,
$h\vert_V$ admits only finitely many eigenvalues by Corollary \ref{finitespec}.

By elementary $\fsl(2)$-theory, the highest eigenvalue is a nonnegative integer $\ell$. By simplicity of the
superalgebra $\L^\sigma$, the action of $\fsl(2)$ is not trivial, therefore $\ell>0$.

Moreover $h(t^2)\vert_{V^{(\ell)}}$ is invertible by Corollary \ref{corh=1}.
\end{proof}

Set $\frac{h(t^2)}{\ell}=t^2$, that is we identify the commutative algebra
$\C[\frac{h(t^2)}{\ell}]$ with the subalgebra
$\C[t^2]\subset \C[t,t^{-1}]$. 
  Write as  $v_1(1),\cdots v_{p}(1)$  a basis of $V_0^{(\ell)}$  and 
$v_{p+1}(t),\cdots v_{p+q}(t)$  a basis of 
$V_1^{(\ell)}$, where $p=\dim\,V_0^{(l)}$ and
$q=\dim\,V_1^{(l)}$. 
Define the shifts $\tau_i$ as
$$\tau_i=0\text{ if } i\leq p\text{ and } \tau_i=1
\text{ if } i>p.$$
Thus the set
$$\{v_i(t^{\tau_i})\mid 1\leq i\leq p+q\}$$
is a basis of $V_0^{(l)}\oplus V_1^{(l)}$. 
At this stage, the expression $v_i(t^{\tau_i})$ are merely
formal notations.

By  Lemma \ref{prepa}, $V^{(\ell)}$ is a $\C[t^2,t^{-2}]$-module, so we can define
$$v_i(t^{2n+\tau_i}):=t^{2n} v_i(t^{\tau_i})\text{ for } n\in\Z.$$

Then the set 
$$\{v_i(t^n)\mid 1\leq i\leq p+q \text{ and }
n\equiv\tau_i \mod 2\}$$
is a basis of $V$.

Let $a\in (L_0)^\sigma\cup (L_0)^{-\sigma}$, where
$(L_0)^{-\sigma}$ is the set of $a\in \L_0$ which are fixed by $-\sigma$. Set $\tau_a=0$ if $a\in (L_0)^\sigma$ and
$\tau_a=1$ otherwise. Then $a(t^n)$ belongs to
$\L^\sigma$ if $n\equiv \tau_a\mod 2$.
For $1\leq i\leq p+q$ set
$$D_i=\{(n,m)\Z^2\mid n\equiv \tau_a \text{ and } m\equiv\tau_i\mod 2\},$$

We define the functions
$$\gamma^{i,j}_a: D_i\to \C,\,(n,m)\mapsto
\gamma^{i,j}_a(n,m)$$
by the formula
$$a(t^n).v_i(t^m)=\sum_j\,\gamma^{i,j}_a(n,m)\,v_j(t^{n+m}).$$ 
In the formula, it is assumed that $n\equiv \tau_a\mod 2$,
$m\equiv \tau_i\mod 2$. Moreover we have 
$\gamma^{i,j}_a(n,m)=0$ if $\tau_a+\tau_i+\tau_j$ is odd.

\begin{lemma}\label{polynomial} For $i,j\in\{1,\cdots,p+q\}$, the function
$$\gamma^{i,j}_a: D_i\to \C$$ 
is a polynomial.
\end{lemma}

\begin{proof} To avoid triviality, we assume that
$\tau_a+\tau_i+\tau_j$ is even.

Consider the formal distribution
$v_i[z]=\sum_{n\in \Z} v_i(t^{2n+\tau_i})\,z^{2n+\tau_i}$.
We observe that 
$A:(\L^\sigma)^{(-2)}$ and $B:=V^{(\ell)}$ are  modules over the Lie algebra $\fg:=\Vir\ltimes h\otimes\C[t^2,t^{-2}]$
and the action provides a morphism of $\fg$-modules
$$\pi: A\otimes B\to C$$
where $C=V^{(\ell-2)}$.

It follows from Theorem \ref{local} that
$v_i[z]$ is {\it mutually semi-local} with any distribution
$\phi[z]$ in $\Conf(\L^\sigma)^{(-2)}$, that is
$$(z_1^2-z_2^2)^N \phi[z_1] v_i[z_2]=0\text{ for some } N>>0.$$
Moreover it is obvious that $\phi[z_1].v_i[z_2]=0$ for 
$\phi[z_1]\in \Conf(\L^\sigma)^{(2)}$.

By Lemma \ref{keygenerate}, the algebra $\Conf(\L^\sigma)$ is
generated by $\Conf(\L^\sigma)^{(\pm2)}$. By Dong Lemma, $v_i[z]$ is mutually semi-local with any 
distribution in $\Conf(\L^\sigma)$. Thus
$$(z_1^2-z_2^2)^{N+1} a[z_1].v_i[z_2]=0\text{ for some } N>>0.$$

We deduce that for a fixed integer $m$ the function
\begin{equation}\label{poly1}
k\mapsto \gamma^{i,j}_a(2k+\tau_a,m-(2k+\tau_a))
\end{equation}
is a polynomial of degree $\leq N$ in the variable $k$.

Let $n\equiv\tau_a\mod 2$ be an arbitrary integer.
By Lemma \ref{diffoperator}, we have
$$[t^2,[t^2,a(t^n)]]=0$$
hence $a(t^n)$ acts on the $\C[t^2,t^{-2}]$-module 
$V^{(\ell)}$ as a differential operator of order $\leq 1$.
Thus the function
\begin{equation}\label{poly2}
k\mapsto \gamma^{i,j}_a(n,2k+\tau_i)
\end{equation}
is a polynomial of degree $\leq 1$ in the variable $k$.

Thus by Formulas \ref{poly1} and \ref{poly2} the function
$\gamma^{i,j}_a$ is a polynomial on its domain of definition $D_i$.
\end{proof}

Set $W^{(\ell)}=\C[t,t^-1]\otimes_{\C[t^2,t^{-2}]}\,V^{(\ell)}$.
The space $W^{(\ell)}$ has a natural grading. As a
$\C[t^2,t^{-2}]$-module it decomposes as 
$$W^{(\ell)}=V^{(\ell)}\oplus \tilde{V}^{(\ell)},$$
where $\tilde{V}^{(\ell)}=t.V^{(\ell)}$.

\begin{lemma} The space
$W^{(\ell)}$ carries a structure
of $(\sigma, \L^{(0)})$-module such that
\begin{enumerate}
\item[(a)] $(W^{(\ell)}))^\sigma=V^{(\ell)}$ as an 
$(\L^{(0)})^\sigma$-module, and 
\item[(b)]
$(W^{(\ell)})^{-\sigma}=\tilde{V}^{(\ell)}$. 
\end{enumerate}
\end{lemma}

\begin{proof} The action of $\sigma$ is easy to define: it acts as $1$ on $V^{(\ell)}$ and as $-1$ on $\tilde{V}^{(\ell)}$.

By definition, $W^{(\ell)}$ has basis
$$\{v_i(t^n)\mid 1\leq i\leq p+q\},$$
where $v_i(t^n)=t^{n-\tau_i} v_i(\tau_i)$.

Let $a\in \in (L_0)^\sigma\cup (L_0)^{-\sigma}$. 
By Lemma \ref{polynomial} there are polynomial functions
$$\tilde{\gamma}^{i,j}_a: \Z^2\to \C$$ 
which extend the function $\gamma^{i,j}_a: D_i\to \C$.
Then we define the action of $a(t^n)$ on $W^{(\ell)}$
by 
$$a(t^n).v_i(t^m)=\sum_j\,\tilde{\gamma}^{i,j}_a(n,m)\,v_j(t^{n+m}).$$ 
By the principle of extension of polynomial identies,
the formula provides a structure of $\L^{(0)}$-module on the space $W^{(\ell)}$.
\end{proof}

\subsection{The double of $V$}

Let $\L$ be as before. Let $P\subset \L$
and  be the parabolic 
subalgebra 
$$P=C_\L(h) \oplus \L^{(2)}.$$
 We consider the $C_\L(h)$-module $W^{(\ell)}$ as a $P$-module with a trivial action of $\L^{(2)}$. The subspaces $V^{(\ell)}, \tilde{V}^{(\ell)}\subset W^{(\ell)}$ are $P^\sigma$-module, where $P^\sigma=P\cap \L^{\sigma}$.
 
 The generalized Verma modules 
 $$M=\Ind_P^\L\,W^{(\ell)}\text{, }
N=\Ind_{P^\sigma}^{\L^\sigma}\,V^{(\ell)}
\text{ and }
\tilde{N}=\Ind_{P^\sigma}^{\L^\sigma}\,\tilde{V}^{(\ell)}$$
admit  eigenspace decompositions
$$M=\oplus_{k\geq 0}\, M^{(\ell-2k)}\text{, }
N=\oplus_{k\geq 0}\, N^{(\ell-2k)}\text{ and }
\tilde{N}=\oplus_{k\geq 0}\, \tilde{N}^{(\ell-2k)}.$$
We have $M^{(\ell)}=W^{(\ell)}$, $N^{(\ell)}=
V^{(\ell)}$ and $\tilde{N}^{(\ell)}=
\tilde{V}^{(\ell)}$. Set 
$$W:=M/Z(M),$$ 
where
$Z(M)$ is the largest
$\L$-module contained in
$\oplus_{k> 0}\, M^{(\ell-2k)}.$

Similarly, let $Z(N)$ and $Z(\tilde{N})$ be the largest
$\L^\sigma$-modules contained in
$$\oplus_{k> 0}\, N^{(\ell-2k)}\text{ and }
\oplus_{k> 0}\, \tilde{N}^{(\ell-2k)}.$$
Obviously we have $V=N/Z(N)$. Set
$$\tilde{V}=\tilde{N}/Z(\tilde{N}).$$

Fix an integer $k\geq 1$ and elements
$a_1,\cdots,a_k\in \L^{(+2)}$,
$b_1,\cdots,b_k\in \L^{(-2)}$.
For any  
  $(m_0,\cdots,m_{2k})\in\Z^{2k+1}$, we have

$$a_1(t^{m_1})\cdots a_k(t^{m_k})b_1(t^{m_{k+1}})\cdots b_k(t^{m_{2k}}) W_{m_0}\subset W_{m_0+\cdots+m_{2k}}.$$
So there are structure constants 
$\theta^{i,j}(m_0,\cdots,m_{2k})$ for $i,j\leq p+q$
such that
$$a_1(t^{m_1})\cdots a_k(t^{m_dk})b_1(t^{m_{k+1}})\cdots b_k(t^{m_{2k}}) v_i(t^{m_0})$$
$$=\sum_{j=1}^{p+q}\,
\theta^{i,j}(m_0,\cdots,m_{2k}) v_j(t^{m_0+\cdots+m_{2k}}).$$

\begin{lemma}\label{poly3} The structure constants 
$\theta^{i,j}(m_0,\cdots,m_{2k})$ are polynomials in
$m_0,\cdots,m_{2k}\in\Z^{2k+1}$.

Moreover their  degrees are bounded by a constant 
$N(k)$ which only depends on $k$.
\end{lemma}

\begin{proof} 
Set $I=U(\L)\L^{(2)}\cap U(\L)^{(0)}$. By Poincar\'e-Birkhoff-Witt Theorem,
we have 
$$C_{U(\L)}(h)=U(\L^{(0)})\oplus I.$$
The Harish-Chandra homomorphism 
$$\pi:C_{U(\L)}(h)\to U(\L^{(0)})$$
is the projection on the first factor.
Using the explicit formulas of section \ref{defK}, it is clear that the mutidistribution
$$\pi \Bigl(a_1(z_1)\cdots a_k(z_k)b_1(z_{k+1})\cdots b_k(z_{2k})\Bigr)$$
lies in the algebra of multidistributions of 
$\L^{(0)}$. Thus the structure constants $\theta^{i,j}$ are polynomial by Lemma \ref{poly3}.

Moreover the multidistributions
$$\sum_{m\in \Z^{2k+1}}a_1(t^{m_1})\cdots a_k(t^{m_k})b_1(t^{m_{k+1}})\cdots b_k(t^{m_{2k}}) v_i(t^{m_0})\,\, z_0^{m_0}z_1^{m_1}\cdots 
z_{2k}^{m_{2k}}$$
generate a finite dimension space, thus the degrees 
of the polynomials
$$\theta^{i,j}(m_0,\cdots,m_{2k})$$ 
are uniformly bounded.
\end{proof}

\begin{lemma}\label{h0h0} We have
\begin{enumerate}
\item[(a)] $W^{(\ell-2k)}=0$ if $k>\ell$
\item[(b)] the dimensions $\dim\, W^{(\ell-2k)}_n$ are finite and uniformly bounded,
\item[(c)] $H_0((\L^\sigma)^{(-2)}, W)\simeq W^{(\ell)}$,
\item[(d)] $H^0((\L^\sigma)^{(2)}, W)=W^{(\ell)}$.
\end{enumerate}
\end{lemma}

\begin{proof} We will repeatedly use that 
$$\Bigl(\L^{(2)}\Bigr)^k.v\neq 0$$
for any nonzero vector $v\in W^{(\ell-2k)}$.
Choose bases $B_\pm$ of $\L^{(\pm2)}$ which consist of elements in $\L^\sigma\cup \L^{-\sigma}$. For 
$b\in B_\pm$ we write $\tau(b)=0$ if $\sigma(b)=b$ and
$\tau(b)=1$ otherwise. Also set $\tau(i)=0$ if $i\leq p$ and
$\tau(i)=1$ otherwise.

Since $V$ is cuspidal, we have $V^{(\ell-2k)}=0$ if 
$k>\ell$. Therefore for $k>\ell$, all polynomials 
$\theta^{i,j}(m_0,\cdots,m_{2k})$ vanish on the Zariski-dense subset
$$\{(m_0,\cdots,m_{2k})\mid m_0\equiv \tau(i),
m_1\equiv\tau(a_1),\cdots,
m_{2k}\equiv \tau(b_k)\mod 2\}$$
so all $\theta^{i,j}$ are identically  zero. Thus
$$a_1(t^{m_1})\cdots a_k(t^{m_k})b_1(t^{m_{k+1}})\cdots b_k(t^{m_{2k}}) W^{(\ell)}=0$$
for all $a_i\in\L^{(2)}$ and $b_i\in\L^{(-2)}$. It follows that
$$W^{(\ell-2k)}=(\L^{(-2)})^k  W^{(\ell)}=0,$$
which proves Assertion (a).

For a $k$-uples ${\bf a}:=(a_1,\cdots,a_k)\in (B_+)^k$,
${\bf b}:=(b_1,\cdots,b_k)\in (B_-)^k$ and an integer $i\leq p+q$,
consider the multidistributions
$$X_{\bf a}(z_{k+1},\cdots,z_{2k})
=
\sum_{m\in \Z^{k}} a_1(t^{m_{1}})\cdots a_k(t^{m_{k}}) \,\, z_{k+1}^{m_{1}}\cdots 
z_{2k}^{m_{k}}\text{ and }$$
 $$Y_{\bf b}^i(z_{0},\cdots,z_{k})
=\sum_{m\in \Z^{k+1}} b_1(t^{m_{1}})\cdots b_k(t^{m_{k}}) v_i(t^{m_0})\,\, z_0^{m_0}z_{1}^{m_1}\cdots 
z_{k}^{m_{k}}.$$

Assume given ${\bf b}$ and $i$. Since the structure constants
$\theta^{i,j}$ are polynomials,
there exist an integer $N$ such that
$$(z_i-z_j)^N  X_{\bf a}(z_{k+1},\cdots,z_{2k})
Y_{\bf b}^i(z_0,\cdots,z_{k})=0$$
for any ${\bf a}\in B_+^k$ and any 
$0\leq i,j\leq k$. Thus
$$\Bigl(\L^{(2)}\Bigr)^k
(z_i-z_j)^N Y_{\bf b}^i(z_0,\cdots,z_{k})=0$$
which implies that the multidistribution 
$Y_{\bf b}^i(z_0,\cdots,z_{k})$ is local.
Since $W^{(\ell-2k)}$ is generated by the modes of a finite set of local multidistributions, it follows from
Corollary \ref{mode} that $W^{(\ell-2k)}$ is a 
$\Z$-graded space of growth one, which proves Assertion (b).

Let $N$ be an arbitrary integer.
The set of all $(k+1)$-uples
$(m_0,\cdots,m_{k})$ of integers such that
$$m_0+\cdots+
m_{k} =N\text{ and }
m_{1}\equiv \tau(b_{1}),\cdots, m_{k}\equiv \tau(b_{k})$$
is Zariski-dense in the hyperplane
$$\{(x_0,\cdots, x_{k})\in\C^{1+k}\mid x_0+\cdots
+x_{k} =N\}.$$
It follows from Corollary \ref{mode}(b) that the space of modes of  $Y_{\bf b}^i(z_0,\cdots,z_{k})$ is generated by the elements
$$b_1(t^{m_1})\cdots b_k(t^{m_{k}}) v_i(t^{m_0})$$
where $m_{j}\equiv\tau({b_j})\mod 2$ for $1\leq j\leq k$.
Thus $W$ is generated by $W^{(\ell)}$ as a 
$(\L^\sigma)^{(-2)}$-module, which proves Assertion (c).

Replacing $\L^{(-2)}$ by  $\L^{(2)}$ and  considering the dual graded of $W^*$ instead of $W$, we deduce 
Assertion (d) from Assertion (c).
\end{proof}

\begin{lemma}\label{double} The $\L^\sigma$-module $\tilde{V}$ is cuspidal and
$$W=V\oplus\tilde{V}.$$
\end{lemma}

\begin{proof} The $\L^\sigma$-submodule generated by
$V^{(\ell)}$ is a quotient of
$N=\Ind_{P^\sigma}^{\L^\sigma}\,V^{(\ell)}$. By Lemma \ref{h0h0} the
image of $Z(N)$ in $W$ is zero, hence 
$U(\L^\sigma) V^{(\ell)}\simeq V$. Similarly 
we have
$$U(\L^\sigma) \tilde{V}^{(\ell)}\simeq \tilde{V},$$
which proves the lemma.
\end{proof}

\subsection{Proof of Theorems \ref{ClassK2} and \ref{ClassK2(4)}}

\noindent{\it Proof of Theorem  \ref{ClassK2} and \ref{ClassK2(4)}.}
Set $\L=\K(2m)$ if $m\geq 3$ and $\L=\widehat{\K(4)}$
otherwise.

Let $V$ be a cuspidal $\K^{(2)}(2m)$-module.
By Lemma \ref{double},  there  exists another $\K^{(2)}(2m)$-module $\tilde{V}$ such that
\begin{enumerate} 
\item[(a)] $W:=V\oplus\tilde{V}$ carries a structure of
$(\sigma, \L)$-module extending the given 
$\K^{(2)}(2m)$-action, and
\item[(b)] the involution $\sigma$ acts as $1$ on $V$ and $-1$ on $\tilde{V}$.
\end{enumerate}

If $M$ is a simple $\L$-module,
then $W$ is a cuspidal $\L$-module of second type and $V=W^\sigma$. Thus $V$ satisfies Assertion (b).

Otherwise $W$ is a $\L$-module of lenght $2$. Let
$W'\subset W$ be a simple $\L$-submodule. We have
$$\L^{-\sigma}\tilde{V}\subset  V,$$
where  
$$\L^{-\sigma}=\{g\in\L\mid \sigma(g)=-g\}.$$
We deduce that $W'\cap \tilde{V}=0$. Thus the first factor projection
$$\pi:W\to V$$
induces a isomorphism of $\K^{(2)}(2m)$-modules
$W'\simeq V$, which shows that $V$ is isomorphic to
the  $\L$-module $W'$. Since $\L$-modules of second type are not irreducible as a $\K^{(2)}(2m)$-module,
$W'$ is a module of first type. Thus $V$ satisfies Assertion (a). \qed

\end{document}